\documentclass[11pt,twoside]{amsbook}

\usepackage[margin=1in]{geometry}
\usepackage{mathtools,amssymb}
\usepackage{amsfonts}
\usepackage{amsmath}
\usepackage{amssymb}
\usepackage{mathrsfs}
\usepackage{multicol}
\usepackage{indentfirst}
\usepackage{graphics}
\usepackage{stmaryrd}
\usepackage{cite}
\usepackage{color}
\usepackage{caption}
\usepackage{esint}
\usepackage{graphicx}
\usepackage{epstopdf}
\usepackage{subfigure}
\usepackage{bm}
\usepackage{enumerate}
\usepackage{ulem}

\theoremstyle{plain}
\newtheorem{theorem}{Theorem}[section]
\newtheorem{lemma}[theorem]{Lemma}
\newtheorem{corollary}[theorem]{Corollary}
\newtheorem{example}[theorem]{Example}

\newtheorem{remark}[theorem]{Remark}

\newtheorem{proposition}[theorem]{Proposition}
\newtheorem{question}[theorem]{Question}
\numberwithin{equation}{section}
\allowdisplaybreaks
\renewcommand{\thesection}{\thechapter.\arabic{section}}

\usepackage{chngcntr}

\counterwithout{equation}{section}
\counterwithin{equation}{chapter}

\counterwithout{theorem}{section}
\counterwithin{theorem}{chapter}

\title{Qualitative Properties of the Principal Floquet Bundle and Their Applications}

 \author{Yun Li\footnote{Partially supported by the Postdoctoral Fellowship Program of China Postdoctoral Science Foundation under Grant Number  GZB20250712, by the China Postdoctoral Science Foundation under Grant Number 2026M793368, and by National Natural Science Foundation of China (12601331).}}
 \address{School of Mathematics and Statistics, Lanzhou University, Lanzhou 730000, China}
 \email{liyun17@lzu.edu.cn}

 \author{Yuan Lou\footnote{Partially supported by National Natural Science Foundation of China (12261160366, 12250710674, 12426601).}}
 \address{School of Mathematical Science, CMA-Shanghai, Shanghai Jiao Tong University, Shanghai, 200240, China}
 \email{yuanlou@sjtu.edu.cn}

 \author{Zhi-Cheng Wang\footnote{Corresponding author. Partially supported by National Natural Science Foundation of China (12471164, 12071193).}}
 \address{School of Mathematics and Statistics, Lanzhou University, Lanzhou 730000, China}
 \email{wangzhch@lzu.edu.cn}

\begin{document}
\frontmatter
\maketitle
\tableofcontents

\chapter*{Abstract}
Floquet bundle theory serves as a natural extension of the eigenvalue theory for elliptic and periodic parabolic operators to the spectral theory of non-autonomous, non-periodic parabolic operators, which provides a useful spectral tool for studying threshold dynamics in more general non-autonomous systems. However, current understanding and applications of the normalized principal Floquet bundle remain rather limited compared with the eigenvalue theory for elliptic or periodic parabolic problems. In this work, we first extend several fundamental properties of the principal eigenvalue to the normalized principal Floquet bundle of non-autonomous, non-periodic parabolic operators, and obtain a series of parallel results. Furthermore, we focus on the asymptotic behavior of the principal Floquet exponent with respect to the generalized frequency and diffusion rate under various typical parameter limits (including both independent and coupled parameter regimes) for the operator under homogeneous Neumann boundary conditions. The main tools are a parabolic comparison principle developed in this work for long-time limit processes, which is used to derive bounds for the principal Floquet exponent via suitably constructed sub- and super-solutions, supplemented by Harnack's inequality and long-time averaging techniques. Finally, we employ the normalized principal Floquet bundle to define two critical numbers (as extensions of the basic reproduction number) for a non-autonomous spatially diffusive SIS epidemic model, examine the extinction and weak persistence of the disease, and investigate the limiting behavior of the basic critical numbers with respect to the generalized frequency and diffusion rate.

\medskip
\noindent \textbf{Keywords}: Principal Floquet bundle; Qualitative property; Asymptotic behavior; SIS epidemic model.

\noindent \textbf{MSC (2020)}: Primary: 35P20; 35P05; 35P15; 35K57; 35Q92; 92D40.

\mainmatter

\chapter{Introduction}\label{chp1}
The spectral theory of linear evolution problems is often required as a crucial tool to consider the stability or invasion of ecological systems modeled by nonlinear equations. This theory is well-developed for several classes of evolution equations, such as time-independent problems \cite{Berestycki1994The, Hamel2011Rearrangement, Chen2012Effects} and temporally periodic parabolic problems \cite{Peng2015Effects, Liu2021AsymptoticsI, Liu2021AsymptoticsII}. The well-known Krein-Rutman theorem guarantees that the elliptic or time-periodic parabolic eigenvalue problem admits a unique simple principal eigenvalue. In studying the dynamics of population model, the principal eigenvalue together with the corresponding principal eigenfunction is frequently used to determine the locally linear stability or instability of the trivial solution as well as other equilibria. However, without independence or periodicity assumptions on $t$ of the coefficients, or when the associated operators lack compactness, the Krein-Rutman theorem is no longer applicable. Consequently, considerable effort has been devoted to extending Krein-Rutman type results to general time-dependent parabolic problems. We briefly review several relevant contributions.

Shen and Vickers studied the principal Lyapunov exponent $\lambda_L:=\limsup_{t-s\to\infty}\frac{\ln\| \Phi(t,s)\|}{t-s}$ of the evolution operator $\Phi(t,s)$ generated by a non-autonomous evolution equation \cite{Shen2007Spectral}. They also introduced the concepts of principal spectrum and forward uniform persistence to analyze nonlinear parabolic boundary value problems of Kolmogorov type \cite{Mierczynski2011Persistence}. Recently, Berestycki et al. \cite{Berestycki2025Generalized} introduced new generalized principal eigenvalues for linear parabolic operators with heterogeneous coefficients in space and time, defined on bounded spatial domains and unbounded time intervals. These eigenvalues were shown to play a key role in understanding the long-time behavior and entire solutions of heterogeneous Fisher-KPP type equations. More generally, the principal eigenvalue, principal Lyapunov exponent, principal spectrum and generalized principal eigenvalue are all the quantities which play the role of the Floquet exponent in Floquet theory, and the Floquet bundles extend the concepts of eigenvalues and eigenfunctions of elliptic and time-periodic parabolic problems in a natural way. Under mild regularity assumptions on the coefficients and domain, H\'{u}ska and Pol\'{a}\v{c}ik \cite{Huska2004The} established the existence of a principal Floquet bundle that is exponentially separated from a complementary invariant bundle. The properties they obtained, such as boundedness, decomposition, exponential separation, and positivity of the principal Floquet bundle (detailed in Proposition \ref{proposition3.3}), are natural extensions of Krein-Rutman type results to non-periodic parabolic problems. These properties have served as key ingredients in proving the uniqueness of positive entire solutions for parabolic equations on unbounded time intervals. In their recent work \cite{Berestycki2025Generalized}, Berestycki et al. also clarified the relations among the Lyapunov exponent, generalized principal eigenvalues, and Floquet bundles. For further details on these spectral-theoretic quantities, we refer to \cite{Mierczynski2008Time, Mierczynski2008Spectral, Mierczynski2013Spectral,Chow1993Floquet, Chow1995Floquet, Mierczynski1995Globally, Polacik1993Exponential, Huska2006Harnack, Huska2006Exponential, Shen2007Almost} and the references therein.

In the present paper, we would like to investigate some qualitative properties of the Floquet bundle, which are natural generalizations to the principal eigenvalue of elliptic or time-periodic parabolic eigenvalue problems. Let $\Omega\subset\mathbb{R}^N$, $N\ge 1$ be a bounded smooth domain, $c(x,t)$ be a function in $C^{\delta,\delta/2}(\overline{\Omega}\times\mathbb{R})$, $\delta\in(0,1)$, and $A(x)=(a_{ij}(x))_{N\times N}\in C^{2+\delta}(\overline{\Omega};\mathbb{R}^{N^2})$ be a symmetric matrix-valued function satisfying uniform ellipticity condition, namely, there exists a positive constant $\Lambda$ such that
 \[\frac{1}{\Lambda}|\xi|^2\le \xi\cdot (A(x)\xi)\le \Lambda |\xi|^2,\quad~\forall~x\in\overline{\Omega},~\xi\in\mathbb{R}^N.\]
The pair $(H,\varphi)\in C(\mathbb{R})\times (C^{2,1}(\Omega\times\mathbb{R})\cap C^{1,1}(\overline{\Omega}\times\mathbb{R}))$ is called the {\bf normalized principal Floquet bundle} if it satisfies
\begin{equation}\label{eq1.1}
  \left\{
  \begin{aligned}
    &\omega\partial_t\varphi-d\operatorname{div}\left(A(x)\nabla\varphi\right)+c(x,t)\varphi=H(t)\varphi,&&(x,t)\in\Omega\times\mathbb{R},\\
    &\mathbf{n}\cdot (A(x)\nabla\varphi)=0,&&(x,t)\in\partial\Omega\times\mathbb{R},\\
    &\varphi(x,t)>0,&&(x,t)\in\overline{\Omega}\times\mathbb{R},\\
    &\int_{\Omega}\varphi(x,t)\,\mathrm{d}x=1,&&t\in\mathbb{R}
  \end{aligned}
  \right.
\end{equation}
and $\psi\in C^{2,1}(\Omega\times\mathbb{R})\cap C^{1,1}(\overline{\Omega}\times\mathbb{R})$ is called the {\bf normalized adjoint principal Floquet bundle} if it satisfies
\begin{equation}\label{eq1.2}
  \left\{
  \begin{aligned}
    &-\omega\partial_t\psi-d\operatorname{div}\left(A(x)\nabla\psi\right)+c(x,t)\psi=H(t)\psi,&&(x,t)\in\Omega\times\mathbb{R},\\
    &\mathbf{n}\cdot(A(x)\nabla\psi)=0,&&(x,t)\in\partial\Omega\times\mathbb{R},\\
    &\psi(x,t)>0,&&(x,t)\in\overline{\Omega}\times\mathbb{R},\\
    &\int_{\Omega}\varphi(x,t)\psi(x,t)\,\mathrm{d}x=1,&&t\in\mathbb{R},
  \end{aligned}
  \right.
\end{equation}
where
 \[\operatorname{div}\left(A(x)\nabla\varphi\right)=\sum^{N}_{i,j=1}\partial_{x_i}(a_{ij}(x) \partial_{x_j}\varphi),\qquad (x,t)\in\Omega\times\mathbb{R},\]
$\mathbf{n}=(n_1,n_2,\cdots,n_N)(x)$ is outward unit normal vector at $x\in\partial\Omega$, and
 \[\mathbf{n}\cdot(A(x)\nabla\varphi)=\sum^{N}_{i,j=1}n_i(x)a_{ij}(x)\partial_{x_j}\varphi,\qquad (x,t)\in\partial\Omega\times\mathbb{R}.\]
We also refer to the triplet $(H,\varphi,\psi)$ as the normalized principal Floquet bundle when no confusion can arise. It has been shown in \cite{Cantrell2021On,Lam2022Introduction} that $(H, \varphi,\psi)\in C^{\delta/2}(\mathbb{R})\times C^{2+\delta,1+\delta/2}(\overline{\Omega}\times\mathbb{R}) \times C^{2+\delta,1+\delta/2}(\overline{\Omega}\times\mathbb{R})$ exists uniquely. In equations \eqref{eq1.1} and \eqref{eq1.2}, the constant $\omega>0$ represents the rate of environmental variation (how fast the environment changes), and is referred to as the \textbf{generalized frequency} (hereafter simply \textbf{frequency}). The constant $d>0$ is the diffusion rate, and the function $c$ is the potential function. To indicate the dependence on parameters $\omega$, $d$ or $c$, we also represent by $(H(\cdot;\omega,d),\varphi(\cdot,\cdot;\omega,d),\psi(\cdot,\cdot;\omega,d))$, $(H(\cdot;c),\varphi(\cdot,\cdot;c),\psi(\cdot,\cdot;c))$, etc. the normalized principal Floquet bundle.

As mentioned above, H\'{u}ska, Pol\'{a}\v{c}ik et al. (see, \cite{Polacik2005On, Huska2004The,Huska2006Harnack,Huska2006Exponential}) weakened the smoothness assumption on coefficients and obtained the continuous dependence of the principal Floquet bundle on coefficients. Recently, Cantrell and Lam \cite{Cantrell2021On} obtained the smooth dependence of the principal Floquet bundle on coefficients by assuming that the coefficients belong to H\"{o}lder space $C^{\delta,\delta/2}(\overline{\Omega}\times\mathbb{R})$. Later on, Lam and Lou \cite{Lam2024The} studied the asymptotic behavior of the normalized principal Floquet bundle as the diffusion rate $d$ converges to infinity. One may refer to \cite{Huska2007Harnack,Huska2008Exponential,Lam2023A} and literature therein for more theoretical results and applications about the theory of the normalized principal Floquet bundle.

A special situation is that the function $c$ is independent of spatial argument $x$, which means that $c$ depends on $t$ only. This situation is trivial and one easily verifies that $(c(\cdot),1/|\Omega|)$ is the normalized principal Floquet bundle; see Lemma \ref{lemma3.5} for details. The environment is said to be spatially heterogeneous in an average sense in time if
\begin{equation}\label{H1}
\tag{H1}
\liminf\limits_{T\to\infty}\fint^{T}_{0}\fint_{\Omega}\left|c(x,s)- \fint_{\Omega}c(x,s)\,\mathrm{d}x\right|^2\,\mathrm{d}x\mathrm{d}s>0.
\end{equation}
Hereafter, the notations ``$\fint^{T}_{0}$" and ``$\fint_\Omega$" represent the integral averages of the integrand with respect to the time argument $t$ over $[0,T]$ and with respect to the space argument $x$ over $\Omega$.

Define
\[\underline{\lambda}_{H}(\omega,d)=\liminf_{T\to\infty}\fint^{T}_{0}H(s;\omega,d)\,\mathrm{d}s\quad\text{and}\quad \overline{\lambda}_{H}(\omega,d)=\limsup_{T\to\infty}\fint^{T}_{0}H(s;\omega,d)\,\mathrm{d}s.\]
To highlight the dependence of these two limits on the potential function $c$, we shall write them as $\underline{\lambda}_{H}(c)$ and $\overline{\lambda}_{H}(c)$. For autonomous or time-periodic problems,
\[\underline{\lambda}_{H}(\omega,d)=\overline{\lambda}_{H}(\omega,d)= \text{the principal eigenvalue of the associated problem}.\]
The dependence of this principal eigenvalue on various parameters is well understood, with a vast literature available; see \cite{Hess1991Periodic,Hutson2001The,Lou2006Evolution, Liu2019Monotonicity,Bai2020Asymptotic,Liu2022Classifying,Bai2023Dynamics,Li2026Effects,Li2026Optimization} and references therein. In contrast, the more general and practically relevant nonautonomous setting has received considerably less attention. Note that the identity $\underline{\lambda}_{H}(\omega,d)=\overline{\lambda}_{H}(\omega,d)$ typically breaks down in this general framework. Yet, these two quantities remain fundamentally linked to the system's dynamics, acting as thresholds that govern the long-term behavior of the associated nonlinear evolution; see \cite{Cantrell2021On,Lam2023A,Lam2024The} and references therein. The dynamical behavior within the interval $[\underline{\lambda}_{H}(\omega,d),\overline{\lambda}_{H}(\omega,d)]$ remains elusive yet highly intriguing. Specifically,

\medskip

\noindent\textbf{Open problem.} {\it How do these thresholds translate into biological interpretations in the context of biomathematics? Moreover, what biological implications does the interval length $L(\omega,d) = \overline{\lambda}_{H}(\omega,d)-\underline{\lambda}_{H}(\omega,d)$ hold?}

\medskip

\noindent While these issues have received little attention in the autonomous or time-periodic framework, we are convinced that they are crucial and merit particular emphasis in the study of nonautonomous, non-periodic problems. The present work is devoted to investigating how $\underline{\lambda}_{H}(\omega,d)$, $\overline{\lambda}_{H}(\omega,d)$, and the intermediate values therein depend on the parameters $(\omega,d)$. This motivates the following natural and pressing questions.

\medskip

\noindent\textbf{Question.} {\it What is the impact of the parameters $(\omega,d)$ on $\underline{\lambda}_{H}(\omega,d)$ and $\overline{\lambda}_{H}(\omega,d)$, and how does this impact translate into the system's dynamical behavior in terms of frequency and diffusion effects? How does the asymptotic behavior of the pair $(\underline{\lambda}_{H}(\omega,d),\overline{\lambda}_{H}(\omega,d))$ compare with that of the principal eigenvalue?}

\medskip

Suppose $\delta\in(0,1)$ throughout this paper. Let $\mathcal{C}$ be the set defined as
\[\mathcal{C}:=\left\{c\in C^{\delta,\delta/2}(\overline{\Omega}\times\mathbb{R})~\big|~\|c(\cdot,\cdot)\|_\infty<\infty,~\sup_{t\in\mathbb{R}} [c(\cdot,t)]_{\delta;\Omega}<\infty\right\},\]
where $\|\cdot\|_\infty$ denotes the supremum norm and $[\cdot]_{\delta;\Omega}$ denotes the H\"{o}lder seminorm. Indeed, since one can replace $c$ by $c-\fint_{\Omega}c(x,\cdot)\,\mathrm{d}x$ and $H$ by $H-\fint_{\Omega}c(x,\cdot)\,\mathrm{d}x$ simultaneously, it is sufficient to assume that $\|c(\cdot,\cdot)-\fint_{\Omega}c(x,\cdot)\,\mathrm{d}x\|_{\infty}<\infty$ in $\mathcal{C}$, and our conclusions about the normalized principal Floquet bundle $H$ are valid for $\tilde{H}:=H-\fint_{\Omega}c(x,\cdot)\,\mathrm{d}x$. We prepare some properties of $c\in\mathcal{C}$.
\begin{proposition}\label{proposition1.1}
  Let $c\in\mathcal{C}$ be a function and $\{T_n\}_{n\in\mathbb{N}^+}$ be a sequence satisfying $T_n\to\infty$ as $n\to\infty$, and define $\hat{c}_n(x):=\fint^{T_n}_{0}c(x,s)\,\mathrm{d}s$, $x\in\overline{\Omega}$, $n\in\mathbb{N}^+$.
  Then the following assertions hold.
  \begin{enumerate}[{\rm (i)}]
    \item $\{\hat{c}_n\}_{n\in\mathbb{N}^+}$ admits a convergent subsequence $\{\hat{c}_{n_j}\}_{j\in\mathbb{N}^+}$ in the sense that $\hat{c}_{n_j}(\cdot)\to \hat{c}(\cdot)$ in $C^{\delta'}(\overline{\Omega})$ as $j\to\infty$ for some $\delta'\in(0,\delta)$.
    \item The functions $\hat{c}^\infty(x):=\limsup_{n\to\infty}\hat{c}_n(x)$ and $\hat{c}_\infty(x):=\liminf_{n\to\infty}\hat{c}_n(x)$ are continuous on $\overline{\Omega}$. {\rm (}Note that the functions $\hat{c}^\infty(x)$ and $\hat{c}_\infty(x)$ are defined pointwise in $x\in\overline{\Omega}$ here.{\rm )}
    \item Similarly, the limiting functions $\limsup\limits_{T\to\infty}\fint^{T}_{0}c(x,s)\,\mathrm{d}s$ and $\liminf\limits_{T\to\infty}\fint^{T}_{0}c(x,s)\,\mathrm{d}s$ are also continuous on $\overline{\Omega}$.
  \end{enumerate}
\end{proposition}

In the discussion of this paper, it is necessary to clearly distinguish between two modes of convergence for the time-averaged function $\hat{c}_T(\cdot)=\fint^{T}_{0}c(\cdot,s)\,\mathrm{d}s$ as $T\to\infty$ (or along a sequence $\{T_n\}_{n\in\mathbb{N}^+}$): the first is pointwise convergence on $\overline{\Omega}$, e.g., $\liminf_{T\to\infty}\hat{c}_T(x)$ for $x\in\overline{\Omega}$; the second is the convergence of the family $\{\hat{c}_{T_n}\}_{n\in\mathbb{N}^+}$ as a whole to a certain function. Although the specific meaning will be made clear at each occurrence in the text, we emphasize this distinction here and hope to draw the reader's attention to it.

\begin{remark}\label{remark1.2}{\rm
To better understand the concept of the normalized principal Floquet bundle, we make some comparisons between the normalized principal Floquet bundle and the principal eigenvalue and its associated eigenfunction. If the function $c$ is independent of $t$, then $H(t)$ is equal to the principal eigenvalue of the corresponding elliptic problem for all $t\in\mathbb{R}$. This fact is not true for the time-periodic case. We take $A=\mathrm{diag}(1,1,\cdots,1)$ and $\omega=1$ for simplicity. Without loss of generality, let $c(\cdot,t)$ be a $1$-periodic function in $t$ and $(\lambda,w)$ be the principal eigenpair of $\partial_t-d\Delta+c$ in $\Omega$ subject to homogeneous Neumann boundary condition. The normalized principal Floquet bundle is (see the proof of Lemma \ref{lemma3.5} for the computational details)
\[H(t)=\lambda-\frac{\mathrm{d}}{\mathrm{d}t}\ln\int_{\Omega}w(x,t)\,\mathrm{d}x,\qquad t\in\mathbb{R}.\]
This indicates that $H(t)$ is not the direct extension of the principal eigenvalue in the pointwise sense in $t$. However, one can easily check that
 \[\lambda=\lim\limits_{T\to\infty}\fint^{T}_{0}H(s)\,\mathrm{d}s.\]
This motivates us to study the qualitative properties of normalized principal Floquet bundle via the time average in the sense that $\fint^{T}_{0}H(s)\,\mathrm{d}s$ as $T\to\infty$ for non-periodic case. Any subsequential limit of $\fint^{T}_{0}H(s)\,\mathrm{d}s$ along a sequence $T_n\to\infty$ is called a \textbf{normalized principal Floquet exponent}. Subsequently, for brevity, we shall call it simply the \textbf{principal Floquet exponent}. The set of such limits is called the principal spectrum. Consequently, our main contributions can be understood as a generalization to the non-autonomous framework of the qualitative properties of the principal eigenvalue, which are well-established in the classical elliptic and periodic parabolic settings.
}
\end{remark}

\section{Main Results}\label{sec1.1}
First of all, we state some qualitative estimates for the principal Floquet bundle as follows. Let $c\in\mathcal{C}$ be given, denote by $\hat{c}_T(\cdot):=\fint^{T}_{0}c(\cdot,s)\,\mathrm{d}s$ and define the set $\mathscr{C}$ of functions as
 \[\mathscr{C}:=\left\{\hat{c}\in C (\overline{\Omega})~|~\text{there exists }T_n\to\infty\text{ such that }\|\hat{c}_{T_n}-\hat{c}\|_{C(\overline{\Omega})}\to 0\right\}.\]
It is easy to check that $\mathscr{C}$ is a compact (see Lemma \ref{lemma2.1}), uniformly bounded metric space induced by norm $\|\cdot\|_\infty$. Here, the family $\mathscr{C}$ is said to be uniformly bounded provided that $\sup_{\hat{c}\in\mathscr{C}}\|\hat{c}\|_\infty<\infty$. Denote by $\mu^{\infty}(d,m)$ the principal eigenvalue of
  \begin{equation}\label{eq1.3}
    \left\{
    \begin{aligned}
      &-d\operatorname{div}(A(x)\nabla\hat{\phi})+m(x)\hat{\phi}=\mu\hat{\phi},&&x\in\Omega,\\
      &\mathbf{n}\cdot(A(x)\nabla\hat{\phi})=0, &&x\in\partial\Omega.
    \end{aligned}
    \right.
  \end{equation}

The basic results presented below for the normalized principal Floquet bundle generalize prior results on the principal eigenvalue; they are also essential for the subsequent analysis of asymptotic behavior.
\begin{theorem}\label{theorem1.3}
Let $c\in\mathcal{C}$ be a given function, and let $H$ denote the normalized principal Floquet bundle of \eqref{eq1.1}. The following statements hold.
  \begin{enumerate}[{\rm (i)}]
    \item  One has
     \[\underline{\lambda}_{H}(\omega,d)\le \inf\limits_{\hat{c}\in\mathscr{C}}\mu^{\infty}(d,\hat{c})~\text{ and }~ \overline{\lambda}_{H}(\omega,d)\le \sup\limits_{\hat{c}\in\mathscr{C}}\mu^{\infty}(d,\hat{c}).\]
     Furthermore, the following inequalities are satisfied,
       \[\liminf_{T\to\infty} \fint^{T}_{0}\min_{x\in\overline{\Omega}}c(x,s)\,\mathrm{d}s\le \underline{\lambda}_{H}(\omega,d)\le \liminf_{T\to\infty} \fint^{T}_{0}\fint_\Omega c(x,s) \,\mathrm{d}x\mathrm{d}s\]
      and
       \[\limsup_{T\to\infty} \fint^{T}_{0}\min_{x\in\overline{\Omega}}c(x,s)\,\mathrm{d}s \le \overline{\lambda}_{H}(\omega,d) \le\limsup_{T\to\infty}\fint^{T}_{0}\fint_{\Omega}c(x,s)\,\mathrm{d}x\mathrm{d}s.\]
    \item Suppose that there is a $t_0\in\mathbb{R}$ such that $c_1(x,t)\ge c_2(x,t)$ for all $(x,t)\in\Omega\times[t_0,\infty)$, where $c_1,c_2\in \mathcal{C}$. Then $H(\cdot;c)$ is monotonically increasing with respect to $c$ in the sense that
         \[\liminf\limits_{T\to\infty}\fint^{T}_{0}[H(s;c_1)-H(s;c_2)]\,\mathrm{d}s\ge 0.\]
        Consequently,
         \[\underline{\lambda}_{H}(c_1)\ge \underline{\lambda}_{H}(c_2)\quad
        \text{and}
         \quad\overline{\lambda}_{H}(c_1)\ge \overline{\lambda}_{H}(c_2).\]
        The three inequalities above become strict if $\liminf\limits_{T\to\infty}\fint^{T}_{0} \int_{\Omega}[c_1(x,s)-c_2(x,s)] \,\mathrm{d}x\mathrm{d}s>0$.
    \item  Assume that $c_1,c_2\in \mathcal{C}$ satisfy $\lim\limits_{t\to\infty} \left\|c_1(\cdot,t)-c_2(\cdot,t)\right\|_\infty=0$. Then
  \[\lim\limits_{T\to\infty}\fint^{T}_{0}[H(s;c_1)-H(s;c_2)]\,\mathrm{d}s= 0.\]
        Consequently,
        \[\underline{\lambda}_{H}(c_1)=\underline{\lambda}_{H}(c_2)\quad\text{and}
         \quad\overline{\lambda}_{H}(c_1)=\overline{\lambda}_{H}(c_2).\]

  \end{enumerate}
\end{theorem}

These results generalize the theory of the principal eigenvalue for periodic parabolic problems. Specifically, the analogue of Theorem \ref{theorem1.3}(i) is contained in \cite[Theorem 1.2]{Liu2022Classifying}, and those of Theorem \ref{theorem1.3}(ii)-(iii) in \cite[Lemmas 15.5 and 15.7]{Hess1991Periodic}.

\begin{remark}{\rm
  The estimate established in Theorem \ref{theorem1.3}(i) is optimal (sharp) in the sense of parameter limits involving $\omega$ and $d$. This optimality is manifested in the following three limiting regimes.
\begin{enumerate}[{\rm (a)}]
  \item The upper bounds $\inf\limits_{\hat{c}\in\mathscr{C}}\mu^\infty(d,\hat{c})$ and $\sup\limits_{\hat{c}\in\mathscr{C}}\mu^\infty(d,\hat{c})$ can be achieved as the frequency $\omega\to\infty$ while $d>0$ is fixed; see Theorem \ref{theorem1.6}(ii).
  \item The lower bounds $\liminf_{T\to\infty} \fint^{T}_{0}\min_{x\in\overline{\Omega}}c(x,s)\,\mathrm{d}s$ and $\limsup_{T\to\infty} \fint^{T}_{0}\min_{x\in\overline{\Omega}}c(x,s)\,\mathrm{d}s$ can be achieved in the limit $\left(d, \frac{\omega}{d}\right) \to (0,0)$; see Theorem \ref{theorem1.10}(i).
  \item The upper bounds $\liminf_{T\to\infty} \fint^{T}_{0}\fint_\Omega c(x,s) \,\mathrm{d}x\mathrm{d}s$ and $\limsup_{T\to\infty} \fint^{T}_{0}\fint_\Omega c(x,s) \,\mathrm{d}x\mathrm{d}s$ can be achieved as the diffusion rate $d\to\infty$; see Theorem \ref{theorem1.9}(i).
\end{enumerate}
}
\end{remark}

\subsection{Asymptotic behavior: Single limits}
We now proceed to present the asymptotic behavior of the principal Floquet exponent under various typical parameter limits. For any given constant $T>0$ and function $c\in\mathcal{C}\cap C^{2,1}(\overline{\Omega}\times\mathbb{R})$, we define
\[\mathcal{M}(x,t;T):=t\fint^{T}_{0}c(x,s)\,\mathrm{d}s-\int^{t}_{0}c(x,s)\,\mathrm{d}s, \qquad (x,t)\in\overline{\Omega}\times[0,T].\]
Assume that
\begin{equation}\label{H2}
\tag{H2}
\limsup\limits_{T\to\infty}\sup\limits_{(x,t)\in\Omega\times[0,T]}(|\operatorname{div}(A(x)\nabla\mathcal{M}(x,t;T))|+|\nabla \mathcal{M}(x,t;T)\cdot (A(x)\nabla \mathcal{M}(x,t;T))|)<\infty.
\end{equation}
The validity of hypothesis \eqref{H2} follows from its verifiability for finite sums of periodic functions (with possibly different periods), as shown in Lemma \ref{lemma8.1}.

\medskip
We first examine the limiting behavior of the normalized principal Floquet exponent in the regime where $d\to0$ and $\omega>0$ remains constant. To stress the dependence on the diffusion parameter, we employ the notation $H(\cdot;d)$.
\begin{theorem}\label{theorem1.5}
Let $H(\cdot;d)$ be the normalized principal Floquet bundle of \eqref{eq1.1}. Then the following assertions are valid.
\begin{enumerate}[{\rm (i)}]
  \item {\rm (}The case $d\to0${\rm )} Assume that $c\in\mathcal{C}\cap C^{2,1}(\overline{\Omega}\times\mathbb{R})$ satisfies \eqref{H2}. One has
        \[\lim\limits_{d\to0}\underline{\lambda}_{H}(d)= \min\limits_{x\in\overline{\Omega}}\liminf\limits_{T\to\infty} \fint^{T}_{0}c(x,s)\,\mathrm{d}s= \inf_{\hat{c}\in\mathscr{C}}\min_{x\in\overline{\Omega}}\hat{c}(x)\]
      and
        \[\lim\limits_{d\to0}\overline{\lambda}_{H}(d)= \limsup\limits_{T\to\infty}\min\limits_{x\in\overline{\Omega}} \fint^{T}_{0}c(x,s)\,\mathrm{d}s= \sup_{\hat{c}\in\mathscr{C}}\min_{x\in\overline{\Omega}}\hat{c}(x).\]
  \item {\rm (}The case $d\to\infty${\rm )} Assume that the function $c\in C^{\delta,\delta/2}(\overline{\Omega}\times\mathbb{R})$ satisfies assumption \eqref{H1} and $\left\|c(\cdot,\cdot)-\fint_\Omega c(x,\cdot)\,\mathrm{d}x\right\|_{C^{\delta,\delta/2}(\overline{\Omega} \times\mathbb{R} )}<\infty$. One has
        \[ \lim\limits_{d\to\infty} \underline{\lambda}_{H}(d) = \liminf_{T\to\infty}\fint^{T}_{0}\fint_\Omega c(x,s)\,\mathrm{d}x\mathrm{d}s\quad
     \text{and}
        \quad\lim\limits_{d\to\infty} \overline{\lambda}_{H}(d) = \limsup_{T\to\infty}\fint^{T}_{0}\fint_\Omega c(x,s)\,\mathrm{d}x\mathrm{d}s.\]
\end{enumerate}

\end{theorem}

In the theorem stated above, the primary contribution of this work lies in establishing the limiting behavior in the small diffusion regime; and this result may be viewed as a generalization of the asymptotic analysis of the principal eigenvalue in periodic parabolic problems, with relevant results available in \cite[Lemma 2.4]{Hutson2001The}, \cite[Theorem 1.1]{Bai2020Asymptotic} and the references therein. By contrast, the large diffusion limit is a direct consequence of the results by Lam and Lou \cite[Proposition 2.6, Corollary 2.8]{Lam2024The}, and they yield additionally not only pointwise estimates for $H(t)$ but also its monotonicity in the fast diffusion parameter $d>>1$.

A natural question is to consider the effect of varying the temporal frequency $\omega$. The following theorem examines two limiting regimes: slow temporal frequency ($\omega\to0$) and fast temporal frequency ($\omega\to\infty$). The asymptotic behavior of the normalized principal Floquet exponent is related to two auxiliary  elliptic eigenvalue problems. One of them is \eqref{eq1.3}, and the other is that for each $t\in\mathbb{R}$,
\begin{equation}\label{eq1.4}
  \left\{
  \begin{aligned}
    &-d\operatorname{div}(A(x)\nabla\phi^{0})+c(x,t)\phi^{0}=\mu^0(t)\phi^{0},&&x\in\Omega,\\
    &\mathbf{n}\cdot(A(x)\nabla\phi^{0})=0,&&x\in\partial\Omega,\\
    &\|\phi^{0}(\cdot,t)\|_\infty=1.
  \end{aligned}
  \right.
\end{equation}

\begin{theorem}\label{theorem1.6}
Let $H(\cdot;\omega)$ be the normalized principal Floquet bundle of \eqref{eq1.1} with potential $c$.
\begin{enumerate}[{\upshape (i)}]
  \item {\rm (}The case $\omega\to0${\rm )} Assume $c\in \mathcal{C}\cap C^{0,1}(\overline{\Omega} \times\mathbb{R})$ satisfies $\sup_{t\in\mathbb{R}}\|\partial_t c(\cdot,t)\|_\infty<\infty$. Denote by $\mu^{0}(t)$, $t\in\mathbb{R}$ the principal eigenvalue of \eqref{eq1.4} with $c$. Then
       \[\lim_{\omega\to0}\underline{\lambda}_{H}(\omega)= \liminf_{T\to\infty}\fint^{T}_{0}\mu^0(s)\,\mathrm{d}s \quad\text{and}\quad\lim_{\omega\to0}\overline{\lambda}_{H}(\omega)= \limsup_{T\to\infty}\fint^{T}_{0}\mu^0(s)\,\mathrm{d}s.\]
  \item {\rm (}The case $\omega\to\infty${\rm )} Assume that $c\in \mathcal{C}\cap C^{2,1}(\overline{\Omega}\times\mathbb{R})$ satisfies assumption \eqref{H2}. Denote by $\mu^\infty$ the principal eigenvalue of \eqref{eq1.3}. Then
       \[\lim_{\omega\to\infty}\underline{\lambda}_{H}(\omega)= \inf_{\hat{c}\in\mathscr{C}}\mu^\infty(\hat{c})\quad\text{and}
       \quad\lim_{\omega\to\infty}\overline{\lambda}_{H}(\omega)= \sup_{\hat{c}\in\mathscr{C}}\mu^\infty(\hat{c}).\]
\end{enumerate}
\end{theorem}

When the potential function c is time-periodic, the above results reduce to the limiting results with respect to the frequency for the principal eigenvalue of periodic parabolic problems; specifically, one may refer to \cite{Liu2019Monotonicity}, where they also investigated the monotonicity of the principal eigenvalue with respect to the frequency. For the regime of small frequencies ($\omega \to 0$), the following asymptotic result on finite intervals for the principal Floquet bundle holds.
\begin{theorem}\label{theorem1.7}
Let $c\in \mathcal{C}\cap C^{0,1}(\overline{\Omega}\times\mathbb{R})$ be a function satisfying $\sup_{t\in\mathbb{R}}\|\partial_tc(\cdot,t)\|_\infty<\infty$, $H(\cdot;\omega)$ be the normalized principal Floquet bundle of \eqref{eq1.1} with potential $c$, and $\mu^{0}(t)$ be the principal eigenvalue of \eqref{eq1.4} with potential $c$ for each $t\in\mathbb{R}$. For any $-\infty<T_1<T_2<\infty$, there holds
  \[ \lim_{\omega\to0}\fint^{T_2}_{T_1}H(s;\omega)\,\mathrm{d}s= \fint^{T_2}_{T_1}\mu^0(s)\,\mathrm{d}s.\]
\end{theorem}

We now contrast the results of Theorems \ref{theorem1.5} and \ref{theorem1.6}. Theorem \ref{theorem1.5} provides the limits of the principal Floquet exponent as $d\to 0$ for a fixed $\omega>0$, and these limits do not depend on the frequency. Theorem \ref{theorem1.6}, on the other hand, gives the limits as $\omega\to 0$ or $\omega\to\infty$ for a fixed $d>0$, and these limits evidently depend on the diffusion rate via the principal eigenvalues $\mu^0$ and $\mu^\infty$. This reveals a noticeable difference. It naturally leads to the question of understanding more thoroughly how these limiting values (and consequently the principal Floquet exponent itself) relate to the diffusion rate $d>0$, encompassing properties such as monotonicity in $d$ and its asymptotic behavior as $d\to 0$ or $d\to\infty$.

\begin{corollary}\label{corollary1.8}
Let $H(\cdot;\omega,d)$ be the normalized principal Floquet bundle of \eqref{eq1.1} with potential $c\in \mathcal{C}$; denote by $\mu^{0}(t;d)$ the principal eigenvalue of \eqref{eq1.4} with potential $c(\cdot,t)$, $t\in\mathbb{R}$.
\begin{enumerate}[{\rm (i)}]
  \item If $c\in C^{0,1}(\overline{\Omega}\times\mathbb{R})$ fulfills $\sup_{t\in\mathbb{R}}\|\partial_tc(\cdot,t)\|_\infty<\infty$, then
      \[\lim_{\omega\to0}\underline{\lambda}_{H}(\omega,d)\quad\text{and}\quad \lim_{\omega\to0}\overline{\lambda}_{H}(\omega,d)\]
      are continuous and increasing in $d>0$, and satisfy
      \begin{align*}
         & \lim_{d\to 0}\lim_{\omega\to0}\underline{\lambda}_{H}(\omega,d) =\liminf_{T\to\infty}\fint^{T}_{0}\min_{x\in\overline{\Omega}}c(x,s)\,\mathrm{d}s, \\
         & \lim_{d\to 0}\lim_{\omega\to0}\overline{\lambda}_{H}(\omega,d) =\limsup_{T\to\infty}\fint^{T}_{0} \min_{x\in\overline{\Omega}}c(x,s)\,\mathrm{d}s, \\
         & \lim_{d\to \infty}\lim_{\omega\to0}\underline{\lambda}_{H}(\omega,d)=\liminf_{T\to\infty}\fint^{T}_{0}\int_\Omega c(x,s)\,\mathrm{d}x\mathrm{d}s, \\
         &\lim_{d\to \infty}\lim_{\omega\to0}\overline{\lambda}_{H}(\omega,d)=\limsup_{T\to\infty}\fint^{T}_{0}\int_\Omega c(x,s)\,\mathrm{d}x\mathrm{d}s.
      \end{align*}
  \item Assume in addition that $c\in C^{2,1}(\overline{\Omega}\times\mathbb{R})$ and that $c$ satisfies assumption \eqref{H2}. Then the limits
      \[\lim_{\omega\to\infty}\underline{\lambda}_{H}(\omega,d)\quad\text{and}\quad \lim_{\omega\to\infty}\overline{\lambda}_{H}(\omega,d) \]
      are continuous and nondecreasing in $d>0$, and
\begin{align*}
   &\lim_{d\to0}\lim_{\omega\to\infty}\underline{\lambda}_{H}(\omega,d) =\min_{x\in\overline{\Omega}} \liminf_{T\to\infty}\fint^{T}_{0}c(x,s)\,\mathrm{d}s, \\
   &\lim_{d\to0}\lim_{\omega\to\infty}\overline{\lambda}_{H}(\omega,d) =\limsup_{T\to\infty}\min_{x\in\overline{\Omega}}\fint^{T}_{0}c(x,s)\,\mathrm{d}s, \\
   &\lim_{d\to\infty}\lim_{\omega\to\infty}\underline{\lambda}_{H}(\omega,d) =\liminf_{T\to\infty}\fint^{T}_{0} \fint_{\Omega}c(x,s)\,\mathrm{d}x\mathrm{d}s,\\
   &\lim_{d\to\infty}\lim_{\omega\to\infty}\overline{\lambda}_{H}(\omega,d) =\limsup_{T\to\infty}\fint^{T}_{0} \fint_{\Omega}c(x,s)\,\mathrm{d}x\mathrm{d}s.
\end{align*}
\end{enumerate}
\end{corollary}

\subsection{Asymptotic behavior: Double limits} The foregoing analysis shows that the limiting processes for the principal Floquet exponent with respect to the frequency and the diffusion cannot be interchanged arbitrarily. We therefore conjecture that the coupling effect of frequency and diffusion likely plays a significant role in shaping its asymptotic behavior and investigate its asymptotic behavior when both parameters vary simultaneously. The next theorem considers two joint limiting processes: one where $d\to\infty$ and $\omega/d$ tends to a finite or infinite limit, and another where $(\omega,d)\to(\infty,0)$. These results reveal how the interplay between diffusion and temporal frequency shapes the normalized principal Floquet exponent.
\begin{theorem}\label{theorem1.9}
Let $H(\cdot;\omega,d)$ be the normalized principal Floquet bundle of \eqref{eq1.1} with potential $c\in\mathcal{C}\cap C^{2,1}(\overline{\Omega}\times\mathbb{R})$.
\begin{enumerate}[{\upshape (i)}]
  \item Let $\vartheta\in[0,\infty]$ be given. Assume that $\sup_{t\in\mathbb{R}}\|\partial_t c(\cdot,t)\|_{\infty}<\infty$. Suppose further that $c$ satisfies assumption \eqref{H2} if $\frac{\omega}{d}\to\infty$. Then
       \[ \lim\limits_{\left(d,\frac{\omega}{d}\right)\to(\infty,\vartheta)} \underline{\lambda}_{H}(\omega,d) = \liminf_{T\to\infty}\fint^{T}_{0}\fint_\Omega c(x,s)\,\mathrm{d}x\mathrm{d}s\]
      and
       \[\lim\limits_{\left(d,\frac{\omega}{d}\right)\to(\infty,\vartheta)} \overline{\lambda}_{H}(\omega,d) = \limsup_{T\to\infty}\fint^{T}_{0}\fint_\Omega c(x,s)\,\mathrm{d}x\mathrm{d}s.\]
  \item Suppose that $c$ satisfies \eqref{H2}. Then
       \[\lim_{(\omega,d)\to(\infty,0)}\underline{\lambda}_{H}(\omega,d)= \min_{x\in\overline{\Omega}}\liminf_{T\to\infty}\fint^{T}_{0}c(x,s)\,\mathrm{d}s= \inf_{\hat{c}\in\mathscr{C}}\min_{x\in\overline{\Omega}}\hat{c}(x)\]
      and
       \[\lim_{(\omega,d)\to(\infty,0)}\overline{\lambda}_{H}(\omega,d)= \limsup_{T\to\infty}\min_{x\in\overline{\Omega}}\fint^{T}_{0}c(x,s)\,\mathrm{d}s= \sup_{\hat{c}\in\mathscr{C}}\min_{x\in\overline{\Omega}}\hat{c}(x).\]
\end{enumerate}
\end{theorem}

We now consider the limiting case $(\omega, d) \to (0,0)$. In this regime, we need to assume that $c\in\mathcal{C}$ satisfies a uniform H\"{o}lder condition, i.e., there exists a constant $L > 0$ such that
\begin{equation}\label{H3}
  L:=\max_{x\in\overline{\Omega}}\sup_{t,s\in\mathbb{R}}\frac{|c(x,t)-c(x,s)|}{|t-s|^{\frac{\delta}{2}}}<\infty.
  \tag{H3}
\end{equation}

\begin{theorem}\label{theorem1.10}
  Let $H(\cdot;\omega,d)$ be the normalized principal Floquet bundle of \eqref{eq1.1} with potential $c\in \mathcal{C}\cap C^{2,1}(\overline{\Omega}\times\mathbb{R})$.
\begin{enumerate}[{\rm (i)}]
  \item Assume that the function $c$ satisfies uniform H\"{o}lder condition \eqref{H3}. Then
      \begin{align*}
        \lim_{\left(d,\frac{\omega}{\sqrt{d}}\right)\to(0,0)}\underline{\lambda}_{H}(\omega,d) =\liminf_{T\to\infty}\fint^{T}_{0}\min_{x\in\overline{\Omega}}c(x,s)\,\mathrm{d}s
        =\inf_{x(\cdot)\in C(\mathbb{R};\overline{\Omega})}\liminf_{T\to\infty} \fint^{T}_{0}c(x(s),s)\,\mathrm{d}s
      \end{align*}
     and
      \begin{align*}
        \lim_{\left(d,\frac{\omega}{\sqrt{d}}\right)\to(0,0)}\overline{\lambda}_{H}(\omega,d) =\limsup_{T\to\infty}\fint^{T}_{0}\min_{x\in\overline{\Omega}}c(x,s)\,\mathrm{d}s =\inf_{x(\cdot)\in C(\mathbb{R};\overline{\Omega})}\limsup_{T\to\infty} \fint^{T}_{0}c(x(s),s)\,\mathrm{d}s.
      \end{align*}
  \item Assume that the function $c$ satisfies assumption \eqref{H2}. Then
       \[\lim\limits_{\left(\omega,\frac{\omega}{\sqrt{d}}\right)\to(0,\infty)} \underline{\lambda}_{H}(\omega,d) = \min\limits_{x\in\overline{\Omega}}\liminf_{T\to\infty}\fint^{T}_{0}c(x,s)\,\mathrm{d}s = \inf_{\hat{c}\in\mathscr{C}}\min_{x\in\overline{\Omega}}\hat{c}(x)\]
      and
       \[\lim\limits_{\left(\omega,\frac{\omega}{\sqrt{d}}\right)\to(0,\infty)} \overline{\lambda}_{H}(\omega,d) = \limsup_{T\to\infty}\min\limits_{x\in\overline{\Omega}}\fint^{T}_{0}c(x,s)\,\mathrm{d}s = \sup_{\hat{c}\in\mathscr{C}}\min_{x\in\overline{\Omega}}\hat{c}(x).\]
\end{enumerate}
\end{theorem}

For the case of $\left(d,\frac{\omega}{\sqrt{d}}\right)\to(0,0)$, we present the following local result.
\begin{theorem}\label{theorem1.11}
Let $H(\cdot;\omega,d)$ be the normalized principal Floquet bundle of \eqref{eq1.1} with potential $c\in \mathcal{C}$. For any $T>0$, there holds
 \[\lim_{\left(d,\frac{\omega}{\sqrt{d}}\right)\to(0,0)}\fint^{T}_{0}H(s;\omega,d)\,\mathrm{d}s= \fint^{T}_{0}\min_{x\in\overline{\Omega}}c(x,s)\,\mathrm{d}s,\]
and the above limit holds uniformly in $T>0$.
\end{theorem}

Most relevant to our work, Liu and Lou \cite{Liu2022Classifying} offered a thorough classification of the coupled $(\omega,d)\to(0,0)$ limit. Analyzing the Laplacian case, they detailed the eigenvalue's monotonicity and level set topology in the frequency-diffusion plane. Their work inspired us to decompose the limit $(\omega,d)\to(0,0)$ into three regimes: (a) $\bigl(d,\frac{\omega}{\sqrt{d}}\bigr)\to(0,0)$; (b) $\bigl(d,\frac{\omega}{\sqrt{d}}\bigr)\to(0,\vartheta)$ for $\vartheta\in(0,\infty)$; (c) $\bigl(d,\frac{\omega}{\sqrt{d}}\bigr)\to(0,\infty)$. We conjecture that the asymptotic behavior of the normalized principal Floquet bundle in Case (b) is related to a Hamilton-Jacobi equation problem.

\medskip

We conclude by offering some remarks on the main results, the challenges involved, and the methodology adopted in this paper. Compared with the principal eigenvalues in elliptic and periodic parabolic problems, our work investigates the asymptotic behavior of the spectrum for a second-order parabolic operator in a broader context. As noted in Remark \ref{remark1.2}, we compared the principal eigenvalue in periodic parabolic problems with its associated normalized principal Floquet bundle, highlighting both distinctions and links. Motivated by this comparison, we introduced the normalized principal Floquet exponent through long-time average limits, which coincides with the usual principal eigenvalue in the periodic parabolic setting. In general non-autonomous and non-periodic scenarios, such long-time average limits may fail to exist. Consequently, we examine the upper and lower limits of these long-time averages for the principal Floquet bundle, which is reflected in all the main results presented above. More precisely, given that it is only known that $c\in\mathcal{C}$ and that $c$ and $H$ are uniformly bounded in $\overline{\Omega}\times\mathbb{R}$, the limits
\[\lim\limits_{T\to\infty}\fint^{T}_{0}c(\cdot,s)\,\mathrm{d}s \quad\text{and}\quad \lim\limits_{T\to\infty}\fint^{T}_{0}H(s)\,\mathrm{d}s\]
may not exist. For this reason, our analysis will primarily proceed along a sequence $\{T_n\}_{n\in\mathbb{N}^+}$ satisfying $T_n\to\infty$ as $n\to\infty$. The properties of the normalized principal Floquet bundle obtained in the subsequent chapters are, in fact, established for such sequences. We observe that if the limit
\[\lim\limits_{n\to\infty}\fint^{T_n}_{0}c(\cdot,s)\,\mathrm{d}s\quad\text{exists in the pointwise sense on }\overline{\Omega},\]
then by Proposition \ref{proposition1.1}, the limit function is continuous on $\overline{\Omega}$. The set of all limiting values of $\fint^{T}_{0}H(s)\,\mathrm{d}s$ under convergent sequences of $\left\{\fint^{T}_{0} H(s)\,\mathrm{d}s\right\}_{T>0}$ is called the (normalized) principal Floquet exponent. It thus forms a set of constants, also referred to as the principal spectrum.
\section{Almost periodic cases and other extensions}
Assumption \eqref{H2} is a very strong assumption on the potential function, which severely limits the applicability of the main results. Therefore, based on \eqref{H2} and combined with some basic properties of the normalized principal Floquet bundle (for example, Theorem \ref{theorem1.3}(ii)-(iii)), we intend to extend the main results so that they can be applied to a broad class of potential functions, including the uniformly almost periodic case.

Denote by
\[\mathscr{X}_1:=\{c\in \mathcal{C}\cap C^{2,1}(\overline{\Omega}\times\mathbb{R}) ~|~c\text{ satisfies } \eqref{H2}\},~\mathscr{X}_2:=\{\eta\in C^{\delta,\delta/2}(\overline{\Omega}\times\mathbb{R})~|~\lim\limits_{t\to\infty}\|\eta(\cdot,t)\|_\infty=0\}.\]
Let us state the asymptotic behavior of $\fint^{T}_{0}H(s;\omega,d)\,\mathrm{d}s$ with the potential $c\in\mathscr{X}$, where $\mathscr{X}$ is defined as
\begin{equation}\label{eq1.5}
\mathscr{X}:=\{c\in\mathcal{C}~|~c=\tilde{c}+\eta, \text{ where }\tilde{c}\in \mathscr{X}_1,~\eta\in \mathscr{X}_2\}.
\end{equation}
Using Theorem \ref{theorem1.3}(iii) and Theorem \ref{theorem1.5}, we have the following corollary.
\begin{corollary}\label{corollary1.12}
  The conclusions of Theorems {\rm \ref{theorem1.5}, \ref{theorem1.6}(ii), \ref{theorem1.9}(ii)} and {\rm \ref{theorem1.10}(ii)} continue to hold even when hypothesis \eqref{H2} on the potential function $c$ is replaced by ``$c\in \mathscr{X}$".
\end{corollary}

Corollary \ref{corollary1.12} extends the applicability of the theorem by decomposing the potential function. Furthermore, the essence of this extension can be understood from the perspective of uniform approximation. In other words, the relevant limiting properties are preserved as long as the potential can be uniformly approximated, to any desired accuracy over long time scales, by a function satisfying the original conditions. This provides a general approach for handling a broader class of non-periodic and non-decaying potentials. In this regard, let us examine an example.
\begin{example}{\rm
Let $\hat{c}\in C^\infty(\overline{\Omega})$ be a nonnegative function satisfying $\|\Delta \hat{c}\|_\infty+\|\nabla\hat{c}\|_\infty>0$. Consider a smooth function $c\in C^\infty(\overline{\Omega}\times\mathbb{R})$ which takes the form
  \[c(x,t)=\frac{\hat{c}(x)}{\sqrt{t+1}}\qquad\text{ on }\overline{\Omega}\times[0,\infty).\]
Let $(H(\cdot;c),\varphi)$ be the normalized principal Floquet bundle of \eqref{eq1.1} with potential $c$. Now we study the auxiliary function
  \[\mathcal{M}(x,t;T)=t\fint^{T}_{0}c(x,s)\,\mathrm{d}s-\int^{t}_{0}c(x,s)\,\mathrm{d}s, \qquad (x,t)\in\overline{\Omega}\times[0,T].\]
By direct computations, we have
   \[ \mathcal{M}(x,t;T)=  2\hat{c}(x)\left[\frac{t}{T}(\sqrt{T+1}-1)-\sqrt{t+1}+1\right].\]
Obviously, it holds
   \[\limsup\limits_{T\to\infty}\sup\limits_{(x,t)\in\Omega\times[0,T]}(|\operatorname{div}(A(x)\nabla\mathcal{M}(x,t;T))|+|\nabla \mathcal{M}(x,t;T)\cdot (A(x)\nabla \mathcal{M}(x,t;T))|)=\infty.\]

Now we turn to compute $\fint^{T}_{0}H(s;c)\,\mathrm{d}s$ as $T\to\infty$. For any $\varepsilon>0$, there is a $t_0>0$ such that
  \[\sup_{(x,t)\in\Omega\times[t_0,\infty)}c(x,t)<\varepsilon.\]
This observation, together with Theorem \ref{theorem1.3}(iii), yields
  \[\liminf\limits_{T\to\infty}\fint^{T}_{0}[H(s;\varepsilon)-H(s;c)]\,\mathrm{d}s\ge 0.\]
In view of the facts that $H(\cdot;c)\ge0$ (see Lemma \ref{lemma3.5}) and $H(\cdot;\varepsilon)=\varepsilon$ (since $H(\cdot;\varepsilon)$ is the normalized principal Floquet bundle with potential $\varepsilon$), we have
  \[0\le\liminf_{T\to\infty}\fint^{T}_{0}H(s;c)\,\mathrm{d}s\le\limsup_{T\to\infty}\fint^{T}_{0}H(s;c)\,\mathrm{d}s\le\varepsilon.\]
By the arbitrariness of $\varepsilon$, one has $\lim_{T\to\infty}\fint^{T}_{0}H(s;c)\,\mathrm{d}s=0$.
}
\end{example}

\medskip
Owing to the fact that the limit $\lim_{T\to\infty}\fint^{T}_{0}c(\cdot,s)\,\mathrm{d}s$ may fail to exist in $C(\overline{\Omega})$, the formulations of the main results in Section \ref{sec1.1} are necessarily somewhat intricate. In the hypothetical situation where the limit does exist, the corresponding statements would, by contrast, be much more concise. While we do not elaborate on the main results under that hypothesis, it is worth noting that a frequently employed class of functions (i.e., the uniformly almost periodic functions) not only enjoys the existence of the long-time average but also possesses, in addition, further favourable properties. A key feature of such functions is that they can be uniformly approximated by finite sums of periodic functions, whose periods need not be equal. Consequently, under suitable regularity conditions, finite sums of periodic functions satisfy hypotheses \eqref{H2} and \eqref{H3}; see Lemma \ref{lemma8.1}.

A function $f(x,t)$ defined on $\overline{\Omega}\times\mathbb{R}$ can be treated as a family $\mathcal{F}=\{f(x)(\cdot)\}_{x\in\overline{\Omega}}$ of functions defined on $\mathbb{R}$. A family $\mathcal{F}$ of almost periodic functions is called a \textbf{uniformly almost periodic family} if it is uniformly bounded, and if given $\epsilon>0$, then $\mathcal{T}(\mathcal{F},\epsilon)=\cap_{f\in\mathcal{F}}\mathcal{T}(f,\epsilon)$ is relatively dense and includes an interval about 0, where $\mathcal{T}(f,\epsilon)$ is $\epsilon-$translation set of $f$; see \cite[Definition 2.1]{Fink2006Almost}. Using Theorem \ref{theorem1.3} and Theorem \ref{theorem1.5}, we can determine the asymptotic behavior of $\fint^{T}_{0}H(s;\omega,d)\,\mathrm{d}s$ for the case that $\{c(x,\cdot)\}_{x\in\overline{\Omega}}$ is a uniformly almost periodic family. It is important to note that the limit
\[\lim_{T\to\infty}\fint^{T}_{0}c(x,s)\,\mathrm{d}s\quad \text{exists}\]
if $\{c(x,\cdot)\}_{x\in\overline{\Omega}}$ is a uniformly almost periodic family; see \cite[Theorem 3.1]{Fink2006Almost}.
\begin{theorem}\label{theorem1.14}
Assume that $c\in\mathcal{C}$ and $\{c(x,\cdot)\}_{x\in\overline{\Omega}}$ is a uniformly almost periodic family and denote by $\hat{c}(\cdot):=\lim_{T\to\infty}\fint^{T}_{0}c(\cdot,s)\,\mathrm{d}s$. Let $(H,\varphi)$ be the normalized principal Floquet bundle of \eqref{eq1.1}. The following statements are valid.
\begin{enumerate}[{\rm (i)}]
  \item For any fixed $\omega>0$, there holds
       \[\lim\limits_{d\to0}\underline{\lambda}_{H}(d)= \lim\limits_{d\to0}\overline{\lambda}_{H}(d)= \min\limits_{x\in\overline{\Omega}}\hat{c}(x);\]
  \item Let $\mu^{0}(t)$ be the principal eigenvalue of \eqref{eq1.4} with potential $c(\cdot,t)$ for each $t\in\mathbb{R}$. Then for any fixed $d>0$, there holds
       \[\lim_{\omega\to0}\underline{\lambda}_{H}(\omega)= \lim_{\omega\to0}\overline{\lambda}_{H}(\omega)= \lim_{T\to\infty}\fint^{T}_{0}\mu^0(s)\,\mathrm{d}s;\]
  \item Let $\mu^\infty(\hat{c})$ be the principal eigenvalue of \eqref{eq1.3} with potential $\hat{c}$. For any fixed $d>0$, there holds
       \[\lim_{\omega\to\infty}\underline{\lambda}_{H}(\omega)= \lim_{\omega\to\infty}\overline{\lambda}_{H}(\omega)=\mu^\infty(\hat{c});\]
  \item Let $\vartheta\in[0,\infty]$ be a given constant. There holds
       \[\lim\limits_{\left(d,\frac{\omega}{d}\right)\to(\infty,\vartheta)} \underline{\lambda}_{H}(\omega,d)= \lim\limits_{\left(d,\frac{\omega}{d}\right)\to(\infty,\vartheta)} \overline{\lambda}_{H}(\omega,d) = \fint_\Omega \hat{c}(x)\,\mathrm{d}x;\]
  \item There holds
       \[\lim_{ (\omega,d)\to(\infty,0)}\underline{\lambda}_{H}(\omega,d)= \lim_{(\omega,d)\to(\infty,0)}\overline{\lambda}_{H}(\omega,d)= \min_{x\in\overline{\Omega}}\hat{c}(x);\]
  \item There holds
       \[\begin{aligned}
        \lim_{\left(d,\frac{\omega}{\sqrt{d}}\right)\to(0,0)}\underline{\lambda}_{H}(\omega,d)
       = \lim_{\left(d,\frac{\omega}{\sqrt{d}}\right)\to(0,0)}\overline{\lambda}_{H}(\omega,d)
       =\lim_{T\to\infty}\fint^{T}_{0}\min_{x\in\overline{\Omega}}c(x,s)\,\mathrm{d}s;
       \end{aligned}\]
  \item There holds
       \[\lim_{\left(d,\frac{\omega}{\sqrt{d}}\right)\to(0,\infty)} \underline{\lambda}_{H}(\omega,d)=\lim_{\left(d, \frac{\omega}{\sqrt{d}}\right) \to(0,\infty)}\overline{\lambda}_{H}(\omega,d)
       =\min_{x\in\overline{\Omega}}\hat{c}(x).\]
\end{enumerate}
\end{theorem}

\section{Outline of contents}
The rest of this paper is organized as follows.

Chapter \ref{chp2} provides the necessary preparations, including basic properties of the potential function, a parabolic comparison principle developed for long-time limit processes, and several useful results on elliptic eigenvalue problems.

Chapter \ref{chp3} studies the basic properties of the normalized principal Floquet bundle (referred to as the principal Floquet exponent), including its estimates and monotonicity with respect to the potential function. These properties extend the corresponding ones of the principal eigenvalue for elliptic and periodic parabolic operators.

Chapter \ref{chp4} investigates the asymptotic behavior of the principal Floquet exponent as the diffusion rate $d$ tends to zero while the generalized frequency $\omega>0$ is fixed, showing that it converges to the spatial minimum of the long-time average of the potential function $c$.

Chapter \ref{chp5} is devoted to the asymptotics of the principal Floquet exponent with respect to the generalized frequency $\omega$ for a fixed diffusion rate $d>0$. Two limiting regimes are considered: as $\omega\to0$, the limit equals the long-time average of the principal eigenvalue of a family of elliptic problems parameterized by $t\in\mathbb{R}$; as $\omega\to\infty$, the limit reduces to the principal eigenvalue of an elliptic problem with the long-time average of $c$ as its potential. We also study the asymptotic behavior of these limits, which depend obviously on $d>0$.

Chapter \ref{chp6} considers the large-diffusion and large-frequency regimes, which are divided into two parameter limits: $(d,\omega/d)\to(\infty,\vartheta)$ with $\vartheta\in[0,\infty]$, and $(\omega,d)\to(\infty,0)$. In the former, the principal Floquet exponent converges to the spatio-temporal average of $c$; in the latter, it converges to the spatial minimum of the long-time average function of $c$.

Chapter \ref{chp7} examines the small-diffusion coupled with small-frequency regimes, focusing on two limits: $(d,\omega/\sqrt{d})\to(0,0)$ and $(d,\omega/\sqrt{d})\to(0,\infty)$. In the first case the principal Floquet exponent converges to the long-time average of the spatial minimum function of $c$; in the last case it converges to the spatial minimum of the long-time average function of $c$.

Chapter \ref{chp8} deals with the case where the potential function is uniformly almost periodic, and studies the limiting behavior of the corresponding principal Floquet exponent with respect to the frequency and diffusion.

Chapter \ref{chp9}, as an application, considers a non-autonomous, non-periodic spatially diffusive SIS epidemic model. Using the principal Floquet bundle we define two critical quantities $\underline{\mathcal{R}}_0$ and $\overline{\mathcal{R}}_0$, which generalize the notion of the basic reproduction number. We give the threshold-type dynamics of the SIS system via the two critical values and then analyze their limiting behavior with respect to the frequency and the diffusion rate.

\chapter{Preliminaries}\label{chp2}
This chapter presents the necessary preliminary results and foundational tools for the subsequent analysis. We first investigate the convergence and continuity properties of the temporal averages for functions belonging to the space $\mathcal{C}$. A comparison principle for parabolic operators and its variants are then established. Finally, we study several elliptic eigenvalue problems, which will play an important role in the subsequent analysis of the asymptotic behavior of the normalized principal Floquet bundle.

\section{Notation and terminology}
For the reader's convenience, we list the main notations used in this paper as follows.
\begin{itemize}
  \item $\fint^{T}_{0}$, $\fint_\Omega$, etc. denote the integral averages over $[0,T]$, $\Omega$, etc, respectively;
  \item $\lfloor\cdot\rfloor$ stands the floor function (integer part of a real number);
  \item $\|\cdot\|_\infty$ represents the standard supremum norm;
  \item $H$, $H(\cdot)$: denote the normalized principal Floquet bundle of \eqref{eq1.1};
  \item $\mathcal{C}$: the set of functions defined as
      \[\mathcal{C}=\left\{c(x,t)\in C^{\delta,\delta/2}(\overline{\Omega}\times\mathbb{R})~\big|~\|c(\cdot,\cdot)\|_\infty<\infty,\sup_{t\in\mathbb{R}} [c(\cdot,t)]_{\delta;\Omega}<\infty\right\};\]
  \item $\hat{c}_T$, $\hat{c}_\infty$, $\hat{c}_n$: for a given $c\in\mathcal{C}$, $\hat{c}_T$ is the integral average function of $c$ with respect to the time variable over $[0,T]$ (namely $\hat{c}_T:=\fint^{T}_{0}c(\cdot,s)\,\mathrm{d}s$), $\hat{c}_n:=\hat{c}_{T_n}$ for some sequence $\{T_n\}_{n\in\mathbb{N}^+}$, $\hat{c}_\infty$ is the limiting function of $\hat{c}_T$ up to a sequence satisfying $T_n\to\infty$ as $n\to\infty$;
  \item $\mathscr{C}$ and $\mathscr{C}_*$ are the sets of limit functions of long-time averages, namely, for a given $c\in\mathcal{C}$, define
     \[\mathscr{C}:=\left\{\hat{c}\in C (\overline{\Omega})~|~\text{there exists }T_n\to\infty\text{ such that }\|\hat{c}_{T_n}-\hat{c}\|_{C(\overline{\Omega})}\to 0\right\},\]
     where $\hat{c}_T(\cdot)=\fint^{T}_{0}c(\cdot,s)\,\mathrm{d}s$; and for a given sequence $\{T_n\}_{n\in\mathbb{N}^+}$ of $T\to\infty$,
     \[\mathscr{C}_*=\left\{\hat{c}\in C (\overline{\Omega})~|~\left\{\hat{c}_n \right\}_{n\in\mathbb{N}^+}\text{ admits a convergent subsequence converging to }\hat{c}\text{ in } C(\overline{\Omega})\right\},\]
     where $\hat{c}_n(\cdot)=\hat{c}_{T_n}(\cdot)$;
  \item $H(\cdot;d)$, $H(\cdot;\omega)$, $H(\cdot;\omega,d)$, $H(\cdot;c)$, etc. are used to emphasize that the normalized principal Floquet bundle $H$ depends on the diffusion rate $d$ with the frequency $\omega$ fixed, on the frequency $\omega$ with the diffusion rate $d$ fixed, on both the diffusion rate $d$ and frequency $\omega$ simultaneously, on potential $c$, other parameters of interest, respectively;
  \item For function $c\in C^{2,1}(\overline{\Omega}\times\mathbb{R})$ and $T>0$, define
      \[\mathcal{M}(x,t;T)=t\fint^{T}_{0}c(x,s)\,\mathrm{d}s-\int^{t}_{0}c(x,s)\,\mathrm{d}s, \qquad (x,t)\in\overline{\Omega}\times[0,T];\]
  \item In the subsequent chapters, $C>0$ generally denotes a positive constant that does not depend on the parameters under consideration (i.e., $\omega$, $d$, or $(\omega,d)$), but its value may change from one occurrence to another.
\end{itemize}

\section{Some properties of functions in $\mathcal{C}$}\label{sec2.1}
\begin{lemma}\label{lemma2.1}
  Let $c\in\mathcal{C}$ be a given function and let $\left\{T_n\right\}_{n\in\mathbb{N}^+}$ be a given sequence. Then both $\mathscr{C}$ and $\mathscr{C}_*$ are compact sets in $C(\overline{\Omega})$.
\end{lemma}
\begin{proof}
We work in the Banach space $C(\overline{\Omega})$. By the Arzel\`{a}-Ascoli theorem, a family of continuous functions on the compact metric space $\overline{\Omega}$ is relatively compact if and only if it is uniformly bounded and equi-continuous. It is easy to show that the family $\left\{\hat{c}_T\right\}_{T>0}$ is uniformly bounded and equi-continuous, hence relatively compact.

Note that $\mathscr{C}$ consists of all limit points of $\left\{\hat{c}_T\right\}_{T>0}$. It suffices to prove that $\mathscr{C}$ is closed. Suppose $\left\{\hat{c}^{(k)}\right\}_{k\in\mathbb{N}^+} \subset \mathscr{C}$ and $\hat{c}^{(k)} \to \hat{c}$ in $C(\overline{\Omega})$. Please note the distinction from notations used elsewhere in this paper: here $\hat{c}^{(k)}$ does not denote the time-average integral of some $c\in\mathcal{C}$ over $[0,T_k]$, but rather a function in $\mathscr{C}$; the time average of $c$ over $[0,T_k]$ is denoted by $\hat{c}_{T_k}$. For each $k\in\mathbb{N}^+$, because of $\hat{c}^{(k)} \in \mathscr{C}$, there exists $T_k > k$ such that
 \[\left\|\hat{c}_{T_k} - \hat{c}^{(k)}\right\|_\infty < \frac{1}{k}.\]
By the triangle inequality,
 \[\|\hat{c}_{T_k} - \hat{c}\|_\infty \le \left\|\hat{c}_{T_k} - \hat{c}^{(k)}\right\|_\infty + \left\|\hat{c}^{(k)} - \hat{c}\right\|_\infty < \frac{1}{k} + \|\hat{c}^{(k)} - \hat{c}\|_\infty .\]
Letting $k\to\infty$ gives $\|\hat{c}_{T_k} - \hat{c}\|_\infty \to 0$; hence $\hat{c} \in \mathscr{C}$. Thus $\mathscr{C}$ is closed. In a complete metric space the closure of a relatively compact set is compact. Since $\mathscr{C}$ is a closed and contained in the compact closure of the relatively compact set $\left\{\hat{c}_T\right\}_{T>0}$, it is compact.

We now show the compactness of $\mathscr{C}_*$. The subfamily $\left\{\hat{c}_{T_n}\right\}_{n\in\mathbb{N}^+}$ inherits uniform boundedness and equi-continuity from $\left\{\hat{c}_T\right\}_{T>0}$, so it is relatively compact in $C(\overline{\Omega})$. The set $\mathscr{C}_*$ is the collection of all limit points of $\left\{\hat{c}_{T_n}\right\}_{n\in\mathbb{N}^+}$. Again we only need to prove that $\mathscr{C}_*$ is closed. Assume $\left\{\hat{c}^{(k)}\right\}_{k\in\mathbb{N}^+} \subset \mathscr{C}_*$ and $\hat{c}^{(k)} \to \hat{c}$ in $C(\overline{\Omega})$. For each $k\in\mathbb{N}^+$, because of $\hat{c}^{(k)} \in \mathscr{C}_*$, there exists a strictly increasing subsequence $\left\{n_{k,m}\right\}_{m\in\mathbb{N}^+}$ such that
 \[\left\|\hat{c}_{T_{n_{k,m}}} - \hat{c}^{(k)}\right\|_\infty < \frac{1}{m}, \qquad \forall~ m\in\mathbb{N}^+ .\]
Choose a diagonal subsequence: take $n_1 = n_{1,1}$, then $\|\hat{c}_{T_{n_1}} - \hat{c}_1\|_\infty < 1$; choose $m_k$ sufficiently large that $n_{k,m_k}>n_{k-1}$ and $\|\hat{C}_{T_{n_{k,m_k}}}-\hat{c}^{(k)}\|_\infty<1/k$; set $n_k=n_{k,m_k}$. Consequently,
 \[\left\|\hat{c}_{T_{n_k}} - \hat{c}\right\|_\infty \le \left\|\hat{c}_{T_{n_k}} - \hat{c}^{(k)}\right\|_\infty + \left\|\hat{c}^{(k)} - \hat{c}\right\|_\infty < \frac{1}{k} + \left\|\hat{c}^{(k)} - \hat{c}\right\|_\infty .\]
Letting $k\to\infty$ yields $\|\hat{c}_{T_{n_k}} - \hat{c}\|_\infty \to 0$; therefore $\hat{c} \in \mathscr{C}_*$. Hence $\mathscr{C}_*$ is closed. As before, $\mathscr{C}_*$ is a closed and contained in the compact closure of the relatively compact set $\left\{\hat{c}_{T_n}\right\}_{n\in\mathbb{N}^+}$, and thus is compact. This completes the proof.
\end{proof}
\begin{remark}\label{remark2.2}{\rm
It is clear that $\mathscr{C}_* \subset \mathscr{C}$, but the converse is false. If $c(\cdot,t)$ is periodic in $t>0$ or uniformly almost periodic, then the time-average limit exists and is unique, i.e.
 \[\lim_{T \to \infty} \hat{c}_T(x) = \hat{c}(x)\quad\text{uniformly for all }x \in \overline{\Omega}.\]
In this case $\mathscr{C} = \mathscr{C}_* = \left\{\hat{c}\right\}$ is a singleton compact set. The time regularity $C^{\delta/2}$ in the hypothesis $c(x,t) \in C^{\delta,\delta/2}(\overline{\Omega} \times \mathbb{R})$ of the lemma is not actually needed; it suffices to require
 \[\|c\|_{L^\infty(\overline{\Omega} \times \mathbb{R})} < \infty,\qquad   \sup_{t \in \mathbb{R}} [c(\cdot,t)]_{\delta;\Omega} < \infty,\]
for the conclusion to hold. In fact, both limit sets are nonempty and compact by the Arzel\`{a}-Ascoli theorem.
}
\end{remark}

\begin{lemma}\label{lemma2.3}
Let $c\in\mathcal{C}$ and $\left\{T_n\right\}_{n\in\mathbb{N}^+}$ be a sequence satisfying $T_n\to\infty$ as $n\to\infty$. Then
 \[\limsup_{n\to\infty} \min_{x\in\overline{\Omega}}  \fint^{T_n}_{0} c(x,s)\,\mathrm{d}s = \sup_{\hat{c}\in\mathscr{C}_*} \min_{x\in\overline{\Omega}} \hat{c}(x)\quad\text{and} \quad\limsup_{T\to\infty} \min_{x\in\overline{\Omega}}  \fint_0^T c(x,s)\,\mathrm{d}s = \sup_{\hat{c}\in\mathscr{C}} \min_{x\in\overline{\Omega}} \hat{c}(x).\]
\end{lemma}
\begin{proof}
We proceed in two steps.

\medskip
\noindent\textbf{Step 1.} Show that
 \[\sup_{\hat{c}\in\mathscr{C}_*} \min_{x\in\overline{\Omega}} \hat{c}(x) \leq \limsup_{n\to\infty} \min_{x\in\overline{\Omega}} \hat{c}_{T_n}(x).\]
Take any $\hat{c} \in \mathscr{C}_*$. By the definition of the limit set $\mathscr{C}_*$, there exists a subsequence $\left\{T_{n'}\right\}_{n'\in\mathbb{N}^+}$ of $\left\{T_{n}\right\}_{n\in\mathbb{N}^+}$ such that
 \[\lim_{n'\to\infty} \|\hat{c}_{T_{n'}}(\cdot) - \hat{c}(\cdot)\|_{\infty} = 0.\]
Since the minimum is continuous under uniform convergence,
 \[\lim_{n'\to\infty} \min_{x\in\overline{\Omega}} \hat{c}_{T_{n'}}(x) = \min_{x\in\overline{\Omega}} \hat{c}(x).\]
Thus $\min_{x\in\overline{\Omega}} \hat{c}(x)$ is a subsequential limit of the family $\left\{\min_{x\in\overline{\Omega}} \hat{c}_{T_n}(x)\right\}_{n\in\mathbb{N}^+}$. By the definition of the limit superior, every subsequential limit does not exceed the limit superior, hence
 \[\min_{x\in\overline{\Omega}} \hat{c}(x) \leq \limsup_{n\to\infty} \min_{x\in\overline{\Omega}} \hat{c}_{T_n}(x).\]
Taking the supremum over all $\hat{c}\in\mathscr{C}_*$ gives the desired inequality.

\medskip
\noindent\textbf{Step 2.} Show that
 \[\limsup_{n\to\infty} \min_{x\in\overline{\Omega}} \hat{c}_{T_n}(x) \leq \sup_{\hat{c}\in\mathscr{C}_*} \min_{x\in\overline{\Omega}} \hat{c}(x).\]
By the definition of the limit superior, we can pick a subsequence of $\left\{T_{n}\right\}_{n\in\mathbb{N}^+}$, denoted still by $\left\{T_{n'}\right\}_{n'\in\mathbb{N}^+}$, with
 \[\lim_{n'\to\infty} \min_{x\in\overline{\Omega}} \hat{c}_{T_{n'}}(x) = \limsup_{n\to\infty} \min_{x\in\overline{\Omega}} \hat{c}_{T_n}(x).\]
The family $\left\{\hat{c}_{T_n}\right\}_{n\in\mathbb{N}^+}$ is easily seen to be uniformly bounded and equi-continuous in $C(\overline{\Omega})$. By the Arzel\`{a}-Ascoli theorem it is relatively compact; consequently the sequence $\left\{\hat{c}_{T_{n'}}\right\}_{n'\in\mathbb{N}^+}$ admits a subsequence $\left\{\hat{c}_{T_{n'_k}}\right\}_{k\in\mathbb{N}^+}$ that converges uniformly on $\overline{\Omega}$. Denote its limit by $\hat{c}$; then $\hat{c} \in \mathscr{C}_*$ by the definition of $\mathscr{C}_*$.
Again using the continuity of the minimum under uniform convergence, we obtain
 \[\limsup_{n\to\infty} \min_{x\in\overline{\Omega}} \hat{c}_{T_n}(x)=\lim_{k\to\infty} \min_{x\in\overline{\Omega}} \hat{c}_{T_{n'_k}}(x) = \min_{x\in\overline{\Omega}} \hat{c}(x)\leq \sup_{\hat{c}'\in\mathscr{C}_*} \min_{x\in\overline{\Omega}} \hat{c}'(x).\]
This is exactly the reverse inequality. Combining the two steps yields the required equality.

Arguing similarly, we can prove the second identity in the lemma. This ends the proof.
\end{proof}

\medskip
In general, the inequality
 \[\limsup_{n\to\infty} \min_{x\in\overline{\Omega}}  \fint^{T_n}_{0} c(x,s)\,\mathrm{d}s \le \min_{x\in\overline{\Omega}}\limsup_{n\to\infty} \fint^{T_n}_{0} c(x,s)\,\mathrm{d}s\]
 \[\left(resp.\qquad \limsup_{T\to\infty} \min_{x\in\overline{\Omega}}  \fint^{T}_{0} c(x,s)\,\mathrm{d}s \le \min_{x\in\overline{\Omega}}\limsup_{T\to\infty} \fint^{T}_{0} c(x,s)\,\mathrm{d}s\right)\]
typically holds, where equality may fail. On the other hand, when $\limsup$ is substituted by $\liminf$, the operations $\min$ and $\liminf$ commute, leading to the following result.

\begin{lemma}\label{lemma2.4}
Let $c\in\mathcal{C}$ and $\left\{T_n\right\}_{n\in\mathbb{N}^+}$ be a sequence satisfying $T_n\to\infty$ as $n\to\infty$. Then
  \[\min_{x\in\overline{\Omega}}\liminf_{n\to\infty}\fint^{T_n}_{0} c(x,s)\,\mathrm{d}s =\liminf_{n\to\infty}\min_{x\in\overline{\Omega}}\fint^{T_n}_{0} c(x,s)\,\mathrm{d}s=\inf_{\hat{c}\in\mathscr{C}_*}\min_{x\in\overline{\Omega}}\hat{c}(x)\]
and
   \[\min_{x\in\overline{\Omega}}\liminf_{T\to\infty}\fint^{T}_{0} c(x,s)\,\mathrm{d}s =\liminf_{T\to\infty}\min_{x\in\overline{\Omega}}\fint^{T}_{0} c(x,s)\,\mathrm{d}s=\inf_{\hat{c}\in\mathscr{C}}\min_{x\in\overline{\Omega}}\hat{c}(x).\]
\end{lemma}
\begin{proof}
Since for each $n\in\mathbb{N}^+$,
\[\fint^{T_n}_{0}c(x,s)\,\mathrm{d}s\ge \min\limits_{y\in\overline{\Omega}} \fint^{T_n}_{0}c(y,s)\,\mathrm{d}s,\quad \forall~x\in\overline{\Omega}, \]
we easily obtain
\[\min\limits_{x\in\overline{\Omega}} \liminf_{n\to\infty} \fint^{T_n}_{0}c(x,s)\,\mathrm{d}s\ge \liminf_{n\to\infty}\min\limits_{x\in\overline{\Omega}} \fint^{T_n}_{0}c(x,s)\,\mathrm{d}s.\]
So we need only to show the inverse inequality. Recall that the function $c\in \mathcal{C}$ satisfies
\[\sup_{t\in\mathbb{R}}[c(\cdot,t)]_\delta<\infty,\]
which implies directly that
\begin{equation}\label{eq2.1}
\left|\fint^{T_{n}}_{0}c(x,s)\,\mathrm{d}s-\fint^{T_{n}}_{0}c(y,s)\,\mathrm{d}s\right| =\left|\fint^{T_{n}}_{0}(c(x,s)-c(y,s))\,\mathrm{d}s\right|<C|x-y|^\delta,\quad \forall~x,y\in\overline{\Omega}
\end{equation}
for some constant $C>0$ independent of $n$. For any real number $\varepsilon>0$, there is a subsequence $\left\{T_{n_k}\right\}_{k\in\mathbb{N}^+}$ of $\left\{T_n\right\}_{n\in\mathbb{N}^+}$ such that
\[\min\limits_{x\in\overline{\Omega}} \fint^{T_{n_k}}_{0}c(x,s)\,\mathrm{d}s< \liminf_{n\to\infty}\min\limits_{x\in\overline{\Omega}} \fint^{T_{n}}_{0}c(x,s)\,\mathrm{d}s+\frac{\varepsilon}{2},\qquad\forall~k\in\mathbb{N}^+.\]
On the other hand, let $x_{k}\in\overline{\Omega}$, $k\in\mathbb{N}^+$ be points such that
\[\fint^{T_{n_k}}_{0}c(x_{k},s)\,\mathrm{d}s=\min\limits_{x\in\overline{\Omega}} \fint^{T_{n_k}}_{0}c(x,s)\,\mathrm{d}s,\quad k\in\mathbb{N}^+.\]
Due to the compactness of $\overline{\Omega}$, we may assume that $x_{k}\to x_*\in\overline{\Omega}$ as $k\to\infty$ by choosing a subsequence of $\left\{x_{k}\right\}_{k\in\mathbb{N}^+}$ if necessary.
So \eqref{eq2.1} implies that
\[ \left|\fint^{T_{n_k}}_{0}c(x_*,s)\,\mathrm{d}s-\fint^{T_{n_k}}_{0}c(x_{k},s)\,\mathrm{d}s\right|\le C|x_*-x_{k}|^\delta,\quad\forall~k\ge 1.\]
As a result, there is an integer $K_0>0$ such that
\[\fint^{T_{n_k}}_{0}c(x_*,s)\,\mathrm{d}s<\fint^{T_{n_k}}_{0}c(x_{k},s)\,\mathrm{d}s+\frac{\varepsilon}{2},\quad\forall~k>K_0.\]
In summary, one obtains that
\begin{align*}
        \min\limits_{x\in\overline{\Omega}}\liminf_{n\to\infty}\fint^{T_{n}}_{0} c(x,s)\,\mathrm{d}s
  \le & \liminf_{n\to\infty}\fint^{T_{n}}_{0} c(x_*,s)\,\mathrm{d}s \\
  \le & \liminf_{k\to\infty}\fint^{T_{n_k}}_{0} c(x_*,s)\,\mathrm{d}s \\
  \le & \liminf_{k\to\infty}\left(\fint^{T_{n_k}}_{0}c(x_{k},s)\,\mathrm{d}s+\frac{\varepsilon}{2}\right) \\
  \le & \liminf_{k\to\infty}\fint^{T_{n_k}}_{0}c(x_k,s)\,\mathrm{d}s+\frac{\varepsilon}{2} \\
  \le & \liminf_{n\to\infty}\min\limits_{x\in\overline{\Omega}} \fint^{T_{n}}_{0}c(x,s)\,\mathrm{d}s+\varepsilon.
\end{align*}
The arbitrariness of $\varepsilon>0$ implies that
\[\min\limits_{x\in\overline{\Omega}} \liminf_{n\to\infty} \fint^{T_n}_{0}c(x,s)\,\mathrm{d}s\le \liminf_{n\to\infty}\min\limits_{x\in\overline{\Omega}} \fint^{T_n}_{0}c(x,s)\,\mathrm{d}s.\]

In a manner analogous to the proof of Lemma \ref{lemma2.3}, one easily establishes that
  \[\liminf_{n\to\infty}\min_{x\in\overline{\Omega}}\fint^{T_n}_{0} c(x,s)\,\mathrm{d}s=\inf_{\hat{c}\in\mathscr{C}_*}\min_{x\in\overline{\Omega}}\hat{c}(x).\]
The last two identities can be proven similarly. This finishes the proof.
\end{proof}

\begin{remark}\label{remark2.5}{\rm
The crucial assumption of uniform H\"{o}lder continuity (w.r.t. spatial argument) in Lemmas \ref{lemma2.3} and \ref{lemma2.4} (i.e., $c \in \mathcal{C}$) ensures the equi-continuity of the families $\left\{c_{T_n}\right\}_{n\in\mathbb{N}^+}$ etc.. Hence, by an application of the Arzel\`{a}-Ascoli theorem, we are guaranteed the existence of a uniformly convergent subsequence. Without this uniform H\"{o}lder continuity, we would only obtain the trivial inequalities
 \[\sup_{\hat{c}\in\mathscr{C}_*} \min_{x\in\overline{\Omega}} \hat{c}(x)=\limsup_{n\to\infty}\min_{x\in\overline{\Omega}} \fint^{T_n}_{0}c(x,s)\,\mathrm{d}s\le \min_{x\in\overline{\Omega}}\limsup_{n\to\infty} \fint^{T_n}_{0}c(x,s)\,\mathrm{d}s \]
and
 \[\sup_{\hat{c}\in\mathscr{C}} \min_{x\in\overline{\Omega}} \hat{c}(x)=\limsup_{T\to\infty}\min_{x\in\overline{\Omega}} \fint^{T}_{0}c(x,s)\,\mathrm{d}s\le \min_{x\in\overline{\Omega}}\limsup_{T\to\infty} \fint^{T}_{0}c(x,s)\,\mathrm{d}s. \]
  }
\end{remark}

As direct corollary of Lemmas \ref{lemma2.3} and \ref{lemma2.4}, we have the following result.
\begin{corollary}\label{corollary2.6}
Let $c\in\mathcal{C}$ and $\left\{T_n\right\}_{n\in\mathbb{N}^+}$ be a sequence satisfying $T_n\to\infty$ as $n\to\infty$. If $\left\{\fint^{T_n}_{0}c(\cdot,s)\,\mathrm{d}s\right\}_{n\in\mathbb{N}^+}$ converges in $C(\overline{\Omega})$, then
  \[\lim_{n\to\infty}\min_{x\in\overline{\Omega}}\fint^{T_n}_{0}c(x,s)\,\mathrm{d}s= \min_{x\in\overline{\Omega}}\lim_{n\to\infty}\fint^{T_n}_{0}c(x,s)\,\mathrm{d}s.\]
\end{corollary}

In the following, we present the identities in spatial integral form.
\begin{lemma}\label{lemma2.7}
Let $c\in\mathcal{C}$ and $\left\{T_n\right\}_{n\in\mathbb{N}^+}$ be a sequence satisfying $T_n\to\infty$ as $n\to\infty$. Then
  \[\liminf_{n\to\infty}\fint^{T_n}_{0}\fint_{\Omega} c(x,s)\,\mathrm{d}x\mathrm{d}s =\inf_{\hat{c}\in\mathscr{C}_*}\fint_{\Omega}\hat{c}(x)\,\mathrm{d}x\quad\text{and}\quad \limsup_{n\to\infty}\fint^{T_n}_{0}\fint_{\Omega} c(x,s)\,\mathrm{d}x\mathrm{d}s =\sup_{\hat{c}\in\mathscr{C}_*}\fint_{\Omega}\hat{c}(x)\,\mathrm{d}x.\]
Furthermore, one has
 \[\liminf_{T\to\infty}\fint^{T}_{0}\fint_{\Omega} c(x,s)\,\mathrm{d}x\mathrm{d}s =\inf_{\hat{c}\in\mathscr{C}}\fint_{\Omega}\hat{c}(x)\,\mathrm{d}x\quad\text{and}\quad \limsup_{T\to\infty}\fint^{T}_{0}\fint_{\Omega} c(x,s)\,\mathrm{d}x\mathrm{d}s =\sup_{\hat{c}\in\mathscr{C}}\fint_{\Omega}\hat{c}(x)\,\mathrm{d}x.\]
\end{lemma}
\begin{proof}
By the compactness of $\mathscr{C}_*$ (see Lemma \ref{lemma2.1}), we choose a subsequence $\left\{T_{n'}\right\}_{n'\in\mathbb{N}^+}$ of $\left\{T_n\right\}_{n\in\mathbb{N}^+}$ such that $\left\{\fint^{T_{n'}}_{0}c(\cdot,s)\,\mathrm{d}s\right\}_{n'\in\mathbb{N}^+}$ converges in $C(\overline{\Omega})$ and
 \[\inf_{\hat{c}\in\mathscr{C}_*}\fint_{\Omega}\hat{c}(x)\,\mathrm{d}x=\fint_{\Omega}\lim_{n'\to\infty} \fint^{T_{n'}}_{0}c(x,s)\,\mathrm{d}s\mathrm{d}x=\lim_{n'\to\infty} \fint^{T_{n'}}_{0}\fint_{\Omega}c(x,s)\,\mathrm{d}s\mathrm{d}x,\]
where the Lebesgue dominated convergence theorem has been used for the second identity. Then it is easy to check that
 \[\liminf_{n\to\infty}\fint^{T_n}_{0}\fint_{\Omega} c(x,s)\,\mathrm{d}x\mathrm{d}s \le\inf_{\hat{c}\in\mathscr{C}_*}\fint_{\Omega}\hat{c}(x)\,\mathrm{d}x.\]
On the other hand, we can also take a subsequence of $\left\{T_n\right\}_{n\in\mathbb{N}^+}$, denoted still by $\left\{T_{n'}\right\}_{n'\in\mathbb{N}^+}$, such that
 \[\liminf_{n\to\infty}\fint^{T_n}_{0}\fint_{\Omega} c(x,s)\,\mathrm{d}x\mathrm{d}s=\lim_{n'\to\infty}\fint^{T_{n'}}_{0}\fint_{\Omega} c(x,s)\,\mathrm{d}x\mathrm{d}s.\]
Without loss of generality, we may assume that $\left\{\fint^{T_{n'}}_{0} c(\cdot,s)\,\mathrm{d}s\right\}_{n'\in\mathbb{N}^+}$ converges to some $\hat{c}'\in\mathscr{C}_*$. Using the Lebesgue dominated convergence theorem again gives
 \[\lim_{n'\to\infty}\fint^{T_{n'}}_{0}\fint_{\Omega} c(x,s)\,\mathrm{d}x\mathrm{d}s=\fint_{\Omega}\lim_{n'\to\infty}\fint^{T_{n'}}_{0} c(x,s)\,\mathrm{d}s\mathrm{d}x=\fint_{\Omega}\hat{c}'(x)\,\mathrm{d}x.\]
Combining with these identities yields
 \[\liminf_{n\to\infty}\fint^{T_n}_{0}\fint_{\Omega} c(x,s)\,\mathrm{d}x\mathrm{d}s=\fint_{\Omega}\hat{c}'(x)\,\mathrm{d}x.\]
As a result, we can obtain
 \[\liminf_{n\to\infty}\fint^{T_n}_{0}\fint_{\Omega} c(x,s)\,\mathrm{d}x\mathrm{d}s \ge\inf_{\hat{c}\in\mathscr{C}_*}\fint_{\Omega}\hat{c}(x)\,\mathrm{d}x.\]
The first identity in the lemma has been proven. Similarly, one can verify the remaining identities in the lemma. The proof is complete.
\end{proof}

\medskip
We now derive a useful result from the uniform H\"{o}lder condition \eqref{H3}.
\begin{lemma}\label{lemma2.8}
If the function $c\in\mathcal{C}$ satisfies the uniform H\"{o}lder condition \eqref{H3}, then the following limits
\[
  \lim_{\Delta t\to 0}\fint^{t+\Delta t}_{t}\min_{x\in\overline{\Omega}}c(x,s)\,\mathrm{d}s =\min_{x\in\overline{\Omega}}c(x,t)\quad
  \text{ and }\quad \lim_{\Delta t\to 0}\fint^{t+\Delta t}_{t}c(x,s)\,\mathrm{d}s= c(x,t),~\forall~x\in\overline{\Omega}
\]
 hold uniformly in $t\in\mathbb{R}$.
\end{lemma}
\begin{proof}
Since $c\in\mathcal{C}$ satisfies the uniform H\"{o}lder condition \eqref{H3},
  \[L=\max_{x\in\overline{\Omega}}\sup_{t,s\in\mathbb{R}}\frac{|c(x, t) - c(x, s)|}{|t - s|^{\frac{\delta}{2}}}<\infty,\]
we have
\[\left|\fint^{t+\Delta t}_{t}c(x,s)\,\mathrm{d}s-c(x,t)\right|\le \left|\fint^{t+\Delta t}_{t}[c(x,s)-c(x,t)]\,\mathrm{d}s\right|\le \frac{L}{1+\frac{\delta}{2}}|\Delta t|^{\frac{\delta}{2}}\le L |\Delta t|^{\frac{\delta}{2}}.\]
This proves that the second limit in the lemma uniformly holds in $t\in\mathbb{R}$. To show the first one, it needs only to show that $c_m(t):=\min_{x\in\overline{\Omega}}c(x,t)$ is also uniformly H\"{o}lder continuous in $t\in\mathbb{R}$. For this purpose, denote by $c_m(t):=c(x(t),t)=\min_{x\in\overline{\Omega}}c(x,t)$, $t\in\mathbb{R}$, where $x(t)$ is the point that $c$ attains its minimum at $t\in\mathbb{R}$. Then for any $t,s\in\mathbb{R}$, one has
\[c_m(t)-c_m(s)= c(x(t),t)-c_m(s)\ge c(x(t),t)-c(x(t),s)\ge -L|t-s|^{\frac{\delta}{2}}\]
and
\[c_m(t)-c_m(s)= c_m(t)-c(x(s),s)\le c(x(s),t)-c(x(s),s)\le L|t-s|^{\frac{\delta}{2}},\]
which follows that
 \[\frac{|c_m(t)-c_m(s)|}{|t-s|^{\frac{\delta}{2}}}<L,\qquad\forall~t,s\in\mathbb{R}.\]
This ends the proof.
\end{proof}

\medskip
Cantrell et al. \cite[Lemma B.1]{Cantrell2021Ideal} established the result below under the assumption of $T$-periodicity. However, it should be noted that a line-by-line examination of their proof reveals that periodicity is not a necessary condition for the equality to hold; it was merely a contextual assumption of the model they studied. Consequently, the result continues to hold without the periodicity assumption. We record the result, omitting the proof.
\begin{lemma}\label{lemma2.9}
Let $T>0$ and $c\in C(\overline{\Omega}\times[0,T])$ be given. Then
  \[\inf_{x(\cdot)\in C([0,T];\overline{\Omega})}\fint^{T}_{0} c(x(s),s)\,\mathrm{d}s=\fint^{T}_{0} \min_{x\in\overline{\Omega}}c(x,s)\,\mathrm{d}s.\]
\end{lemma}

\medskip
 The following lemma is a generalized version of the result \cite[Lemma B.1]{Cantrell2021Ideal} obtained by Cantrell et al.. We denote by $C(\mathbb{R};\overline{\Omega})$ the set of continuous curves from $\mathbb{R}$ to $\overline{\Omega}$.
\begin{lemma}\label{lemma2.10}
Assume that the function $c\in\mathcal{C}$ satisfies the uniform H\"{o}lder condition \eqref{H3}. Then
  \[\inf_{x(\cdot)\in C(\mathbb{R};\overline{\Omega})}\liminf_{T\to\infty}\fint^{T}_{0} c(x(s),s)\,\mathrm{d}s=\liminf_{T\to\infty}\fint^{T}_{0} \min_{x\in\overline{\Omega}}c(x,s)\,\mathrm{d}s\]
 and
  \[\inf_{x(\cdot)\in C(\mathbb{R};\overline{\Omega})}\limsup_{T\to\infty}\fint^{T}_{0} c(x(s),s)\,\mathrm{d}s=\limsup_{T\to\infty}\fint^{T}_{0} \min_{x\in\overline{\Omega}}c(x,s)\,\mathrm{d}s.\]
\end{lemma}
\begin{proof}
The proof proceeds in two steps.

\medskip
\noindent\textbf{Step 1.} For any continuous trajectory $x:~\mathbb{R}\to\overline{\Omega}$, we have
 \[c(x(t),t)\ge\min_{x\in\overline{\Omega}}c(x,t)=:c_m(t)\quad\forall~t\in\mathbb{R}.\]
For any given $T>0$, integrating the above inequality over $[0,T]$ and then dividing both sides by $T$ give
 \[\fint_{0}^{T}c(x(s),s)\,\mathrm{d}s\ge\fint_{0}^{T}c_m(s)\,\mathrm{d}s.\]
Taking $T\to\infty$ on both sides yields
 \[\liminf_{T\to\infty}\fint_{0}^{T}c(x(s),s)\,\mathrm{d}s\ge\liminf_{T\to\infty}\fint_{0}^{T} c_m(s)\,\mathrm{d}s\]
and
 \[\limsup_{T\to\infty}\fint_{0}^{T}c(x(s),s)\,\mathrm{d}s\ge\limsup_{T\to\infty}\fint_{0}^{T} c_m(s)\,\mathrm{d}s.\]
Since this holds for all continuous $x(\cdot)$, taking the infimum over such trajectories gives
  \[\inf_{x(\cdot)\in C(\mathbb{R};\overline{\Omega})}\liminf_{T\to\infty}\fint^{T}_{0} c(x(s),s)\,\mathrm{d}s\ge\liminf_{T\to\infty}\fint^{T}_{0} \min_{x\in\overline{\Omega}}c(x,s)\,\mathrm{d}s\]
and
  \[\inf_{x(\cdot)\in C(\mathbb{R};\overline{\Omega})}\limsup_{T\to\infty}\fint^{T}_{0} c(x(s),s)\,\mathrm{d}s\ge\limsup_{T\to\infty}\fint^{T}_{0} \min_{x\in\overline{\Omega}}c(x,s)\,\mathrm{d}s.\]
\medskip
\noindent\textbf{Step 2.} For any $\varepsilon>0$, we construct a continuous trajectory $x:~\mathbb{R}\to\overline{\Omega}$ such that
 \[\liminf_{T \to \infty} \fint_0^T c(x(s),s)\,\mathrm{d}s \le \liminf_{T\to\infty}\fint^{T}_{0} \min_{x\in\overline{\Omega}}c(x,s)\,\mathrm{d}s+\varepsilon\]
and
 \[\limsup_{T \to \infty} \fint_0^T c(x(s),s)\,\mathrm{d}s \le \limsup_{T\to\infty}\fint^{T}_{0} \min_{x\in\overline{\Omega}}c(x,s)\,\mathrm{d}s+\varepsilon.\]
Since $c\in \mathcal{C}$ is uniformly bounded and satisfies uniform H\"{o}lder condition \eqref{H3}, the function $c_m(t)=\min_{x \in \overline{\Omega}}c(x,t)$ is also uniformly H\"{o}lder continuous with H\"{o}lder coefficient $L$; see the proof of Lemma \ref{lemma2.8}.
Fix $\varepsilon>0$. Choose $\varrho>0$ and $\sigma\in(0,1)$ such that
\[L\varrho^{\frac{\delta}{2}} + \|c\|_\infty\sigma<\frac{\varepsilon}{2}.\]
Partition the time axis into intervals of length $\varrho$: $I_k=[k\varrho,(k+1)\varrho)$ for $k = 0, 1, 2, \cdots$. For each $k$, pick $x_k\in\overline{\Omega}$ such that
\[c(x_k,k\varrho) = c_m(k\varrho).\]
Since $\Omega$ is path-connected, we construct the continuous curve $x(\cdot)$ as follows: on each $I_k$, let
 \[
  x(t) =\left\{
   \begin{aligned}
    &x_k, &&t \in [k\varrho, (k+1)\varrho - \sigma\varrho), \\
    &\text{a continuous path from }x_k\text{ to } x_{k+1}, &&t \in[(k+1)\varrho-\sigma\varrho, (k+1)\varrho).
   \end{aligned}\right.
 \]
That is, during the first part of each interval, we stay at $x_k$; during the last part, we move continuously to the next point $x_{k+1}$. This yields a globally continuous trajectory. Now we estimate the integral. For a full interval $I_k = [k\varrho, (k+1)\varrho]$, on the stationary part $t\in[k\varrho,(k+1)\varrho-\sigma\varrho)$,
 \[|c(x(t), t) - c_m(t)| \le |c(x_k, t) - c(x_k, k\varrho)| + |c_m(k\varrho) -c_ m(t)| \le 2L\varrho^{\frac{\delta}{2}}.\]
On the moving part $t\in[(k+1)\varrho-\sigma\varrho,(k+1)\varrho]$, we have $|c(x(t), t) - c_m(t)| \le 2\|c\|_\infty$. The length of the moving part is $\sigma\varrho$. Hence, for any $k=0,1,2,\cdots$, one has
 \[\left|\int_{I_k}[c(x(s),s)-c_m(s)]\,\mathrm{d}s \right|\le2L\varrho^{\frac{\delta}{2}}\cdot(1-\sigma)\varrho +2\|c\|_\infty\cdot\sigma\varrho =2\left[L\varrho^{\frac{\delta}{2}}(1-\sigma)+\|c\|_\infty\sigma\right]\varrho.\]
Thus,
 \[\left|\fint_{I_k}c(x(s),s)\,\mathrm{d}s-\fint_{I_k}c_m(s)\,\mathrm{d}s\right|\le 2L\varrho^{\frac{\delta}{2}}(1-\sigma) +2\|c\|_\infty\sigma\le 2L\varrho^{\frac{\delta}{2}} + 2\|c\|_\infty\sigma,\quad k=0,1,2,\cdots.\]
Since $L\varrho^{\frac{\delta}{2}} + \|c\|_\infty\sigma \le\varepsilon/2$, we have
 \begin{equation}\label{eq2.2}
   \fint_{I_k}c(x(s),s)\,\mathrm{d}s\le\fint_{I_k}c_m(s)\,\mathrm{d}s+\varepsilon,\quad k=0,1,2,\cdots.
 \end{equation}

For an arbitrary $T>0$, write $T=\left\lfloor\frac{T}{\varrho}\right\rfloor\varrho+r$ and $0 \le r < \varrho$, where $\lfloor\cdot\rfloor$ stands the floor function (integer part of a positive real number). Then
 \[\int_{0}^{T}c(x(s),s)\,\mathrm{d}s=\sum_{k=0}^{\left\lfloor\frac{T}{\varrho}\right\rfloor-1}\int_{k\varrho}^{(k+1)\varrho}c(x(s),s)\,\mathrm{d}s +\int_{\left\lfloor\frac{T}{\varrho}\right\rfloor\varrho}^{T}c(x(s),s)\,\mathrm{d}s.\]
Using \eqref{eq2.2}, one has
 \[\sum_{k=0}^{\left\lfloor\frac{T}{\varrho}\right\rfloor-1}\int_{k\varrho}^{(k+1)\varrho}c(x(s),s)\,\mathrm{d}s\le\sum_{k=0}^{\left\lfloor\frac{T}{\varrho}\right\rfloor-1} \int_{k\varrho}^{(k+1)\varrho}c_m(s)\,\mathrm{d}s+\varepsilon\left\lfloor\frac{T}{\varrho}\right\rfloor\varrho= \int_{0}^{\left\lfloor\frac{T}{\varrho}\right\rfloor\varrho}c_m(s)\,\mathrm{d}s+\varepsilon\left\lfloor\frac{T}{\varrho}\right\rfloor\varrho.\]
We further have
\[\int_{\left\lfloor\frac{T}{\varrho}\right\rfloor\varrho}^{T}c(x(s),s)\,\mathrm{d}s\le \|c\|_\infty r < \|c\|_\infty\varrho.\]
Combining the two estimates above gives
 \begin{equation}\label{eq2.3}
   \fint_{0}^{T}c(x(s),s)\,\mathrm{d}s\le\frac{1}{T}\int_{0}^{\left\lfloor\frac{T}{\varrho}\right\rfloor \varrho} c_ m(s)\,\mathrm{d}s+\varepsilon\left\lfloor\frac{T}{\varrho}\right\rfloor\frac{\varrho}{T} +\frac{\|c\|_\infty\varrho}{T}.
 \end{equation}
Take $\liminf_{T \to \infty}$ and $\limsup_{T \to \infty}$, respectively. As $T \to \infty$, we have $\left\lfloor\frac{T}{\varrho}\right\rfloor \frac{\varrho }{T} \to 1$ and $\frac{\|c\|_\infty\varrho}{T}\to 0$. Moreover, since
 \[\frac{1}{T}\int_0^{\left\lfloor\frac{T}{\varrho}\right\rfloor\varrho}c_m(s)\,\mathrm{d}s =\fint_{0}^{T}c_m(s)\,\mathrm{d}s-\frac{1}{T}\int^{T}_{\left\lfloor\frac{T}{\varrho}\right\rfloor\varrho}c_m(s)\,\mathrm{d}s \le\fint_{0}^{T}c_m(s)\,\mathrm{d}s+\frac{\rho\|c\|_\infty}{T}\]
by letting $T\to\infty$ we have
 \[\liminf_{T\to\infty}\frac{1}{T}\int_0^{\left\lfloor\frac{T}{\varrho}\right\rfloor\varrho}c_m(s) \,\mathrm{d}s\le\liminf_{T\to\infty}\fint_{0}^{T}c_m(s)\,\mathrm{d}s\]
and
 \[\limsup_{T\to\infty}\frac{1}{T}\int_0^{\left\lfloor\frac{T}{\varrho}\right\rfloor\varrho}c_m(s) \,\mathrm{d}s\le\limsup_{T\to\infty}\fint_{0}^{T}c_m(s)\,\mathrm{d}s.\]
Hence, from \eqref{eq2.3}, one obtains
 \[\liminf_{T \to \infty}\fint_{0}^{T}c(x(s),s)\,\mathrm{d}s\le\liminf_{T\to\infty}\fint_{0}^{T} c_m(s)\,\mathrm{d}s+\varepsilon\]
and
 \[\limsup_{T \to \infty}\fint_{0}^{T}c(x(s),s)\,\mathrm{d}s\le\limsup_{T\to\infty}\fint_{0}^{T} c_m(s)\,\mathrm{d}s+\varepsilon.\]
Since $\varepsilon>0$ is arbitrary, we get
 \[\inf_{x(\cdot)\in C(\overline{\Omega};\mathbb{R})}\liminf_{T\to\infty}\fint^{T}_{0} c(x(s),s)\,\mathrm{d}s\le\liminf_{T\to\infty}\fint^{T}_{0} \min_{x\in\overline{\Omega}}c(x,s) \,\mathrm{d}s\]
and
 \[\inf_{x(\cdot)\in C(\overline{\Omega};\mathbb{R})}\limsup_{T\to\infty}\fint^{T}_{0} c(x(s),s)\,\mathrm{d}s\le\limsup_{T\to\infty}\fint^{T}_{0} \min_{x\in\overline{\Omega}}c(x,s) \,\mathrm{d}s.\]
This proves the lemma.
\end{proof}
\section{Comparison principle}
The parabolic comparison principle plays a central role in our analysis. This section is therefore devoted to developing a comparison principle and its variants suitable for long-time asymptotic regimes. First, we prepare a Hopf-type boundary lemma.
\begin{lemma}\label{lemma2.11}
Let $T>0$ be given, and $u\in C^{2,1}(\Omega\times(0,T])\cap C^{1,0}(\overline{\Omega}\times[0,T])$ satisfy $u(\cdot,0)\ge u(\cdot,T)$ and
  \[\omega\partial_t u-d\operatorname{div}(A(x)\nabla u)+c(x,t)u\ge 0,\quad(x,t)\in\Omega\times(0,T].\]
If $u>0$ in $\Omega\times(0,T]$ and $u=0$ at some point $(x_0,t_0)\in\partial\Omega\times(0,T]$, then
\[\left.\mathbf{n}\cdot (A(x)\nabla u)\right|_{(x_0,t_0)}<0,\]
where $\mathbf{n}$ is the outward unit normal vector at $(x_0,t_0)$.
\end{lemma}

It suffices to observe that at $(x_0,t_0)$, $\mathbf{n} = -\frac{\nabla u}{|\nabla u|}$ with $|\nabla u| > 0$, and then apply the uniform ellipticity of $A(\cdot)$ to obtain the lemma. The details are omitted.

Motivated by the work of Bai and He \cite[Proposition 2.4]{Bai2020Asymptotic}, we formulate the following comparison principle. Their proof carries over directly, requiring only that the Hopf lemma is replaced by Lemma \ref{lemma2.11}. The proof details are omitted here.
\begin{lemma}\label{lemma2.12}
  Let $T>0$, $\underline{\lambda}$, $\overline{\lambda}$ be given constants, and $\underline{\varphi}$, $\overline{\varphi}\in C^{2,1}(\Omega\times (0,T])\cap C^{1,0}(\overline{\Omega}\times [0,T])$ be given functions. Suppose that $\overline{\varphi}>0$ for all $(x,t)\in\overline{\Omega}\times[0,T]$ and $\underline{\varphi}\ge,\not\equiv0$ on $\overline{\Omega}\times[0,T]$. If the pairs $(\underline{\lambda},\underline{\varphi})$ and $(\overline{\lambda},\overline{\varphi})$ satisfy
  \begin{equation*}
    \left\{
    \begin{aligned}
      &\omega\partial_t\overline{\varphi}-d\operatorname{div}(A(x)\nabla\overline{\varphi})+c(x,t)\overline{\varphi}\ge\overline{\lambda}\overline{\varphi},&&(x,t)\in\Omega\times(0,T],\\
      &\omega\partial_t\underline{\varphi}-d\operatorname{div}(A(x)\nabla\underline{\varphi})+c(x,t)\underline{\varphi}\le\underline{\lambda}\underline{\varphi},&&(x,t)\in\Omega\times(0,T],\\
      &\mathbf{n}\cdot(A(x)\nabla\overline{\varphi})\ge0\ge \mathbf{n}\cdot(A(x)\nabla\underline{\varphi}),&&(x,t)\in\partial\Omega\times(0,T],\\
      &\overline{\varphi}(x,0)\ge\overline{\varphi}(x,T),~\underline{\varphi}(x,0)\le\underline{\varphi}(x,T),&&x\in\overline{\Omega},
    \end{aligned}
    \right.
  \end{equation*}
  then $\overline{\lambda}\le\underline{\lambda}$.
\end{lemma}

Now we establish the following comparison result, which plays an essential role in studying the asymptotic behavior of the normalized principal Floquet bundle.
\begin{lemma}\label{lemma2.13}
  Let $\left\{T_n\right\}_{n\in\mathbb{N}^+}$ be a sequence satisfying $T_n\to\infty$ as $n\to\infty$, $\left\{\underline{\lambda}_n\right\}_{n\in\mathbb{N}^+},\left\{\overline{\lambda}_n\right\}_{n\in\mathbb{N}^+}$ be given bounded sequences of constants, and $\left\{\underline{\varphi}_n\right\}_{n\in\mathbb{N}^+},\left\{\overline{\varphi}_n\right\}_{n\in\mathbb{N}^+}\subset C^{2,1}(\Omega\times (0,T_n])\cap C^{1,0}(\overline{\Omega}\times [0,T_n])$ be given sequences of functions. Suppose that for any $n\in\mathbb{N}^+$, $\overline{\varphi}_n>0$ for all $(x,t)\in\overline{\Omega}\times[0,T_n]$ and $\underline{\varphi}_n\ge,\not\equiv0$ on $\overline{\Omega}\times[0,T_n]$. If there exists a positive integer $N_0$ such that for all $n>N_0$, the pairs $(\underline{\lambda}_n,\underline{\varphi}_n)$ and $(\overline{\lambda}_n,\overline{\varphi}_n)$ satisfy
  \begin{equation*}
    \left\{
    \begin{aligned}
      &\omega\partial_t\overline{\varphi}_n-d\operatorname{div}(A(x)\nabla\overline{\varphi}_n)+c(x,t)\overline{\varphi}_n\ge\overline{\lambda}_n\overline{\varphi}_n,&&(x,t)\in\Omega\times(0,T_n],\\
      &\omega\partial_t\underline{\varphi}_n-d\operatorname{div}(A(x)\nabla\underline{\varphi}_n)+c(x,t)\underline{\varphi}_n\le\underline{\lambda}_n\underline{\varphi}_n,&&(x,t)\in\Omega\times(0,T_n],\\
      &\mathbf{n}\cdot(A(x)\nabla\overline{\varphi}_n)\ge0\ge\mathbf{n}\cdot(A(x)\nabla\underline{\varphi}_n),&&(x,t)\in\partial\Omega\times(0,T_n],\\
      &\overline{\varphi}_n(x,0)\ge\overline{\varphi}_n(x,T_n),~\underline{\varphi}_n(x,0)\le\underline{\varphi}_n(x,T_n),&&x\in\overline{\Omega},
    \end{aligned}
    \right.
  \end{equation*}
  then
  \[\liminf_{n\to\infty}\left(\underline{\lambda}_n-\overline{\lambda}_n\right)\ge0,\]
  whence
  \[\liminf_{n\to\infty}\underline{\lambda}_n\ge\liminf_{n\to\infty}\overline{\lambda}_n \quad\text{ and }\quad \limsup_{n\to\infty}\underline{\lambda}_n\ge\limsup_{n\to\infty}\overline{\lambda}_n.\]
\end{lemma}
\begin{proof}
Suppose to the contrary that there is a constant $\varepsilon>0$ such that
\[\liminf_{n\to\infty}\left(\underline{\lambda}_n-\overline{\lambda}_n\right)<-\varepsilon.\]
Then there exist subsequences $\left\{\underline{\lambda}_{n_k}\right\}_{k\in\mathbb{N}^+}$ of $\left\{\underline{\lambda}_{n}\right\}_{n\in\mathbb{N}^+}$, $\left\{\overline{\lambda}_{n_k}\right\}_{k\in\mathbb{N}^+}$ of $\left\{\overline{\lambda}_{n}\right\}_{n\in\mathbb{N}^+}$ such that
\[\lim_{k\to\infty}\left(\underline{\lambda}_{n_k}-\overline{\lambda}_{n_k}\right)= \liminf_{n\to\infty}\left(\underline{\lambda}_n-\overline{\lambda}_n\right)<-\varepsilon,\]
so there exists an integer $K_\varepsilon>N_0$ such that
\begin{equation}\label{eq2.4}
\underline{\lambda}_{n_k}-\overline{\lambda}_{n_k}<0,\qquad\forall~k>K_\varepsilon.
\end{equation}

For each $k>K_\varepsilon$, we define $w_{k}:=\overline{\varphi}_{n_k}-\sigma^{*}_{k}\underline{\varphi}_{n_k}$ on $\overline{\Omega}\times[0,T_{n_k}]$, where
 \[\sigma^{*}_{k}:=\sup\left\{\sigma\ge0~|~\overline{\varphi}_{n_k}-\sigma\underline{\varphi}_{n_k} >0,~(x,t)\in\overline{\Omega}\times[0,T_{n_k}]\right\}.\]
Obviously, $\sigma^{*}_{k}$ is well-defined and is strictly positive since $\overline{\varphi}_{n_k}>0$ for all $(x,t)\in\overline{\Omega}\times[0,T_{n_k}]$. Select a constant $M=\|c\|_\infty+\sup_{n\in\mathbb{N}^+}|\underline{\lambda}_{n}|+1$. Then one can check that for each $k>K_\varepsilon$, $w_{k}$ satisfies
  \begin{equation*}
    \left\{
    \begin{aligned}
      &\omega\partial_t w_{k}-d\operatorname{div}(A(x)\nabla w_{k})+[c(x,t)+M]w_{k} \\
      &\hskip 3cm \ge[\underline{\lambda}_{n_k}+M]w_{k} +\sigma^{*}_{k}\underbrace{\left(\overline{\lambda}_{n_k}-\underline{\lambda}_{n_k}\right)}_{>0,\text{ by }\eqref{eq2.4}}\overline{\varphi}_{n_k}>0,&&(x,t)\in\Omega\times(0,T_{n_k}],\\
      &\mathbf{n}\cdot(A(x)\nabla w_{k})\ge0,&&(x,t)\in\partial\Omega\times(0,T_{n_k}],\\
      &w_{k}(x,0)\ge w_{k}(x,T_{n_k}),&&x\in\overline{\Omega}.
    \end{aligned}
    \right.
  \end{equation*}
By the strong maximum principle for parabolic equations, we know that for each $k>K_\varepsilon$, $w_k>0$ in $\Omega\times(0,T_{n_k}]$. If $w_k$ vanishes at some boundary point $(x_0,t_0)\in\partial \Omega\times(0,T_{n_k}]$, then Lemma \ref{lemma2.11} implies that $\mathbf{n}\cdot(A(x_0)\nabla w_k(x_0,t_0))<0$, which contradicts the boundary condition satisfied by $w_k$. Consequently, for each $k>K_\varepsilon$, $w_k>0$ in $\overline{\Omega}\times(0,T_{n_k}]$. Using $w_{k}(\cdot,0)\ge w_{k}(\cdot,T_{n_k})$, we conclude $w_k>0$ in $\overline{\Omega}\times[0,T_{n_k}]$ for each $k>K_\varepsilon$. Then it is easy to show that for each $k>K_\varepsilon$, there is a $\delta'_{k}>0$ such that $\overline{\varphi}_{n_k}-(\sigma^{*}_{k}+\delta'_{k})\underline{\varphi}_{n_k}>0$, which is a contradiction to the definition of $\sigma^{*}_{k}$. This shows that
\[\liminf_{n\to\infty}\left(\underline{\lambda}_n-\overline{\lambda}_n\right)\ge0.\]

Recall the definition of $\liminf_{n\to\infty}$,
\begin{equation}\label{eq2.5}
\liminf_{n\to\infty}(\cdot)=\lim_{n'\to\infty}\inf_{n>n'} (\cdot).
\end{equation}
Note first that for any given $n'>0$, there holds
\[\underline{\lambda}_{n}-\overline{\lambda}_{n}\le \underline{\lambda}_{n}-\inf_{n>n'}\overline{\lambda}_{n},\quad\forall~n>n'.\]
Thus, taking the infimum over $n>n'$ yields
\[\inf_{n>n'}\left(\underline{\lambda}_{n}-\overline{\lambda}_{n}\right)\le \inf_{n>n'} \left(\underline{\lambda}_{n}-\inf_{n>n'}\overline{\lambda}_{n}\right)=\inf_{n>n'} \underline{\lambda}_{n}-\inf_{n>n'}\overline{\lambda}_{n}.\]
Letting $n'\to\infty$ and observing $\liminf_{n\to\infty}\left(\underline{\lambda}_n-\overline{\lambda}_n\right)\ge0$, one gets that
\[\liminf_{n\to\infty}\underline{\lambda}_n\ge\liminf_{n\to\infty}\overline{\lambda}_n,\]
where we have used \eqref{eq2.5}.

Observe that $\liminf_{n\to\infty}\left(\underline{\lambda}_n-\overline{\lambda}_n\right)\ge0$ implies
\[\limsup_{n\to\infty}\left(\overline{\lambda}_n-\underline{\lambda}_n\right)\le0.\]
Recall the definition of $\limsup_{n\to\infty}$,
\[
\limsup_{n\to\infty}(\cdot)=\lim_{n'\to\infty}\sup_{n>n'} (\cdot).
\]
Obviously, for any given $n'>0$, there holds
\[\overline{\lambda}_{n}-\underline{\lambda}_{n}\ge \overline{\lambda}_{n}-\sup_{n>n'}\underline{\lambda}_{n},\quad\forall~n>n'.\]
Take the supremum over $n>n'$ to obtain
\[\sup_{n>n'}\left(\overline{\lambda}_{n}-\underline{\lambda}_{n}\right)\ge \sup_{n>n'} \left(\overline{\lambda}_{n}-\sup_{n>n'}\underline{\lambda}_{n}\right)=\sup_{n>n'} \overline{\lambda}_{n}-\sup_{n>n'}\underline{\lambda}_{n}.\]
By sending $n'\to\infty$, we find that
\[\limsup_{n\to\infty}\overline{\lambda}_n\le\limsup_{n\to\infty}\underline{\lambda}_n.\]

This finishes the proof.
\end{proof}

\medskip
It can be assumed that one of the sequences $\left\{\underline{\lambda}_n\right\}_{n\in\mathbb{N}^+}$ and $\left\{\overline{\lambda}_n\right\}_{n\in\mathbb{N}^+}$ converges as $n\to\infty$.
\begin{corollary}\label{corollary2.14}
  Let $\left\{T_n\right\}_{n\in\mathbb{N}^+}$ be a sequence satisfying $T_n\to\infty$ as $n\to\infty$, $\left\{\underline{\lambda}_n\right\}_{n\in\mathbb{N}^+},\left\{\overline{\lambda}_n\right\}_{n\in\mathbb{N}^+}$ be given bounded sequences of constants, and $\left\{\underline{\varphi}_n\right\}_{n\in\mathbb{N}^+},\left\{\overline{\varphi}_n\right\}_{n\in\mathbb{N}^+}$ be given sequences of functions $\underline{\varphi}_n,\overline{\varphi}_n\in C^{2,1}(\Omega\times (0,T_n])\cap C^{1,0}(\overline{\Omega}\times [0,T_n])$ for each $n\in\mathbb{N}^+$. Suppose that for any $n\in\mathbb{N}^+$, $\overline{\varphi}_n>0$ for all $(x,t)\in\overline{\Omega}\times[0,T_n]$ and $\underline{\varphi}_n\ge,\not\equiv0$ on $\overline{\Omega}\times[0,T_n]$, and moreover, the pairs $(\underline{\lambda}_n,\underline{\varphi}_n)$ and $(\overline{\lambda}_n,\overline{\varphi}_n)$ satisfy
  \begin{equation*}
    \left\{
    \begin{aligned}
      &\omega\partial_t\overline{\varphi}_n-d\operatorname{div}(A(x)\nabla\overline{\varphi}_n)+c(x,t)\overline{\varphi}_n\ge\overline{\lambda}_n\overline{\varphi}_n,&&(x,t)\in\Omega\times(0,T_n],\\
      &\omega\partial_t\underline{\varphi}_n-d\operatorname{div}(A(x)\nabla\underline{\varphi}_n)+c(x,t)\underline{\varphi}_n\le\underline{\lambda}_n\underline{\varphi}_n,&&(x,t)\in\Omega\times(0,T_n],\\
      &\mathbf{n}\cdot(A(x)\nabla\overline{\varphi}_n)\ge0\ge\mathbf{n}\cdot(A(x)\nabla\underline{\varphi}_n),&&(x,t)\in\partial\Omega\times(0,T_n],\\
      &\overline{\varphi}_n(x,0)\ge\overline{\varphi}_n(x,T_n),~\underline{\varphi}_n(x,0)\le\underline{\varphi}_n(x,T_n),&&x\in\overline{\Omega}
    \end{aligned}
    \right.
  \end{equation*}
  for all $n>N_0$, where $N_0$ is some positive integer. Then
  \begin{align*}
     & \limsup_{n\to\infty}\underline{\lambda}_n\ge\liminf_{n\to\infty}\underline{\lambda}_n\ge\lim_{n\to\infty}\overline{\lambda}_n,&&\text{ if the limit } \lim_{n\to\infty}\overline{\lambda}_n \text{ exists};\\
     & \lim_{n\to\infty}\underline{\lambda}_n\ge\limsup_{n\to\infty}\overline{\lambda}_n \ge\liminf_{n\to\infty}\overline{\lambda}_n,&&\text{ if the limit } \lim_{n\to\infty}\underline{\lambda}_n \text{ exists}.
  \end{align*}
\end{corollary}

By selecting suitable sequence $\left\{T_n\right\}_{n\in\mathbb{N}^+}$ of $T\to\infty$ such that
 \[\lim_{n\to\infty}(\underline{\lambda}_{T_n}-\overline{\lambda}_{T_n})=\liminf_{T\to\infty}(\underline{\lambda}_T-\overline{\lambda}_T) \]
 \[\left(resp.~\lim_{n\to\infty}\underline{\lambda}_{T_n} =\liminf_{T\to\infty}\underline{\lambda}_T\text{ or }\lim_{n\to\infty}\overline{\lambda}_{T_n}=\limsup_{T\to\infty}\overline{\lambda}_T\right),\]
and then applying Lemma \ref{lemma2.13} and Corollary \ref{corollary2.14}, one has the following result.
\begin{corollary}\label{corollary2.15}
  Let $\left\{\underline{\lambda}_T\right\}_{T>0},\left\{\overline{\lambda}_T\right\}_{T>0}$ be two given bounded families of constants, and $\left\{\underline{\varphi}_T\right\}_{T>0}$, $\left\{\overline{\varphi}_T\right\}_{T>0}$ be two given families of functions $\underline{\varphi}_T,\overline{\varphi}_T\in C^{2,1}(\Omega\times (0,T])\cap C^{1,0}(\overline{\Omega}\times [0,T])$ for all $T>0$. Suppose that for any $T>0$, $\overline{\varphi}_T>0$ for all $(x,t)\in\overline{\Omega}\times[0,T]$ and $\underline{\varphi}_T\ge,\not\equiv0$ on $\overline{\Omega}\times[0,T]$. If there exists a positive constant $T'$ such that for any $T>T'$, the pairs $(\underline{\lambda}_T,\underline{\varphi}_T)$ and $(\overline{\lambda}_T,\overline{\varphi}_T)$ satisfy
  \begin{equation*}
    \left\{
    \begin{aligned}
      &\omega\partial_t\overline{\varphi}_T-d\operatorname{div}(A(x)\nabla\overline{\varphi}_T)+c(x,t)\overline{\varphi}_T\ge\overline{\lambda}_T\overline{\varphi}_T,&&(x,t)\in\Omega\times(0,T],\\
      &\omega\partial_t\underline{\varphi}_T-d\operatorname{div}(A(x)\nabla\underline{\varphi}_T)+c(x,t)\underline{\varphi}_T\le\underline{\lambda}_T\underline{\varphi}_T,&&(x,t)\in\Omega\times(0,T],\\
      &\mathbf{n}\cdot(A(x)\nabla\overline{\varphi}_T)\ge0\ge\mathbf{n}\cdot(A(x)\nabla\underline{\varphi}_T),&&(x,t)\in\partial\Omega\times(0,T],\\
      &\overline{\varphi}_T(x,0)\ge\overline{\varphi}_T(x,T),~\underline{\varphi}_T(x,0)\le\underline{\varphi}_T(x,T),&&x\in\overline{\Omega},
    \end{aligned}
    \right.
  \end{equation*}
  then
  \[\liminf_{T\to\infty}(\underline{\lambda}_T-\overline{\lambda}_T)\ge0,\]
  whence
  \[\liminf_{T\to\infty}\underline{\lambda}_T\ge\liminf_{T\to\infty}\overline{\lambda}_T \quad\text{ and }\quad \limsup_{T\to\infty}\underline{\lambda}_T\ge\limsup_{T\to\infty}\overline{\lambda}_T.\]
\end{corollary}

\section{Several elliptic eigenvalue problems}
In forthcoming chapters, we will construct suitable sub- and super-solutions and using comparison principle to derive the asymptotic behavior of the normalized principal Floquet exponent with respect to $(\omega,d)$. This section is devoted to the study of some elliptic eigenvalue problems, and we state several useful properties. In fact, some of these results have already been proven and widely used; see \cite{Liu2025Asymptotic,Liu2022Classifying,Evans1989A,Fleming1986PDE,Koike1987An}.

Consider first the following elliptic eigenvalue problem
\begin{equation}\label{eq2.6}
  \left\{
  \begin{aligned}
    &-\operatorname{div}(A(x)\nabla\phi)=\nu\phi,&&x\in\Omega,\\
    &\phi=0,&&x\in\partial\Omega,\\
    &\|\phi\|_\infty=1.
  \end{aligned}
  \right.
\end{equation}
It is well-known that problem \eqref{eq2.6} admits a unique positive principal eigenvalue $\nu^*$ which depends only on $A(\cdot)$ and domain $\Omega$. Furthermore, the corresponding positive principal eigenfunction $\phi^*\in C^{2}(\overline{\Omega})$ is also unique and depends only on matrix $A(\cdot)$ and domain $\Omega$; see, for example, \cite{Strauss2007Partial}.

Recall that for $c\in\mathcal{C}\cap C^{2,1}(\overline{\Omega}\times\mathbb{R})$ and $T>0$,
\[\mathcal{M}(x,t;T)=t\fint^{T}_{0}c(x,s)\,\mathrm{d}s-\int^{t}_{0}c(x,s)\,\mathrm{d}s,\qquad (x,t)\in\overline{\Omega}\times[0,T].\]
\begin{lemma}\label{lemma2.16}
Let $c\in\mathcal{C}\cap C^{2,1}(\overline{\Omega}\times\mathbb{R})$ be given. If assumption \eqref{H2} holds, namely,
  \[\limsup\limits_{T\to\infty}\sup\limits_{(x,t)\in\Omega\times[0,T]}(|\operatorname{div}(A(x)\nabla\mathcal{M}(x,t;T))|+|\nabla \mathcal{M}(x,t;T)\cdot (A(x)\nabla \mathcal{M}(x,t;T))|)<\infty,\]
then there exist constants $T_0>0$ and $\kappa>0$ such that for any $T>T_0$, there holds
  \[\mathbf{n}\cdot (A(x)\nabla \mathcal{M}(x,t;T))-\kappa\mathbf{n}\cdot(A(x)\nabla \phi^*(x))\ge 0,\qquad (x,t)\in\partial\Omega\times[0,T],\]
where $\phi^*$ is the principal eigenfunction of \eqref{eq2.6}.
\end{lemma}
\begin{proof}
By Lemma \ref{lemma2.11}, we see that the outward normal derivative $\mathbf{n}\cdot(A(x)\nabla\phi^*)$ is strictly negative, namely, there is a constant $\epsilon>0$ such that $\min_{\partial\Omega}(-\mathbf{n}\cdot(A(x)\nabla\phi^*))>\epsilon$. By assumption \eqref{H2}, there are positive constants $T_0$ and $C$ such that
\begin{equation*}
\sup_{(x,t)\in\Omega\times[0,T]}(|\operatorname{div}(A(x)\nabla\mathcal{M}(x,t;T))|+|\nabla \mathcal{M}(x,t;T)\cdot (A(x)\nabla \mathcal{M}(x,t;T))|)<C,\quad\forall~T>T_0.
\end{equation*}
Therefore, one can find a constant $\kappa=C/\epsilon+1$ such that for any $T>T_0$,
\[-\kappa\mathbf{n}\cdot(A(x)\nabla\phi^*)\ge \sup\limits_{(x,t)\in\partial\Omega\times[0,T]} |\mathbf{n}\cdot(A(x)\nabla\mathcal{M}(x,t;T))|, \quad\forall~x\in\partial\Omega,~t\in[0,T],\]
This finishes the proof.
\end{proof}

\medskip
We now turn to consider the Neumann elliptic eigenvalue problem \eqref{eq1.3},
  \begin{equation*}
    \left\{
    \begin{aligned}
      &-d\operatorname{div}(A(x)\nabla\hat{\phi})+m(x)\hat{\phi}=\mu\hat{\phi},&&x\in\Omega,\\
      &\mathbf{n}\cdot(A(x)\nabla\hat{\phi})=0, &&x\in\partial\Omega,
    \end{aligned}
    \right.
  \end{equation*}
where $m\in C(\overline{\Omega})$. The principal eigenvalue $\mu^\infty $ is bounded, namely, $|\mu^\infty |\le\|m\|_\infty$. Using the variational characterization of the principal eigenvalue, the following result has been readily obtained.
\begin{lemma}\label{lemma2.17}
  The principal eigenvalue $\mu^\infty (d,m)$ of \eqref{eq1.3} is continuous in $m$. Moreover, if $m_1\le m_2$ on $\overline{\Omega}$, then
  \[\mu^\infty (d,m_1)\le \mu^\infty (d,m_2),\]
  the equality holds if and only if $m_1\equiv m_2$.
\end{lemma}

\begin{lemma}\label{lemma2.18}
Let $\mathcal{S}\subset C(\overline{\Omega})$ be a non-empty bounded and compact subset, and let $\mu^\infty (d,m)$ be the principal eigenvalue of \eqref{eq1.3} with potential $m\in\mathcal{S}$. Then $\lim_{d\to 0}\mu^\infty (d,m)=\min_{\overline{\Omega}}m$, and
 \[ \lim_{d\to0}\inf_{{m}\in \mathcal{S}}\mu^{\infty}(d,m) =\inf_{{m}\in \mathcal{S}}\min_{x\in\overline{\Omega}}m(x) \quad\text{and}\quad \lim_{d\to0}\sup_{{m}\in \mathcal{S}}\mu^{\infty}(d,m)=\sup_{{m}\in \mathcal{S}}\min_{x\in\overline{\Omega}}m(x).\]
\end{lemma}
\begin{proof}
The proof proceeds in three steps.

\medskip
\noindent\textbf{Step 1.} Small-diffusion limit for a single potential.

Although the proof follows a standard pattern (see, for example, \cite[Lemma 3.1]{Lou2006Evolution}), we include a full account for the sake of completeness. The principal eigenvalue $\mu^\infty(d,m)$ admits the Rayleigh quotient characterization:
 \[\mu^{\infty}(d,m)= \min_{\phi\in H^{1}(\Omega),\phi\neq0}  \frac{d\int_{\Omega}\nabla\phi\cdot (A(x)\nabla\phi)\,\mathrm{d}x +\int_{\Omega}m(x)\phi^{2}\,\mathrm{d}x}{\int_{\Omega}\phi^{2}\,\mathrm{d}x}.\]
For any test function $\phi\neq0$,
 \[\int_{\Omega}m(x)\phi^{2}\,\mathrm{d}x\ge\left(\min_{x\in\overline{\Omega}}{m}(x)\right)\int_{\Omega}\phi^{2}\,\mathrm{d}x,\]
hence
 \[\mu^{\infty}(d,{m})\ge\min_{x\in\overline{\Omega}}m(x),\qquad\forall~d>0.\]
On the other hand, given $\varepsilon>0$, choose a ball $B_{r_0}(x_0)\subset\Omega$ centered at $x_0\in\Omega$ with radius $r_0>0$ such that
 \[m(x)\le\min_{x\in\overline{\Omega}}{m}(x)+\frac{\varepsilon}{2},\quad\forall~x\in\overline{B}_{r_0}(x_0).\]
For any $r\in(0,r_0)$, take a smooth cut-off function $\phi_{r}\in C_{c}^{\infty}(B(x_{0},r))$ satisfying
 \[0\le\phi_{r}\le1\text{ in }B(x_0,r),\quad \phi_{r}\equiv1\text{ in }B(x_{0},r/2),\quad\text{ and } |\nabla\phi_{r}|\le C/r,\]
where the constant $C>0$ depends only on the dimension. For $r$ sufficiently small, $|B(x_{0},r)|\ge |B_1|r^{N}$, where $B_1$ is unit ball in $\mathbb{R}^N$. Substituting $\phi_{r}$ into the Rayleigh quotient and using the uniform ellipticity give
\begin{equation*}
\begin{aligned}
 \mu^{\infty}(d,{m}) \le&\frac{d\Lambda\int_{B_r(x_{0})}|\nabla\phi_{r}|^{2}\,\mathrm{d}x+ \int_{B_r(x_{0})}m(x)\phi_{r}^{2}\,\mathrm{d}x}{\int_{B_r(x_{0})}\phi^{2}_{r}\,\mathrm{d}x}  \\
 \le&\frac{d\Lambda\int_{B_r(x_{0})}|\nabla\phi_{r}|^{2}\,\mathrm{d}x+\left(\min_{x\in\overline{\Omega}} m(x)+ \frac{\varepsilon}{2}\right)\int_{B_r(x_{0})}\phi_{r}^{2}\,\mathrm{d}x}{\int_{B_r(x_{0})}\phi^{2}_{r}\,\mathrm{d}x} \\
 \le&\frac{d\Lambda\,C^{2}|B_r(x_{0})|\,r^{-2}}{|B_{r/2}(x_{0})|}+\min_{x\in\overline{\Omega}} m(x)+ \frac{\varepsilon}{2}  \\
 =&O(dr^{-2})+\min_{x\in\overline{\Omega}} m(x)+ \frac{\varepsilon}{2}.
\end{aligned}
\end{equation*}
Choosing $r=d^{1/4}$ yields $O(dr^{-2})=O(d^{1/2})\to0$ as $d\to0$. Consequently, there exists $d_{\varepsilon}>0$ such that for all $d\in(0,d_{\varepsilon})$,
 \[\mu^{\infty}(d,{m})\le\min_{x\in\overline{\Omega}}{m}(x)+\varepsilon.\]
Combining the lower bound and upper bound, we obtain, for each fixed ${m}\in C(\overline{\Omega})$,
 \[\lim_{d\to0}\mu^{\infty}(d,{m})= \min_{x\in\overline{\Omega}}{m}(x).\]

\medskip
\noindent\textbf{Step 2.} Uniform convergence on the compact set.

We now show that $\mu^{\infty}(d,\cdot)$ converges uniformly on the compact set $\mathcal{S}$ to the function $m\mapsto\min_{x\in\overline{\Omega}}{m}(x)$, i.e.,
 \[\lim_{d\to0}\sup_{{m}\in \mathcal{S}}\left|\,\mu^{\infty}(d,{m})-\min_{x\in\overline{\Omega}}{m}(x)\right|=0. \]
It is easy to verify the hypotheses of Dini's theorem:
\begin{enumerate}[{\rm (a)}]
  \item $\mathcal{S}$ is a compact subset of $C(\overline{\Omega})$, and hence $(S,\|\cdot\|_{\infty})$ is a compact metric space;
  \item For each fixed $d>0$, the map $m\mapsto\mu^{\infty}(d,m)$ is continuous;
  \item The limit function $m\mapsto\min_{x\in\overline{\Omega}}{m}(x)$ is also continuous because of
\[\left|\min_{x\in\overline{\Omega}}{m}_{1}(x)-\min_{x\in\overline{\Omega}}{m}_{2}(x)\right|\le\|{m}_{1}-{m}_{2}\|_{\infty};\]
  \item For each fixed ${m}\in \mathcal{S}$, the function $d\mapsto\mu^{\infty}(d,m)$ is monotonically increasing.
\end{enumerate}
Dini's theorem therefore guarantees that on the compact space $\mathcal{S}$ the monotonically increasing family of continuous functions $\mu^{\infty}(d,\cdot)$ converges uniformly to the continuous limit function $m\mapsto\min_{x\in\overline{\Omega}}{m}(x)$. Hence the uniform convergence holds.

\medskip
\noindent\textbf{Step 3.} Uniform convergence implies interchange of extrema.

Set
 \[E(d):=\sup_{m\in\mathcal{S}}\left|\,\mu^{\infty}(d,{m})-\min_{x\in\overline{\Omega}}{m}(x)\right|.\]
For any $m\in \mathcal{S}$, $\mu^{\infty}(d,m)\le\min_{x\in\overline{\Omega}}m(x)+E(d)$. Taking the supremum over $m\in \mathcal{S}$ gives
 \[\sup_{m\in\mathcal{S}}\mu^{\infty}(d,{m})\le\sup_{m\in \mathcal{S}}\min_{x\in\overline{\Omega}}m(x)+E(d).\]
Letting $d\to0$ we obtain
 \[\limsup_{d\to0}\sup_{m\in \mathcal{S}}\mu^{\infty}(d,m)\le\sup_{m\in \mathcal{S}}\min_{x\in\overline{\Omega}}{m}(x). \]
On the other hand, for any $\varepsilon>0$ choose $m_{\varepsilon}\in \mathcal{S}$ such that
 \[\min_{x\in\overline{\Omega}}{m}_{\varepsilon}(x)\ge\sup_{m\in \mathcal{S}}\min_{x\in\overline{\Omega}}{m}(x)-\varepsilon.\]
Using the lower bound of $\mu^{\infty}(d,m_{\varepsilon})$, we have
 \[\mu^{\infty}(d,{m}_{\varepsilon})\ge\min_{x\in\overline{\Omega}}m_{\varepsilon}(x)-E(d)
\ge\sup_{{m}\in \mathcal{S}}\min_{x\in\overline{\Omega}}{m}(x)-\varepsilon-E(d).\]
Hence
 \[\sup_{m\in\mathcal{S}}\mu^{\infty}(d,{m})\ge\mu^{\infty}(d,m_{\varepsilon})\ge\sup_{m\in \mathcal{S}}\min_{x\in\overline{\Omega}}{m}(x)-\varepsilon-E(d).\]
Letting $d\to0$ yields
 \[\liminf_{d\to0}\sup_{m\in\mathcal{S}}\mu^{\infty}(d,m)\ge\sup_{m\in\mathcal{S}}\min_{x\in\overline{\Omega}}m(x)-\varepsilon.\]
Since $\varepsilon>0$ is arbitrary, we have
 \[\liminf_{d\to0}\sup_{m\in\mathcal{S}}\mu^{\infty}(d,m)\ge\sup_{m\in\mathcal{S}}\min_{x\in\overline{\Omega}}m(x).\]
Combining the lower bound and upper bound we finally obtain
 \[\lim_{d\to0}\sup_{m\in\mathcal{S}}\mu^{\infty}(d,m)=\sup_{m\in\mathcal{S}}\min_{x\in\overline{\Omega}}{m}(x).\]

Analogously, for any $m\in\mathcal{S}$, $\mu^{\infty}(d,{m})\ge\min_{x\in\overline{\Omega}}m(x)-E(d)$. Taking the infimum over $m\in\mathcal{S}$ gives
 \[\inf_{m\in\mathcal{S}}\mu^{\infty}(d,{m})\ge\inf_{m\in\mathcal{S}}\min_{x\in\overline{\Omega}}{m}(x)-E(d).\]
Letting $d\to0$ yields
 \[\liminf_{d\to0}\inf_{m\in\mathcal{S}}\mu^{\infty}(d,{m})\ge\inf_{m\in\mathcal{S}}\min_{x\in\overline{\Omega}}{m}(x). \]
For any $\varepsilon>0$, choose $m_{\varepsilon}\in \mathcal{S}$ such that
 \[\min_{x\in\overline{\Omega}}{m}_{\varepsilon}(x)<\inf_{m\in\mathcal{S}}\min_{x\in\overline{\Omega}}{m}(x)+\varepsilon.\]
Then
 \[\mu^{\infty}(d,{m}_{\varepsilon})\le\min_{x\in\overline{\Omega}}{m}_{\varepsilon}(x)+E(d) <\inf_{m\in\mathcal{S}}\min_{x\in\overline{\Omega}}{m}(x)+\varepsilon+E(d),\]
and therefore
 \[\inf_{m\in\mathcal{S}}\mu^{\infty}(d,{m})\le\mu^{\infty}(d,{m}_{\varepsilon})<\inf_{m\in\mathcal{S}}\min_{x\in\overline{\Omega}}{m}(x)+\varepsilon+E(d).\]
Letting $d\to0$ gives
 \[\limsup_{d\to0}\inf_{m\in\mathcal{S}}\mu^{\infty}(d,{m})\le\inf_{m\in\mathcal{S}}\min_{x\in\overline{\Omega}}{m}(x)+\varepsilon.\]
The arbitrariness of $\varepsilon$ implies
 \[\limsup_{d\to0}\inf_{m\in\mathcal{S}}\mu^{\infty}(d,{m})\le\inf_{m\in\mathcal{S}}\min_{x\in\overline{\Omega}}{m}(x).\]
From the lower bound and upper bound of $\limsup_{d\to0}\inf_{m\in\mathcal{S}}\mu^{\infty}(d,{m})$ we conclude
 \[\lim_{d\to0}\inf_{m\in\mathcal{S}}\mu^{\infty}(d,{m})=\inf_{m\in\mathcal{S}}\min_{x\in\overline{\Omega}}{m}(x).\]
Thus both identities of the theorem are proven.
\end{proof}

\medskip
We now prepare some other properties of the principal eigenpair of \eqref{eq1.3}.
\begin{lemma}\label{lemma2.19}
Let $\left\{m_n\right\}_{n\in\mathbb{N}^+}\subset C(\overline{\Omega})$ be a sequence of functions, and let $(\mu_n,\hat\phi_n)$ be the normalized principal eigenpair of problem \eqref{eq1.3} corresponding to the potential $m_n$, $n\in\mathbb{N}^+$. If $m_n(\cdot)\to m(\cdot)$ in $C(\overline{\Omega})$ as $n\to\infty$, then
 \[\lim_{n\to\infty}\mu_n = \mu^\infty,\]
and for every $0<\alpha<1$,
 \[\lim_{n\to\infty}\big\|\hat\phi_n - \hat\phi\big\|_{C^{1,\alpha}(\overline{\Omega})} = 0,\]
where $(\mu^\infty,\hat\phi)$ is the normalized principal eigenpair of problem \eqref{eq1.3} corresponding to the potential $m$.
\end{lemma}
\begin{proof}
This conclusion is a standard result in the theory of elliptic eigenvalue problems; for completeness we give a brief proof here. Standard continuous dependence results for eigenvalue problems imply that $\mu_n\to\mu^\infty$ and $\hat\phi_n\to\hat\phi$ in $L^2(\Omega)$ as $n\to\infty$. We now show that $\hat\phi_n$ converges uniformly in $C^{1,\alpha}(\overline{\Omega})$ for any $0<\alpha<1$. Since $\left\{m_n\right\}_{n\in\mathbb{N}^+}$ converges to $m$ in $C(\overline{\Omega})$, the sequence $\left\{m_n\right\}$ is uniformly bounded in $L^\infty(\Omega)$. Together with the bound $|\mu_n|\le\|m_n\|_{L^\infty(\Omega)}$ for the principal eigenvalue, there exists a constant $C>0$ such that
 \[\|\mu_n - m_n\|_{L^\infty(\Omega)}\le C\quad\text{for all }n\in\mathbb{N}^+.\]
Consider the elliptic boundary value problem
 \[-d\,\mathrm{div}\big(A(x)\nabla\hat\phi_n\big) = \big(\mu_n - m_n(x)\big)\hat\phi_n,\quad x\in\Omega,\qquad  \mathbf{n}\cdot\big(A(x)\nabla\hat\phi_n\big)=0,\quad x\in\partial\Omega.\]
By the De Giorgi-Nash-Moser estimates, $\hat{\phi}_n$ is uniformly bounded in $L^\infty(\Omega)$ with respect to $n$. Hence the right-hand side of the equation is uniformly bounded in $L^\infty(\Omega)$. For any $1<p<\infty$, the $W^{2,p}$ regularity theory for divergence-form elliptic equations with continuous coefficients yields that $\hat{\phi}_n$ is uniformly bounded in $W^{2,p}(\Omega)$. By the Sobolev's embedding theorem, $\hat{\phi}_n$ is uniformly bounded in $C^{1,\alpha}(\overline{\Omega})$ for every $0<\alpha<1$. Consequently, $\left\{\hat\phi_n\right\}_{n\in\mathbb{N}^+}$ is precompact in $C^{1,\alpha}(\overline{\Omega})$. Together with the $L^2$ convergence of $\hat{\phi}_n$ and the uniqueness of the limit, we obtain
 \[\lim_{n\to\infty}\|\hat\phi_n - \hat\phi\|_{C^{1,\alpha}(\overline{\Omega})}=0\qquad\forall~\alpha\in(0,1).\]
This finishes the proof.
\end{proof}

\begin{lemma}\label{lemma2.20}
 Let $\hat{\phi}$ be the normalized principal eigenfunction of \eqref{eq1.3} with $m\in C^1(\overline{\Omega})$. Then there is a constant $C>0$, depending on $\Omega$, $A$ and $\|c\|_{C^1(\overline{\Omega})}$ but not on $d\in(0,1)$ such that
  \[\max_{\overline{\Omega}}|\ln \hat{\phi}|\le \frac{C}{\sqrt{d}}.\]
\end{lemma}
\begin{proof}
In fact, by Harnack's inequality for elliptic equation (see, for example, \cite{Gilbarg2015Elliptic}), there is a constant $C >0$ independent of $d>0$ such that $\sup_{\Omega}\hat{\phi}\le e^{\frac{C }{\sqrt{d}}}\inf_{\Omega}\hat{\phi}$. Since
\[0<\inf_{\Omega}\hat{\phi}\le\hat{\phi} \le \sup_{\Omega}\hat{\phi}=1,\quad\forall~x\in\overline{\Omega},\]
we get
\[1\le e^{\frac{C }{\sqrt{d}}} \inf_{\Omega}\hat{\phi}\le e^{\frac{C }{\sqrt{d}}} \hat{\phi} \le e^{\frac{C }{\sqrt{d}}},\quad\forall~x\in\overline{\Omega},\]
or equivalently,
\[0\le \frac{C }{\sqrt{d}}+\ln\hat{\phi}\le \frac{C }{\sqrt{d}},\quad\forall~x\in\overline{\Omega}.\]
Then it is easy to obtain that
\[0\le -\sqrt{d}\ln\hat{\phi}\le C ,\quad\forall~x\in\overline{\Omega}.\]
The proof is complete.
\end{proof}

\medskip
We consider a one-parameter family of elliptic problems. For each $t\in \mathbb{R}$, let $\mu^{0}(t)$ be the principal eigenvalue of problem
\begin{equation}\label{eq2.7}
  \left\{
  \begin{aligned}
    &-d\operatorname{div}(A(x)\nabla\phi^{0})+c(x,t)\phi^{0}=\mu(t)\phi^{0},&&x\in\Omega,\\
    &\mathbf{n}\cdot(A(x)\nabla\phi^{0})=0,&&x\in\partial\Omega,\\
    &\|\phi^{0}(\cdot,t)\|_\infty=1,
  \end{aligned}
  \right.
\end{equation}
where $c\in \mathcal{C}$. In view of the fact that the constant $C$ in Lemma \ref{lemma2.20} depends on the norm of potential $c(\cdot,t)$ and $\mu^0(t)$ is uniformly bounded in $t\in\mathbb{R}$, we can obtain a uniform estimate of $\phi^{0}(\cdot,t)$ with respect to $t\in\mathbb{R}$.
\begin{lemma}\label{lemma2.21}
  Assume that the function $c\in \mathcal{C}\cap C^{1,0}(\overline{\Omega}\times\mathbb{R})$. For each $t\in\mathbb{R}$, let $(\mu^0(t),\phi^{0}(x,t))$ be the principal eigenpair of \eqref{eq2.7}. Then there is a constant $C>0$, depending on $\Omega$, $A$ and $\sup_{t\in\mathbb{R}}\|c(\cdot,t)\|_{C^1(\overline{\Omega})}$ but not on $d\in(0,1)$ such that
  \[\sup_{t\in\mathbb{R}}\max_{x\in\overline{\Omega}}|\ln \phi^{0}(x,t)|\le \frac{C}{\sqrt{d}}.\]
\end{lemma}

If the diffusion rate $d>0$ is fixed, we have the following estimate about the normalized principal eigenfunction.
\begin{lemma}\label{lemma2.22}
Assume that $c\in \mathcal{C}\cap C^{1,1}(\overline{\Omega}\times\mathbb{R})$ satisfies $\sup_{t\in\mathbb{R}}\|\partial_tc(\cdot,t)\|_\infty<\infty$. For each $t\in\mathbb{R}$, let $(\mu^0(t),\phi^{0}(\cdot,t))$ be the principal eigenpair of \eqref{eq2.7}. Then $\sup_{t\in\mathbb{R}}|\mu^0(t)|<\infty$ and
\[\inf_{t\in\mathbb{R}}\inf_{x\in\Omega}\phi^{0}(x,t)>\frac{1}{C_0}\quad\text{ and }\quad\sup_{(x,t)\in\Omega\times\mathbb{R}}\left|\frac{\partial_t\phi^{0}(x,t)}{\phi^{0}(x,t)}\right|<C_0\]
for some constant $C_0>1$.
\end{lemma}
\begin{proof}
It is obvious that
 \[\sup_{t\in\mathbb{R}}|\mu^0(t)|\le \sup_{(x,t)\in\Omega\times\mathbb{R}}|c(x,t)|.\]

We first clarify the normalization notations and justify the continuous differentiability of the principal eigenpair with respect to $t$. Let $\tilde{\phi}^0(\cdot,t)$ denote the $L^2$-normalized principal eigenfunction: $\|\tilde{\phi}^0(\cdot,t)\|_{L^2(\Omega)} = 1$; let $\phi^0(\cdot,t)$ denote the $L^\infty$-normalized principal eigenfunction in the original lemma: $\|\phi^0(\cdot,t)\|_{L^\infty(\Omega)} = 1$.

For each fixed $t$, the operator
 \[L_t=-d\,\operatorname{div}(A(x)\nabla\cdot)+c(x,t)\]
with Neumann boundary condition has a simple principal eigenvalue $\mu^0(t)$. Indeed, if $\phi_1,\phi_2>0$ were two linearly independent principal eigenfunctions, then the ratio $\phi_1/\phi_2$ would attain its maximum at some interior point, where $\nabla(\phi_1/\phi_2)=0$; substituting into the eigenvalue equation and using the strong maximum principle forces $\phi_1/\phi_2\equiv\text{const}$, a contradiction. Hence $\ker(L_t-\mu^0(t)I)=\operatorname{span}\{\tilde{\phi}^0(\cdot,t)\}$ is one-dimensional.

Define the map
 \[F:\mathbb{R}\times H^1(\Omega)\times\mathbb{R}\longrightarrow H^{-1}(\Omega)\times\mathbb{R},\]
 \[F(t,\phi,\mu)=\bigl(L_t\phi-\mu\phi,\ \|\phi\|_{L^2(\Omega)}^2-1\bigr).\]
Its Fr\'{e}chet derivative with respect to $(\phi,\mu)$ is
 \[D_{(\phi,\mu)}F\cdot(\psi,\nu)=\bigl(L_t\psi-\mu\psi-\nu\phi,\ 2\langle\phi,\psi\rangle_{L^2}\bigr).\]
We verify that this derivative is an isomorphism at $(\phi,\mu)=(\tilde{\phi}^0(t),\mu^0(t))$. Suppose $D_{(\phi,\mu)}F\cdot(\psi,\nu)=(0,0)$. Then
 \[L_t\psi-\mu^0(t)\psi=\nu\tilde{\phi}^0(\cdot,t),\qquad \langle\tilde{\phi}^0(\cdot,t),\psi\rangle_{L^2}=0.\]
Taking the $L^2$-inner product of the first equation with $\tilde{\phi}^0$, the self-adjointness of $L_t$ gives
 \[0=\langle L_t\psi-\mu^0(t)\psi,\tilde{\phi}^0\rangle=\nu\|\tilde{\phi}^0\|_{L^2}^2=\nu,\]
so $\nu=0$. Thus $\psi\in\ker(L_t-\mu^0(t) I)=\operatorname{span}\{\tilde{\phi}^0\}$, and $\langle\tilde{\phi}^0,\psi\rangle=0$ yields $\psi=0$. Therefore $D_{(\phi,\mu)}F$ is injective; since it is a Fredholm operator of index zero, it is bijective, hence an isomorphism.

By the regularity assumption $c\in C^{1,1}(\overline{\Omega}\times\mathbb{R})$ (in particular, $c$ is $C^1$ in $t$), the map $F$ is $C^1$ in all variables. The implicit function theorem then implies that the solution $(\tilde{\phi}^0(\cdot,t),\mu^0(t))$ of $F(t,\phi,\mu)=0$ is continuously differentiable in $t$. Consequently, both $\mu^0_t(t)$ and $\partial_t\tilde{\phi}^0(x,t)$ exist and are continuous on $\overline{\Omega}\times\mathbb{R}$.

We now return to the $L^\infty$-normalized eigenfunction $\phi^0(\cdot,t)$. By definition,
 \[\phi^0(x,t) = \frac{\tilde{\phi}^0(x,t)}{\|\tilde{\phi}^0(\cdot,t)\|_{L^\infty(\Omega)}}.\]
By Lemma \ref{lemma2.21} (the uniform Harnack estimate), $\tilde{\phi}^0(x,t)$ has a uniform positive lower bound over $\overline{\Omega}\times\mathbb{R}$, so the denominator $\|\tilde{\phi}^0(\cdot,t)\|_{L^\infty(\Omega)}$ is uniformly bounded away from zero and is continuously differentiable in $t$. Therefore $\phi^0(x,t)$ is also continuously differentiable in $t$, and $\partial_t\phi^0(x,t)$ exists and is continuous on $\overline{\Omega}\times\mathbb{R}$.

Using Lemma \ref{lemma2.21}, we easily show that for fixed $d>0$ there is a constant $C_0>1$ such that
 \[\inf_{t\in\mathbb{R}}\inf_{x\in\Omega}\phi^{0}(x,t)>\frac{1}{C_0}.\]
We need only to show that there is a positive constant $C'$, such that
 \[\sup_{(x,t)\in\Omega\times\mathbb{R}}\left|\partial_t\phi^{0}(x,t)\right|<C'.\]
Differentiating \eqref{eq2.7} with respect to $t$, we have
\[
\left\{
\begin{aligned}
  &-d\operatorname{div}(A(x)\nabla \partial_t\phi^{0})+(c(x,t)-\mu^0(t)) \partial_t\phi^{0}=\left[\mu^{0}_{t}(t)-c_t(x,t)\right]\phi^{0},&&x\in\Omega,\\
  &\mathbf{n}\cdot(A(x)\nabla\partial_t\phi^{0})=0,&&x\in\partial\Omega,
\end{aligned}
\right.
\]
where $\mu^{0}_{t}(t)=\frac{\mathrm{d}}{\mathrm{d}t}\mu^0(t)$ and $c_t(x,t)=\partial_t c(x,t)$, $x\in\Omega$, $t\in\mathbb{R}$. Multiply it by $\phi^{0}$ and integrate the resulting identity by parts over $\Omega$ to obtain
\[
\mu^{0}_{t}(t)=\frac{\int_{\Omega}c_t(x,t)(\phi^{0}(x,t))^2\,\mathrm{d}x}{\int_{\Omega}(\phi^{0}(x,t))^2\,\mathrm{d}x},
\]
from this we assert that $\mu^{0}_{t}(t)$ is uniformly bounded due to the uniform boundedness of $c_t$ on $\overline{\Omega}\times\mathbb{R}$.

Therefore, $[\mu^{0}_{t}(t)-c_t(x,t)]\phi^{0}$ is uniformly bounded on $\overline{\Omega}\times\mathbb{R}$. By the standard arguments of regularity for elliptic equation and applying Sobolev's embedding theorem, one easily derives the uniform boundedness of $\partial_t\phi^{0}$; namely, there is a positive constant $C'>0$ such that
\[
\sup_{(x,t)\in\Omega\times\mathbb{R}}|\partial_t\phi^{0}(x,t)|<C'.
\]
Consequently, there is a constant, denoted still by $C_0>0$, such that
\[
\sup_{(x,t)\in\Omega\times\mathbb{R}}\left|\frac{\partial_t\phi^{0}(x,t)}{\phi^{0}(x,t)}\right|<C_0.
\]
This finishes the proof.
\end{proof}

\medskip
We consider another one-parameter family of elliptic problems. For given $\varrho\in(0,1)$ and $t\in \mathbb{R}$, let $\mu^\infty (d,t,\varrho)$ be the principal eigenvalue of problem
\begin{equation}\label{eq2.8}
  \left\{
  \begin{aligned}
    &-d\operatorname{div}(A(x)\nabla\phi)+\fint^{t+\varrho}_{t}c(x,s)\,\mathrm{d}s\,\phi=\mu\phi,&&x\in\Omega,\\
    &\mathbf{n}\cdot(A(x)\nabla\phi)=0,&&x\in\partial\Omega,
  \end{aligned}
  \right.
\end{equation}
where $c\in \mathcal{C}$.
\begin{lemma}\label{lemma2.23}
Let $c\in\mathcal{C}$; denote by $\mu^\infty (d,t,\varrho)$ the principal eigenvalue of \eqref{eq2.8}. Then
  $$\lim_{d\to 0}\mu^\infty (d,t,\varrho)=\min_{x\in\overline{\Omega}}\fint^{t+\varrho}_{t}c(x,s)\,\mathrm{d}s\quad\text{holds uniformly in }(t,\varrho)\in\mathbb{R}\times(0,1).$$
\end{lemma}
\begin{proof}
Recall that $c\in\mathcal{C}$. We denote by $L$ the uniform H\"{o}lder constant (w.r.t. spatial argument) of $c$, namely,
 \[L=\sup_{t\in\mathbb{R}} [c(\cdot,t)]_{\delta;\Omega}.\]
If $L=0$, the claim is immediate. From now on, we assume that $L>0$ in this proof. Denote by $V(x;t,\varrho):=\fint^{t+\varrho}_{t}c(x,s)\,\mathrm{d}s$, $t\in\mathbb{R}$, $\varrho\in(0,1)$. Note that the principal eigenvalue $\mu^\infty (d,t,\varrho)$ possesses a variational formulation
  \[\mu^\infty (d,t,\varrho)=\inf_{w\in H^1(\Omega)}\frac{\int_\Omega[d\nabla w\cdot(A(x)\nabla w)+V(x;t,\varrho)w^2]\,\mathrm{d}x}{\int_\Omega w^2\,\mathrm{d}x},\qquad t\in\mathbb{R},~\varrho\in(0,1).\]
From this formulation, one easily obtains that
  $$\mu^\infty (d,t,\varrho)\ge\min_{\overline{\Omega}}\fint^{t+\varrho}_{t}c(x,s)\,\mathrm{d}s,\quad\forall~(t,\varrho)\in\mathbb{R}\times(0,1).$$
Thus, it suffices to derive a uniform upper bound for $\mu^\infty (d,t,\varrho)$. For given $t\in\mathbb{R}$ and $\varrho\in(0,1)$, assume that $V(x;t,\varrho)$ attains its minimum at $x_\varrho(t)$. Let $\eta\in C^{\infty}_{c}(\mathbb{R}^N)$ be a standard cut-off function supported in the unit ball. Namely, $\eta$ satisfies
 \[\eta=0\text{ in }\mathbb{R}^N\setminus B_1,\quad \eta=1\text{ in } B_{1/2},\quad 0\le\eta\le1\text{ and }|\nabla\eta|\le 4\text{ in }\mathbb{R}^N,\]
where $B_1$ and $B_{1/2}$ are balls centered at the origin with radius $1$ and $1/2$ respectively. For any $\varepsilon>0$, define $r_\varepsilon=(\varepsilon/L)^{1/\delta}$, where $L$ is the uniform H\"{o}lder constant of $c(x,s)$ in $x$. Note that $r_\varepsilon$ is independent of both $t$ and $\varrho$. For any $(t,\varrho)\in\mathbb{R}\times(0,1)$, there exists a ball $B_{r_\varepsilon}(x_\varrho(t))$ centered at $x_\varrho(t)\in\overline{\Omega}$ such that
\[V(x;t,\varrho)\leq\min_{x\in\overline{\Omega}}V(x;t,\varrho)+\varepsilon, \quad \forall~x\in B_{r_\varepsilon}(x_\varrho(t))\cap\Omega.\]
This holds for all $t$ and $\varrho$ with the same $r_\varepsilon$, since $\varrho\in(0,1)$ implies that $V$ has the same uniform H\"{o}lder constant $L$ in $x$.
Since the boundary $\partial\Omega$ is smooth (thus, it satisfies the uniform interior sphere property), we find a point $\tilde{x}_{\varrho}(t)\in B_{r_\varepsilon}(x_\varrho(t))\cap\Omega$ and constant $\tilde{r}_\varepsilon\in(0,r_\varepsilon)$ such that
 \[B_{\tilde{r}_\varepsilon}(\tilde{x}_{\varrho}(t))\subset B_{r_\varepsilon}(x_\varrho(t))\cap\Omega.\]
For any $r\in(0,\tilde{r}_\varepsilon)$, we define
\[w(x;t,r)=\eta\left(\frac{x-\tilde{x}_\varrho(t)}{r}\right),\qquad x\in\overline{\Omega}.\]
Obviously, one has
\[|\nabla w|\le \frac{4}{r},\quad\forall~x\in\overline{\Omega}.\]
By the variational formulation for $\mu^\infty (d,t,\varrho)$, we have
 \begin{equation}\label{eq2.9}
   \begin{aligned}
     \mu^\infty (d,t,\varrho)\le & \frac{d\int_\Omega\nabla w\cdot(A(x)\nabla w)\,\mathrm{d}x}{\int_\Omega w^2\,\mathrm{d}x}+\frac{\int_{\Omega}V(x;t,\varrho)w^2\,\mathrm{d}x}{\int_\Omega w^2\,\mathrm{d}x} \\
      \le & \min_{x\in\overline{\Omega}}V(x;t,\varrho)+ \frac{16\Lambda|B_r|}{|B_{r/2}|} \frac{d}{r^2} +\varepsilon \\
      = & \min_{x\in\overline{\Omega}}V(x;t,\varrho)+ 2^{N+4}\Lambda\frac{d}{r^2} +\varepsilon,&&r\in(0,\tilde{r}_\varepsilon).
   \end{aligned}
 \end{equation}
Fix $r\in(0,r_\varepsilon)$ and choose $d_\varepsilon>0$ such that $2^{N+4}\Lambda\frac{d}{r^2}<\varepsilon$. Hence, we obtain a uniform estimate for $\mu^\infty (d,t,\varrho)$:
 \[0\le\mu^\infty(d,t,\varrho)-\min_{x\in\overline{\Omega}}V(x;t,\varrho)\le\varepsilon.\]
This completes the proof.
\end{proof}

\medskip
Similarly, and in fact more straightforwardly, we can prove the following result.
\begin{lemma}\label{lemma2.24}
Let $\mu^0(t;d)$ be the principal eigenvalue of \eqref{eq1.4} with potential $c\in\mathcal{C}$ for each $t$. Then
  $$\lim_{d\to 0}\mu^0 (t;d)=\min_{x\in\overline{\Omega}}c(x,t)\quad\text{holds uniformly in }t\in\mathbb{R}.$$
\end{lemma}

As can be seen from the proof of Lemma \ref{lemma2.23}, the uniform convergence of $\mu^0$ in $t$ is essentially the uniform convergence of the principal eigenvalue of the elliptic problem with respect to the potential function, and this is ultimately guaranteed by the uniform boundedness of the potential. The following lemma for the large-diffusion regime likewise confirms this viewpoint.
\begin{lemma}\label{lemma2.25}
Let $\mathcal{S}\subset C(\overline{\Omega})$ be a non-empty bounded and compact subset, and let $\mu^\infty (d,m)$ be the principal eigenvalue of \eqref{eq1.3} with potential $m\in\mathcal{S}$. Then the limit
  \[\lim_{d\to\infty}\mu^\infty(d, m )=\fint_\Omega m (x)\,\mathrm{d}x\quad\text{ holds uniformly in }m\in\mathcal{S},\]
  where $\mu^\infty(d, m )$ is the principal eigenvalue of \eqref{eq1.3} with potential $ m $.
\end{lemma}
\begin{proof}
We first prove
\[\mu^\infty(d, m )\le\fint_\Omega m (x)\,\mathrm{d}x,\qquad\forall~d>0,~m\in\mathcal{S}.\]
This inequality follows directly from the variational principle by choosing the constant test function. Take $\phi\equiv1$; clearly $\phi\in H^1(\Omega)$ and satisfies the homogeneous Neumann condition. Substituting into the Rayleigh quotient gives
\[\mu^\infty(d, m )\le\frac{d\int_\Omega\nabla\phi\cdot (A(x)\nabla\phi)\,\mathrm{d}x+\int_\Omega m (x)\phi^2\,\mathrm{d}x}{\int_\Omega\phi^2\,\mathrm{d}x} = \frac{\int_\Omega m (x)\,\mathrm{d}x}{|\Omega|} = \fint_\Omega m (x)\,\mathrm{d}x,~\forall~d>0,~m\in\mathcal{S}.\]

We now show that there exists a constant $C>0$, independent of $m\in\mathcal{S}$, such that
\[\fint_\Omega m (x)\,\mathrm{d}x-\mu^\infty(d, m )\le\frac{C}{d}\]
for all sufficiently large $d>0$. Let $\phi$ be the normalized principal eigenfunction corresponding to $\mu^\infty(d, m )$, i.e., $\fint_\Omega\phi^2\,\mathrm{d}x=1$. Then $(\mu^\infty(d, m ),\phi)$ satisfies
\begin{equation}\label{eq2.10}
\mu^\infty(d, m )=d\fint_\Omega\nabla\phi\cdot\left(A(x)\nabla\phi\right)\,\mathrm{d}x+\fint_\Omega m (x)\phi^2\,\mathrm{d}x.
\end{equation}
Denote $\bar\phi=\fint_\Omega\phi\,\mathrm{d}x$ and set $\tilde{\phi}(x)=\phi(x)-\bar\phi$, $x\in\overline{\Omega}$. By construction, $\int_\Omega\tilde{\phi}\,\mathrm{d}x=0$. Expanding the normalization condition $\fint_\Omega\phi^2\,\mathrm{d}x=1$ gives
\[1=\fint_\Omega\phi^2\,\mathrm{d}x = \fint_\Omega(\bar\phi+\tilde{\phi})^2\,\mathrm{d}x = \bar\phi^2+\fint_\Omega\tilde{\phi}^2\,\mathrm{d}x,\]
hence
\begin{equation}\label{eq2.11}
  \bar\phi^2\le1,\qquad \fint_\Omega\tilde{\phi}^2\,\mathrm{d}x\le 1.
\end{equation}
By the uniform ellipticity of $A(\cdot)$, there holds
\[\fint_\Omega \nabla\phi\cdot(A(x)\nabla\phi)\,\mathrm{d}x = \fint_\Omega \nabla\tilde{\phi}\cdot(A(x)\nabla\tilde{\phi})\,\mathrm{d}x \ge \Lambda^{-1}\fint_\Omega\left|\nabla\tilde{\phi}\right|^2\,\mathrm{d}x.\]
Combining with the Poincar\'{e}'s inequality  $\left\|\tilde{\phi}\right\|_{L^2(\Omega)}^{2}\le C_P^{-1}\int_\Omega\left|\nabla\tilde{\phi}\right|^2\,\mathrm{d}x$ with constant $C_P>0$, we obtain
\begin{equation}\label{eq2.12}
  d\int_\Omega \nabla\phi\cdot(A(x)\nabla\phi)\,\mathrm{d}x \ge \frac{d}{\Lambda}\fint_\Omega\left|\nabla\tilde{\phi}\right|^2\,\mathrm{d}x \ge \frac{dC_P}{\Lambda|\Omega|}\left\|\tilde{\phi}\right\|_{L^2}^2.
\end{equation}
Decompose $ m $ into its mean and a mean-zero part
\[ m (\cdot)=\fint_\Omega m (x)\,\mathrm{d}x+\tilde{ m }(\cdot),\qquad \fint_\Omega \tilde{ m }(x)\,\mathrm{d}x=0.\]
Because $m\in\mathcal{S}$ is uniformly bounded, there exists a constant $M>0$ such that for every $m\in\mathcal{S}$,
 \[\left|\fint_\Omega m (x)\,\mathrm{d}x\right|\le M,\qquad \left\|\tilde{ m }\right\|_{\infty}\le2M.\]
Inserting this decomposition into the potential term of \eqref{eq2.10} gives
\begin{equation}\label{eq2.13}
 \int_\Omega m \phi^2\,\mathrm{d}x = \fint_\Omega m (x)\,\mathrm{d}x+\int_\Omega \tilde{ m }(x)\phi^2\,\mathrm{d}x.
\end{equation}
Estimate the term containing $\tilde{ m }$. Using $\phi=\bar{\phi}+\tilde{\phi}$, we have
\begin{equation}\label{eq2.14}
 \begin{aligned}
  \left|\fint_\Omega \tilde{ m }(x)\phi^2\,\mathrm{d}x\right| &=\left|\fint_\Omega \tilde{ m }(x)(\bar\phi+\tilde{\phi})^2\,\mathrm{d}x\right| \\
  &\le 2|\bar\phi|\left|\fint_\Omega \tilde{ m }(x)\tilde{\phi}\,\mathrm{d}x\right|    +\left|\fint_\Omega \tilde{ m }(x)\tilde{\phi}^2\,\mathrm{d}x\right| \\
  &\le 2|\bar\phi|\|\tilde{ m }(\cdot)\|_{L^\infty(\Omega)}|\Omega|^{-1}\left\|\tilde{\phi}\right\|_{L^1(\Omega)}    +\|\tilde{ m }(\cdot)\|_{L^\infty(\Omega)}|\Omega|^{-1}\left\|\tilde{\phi}\right\|_{L^2(\Omega)}^2 \\
  &\le 2M\left(2|\bar\phi||\Omega|^{-1/2}\left\|\tilde{\phi}\right\|_{L^2(\Omega)}+|\Omega|^{-1} \left\|\tilde{\phi}\right\|_{L^2(\Omega)}^2\right) \\
  &=4M|\bar\phi||\Omega|^{-1/2}\left\|\tilde{\phi}\right\|_{L^2(\Omega)}+2M|\Omega|^{-1}\left\|\tilde{\phi}\right\|_{L^2(\Omega)}^2,
 \end{aligned}
\end{equation}
where the Cauchy-Schwarz inequality has been used to obtain the last inequality. Substituting \eqref{eq2.12}, \eqref{eq2.13} and \eqref{eq2.14} back into \eqref{eq2.10} yields
\begin{equation}\label{eqq002.16}
 \mu^\infty(d, m ) \ge \fint_\Omega m (x)\,\mathrm{d}x    +\left(\frac{dC_P}{\Lambda}-2M\right)|\Omega|^{-1}\left\|\tilde{\phi}\right\|_{L^2(\Omega)}^2    -4M|\bar\phi||\Omega|^{-1/2}\left\|\tilde{\phi}\right\|_{L^2(\Omega)}.
\end{equation}
Applying Young's inequality to the last term gives that for any $\varepsilon>0$,
 \[4M|\bar\phi||\Omega|^{-1/2}\left\|\tilde{\phi}\right\|_{L^2(\Omega)} \le \varepsilon\fint_\Omega\tilde{\phi}^2\,\mathrm{d}x+\frac{4M^2\bar\phi^2}{\varepsilon}.\]
Choose $\varepsilon=\dfrac{dC_P}{2\Lambda}$. Then for $d\ge\dfrac{4\Lambda M}{C_P}$ we have
 \[\frac{dC_P}{\Lambda}-2M-\varepsilon = \frac{dC_P}{2\Lambda}-2M\ge0.\]
Now \eqref{eqq002.16} becomes
 \[\mu^\infty(d, m ) \ge \fint_\Omega m (x)\,\mathrm{d}x    -\frac{4M^2\bar\phi^2}{\varepsilon}.\]
Replacing $\varepsilon$ by $\dfrac{dC_P}{2\Lambda}$ and using $\bar\phi^2\le1$ (from \eqref{eq2.11}) give
 \[\frac{4M^2\bar\phi^2}{\varepsilon} = \frac{8\Lambda M^2\bar\phi^2}{dC_P} \le \frac{8\Lambda M^2}{dC_P}.\]
Consequently,
 \[\mu^\infty(d, m ) \ge \fint_\Omega m (x)\,\mathrm{d}x-\frac{8\Lambda M^2}{dC_P}.\]
Taking $C=8\Lambda M^2/C_P$, we have
 \[\left|\mu^\infty(d, m ) - \fint_\Omega m (x)\,\mathrm{d}x\right|\le \frac{C}{d},\quad \forall~m\in\mathcal{S}.\]
Therefore, by letting $d\to\infty$, we obtain that
 \[\lim_{d\to\infty}\mu^\infty(d, m )=\fint_\Omega m (x)\,\mathrm{d}x\quad\text{ holds uniformly in }m\in\mathcal{S}.\]
This completes the proof.
\end{proof}

Similarly, we have the following results.
\begin{lemma}\label{lemma2.26}
For each $t\in\mathbb{R}$, let $\mu^0(t;d)$ be the principal eigenvalue of \eqref{eq1.4} with potential $c\in\mathcal{C}$ be a given function. Then the limit
  $$\lim_{d\to \infty}\mu^0 (t;d)=\fint_\Omega c(x,t)\,\mathrm{d}x\quad\text{holds uniformly in }t\in\mathbb{R}.$$
\end{lemma}

\chapter{Basic properties of the normalized principal Floquet bundle}\label{chp3}
\section{Properties for $H$}
As mentioned in section 1.1, our analysis will primarily proceed along a sequence $\{T_n\}_{n\in\mathbb{N}^+}$ satisfying $T_n\to\infty$ as $n\to\infty$. Recall that, for a given sequence $\{T_n\}_{n\in\mathbb{N}^+}$, we define the set $\mathscr{C}_*$ of functions as
 \[\mathscr{C}_*:=\left\{\hat{c}\in C (\overline{\Omega})~|~\left\{\hat{c}_{T_n}\right\}_{n\in\mathbb{N}^+}\text{ admits a convergent subsequence converging to }\hat{c}\text{ in } C(\overline{\Omega})\right\}.\]
This chapter is devoted to establishing fundamental results for the normalized principal Floquet bundle $H$ of \eqref{eq1.1} in the context of the sequence $\{T_n\}_{n\in\mathbb{N}^+}$ of $T\to\infty$, as well as to proving Theorem \ref{theorem1.3}.
\begin{theorem}\label{theorem3.1}
Let $c\in\mathcal{C}$ be a given function and $(H,\varphi)$ be the normalized principal Floquet bundle of \eqref{eq1.1}. Let $\{T_n\}_{n\in\mathbb{N}^+}$ be a sequence satisfying $T_n\to\infty$ as $n\to\infty$. The following assertions are valid.
  \begin{enumerate}[{\rm (i)}]
    \item {\rm (}Recall that $\hat{c}_n(x)=\fint^{T_n}_{0}c(x,s)\,\mathrm{d}s,~x\in\overline{\Omega}${\rm)}. Then
         \[\liminf\limits_{n\to\infty}\fint^{T_n}_{0} H(s)\,\mathrm{d}s\le\inf\limits_{\hat{c}\in \mathscr{C}_*}\mu^\infty (\hat{c})\quad\text{and}\quad\limsup\limits_{n\to\infty}\fint^{T_n}_{0} H(s)\,\mathrm{d}s\le\sup\limits_{\hat{c}\in \mathscr{C}_*}\mu^\infty (\hat{c}).\]
    \item One has
         \[\liminf\limits_{n\to\infty}\fint^{T_n}_{0} H(s)\,\mathrm{d}s \le\liminf_{n\to\infty}\fint^{T_n}_{0}\fint_{\Omega}c(x,s)\,\mathrm{d}x\mathrm{d}s\]
        and
         \[\limsup\limits_{n\to\infty}\fint^{T_n}_{0} H(s)\,\mathrm{d}s \le\limsup_{n\to\infty}\fint^{T_n}_{0}\fint_{\Omega}c(x,s)\,\mathrm{d}x\mathrm{d}s.\]
    \item Suppose that there is a $t_0\in\mathbb{R}$ such that the functions $c_1,c_2\in \mathcal{C}$ satisfy $c_1(x,t)\ge c_2(x,t)$ for any $(x,t)\in\Omega\times[t_0,\infty)$. Then
         \[\limsup\limits_{n\to\infty}\fint^{T_n}_{0}H(s;c_1)\,\mathrm{d}s \ge \limsup\limits_{n\to\infty}\fint^{T_n}_{0}H(s;c_2)\,\mathrm{d}s\]
        and
         \[\liminf\limits_{n\to\infty}\fint^{T_n}_{0}H(s;c_1)\,\mathrm{d}s \ge \liminf\limits_{n\to\infty}\fint^{T_n}_{0}H(s;c_2)\,\mathrm{d}s;\]
        the two inequalities above are strict if $\liminf\limits_{n\to\infty}\fint^{T_n}_{0} \int_{\Omega}[c_1(x,s)- c_2(x,s)]\,\mathrm{d}x\mathrm{d}s>0$.
  \end{enumerate}
\end{theorem}

We shall prove Theorems \ref{theorem1.3} and \ref{theorem3.1} by establishing a series of auxiliary lemmas.

\begin{remark}\label{remark3.2}{\rm
  If $c$ satisfies further a uniform H\"{o}lder condition (see hypothesis \eqref{H3}), we can use Lemma \ref{lemma2.10} to rewrite the lower bounds in Theorem \ref{theorem1.3}(i) as
   \[\inf_{x(\cdot)\in C(\mathbb{R};\overline{\Omega})}\liminf_{T\to\infty} \fint^{T}_{0}c(x(s),s)\,\mathrm{d}s\le \underline{\lambda}_{H}(\omega,d)
   \text{ and }
   \inf_{x(\cdot)\in C(\mathbb{R};\overline{\Omega})}\limsup_{T\to\infty} \fint^{T}_{0}c(x(s),s)\,\mathrm{d}s \le \overline{\lambda}_{H}(\omega,d).\]
  }
\end{remark}

\medskip
It is necessary to recall some basic properties of the normalized principal Floquet bundle, which have been established in \cite[Appendix A]{Cantrell2021On}, also in \cite[Theorems 4.2.2, 4.3.4, etc.]{Lam2022Introduction}, \cite[Appendix A]{Lam2024The}, where the authors set $\omega=1$. By repeating their arguments with slight modifications, one easily obtains the following Proposition \ref{proposition3.3}.
\begin{proposition}\label{proposition3.3}
 Let $(\omega,d,c)\in(0,\infty)\times(0,\infty)\times C^{\delta,\delta/2}(\overline{\Omega}\times\mathbb{R})$. Then the normalized principal Floquet bundle $(H,\varphi,\psi)$ satisfying \eqref{eq1.1} and \eqref{eq1.2} is unique, and the mapping
  \begin{align*}
    \Pi:~(0,\infty)\times(0,\infty)\times C^{\delta,\delta/2}(\overline{\Omega}\times\mathbb{R})&\to C^{\delta/2}(\mathbb{R})\times C^{2+\delta,1+\delta/2}(\overline{\Omega}\times\mathbb{R})\times C^{2+\delta,1+\delta/2}(\overline{\Omega}\times\mathbb{R})\\
    (\omega,d,c(x,t))\qquad &\mapsto \qquad (H(t),\varphi(x,t),\psi(x,t))
  \end{align*}
 is smooth. Furthermore, one has
 \begin{enumerate}[{\rm (i)}]
   \item {\rm (}Boundedness{\rm )} The estimate $\sup\limits_{t\in\mathbb{R}}|H(t)|\le \sup\limits_{(x,t)\in\Omega\times\mathbb{R}}|c(x,t)|$ holds uniformly in $(\omega,d)\in(0,\infty)\times(0,\infty)$;
   \item {\rm (}Harnack principle{\rm )} There is a constant $C_d>0$ independent of $\omega$ and independent of $d$ if $d$ is large {\rm (}Note that $C_d$ depends on $d$ if $d$ is small!{\rm )}, such that
         \[\frac{1}{C_d}\le \varphi(x,t),\psi(x,t)\le C_d,\qquad \forall~(x,t)\in\overline{\Omega}\times\mathbb{R};\]
   \item {\rm (}Decomposition{\rm )} Set
         \[X^1(t):=\operatorname{span}\{\varphi(\cdot,t)\},~
           X^2(t):=\left\{v_0\in L^2(\Omega):\int_{\Omega}\psi(x,t)v_0(x)\,\mathrm{d}x=0\right\},\quad t\in\mathbb{R}.\]
         Then the two spaces are forward-invariant in the sense that for $i\in\{1,2\}$,
         \[u(\cdot,t;u_0)\in X^i(t)\text{ if }u_0\in X^i(0)\text{ for all }t\in [0,\infty),\]
         where $u(\cdot,\cdot;u_0)$ is the classical solution of
         \begin{equation*}
           \left\{
           \begin{aligned}
             &\omega\partial_tu-d\operatorname{div}(A(x)\nabla u)+c(x,t)u=H(t)u,\quad &&(x,t)\in\Omega\times(0,\infty),\\
             &\mathbf{n}\cdot (A(x)\nabla u)=0 &&(x,t)\in\partial\Omega\times(0,\infty),\\
             &u(x,0)=u_0(x),&&x\in\overline{\Omega}.
           \end{aligned}
           \right.
         \end{equation*}
         Additionally,
         \[L^2(\Omega)=X^1(t)\oplus X^2(t),\quad\forall~t\in\mathbb{R};\]
   \item {\rm (}Exponential separation{\rm )} There exist constants $C,\gamma>0$ such that for any $t>0$, and any $u_0\in X^2(0)\cap C(\overline{\Omega})$ one has
         \[\|u(\cdot,t;u_0)\|_{\infty}\le C\mathrm{e}^{-\gamma t}\|u_0(\cdot)\|_{\infty}.\]
 \end{enumerate}
\end{proposition}

As demonstrated in \cite[Theorem 2.2]{Lam2024The}, the constant appearing in the Harnack principle is independent of the diffusion rate $d$ when $d$ is sufficiently large. Furthermore, this constant remains independent of $\omega>0$. In Section \ref{sec3.2}, we establish a sharp Harnack estimate valid for small $d>0$, namely
 \[{\rm e}^{-\frac{C}{\sqrt{d}}}\le \varphi(x,t),\psi(x,t)\le {\rm e}^{\frac{C}{\sqrt{d}}},\qquad \forall~(x,t)\in\overline{\Omega}\times\mathbb{R},\]
where $C>0$ is a constant independent of $\omega>0$ and $d\in(0,1)$.
\begin{remark}\label{remark3.4}{\rm
More precisely, $X^1(\cdot)$ and $X^2(\cdot)$ are known as the principal Floquet bundle and the complementary Floquet bundle, respectively, as they naturally extend the concepts of the principal eigenfunction and the remaining eigenfunctions. Nevertheless, we also call $H$ (as an extension of the principal eigenvalue) the principal Floquet bundle for convenience in the present paper.
}
\end{remark}

Our discussion commences with a brief analysis of the relationship between the normalized principal Floquet bundle $H$ and potential $c$. We anticipate that the computations within its proof will facilitate a deeper comprehension of the normalized principal Floquet bundle.
\begin{lemma}\label{lemma3.5}
  Let $(H,\varphi)$ be the normalized principal Floquet bundle of \eqref{eq1.1} with $c\in \mathcal{C}$. The following assertions hold.
  \begin{enumerate}[{\rm (i)}]
    \item For each $t\in\mathbb{R}$, if $c(\cdot,t)\ge,\not\equiv 0$ on $\overline{\Omega}$, then $H(t)>0$; if $c(\cdot,t)\le,\not\equiv 0$ on $\overline{\Omega}$, then $H(t)<0$. Consequently, $\min_{\overline{\Omega}}c(\cdot,t)\le H(t)\le \max_{\overline{\Omega}}c(\cdot,t)$.
    \item If $H(t)=0$ at some $t\in\mathbb{R}$, then either $c(\cdot,t)\equiv 0$, or $\min_{\overline{\Omega}}c(\cdot,t)<0<\max_{\overline{\Omega}}c(\cdot,t)$.
    \item If $c(x,t)\equiv c(t)$ is independent of $x$, then $H(t)=c(t)$ for all $t\in\mathbb{R}$.
  \end{enumerate}
\end{lemma}
\begin{proof}
Arguing in the same way as the proof of \cite[Theorem 4.1.2]{Lam2022Introduction}, one can show that for a given function $c(x,t)\in C^{\delta,\delta/2}(\overline{\Omega}\times\mathbb{R})$, the problem
  \begin{equation}\label{eq3.1}
    \left\{
    \begin{aligned}
      &\omega\partial_t{u}-d\operatorname{div}(A(x)\nabla u)+c(x,t){u}=0, &&(x,t)\in\Omega\times\mathbb{R}, \\
      &\mathbf{n}\cdot(A(x)\nabla u)=0,&&(x,t)\in\partial\Omega\times\mathbb{R}, \\
      &{u}(x,t)>0,&&(x,t)\in\overline{\Omega}\times\mathbb{R}, \\
      &\int_{\Omega}{u}(x,0)\,\mathrm{d}x=1
    \end{aligned}
    \right.
  \end{equation}
has a unique solution in $C_{loc}^{2+\delta,1+\delta/2}(\overline{\Omega}\times\mathbb{R})$, which we denote by ${u}$. Then the normalized principal Floquet bundle of \eqref{eq1.1}  is determined by
   \[(H(t),\varphi(x,t))=\left(-\omega\frac{\mathrm{d}}{\mathrm{d}t}\left(\ln \int_{\Omega}{u}(x,t)\,\mathrm{d}x\right), {u}(x,t)\mathrm{e}^{\frac{1}{\omega}\int^{t}_{0}H(s)\,\mathrm{d}s}\right),\quad(x,t)\in\Omega\times\mathbb{R};\]
see \cite[Theorem 4.1.4]{Lam2022Introduction}. Direct computation yields
 \[\frac{\mathrm{d}}{\mathrm{d}t}\int_{\Omega}\varphi(x,t)\,\mathrm{d}x=\left[ \frac{\mathrm{d}}{\mathrm{d}t}\int_{\Omega}u(x,t)\,\mathrm{d}x+\frac{1}{\omega}H(t)\int_{\Omega}u(x,t)\,\mathrm{d}x\right] \mathrm{e}^{\frac{1}{\omega}\int^{t}_{0}H(s)\,\mathrm{d}s}=0,\qquad t\in\mathbb{R}.\]
This implies that
 \[\int_{\Omega}\varphi(x,t)\,\mathrm{d}x=\int_{\Omega}\varphi(x,0)\,\mathrm{d}x=\int_{\Omega}u(x,0) \,\mathrm{d}x=1,\qquad\forall~t\in\mathbb{R}.\]
Integrating equation \eqref{eq3.1} over $\Omega$ by parts and then dividing by $\int_{\Omega}{u}(x,\cdot)\,\mathrm{d}x$, one obtains
   \[H(t)=\frac{\int_{\Omega}c(x,t){u}(x,t)\,\mathrm{d}x}{\int_{\Omega}{u}(x,t)\,\mathrm{d}x} =\int_{\Omega}c(x,t)\varphi(x,t)\,\mathrm{d}x, \quad t\in\mathbb{R},\]
  which yields the conclusions immediately. This completes the proof.
\end{proof}

\medskip
Next, we provide an upper bound estimate for the principal Floquet exponent.
\begin{lemma}\label{lemma3.6}
Let $(H,\varphi)$ be the normalized Floquet bundle of \eqref{eq1.1} with potential  $c\in\mathcal{C}$ and $\{T_n\}_{n\in\mathbb{N}^+}$ be a sequence satisfying $T_n\to\infty$ as $n\to\infty$. Then
  \[\liminf\limits_{n\to\infty}\fint^{T_n}_{0} H(s)\,\mathrm{d}s \le\liminf_{n\to\infty}\fint^{T_n}_{0}\fint_{\Omega}c(x,s)\,\mathrm{d}x\mathrm{d}s\]
and
  \[\limsup\limits_{n\to\infty}\fint^{T_n}_{0} H(s)\,\mathrm{d}s \le\limsup_{n\to\infty}\fint^{T_n}_{0}\fint_{\Omega}c(x,s)\,\mathrm{d}x\mathrm{d}s.\]
Furthermore, one has
  \[\liminf_{T\to\infty}\fint^{T}_{0}H(s)\,\mathrm{d}s\le \liminf_{T\to\infty} \fint^{T}_{0}\fint_\Omega c(x,s) \,\mathrm{d}x\mathrm{d}s\]
and
  \[\limsup\limits_{T\to\infty}\fint^{T}_{0} H(s)\,\mathrm{d}s \le\limsup_{T\to\infty}\fint^{T}_{0}\fint_{\Omega}c(x,s)\,\mathrm{d}x\mathrm{d}s.\]
\end{lemma}
\begin{proof}
Dividing both sides of \eqref{eq1.1} by $\varphi$, we obtain
  \[\omega\partial_t\ln\varphi-d\operatorname{div}(A(x)\nabla\ln\varphi) -d\nabla\ln\varphi\cdot(A(x)\nabla\ln\varphi)+c(x,t)=H(t),\qquad(x,t)\in\Omega\times\mathbb{R}.\]
For any $T>0$, integrating the above identity over $\Omega\times[0,T]$ yields
 \[\frac{\omega}{T}\left(\fint_\Omega\ln\varphi\,\mathrm{d}x\right)\Big|^{T}_{0} -d\fint^{T}_{0}\fint_\Omega \nabla\ln\varphi\cdot(A(x)\nabla\ln\varphi) \,\mathrm{d}x\mathrm{d}s
     +\fint^{T}_{0}\fint_\Omega c(x,s) \,\mathrm{d}x\mathrm{d}s= \fint^{T}_{0}H(s)\,\mathrm{d}s.\]
By the uniform ellipticity assumption on matrix $A(\cdot)$, we know
 \[\fint^{T}_{0}\fint_\Omega \nabla\ln\varphi\cdot(A(x)\nabla\ln\varphi) \,\mathrm{d}x\mathrm{d}s\ge \frac{1}{\Lambda}\fint^{T}_{0}\fint_\Omega |\nabla\ln\varphi|^2 \,\mathrm{d}x\mathrm{d}s\ge 0,\qquad\forall~T>0.\]
Thus, for the sequence $\{T_n\}_{n\in\mathbb{N}^+}$ of $T\to\infty$, one has
 \[\fint^{T_n}_{0}H(s)\,\mathrm{d}s\le \fint^{T_n}_{0}\fint_\Omega c(x,s) \,\mathrm{d}x\mathrm{d}s +\frac{\omega}{T_n}\left(\fint_\Omega\ln\varphi\,\mathrm{d}x\right)\Bigg|^{T_n}_{0}, \qquad\forall~n\in\mathbb{N}^+.\]
Letting $n\to\infty$ yields
 \[\liminf\limits_{n\to\infty}\fint^{T_n}_{0} H(s)\,\mathrm{d}s \le\liminf_{n\to\infty}\fint^{T_n}_{0}\fint_{\Omega}c(x,s)\,\mathrm{d}x\mathrm{d}s\]
and
 \[\limsup\limits_{n\to\infty}\fint^{T_n}_{0} H(s)\,\mathrm{d}s \le\limsup_{n\to\infty}\fint^{T_n}_{0}\fint_{\Omega}c(x,s)\,\mathrm{d}x\mathrm{d}s,\]
where we have used the fact that $\varphi$ is uniformly bounded in $(x,t)\in\overline{\Omega}\times\mathbb{R}$ by the Harnack principle for $\varphi$, for fixed $\omega,d\in(0,\infty)$; see Proposition \ref{proposition3.3}(ii).

We now use the first statement to derive assertion (ii). By selecting the sequence $\{T_n\}_{n\in\mathbb{N}^+}$ such that
  \[\lim_{n\to\infty}\fint^{T_n}_{0}\fint_\Omega c(x,s) \,\mathrm{d}x\mathrm{d}s= \liminf_{T\to\infty} \fint^{T}_{0}\fint_\Omega c(x,s) \,\mathrm{d}x\mathrm{d}s,\]
we easily obtain that
  \[\liminf_{T\to\infty}\fint^{T}_{0}H(s)\,\mathrm{d}s\le \limsup_{n\to\infty}\fint^{T_n}_{0}H(s)\,\mathrm{d}s\le \liminf_{T\to\infty} \fint^{T}_{0}\fint_\Omega c(x,s) \,\mathrm{d}x\mathrm{d}s.\]
On the other hand, if $\{T_n\}_{n\in\mathbb{N}^+}$ is a sequence satisfying $T_n\to\infty$ as $n\to\infty$ such that
  \[\lim_{n\to\infty}\fint^{T_n}_{0}H(s)\,\mathrm{d}s=\limsup_{T\to\infty}\fint^{T}_{0}H(s)\,\mathrm{d}s,\]
then
  \begin{align*}
    \limsup_{T\to\infty}\fint^{T}_{0}H(s)\,\mathrm{d}s= &\lim_{n\to\infty}\fint^{T_n}_{0}H(s) \,\mathrm{d}s \\
    \le & \liminf_{n\to\infty}\fint^{T_n}_{0}\fint_\Omega c(x,s) \,\mathrm{d}x\mathrm{d}s\le \limsup_{T\to\infty}\fint^{T}_{0}\fint_\Omega c(x,s) \,\mathrm{d}x\mathrm{d}s.
  \end{align*}
This completes the proof.
\end{proof}

\medskip
We continue to study the monotonicity result of $H(\cdot;c)$ with respect to $c$.
\begin{lemma}\label{lemma3.7}
Suppose that there is a $t_0\in\mathbb{R}$ such that the functions $c_1,c_2\in \mathcal{C}$ satisfy $c_1(x,t)\ge c_2(x,t)$ for any $(x,t)\in\overline{\Omega}\times[t_0,\infty)$. Let $H_1$ and $H_2$ be the principal Floquet bundles of \eqref{eq1.1} corresponding to $c_1$ and $c_2$ respectively. Then
  \[\liminf\limits_{T\to\infty}\fint^{T}_{0}[H_1(s)-H_2(s)]\,\mathrm{d}s\ge 0,\]
the above inequality is strict if $\liminf\limits_{T\to\infty}\fint^{T}_{0} \int_{\Omega}[c_1(x,s)-c_2(x,s)]\,\mathrm{d}x\mathrm{d}s>0$.
Consequently, for any given sequence $\{T_n\}_{n\in\mathbb{N}^+}$ of $T\to\infty$, one has
 \[\liminf\limits_{n\to\infty}\fint^{T_n}_{0}H_1(s)\,\mathrm{d}s \ge \liminf\limits_{n\to\infty}\fint^{T_n}_{0}H_2(s)\,\mathrm{d}s\quad\text{and}\quad\limsup\limits_{n\to\infty}\fint^{T_n}_{0}H_1(s)\,\mathrm{d}s \ge \limsup\limits_{n\to\infty}\fint^{T_n}_{0}H_2(s)\,\mathrm{d}s,\]
the above inequalities are strict if $\liminf\limits_{n\to\infty}\fint^{T_n}_{0} \int_{\Omega}[c_1(x,s)-c_2(x,s)]\,\mathrm{d}x\mathrm{d}s>0$.
\end{lemma}
\begin{proof}
Let $(H_i,\varphi_i,\psi_i)=(H(\cdot;c_i),\varphi(\cdot,\cdot;c_i),\psi(\cdot,\cdot;c_i))$ be the normalized principal Floquet bundles of \eqref{eq1.1} associated with $c_i\in C^{\delta,\delta/2}(\overline{\Omega}\times\mathbb{R})$, $i=1,2$, namely, they satisfy
\begin{equation*}
  \left\{
  \begin{aligned}
    &\omega\partial_t\varphi_{i}-d\operatorname{div}(A(x)\nabla\varphi_i)+c_i(x,t)\varphi_{i}=H_i(t)\varphi_i,&&(x,t)\in\Omega\times\mathbb{R},\\
    &\mathbf{n}\cdot(A(x)\nabla\varphi_i)=0,&&(x,t)\in\partial\Omega\times\mathbb{R},\\
    &\varphi_i(x,t)>0,&&(x,t)\in\overline{\Omega}\times\mathbb{R},\\
    &\int_{\Omega}\varphi_i(x,t)\,\mathrm{d}x=1,&&t\in\mathbb{R}
  \end{aligned}
  \right.
\end{equation*}
and the adjoint problems
\begin{equation*}
  \left\{
  \begin{aligned}
    &-\omega\partial_t\psi_{i}-d\operatorname{div}(A(x)\nabla\psi_i)+c_i(x,t)\psi_{i}=H_i(t)\psi_i,&&(x,t)\in\Omega\times\mathbb{R},\\
    &\mathbf{n}\cdot(A(x)\nabla\psi_i)=0,&&(x,t)\in\partial\Omega\times\mathbb{R},\\
    &\psi_i(x,t)>0,&&(x,t)\in\overline{\Omega}\times\mathbb{R},\\
    &\int_{\Omega}\varphi_i(x,t)\psi_i(x,t)\,\mathrm{d}x=1,&&t\in\mathbb{R}.
  \end{aligned}
  \right.
\end{equation*}
Since
\[\lim\limits_{T\to\infty}\frac{1}{T}\int^{t_0}_{0}H_i(s)\,\mathrm{d}s=0\text{ and }\lim\limits_{T\to\infty}\frac{1}{T}\int^{t_0}_{0}c_i(x,s)\,\mathrm{d}s=0\text{ for all }x\in\overline{\Omega},\]
we may assume that $t_0=0$ without loss of generality. Multiplying the $\varphi_1$-equation by $\psi_2$ and integrating the resulting identity by parts over $\Omega$, we get
\[\omega\int_\Omega\psi_2\partial_t\varphi_1\,\mathrm{d}x+d\int_\Omega\nabla\psi_2\cdot(A(x)\nabla\varphi_1)\,\mathrm{d}x +\int_\Omega c_1(x,t)\varphi_1\psi_2\,\mathrm{d}x=H_1(t)\int_\Omega\varphi_1\psi_2\,\mathrm{d}x,\quad t\in\mathbb{R}.\]
Similarly, multiplying the $\psi_2$-equation by $\varphi_1$ and integrating the resulting identity by parts over $\Omega$ yield
\[-\omega\int_\Omega\varphi_1\partial_t\psi_2\,\mathrm{d}x+d\int_\Omega\nabla\varphi_1\cdot(A(x)\nabla\psi_2)\,\mathrm{d}x +\int_\Omega c_2(x,t)\varphi_1\psi_2\,\mathrm{d}x=H_2(t)\int_\Omega\varphi_1\psi_2\,\mathrm{d}x,\quad t\in\mathbb{R}.\]
Subtracting the two equations yields
 \[\omega\frac{\mathrm{d}}{\mathrm{d}t}\left(\int_{\Omega}\varphi_1\psi_2\,\mathrm{d}x\right)+ \int_{\Omega}[c_1(x,t)-c_2(x,t)]\varphi_1\psi_2\,\mathrm{d}x =[H_1(t)-H_2(t)]\int_{\Omega}\varphi_1\psi_2\,\mathrm{d}x,\quad t\in\mathbb{R}.\]
We first divide the above equation by $\int_{\Omega}\varphi_1\psi_2\,\mathrm{d}x$, and then for any $T>0$, integrate over $[0,T]$ and divide both sides by $T$ to obtain
\begin{equation}\label{eq3.2}
   \fint^{T}_{0}[H_1(s)-H_2(s)]\,\mathrm{d}s=\frac{\omega}{T}\left(\left.\ln\int_{\Omega}\varphi_1 \psi_2\,\mathrm{d}x\right)\right|^{T}_{t=0}+\fint^{T}_{0}\frac{\int_{\Omega}[c_1(x,s)- c_2(x,s)] \varphi_1\psi_2\,\mathrm{d}x }{\int_{\Omega}\varphi_1\psi_2\,\mathrm{d}x}\,\mathrm{d}s.
\end{equation}
In view of the Harnack principle for the normalized principal Floquet bundle (Proposition \ref{proposition3.3}(ii)) and the fact that $c_1(x,t)\ge c_2(x,t)$ for all $(x,t)\in\overline{\Omega}\times[0,\infty)$, by letting $T\to\infty$, one easily obtains
 \begin{equation}\label{eq3.3}
 \liminf\limits_{T\to\infty}\fint^{T}_{0}[H_1(s)-H_2(s)]\,\mathrm{d}s\ge 0.
 \end{equation}
Furthermore, if $\liminf_{T\to\infty}\fint^{T}_{0}\int_{\Omega}(c_1(x,s)-c_2(x,s)) \,\mathrm{d}x \,\mathrm{d}s>0$, then by using the Harnack principle again (so $\varphi_1(\cdot,t)\psi_2(\cdot,t)/\int_{\Omega}\varphi_1\psi_2\,\mathrm{d}x>C$ for some positive constant $C$ independent of $t\in\mathbb{R}$), one gets that
\[\fint^{T}_{0}\frac{\int_{\Omega}[c_1(x,s)- c_2(x,s)] \varphi_1\psi_2\,\mathrm{d}x }{\int_{\Omega}\varphi_1\psi_2\,\mathrm{d}x}\,\mathrm{d}s\ge C\fint^{T}_{0} \int_{\Omega}[c_1(x,s)- c_2(x,s)]\,\mathrm{d}x \,\mathrm{d}s>0.\]
A combination of this strict inequality and \eqref{eq3.2} gives that the inequality in \eqref{eq3.3} holds strictly.

For any sequence $\{T_n\}_{n\in\mathbb{N}^+}$ of $T\to\infty$ and a given integer $n'>0$, we have
\[\fint^{T_n}_{0}H_1(s)\,\mathrm{d}s-\inf_{n>n'}\fint^{T_{n'}}_{0}H_2(s)\,\mathrm{d}s\ge \fint^{T_n}_{0}H_1(s)\,\mathrm{d}s-\fint^{T_n}_{0}H_2(s)\,\mathrm{d}s,\quad\forall~n>n',\]
which implies that
\[\inf_{n>n'}\left(\fint^{T_n}_{0}H_1(s)\,\mathrm{d}s-\inf_{n>n'}\fint^{T_{n'}}_{0}H_2(s)\,\mathrm{d}s\right)\ge \inf_{n>n'}\left(\fint^{T_n}_{0}H_1(s)\,\mathrm{d}s-\fint^{T_n}_{0}H_2(s)\,\mathrm{d}s\right).\]
Letting $n'\to\infty$ and combining \eqref{eq3.3} yield
\[\liminf\limits_{n\to\infty}\fint^{T_n}_{0}H_1(s)\,\mathrm{d}s \ge \liminf\limits_{n\to\infty}\fint^{T_n}_{0}H_2(s)\,\mathrm{d}s.\]
A similar argument gives
\[\limsup\limits_{n\to\infty}\fint^{T_n}_{0}H_1(s)\,\mathrm{d}s \ge \limsup\limits_{n\to\infty}\fint^{T_n}_{0}H_2(s)\,\mathrm{d}s.\]
The strict inequalities appearing above are consequences of the strict inequality propagated from \eqref{eq3.3} under the condition $\liminf\limits_{n\to\infty}\fint^{T_n}_{0} \int_{\Omega}(c_1(x,s)-c_2(x,s)) \,\mathrm{d}x \,\mathrm{d}s>0$.

This completes the proof.
\end{proof}

If the sequence $\{T_n\}_{n\in\mathbb{N}^+}$ is selected to be the one such that \[\lim\limits_{n\to\infty}\fint^{T_n}_{0}H_1(s)\,\mathrm{d}s= \liminf\limits_{T\to\infty}\fint^{T}_{0}H_1(s)\,\mathrm{d}s\quad \text{ or }\quad\lim\limits_{n\to\infty}\fint^{T_n}_{0}H_2(s)\,\mathrm{d}s= \limsup\limits_{T\to\infty}\fint^{T}_{0}H_2(s)\,\mathrm{d}s\]
respectively, then we have the following corollary.
\begin{corollary}\label{corollary3.8}
 Suppose that the functions $c_1,c_2\in \mathcal{C}$ satisfy $c_1(x,t)\ge c_2(x,t)$ on $\overline{\Omega}\times[t_0,\infty)$ for some $t_0\in\mathbb{R}$. Then
 \begin{enumerate}[{\rm (i)}]
   \item the inequality \[\liminf\limits_{T\to\infty}\fint^{T}_{0}H(s;c_1)\,\mathrm{d}s\ge \liminf\limits_{T\to\infty}\fint^{T}_{0}H(s;c_2)\,\mathrm{d}s\]
       holds, and strictly holds if $\liminf\limits_{T\to\infty}\fint^{T}_{0} \int_{\Omega}[c_1(x,s)-c_2(x,s)] \,\mathrm{d}x \,\mathrm{d}s>0$;
   \item the inequality \[\limsup\limits_{T\to\infty}\fint^{T}_{0}H(s;c_1)\,\mathrm{d}s\ge \limsup\limits_{T\to\infty}\fint^{T}_{0}H(s;c_2)\,\mathrm{d}s\]
    holds, and strictly holds if $\liminf\limits_{T\to\infty}\fint^{T}_{0} \int_{\Omega}[c_1(x,s)-c_2(x,s)] \,\mathrm{d}x \,\mathrm{d}s>0$.
 \end{enumerate}
\end{corollary}

\begin{lemma}\label{lemma3.9}
Let $(H_1,\varphi_1,\psi_1)$ and $(H_2,\varphi_2,\psi_2)$ be the normalized principal Floquet bundles corresponding to $c_1$ and $c_2$ respectively, where the functions $c_1,c_2\in \mathcal{C}$ satisfy
  $$\lim_{t\to\infty} \left\|c_1(\cdot,t)-c_2(\cdot,t)\right\|_{\infty} =0.$$
Then
  \[\lim\limits_{T\to\infty}\fint^{T}_{0}[H_1(s)-H_2(s)]\,\mathrm{d}s= 0,\]
and moreover,
  \[\liminf\limits_{T\to\infty}\fint^{T}_{0}H_1(s)\,\mathrm{d}s= \liminf\limits_{T\to\infty}\fint^{T}_{0}H_2(s)\,\mathrm{d}s,\quad \limsup\limits_{T\to\infty}\fint^{T}_{0}H_1(s)\,\mathrm{d}s= \limsup\limits_{T\to\infty}\fint^{T}_{0}H_2(s)\,\mathrm{d}s.\]
\end{lemma}
\begin{proof}
Obviously, the condition $\lim_{t\to\infty}\left\|c_1(\cdot,t)-c_2(\cdot,t)\right\|_{\infty}=0$ implies
  \[\lim\limits_{T\to\infty}\fint^{T}_{0}\left\|c_1(\cdot,s)-c_2(\cdot,s)\right\|_{\infty} \,\mathrm{d}s=0.\]
Let $(H_i(\cdot),\varphi_i(\cdot,\cdot),\psi_i(\cdot,\cdot))=(H(\cdot;c_i),\varphi(\cdot,\cdot;c_i), \psi(\cdot,\cdot;c_i))$ be the normalized principal Floquet bundles of \eqref{eq1.1} associated with $c_i\in C^{\delta,\delta/2}(\overline{\Omega}\times\mathbb{R})$, $i=1,2$. As the proof of Lemma \ref{lemma3.7}, we can obtain \eqref{eq3.2}, that is,
 \begin{equation*}
   \fint^{T}_{0}[H_1(s)-H_2(s)]\,\mathrm{d}s=\frac{\omega}{T}\left(\left.\ln\int_{\Omega}\varphi_1 \psi_2\,\mathrm{d}x\right)\right|^{T}_{t=0}+\fint^{T}_{0}\frac{\int_{\Omega}[c_1(x,s)- c_2(x,s)] \varphi_1\psi_2\,\mathrm{d}x }{\int_{\Omega}\varphi_1\psi_2\,\mathrm{d}x}\,\mathrm{d}s.
 \end{equation*}
Note that
\[\left|\fint^{T}_{0}\frac{\int_{\Omega}[c_1(x,s)- c_2(x,s)] \varphi_1\psi_2\,\mathrm{d}x }{\int_{\Omega}\varphi_1\psi_2\,\mathrm{d}x}\,\mathrm{d}s\right|\le\fint^{T}_{0}\left\|c_1( \cdot,s)-c_2(\cdot,s)\right\|_{\infty} \,\mathrm{d}s.\]
Then one easily verifies
\begin{equation}\label{eq3.4}
 \left|\fint^{T}_{0}[H_1(s)-H_2(s)]\,\mathrm{d}s\right|\le \left|\frac{\omega}{T}\left.\left(\ln\int_{\Omega}\varphi_1 \psi_2\,\mathrm{d}x\right)\right|^{T}_{t=0}\right|+\fint^{T}_{0}\left\|c_1(\cdot,s)-c_2( \cdot,s)\right\|_{\infty} \,\mathrm{d}s.
\end{equation}
By letting $T\to\infty$, we have $\lim\limits_{T\to\infty}\fint^{T}_{0}[H_1(s)-H_2(s)]\,\mathrm{d}s= 0$.

Observe that
\[\fint^{T}_{0}H_1(s)\,\mathrm{d}s= \fint^{T}_{0}H_2(s)\,\mathrm{d}s+ \fint^{T}_{0}[H_1(s)-H_2(s)]\,\mathrm{d}s.\]
Taking ``$\liminf$'' or ``$\limsup$'' on both sides, one easily gets
  \[\liminf\limits_{T\to\infty}\fint^{T}_{0}H_1(s)\,\mathrm{d}s= \liminf\limits_{T\to\infty}\fint^{T}_{0}H_2(s)\,\mathrm{d}s\]
and
  \[\limsup\limits_{T\to\infty}\fint^{T}_{0}H_1(s)\,\mathrm{d}s= \limsup\limits_{T\to\infty}\fint^{T}_{0}H_2(s)\,\mathrm{d}s.\]
The proof is complete.
\end{proof}

\medskip
In fact, the hypotheses in (ii) and (iii) of Theorem \ref{theorem1.3}, as well as those in Lemma \ref{lemma3.7}, Corollary \ref{corollary3.8} and Lemma \ref{lemma3.9}, admit a suitable relaxation. More precisely, condition in Theorem \ref{theorem1.3}(ii) requires only that $c_1 \ge c_2$ holds almost everywhere on $[t_0,\infty)$, while condition in Theorem \ref{theorem1.3}(iii) can be reduced to $\lim_{T\to\infty} \fint_0^T \|c_1(\cdot,s)-c_2(\cdot,s)\|_\infty \, \,\mathrm{d}s = 0$.
\begin{lemma}\label{lemma3.10}
Let $\{T_n\}_{n\in\mathbb{N}^+}$ be a sequence satisfying $T_n\to\infty$ as $n\to\infty$, and $c_1,c_2\in \mathcal{C}$ be given functions.
\begin{enumerate}[{\rm (i)}]
  \item Denote by
         \[S_\varepsilon:=\{t\ge0~|~c_1(x,t)-c_2(x,t)\ge\varepsilon,~\forall ~ x\in\overline{\Omega}\}\]
  for any given $\varepsilon>0$. If $\lim_{\varepsilon\to 0^+}\limsup_{n\to\infty}\frac{1}{T_n}|S_\varepsilon\cap[0,T_n]|=0$, then
  \[\liminf\limits_{n\to\infty}\fint^{T_n}_{0}[H(s;c_1)-H(s;c_2)]\,\mathrm{d}s\ge 0,\]
the above inequality is strict if $\lim_{n\to\infty}\frac{1}{T_n}|S_{\varepsilon_0}\cap[0,T_n]|=0$ for some $\varepsilon_0>0$. Moreover, if $\lim_{\varepsilon\to 0^+}\limsup_{T\to\infty}\frac{1}{T}|S_\varepsilon\cap[0,T]|=0$, then
  \[\liminf\limits_{T\to\infty}\fint^{T}_{0}[H(s;c_1)-H(s;c_2)]\,\mathrm{d}s\ge 0,\]
the above inequality is strict if $\lim_{T\to\infty}\frac{1}{T}|S_{\varepsilon_0}\cap[0,T]|=0$ for some $\varepsilon_0>0$.
  \item Denote by
         \[S_\varepsilon:=\{t\ge0~|~\|c_1(\cdot,t)-c_2(\cdot,t)\|_\infty\ge\varepsilon\}.\]
  If $\lim_{\varepsilon\to 0^+}\limsup_{n\to\infty}\frac{1}{T_n}|S_\varepsilon\cap[0,T_n]|=0$, then
  \[\lim\limits_{n\to\infty}\fint^{T_n}_{0}[H(s;c_1)-H(s;c_2)]\,\mathrm{d}s= 0.\]
  Moreover, if $\lim_{\varepsilon\to 0^+}\limsup_{T\to\infty}\frac{1}{T}|S_\varepsilon\cap[0,T]|=0$, then
  \[\lim\limits_{T\to\infty}\fint^{T}_{0}[H(s;c_1)-H(s;c_2)]\,\mathrm{d}s= 0.\]
\end{enumerate}
\end{lemma}
\begin{proof}
Firstly, we prove assertion (i). Note that
 \begin{align*}
     & \fint^{T}_{0}\frac{\int_{\Omega}[c_1(x,s)- c_2(x,s)] \varphi_1\psi_2\,\mathrm{d}x }{\int_{\Omega}\varphi_1\psi_2\,\mathrm{d}x}\,\mathrm{d}s \\
   = & \frac{1}{T}\int_{[0,T]\cap S_\varepsilon}\frac{\int_{\Omega}[c_1(x,s)- c_2(x,s)] \varphi_1\psi_2\,\mathrm{d}x }{\int_{\Omega}\varphi_1\psi_2\,\mathrm{d}x}\,\mathrm{d}s + \frac{1}{T}\int_{[0,T]\setminus S_\varepsilon}\frac{\int_{\Omega}[c_1(x,s)- c_2(x,s)] \varphi_1\psi_2\,\mathrm{d}x }{\int_{\Omega}\varphi_1\psi_2\,\mathrm{d}x}\,\mathrm{d}s \\
   \ge & \frac{\varepsilon}{T}|[0,T]\cap S_\varepsilon|-\frac{\|c_1\|_\infty+\|c_2\|_\infty}{T}|[0,T]\setminus S_\varepsilon|.
 \end{align*}
This, together with \eqref{eq3.2} gives
 \[\fint^{T}_{0}[H_1(s)-H_2(s)]\,\mathrm{d}s\ge\frac{\omega}{T}\left(\left.\ln\int_{\Omega}\varphi_1 \psi_2\,\mathrm{d}x\right)\right|^{T}_{t=0}+\frac{\varepsilon}{T}|[0,T]\setminus S_\varepsilon|-\frac{\|c_1\|_\infty+\|c_2\|_\infty}{T}|[0,T]\cap S_\varepsilon|.\]
Letting $n\to\infty$ with $T=T_n$, and $T\to\infty$ respectively give the first assertion.

To prove the second assertion, we note that
 \begin{align*}
    & \fint^{T}_{0}\left\|c_1(\cdot,s)-c_2( \cdot,s)\right\|_{\infty} \,\mathrm{d}s \\
  = & \frac{1}{T}\int_{[0,T]\setminus S_\varepsilon}\left\|c_1(\cdot,s)-c_2( \cdot,s)\right\|_{\infty}\,\mathrm{d}s + \frac{1}{T}\int_{[0,T]\cap S_\varepsilon}\left\|c_1(\cdot,s)-c_2( \cdot,s)\right\|_{\infty}\,\mathrm{d}s \\
   \le & \frac{\varepsilon}{T}|[0,T]\setminus S_\varepsilon|+\frac{\|c_1\|_\infty+\|c_2\|_\infty}{T}|[0,T]\cap S_\varepsilon|.
 \end{align*}
Then we rewrite the estimate \eqref{eq3.4} as
 \[\left|\fint^{T}_{0}[H_1(s)-H_2(s)]\,\mathrm{d}s\right|\le \left|\frac{\omega}{T}\left.\left(\ln\int_{\Omega}\varphi_1 \psi_2\,\mathrm{d}x\right)\right|^{T}_{t=0}\right|+\frac{\varepsilon}{T}|[0,T]\setminus S_\varepsilon|+\frac{\|c_1\|_\infty+\|c_2\|_\infty}{T}|[0,T]\cap S_\varepsilon|,\]
and the second assertion follows by sending $n\to\infty$ with $T=T_n$, and $T\to\infty$ respectively. This completes the proof.
\end{proof}

\medskip
We now give some upper estimates of $H(\cdot)$ via an elliptic eigenvalue problem. The following version of Jensen's inequality for positive definite quadratic form will be used.
\begin{lemma}\label{lemma3.11}
  Let $\mathbf{v}(\cdot)\in C(\mathbb{R};\mathbb{R}^N)\cap L^\infty(\mathbb{R})$. If $A$ is a positive definite matrix, then for any $-\infty<t_1<t_2<\infty$,
  \[\fint^{t_2}_{t_1}\mathbf{v}(s)\,\mathrm{d}s\cdot \left(A\fint^{t_2}_{t_1}\mathbf{v}(s)\,\mathrm{d}s\right)\le \fint^{t_2}_{t_1}\mathbf{v}(s)\cdot(A\mathbf{v}(s))\,\mathrm{d}s.\]
\end{lemma}
A proof of this lemma may proceed via Jensen's inequality for convex functions, or alternatively, by employing the Cholesky decomposition of $A$ (there exists an invertible matrix $B$ with $A = B^{\mathrm{T}}B)$ and then invoking the Cauchy-Schwarz inequality. The details are omitted for brevity. In the special case where $A$ is the identity matrix, the inequality reduces to the standard Cauchy-Schwarz inequality. Recall that $\mu^\infty(m)$ is the principal eigenvalue of
  \begin{equation*}
    \left\{
    \begin{aligned}
      &-d\operatorname{div}(A(x)\nabla\hat{\phi})+m(x)\hat{\phi}=\mu\hat{\phi},&&x\in\Omega,\\
      &\mathbf{n}\cdot A(x)\nabla\hat{\phi}=0&&x\in\partial\Omega.
    \end{aligned}
    \right.
  \end{equation*}
\begin{lemma}\label{lemma3.12}
Let $(H,\varphi)$ be the normalized Floquet bundle of \eqref{eq1.1} with potential $c\in\mathcal{C}$ and let $\{T_n\}_{n\in\mathbb{N}^+}$ be a sequence satisfying $T_n\to\infty$ as $n\to\infty$. Then
 \[\liminf\limits_{n\to\infty}\fint^{T_n}_{0} H(s)\,\mathrm{d}s\le\inf\limits_{\hat{c}\in \mathscr{C}_*}\mu^\infty (\hat{c})\quad\text{and}\quad\limsup\limits_{n\to\infty}\fint^{T_n}_{0} H(s)\,\mathrm{d}s\le\sup\limits_{\hat{c}\in \mathscr{C}_*}\mu^\infty (\hat{c}).\]
Consequently, if $\lim\limits_{n\to\infty}\fint^{T_n }_{0}H(s)\,\mathrm{d}s$ exists, then
  \[\lim\limits_{n\to\infty}\fint^{T_n}_{0} H(s)\,\mathrm{d}s\le\inf\limits_{\hat{c}\in \mathscr{C}_*}\mu^\infty (\hat{c});\]
if the sequence $\{\hat{c}_n\}_{n\in\mathbb{N}^+}$ converges to $\hat{c}$ in $C(\overline{\Omega})$, then
  \[\limsup\limits_{n\to\infty}\fint^{T_n}_{0} H(s)\,\mathrm{d}s \le\mu^\infty (\hat{c}).\]
\end{lemma}
\begin{proof}
Dividing \eqref{eq1.1} both sides by $\varphi$ yields
\[\omega\partial_t\ln\varphi-d\operatorname{div}(A(x)\nabla\ln\varphi)-d\nabla\ln\varphi \cdot(A(x)\nabla\ln\varphi)+c(x,t)=H(t),\qquad (x,t)\in\Omega\times\mathbb{R}.\]
Define $u_n(x)=\exp\left(\fint^{T_n}_{0}\ln\varphi(x,s)\,\mathrm{d}s\right)$, $x\in\overline{\Omega}$, $n\in\mathbb{N}^+$. One can check that
  \begin{align*}
    &d\operatorname{div}(A(x)\nabla u_n)\\
    =&d\left[\fint^{T_n}_{0}\operatorname{div}(A(x)\nabla \ln\varphi)\,\mathrm{d}s+\nabla\left(\fint^{T_n}_{0}\ln\varphi\,\mathrm{d}s\right)\cdot\left(A(x) \nabla\left(\fint^{T_n}_{0}\ln\varphi\,\mathrm{d}s\right)\right)\right]u_n \\
    =&\Bigg[\fint^{T_n}_{0}\left(\omega\partial_t\ln\varphi -d\nabla\ln\varphi \cdot(A(x)\nabla\ln\varphi)+c(x,t)-H(t)\right)\,\mathrm{d}s \\
    &\qquad +d\nabla\left(\fint^{T_n}_{0}\ln\varphi\,\mathrm{d}s\right)\cdot\left(A(x) \nabla\left(\fint^{T_n}_{0}\ln\varphi\,\mathrm{d}s\right)\right)\Bigg]u_n\\
    =&\Bigg[\frac{\omega}{T_n}\ln\varphi\Bigg|^{T_n}_{t=0} -d\fint^{T_n}_{0}\nabla\ln\varphi \cdot(A(x)\nabla\ln\varphi)\,\mathrm{d}s+\fint^{T_n}_{0}c(x,s)\,\mathrm{d}s-\fint^{T_n}_{0}H(s)\,\mathrm{d}s \\
    &\qquad +d\nabla\left(\fint^{T_n}_{0}\ln\varphi\,\mathrm{d}s\right)\cdot\left(A(x) \nabla\left(\fint^{T_n}_{0}\ln\varphi\,\mathrm{d}s\right)\right)\Bigg]u_n\\
    \le&\left(\frac{\omega}{T_n}\ln\varphi\Bigg|^{T_n}_{t=0} +\fint^{T_n}_{0}c(x,s)\,\mathrm{d}s-\fint^{T_n}_{0}H(s)\,\mathrm{d}s\right)u_n,\qquad\qquad \forall~x\in\Omega,
  \end{align*}
where Lemma \ref{lemma3.11} has been used in obtaining the last inequality. Substituting the equation satisfied by $\varphi$ into the above inequality, one gets that
  \begin{equation}\label{eq3.5}
    \left\{
    \begin{aligned}
      &\frac{\omega}{T_n}\left.\ln\varphi(x,t)\right|^{T_n}_{t=0}u_n-d\operatorname{div}(A(x)\nabla u_n) +\left(\fint^{T_n}_{0}c(x,s)\,\mathrm{d}s\right)u_n\ge \left(\fint^{T_n}_{0}H(s)\,\mathrm{d}s\right)u_n,&&x\in\Omega,\\
      &\mathbf{n}\cdot(A(x)\nabla u_n)=0&&x\in\partial\Omega.
    \end{aligned}
    \right.
  \end{equation}
Since $c$ is uniformly bounded in $\Omega\times\mathbb{R}$, by choosing a suitable subsequence of $\{T_n\}_{n\in\mathbb{N}^+}$, we may assume that the sequence $\{\hat{c}_n\}_{n\in\mathbb{N}^+}$ converges to some continuous function $\hat{c}$. Since $u_n(\cdot)\in C^{2+\delta}(\Omega)\cap C^{1+\delta/2}(\overline{\Omega})$ is uniformly bounded in $n\in\mathbb{N}^+$, there is a subsequence $\{T_{n_k}\}_{k\in\mathbb{N}^+}$ of $\{T_n\}_{n\in\mathbb{N}^+}$ such that $u_{n_k}(\cdot)$ converges to some strictly positive function $u (\cdot)\in C^{2+\delta}(\Omega)\cap C^{1+\delta/2}(\overline{\Omega})$. The positivity of $u $ follows from the Harnack principle of $\varphi$. Letting $n\to\infty$ in \eqref{eq3.5} yields
  \begin{equation*}
    \left\{
    \begin{aligned}
      &-d\operatorname{div}(A(x)\nabla u ) +\hat{c}(x)u \ge \liminf\limits_{n\to\infty}\left(\fint^{T_n}_{0} H(s)\,\mathrm{d}s\right)u ,&&x\in\Omega,\\
      &\mathbf{n}\cdot(A(x)\nabla u )=0,&&x\in\partial\Omega.
    \end{aligned}
    \right.
  \end{equation*}
By the comparison principle for elliptic eigenvalue problem\cite[Lemma 4.1]{Devinatz1974The}, we conclude that
    \[\liminf\limits_{n\to\infty}\fint^{T_n}_{0} H(s)\,\mathrm{d}s\le\mu^\infty(d,\hat{c}).\]
Since $\hat{c}$ is arbitrary in $\mathscr{C}_*$, we obtain that
 \[\liminf\limits_{n\to\infty}\fint^{T_n}_{0} H(s)\,\mathrm{d}s\le\inf\limits_{\hat{c}\in \mathscr{C}_*}\mu^\infty (\hat{c}).\]
On the other hand, if we choose a subsequence $\{T_{n'}\}_{n'\in\mathbb{N}^+}$ of $\{T_n\}_{n\in\mathbb{N}^+}$ such that
 \[\lim_{n'\to\infty}\fint^{T_{n'}}_{0} H(s)\,\mathrm{d}s= \limsup_{n\to\infty}\fint^{T_n}_{0} H(s)\,\mathrm{d}s,\]
then for any convergence point $\hat{c}$ of $\{\hat{c}_{n'}\}_{n'\in\mathbb{N}^+}$, letting $n'\to\infty$ in \eqref{eq3.5} yields
  \begin{equation*}
    \left\{
    \begin{aligned}
      &-d\operatorname{div}(A(x)\nabla u ) +\hat{c}(x)u \ge \lim\limits_{n'\to\infty}\left(\fint^{T_{n'}}_{0} H(s)\,\mathrm{d}s\right)u ,&&x\in\Omega,\\
      &\mathbf{n}\cdot(A(x)\nabla u )=0,&&x\in\partial\Omega.
    \end{aligned}
    \right.
  \end{equation*}
In the above limiting process, we naturally assume the convergence of $u_{n'}$ in $C^{2+\delta}(\Omega)\cap C^{1+\delta/2}(\overline{\Omega})$, a further subsequence can be taken if necessary. From this we conclude that
 \[\limsup\limits_{n\to\infty}\fint^{T_n}_{0} H(s)\,\mathrm{d}s\le\sup\limits_{\hat{c}\in \mathscr{C}_*}\mu^\infty (\hat{c}).\]
This completes the proof.
\end{proof}
\begin{corollary}\label{corollary3.13}
Let $H$ be the normalized principal Floquet bundle of \eqref{eq1.1} with potential $c\in\mathcal{C}$. Then
     \[\liminf\limits_{T\to\infty}\fint^{T}_{0}H(s)\,\mathrm{d}s\le \inf\limits_{\hat{c}\in\mathscr{C}}\mu^\infty(d,\hat{c})~\text{ and }~ \limsup\limits_{T\to\infty}\fint^{T}_{0}H(s)\,\mathrm{d}s\le \sup\limits_{\hat{c}\in\mathscr{C}}\mu^\infty(d,\hat{c}).\]
\end{corollary}
\begin{proof}
We now show the first inequality. For any $\varepsilon>0$, there is a function $\hat{c}_{\varepsilon}\in \mathscr{C}$ and a sequence $\{T_n\}_{n\in\mathbb{N}^+}$ of $T\to\infty$, such that
  \[\hat{c}_{\varepsilon}=\lim_{n\to\infty}\fint^{T_n}_{0}c(\cdot,s)\,\mathrm{d}s,\qquad \mu^\infty(\hat{c}_{\varepsilon})\le \inf\limits_{\hat{c}\in\mathscr{C}}\mu^\infty(d,\hat{c})+\varepsilon.\]
Then by Lemma \ref{lemma3.12}, we have
 \[\limsup\limits_{n\to\infty}\fint^{T_n}_{0} H(s)\,\mathrm{d}s \le\mu^\infty(\hat{c}_{\varepsilon}).\]
We then deduce that
 \[\liminf\limits_{T\to\infty}\fint^{T}_{0} H(s)\,\mathrm{d}s\le\limsup_{n\to\infty}\int^{T_n}_{0}H(s)\,\mathrm{d}s\le \inf\limits_{\hat{c}\in\mathscr{C}}\mu^\infty(d,\hat{c})+\varepsilon.\]
By the arbitrariness of $\varepsilon$, we obtain the first estimate.

On the other hand, by selecting a sequence $\{T_n\}_{n\in\mathbb{N}^+}$ of $T\to\infty$ such that
  \[\lim\limits_{n\to\infty}\fint^{T_n}_{0}H(s)\,\mathrm{d}s=\limsup\limits_{T\to\infty}\fint^{T}_{0}H(s)\,\mathrm{d}s,\]
then the second estimate follows directly from Lemma \ref{lemma3.12}. This finishes the proof.
\end{proof}

\medskip
Theorem \ref{theorem1.3} and Theorem \ref{theorem3.1} follow from Lemmas \ref{lemma3.7}-\ref{lemma3.9} and Corollary \ref{corollary3.13}.

\section{A Harnack estimate}\label{sec3.2}
This section is devoted to establishing the Harnack estimate for the principal Floquet bundle (i.e., Proposition \ref{proposition3.3}(ii)), with a precise characterization of the dependence of the estimate constant on the coefficient $d\in(0,1)$. The results obtained in this section not only serve as a crucial tool for analyzing the limiting behavior of the principal Floquet exponent in later chapters, but also hold independent interest. It is worth noting that, under slightly weaker regularity assumptions on the parameters compared to other parts of this paper, we derive an elliptic-type Harnack inequality for a parabolic problem.

Let $\Omega\subset\mathbb{R}^n$ ($n\ge 1$) be a bounded domain with $C^2$ boundary. Consider nonnegative solutions of the following linear parabolic equation:
\begin{equation}\label{eq3.6}
\begin{cases}
\omega\,\partial_t u-d\,\mathrm{div}\big(A(x)\nabla u\big)
+c(x,t)\,u=0 &(x,t)\in\Omega\times\mathbb{R},\\[4pt]
\mathbf{n}\cdot A(x)\nabla u=0 &(x,t)\in\partial\Omega\times\mathbb{R},
\end{cases}
\end{equation}
where $\omega>0$ is the frequency parameter and $d>0$ is the diffusion coefficient. Denote by $\mathrm{diam}(\Omega)$ the Euclidean diameter. Define the \textbf{intrinsic diameter} by
 \[L_\Omega:=\sup_{x_1,x_2\in\Omega}\inf\{\ell(\gamma):\gamma\subset\Omega
\text{ is a rectifiable curve connecting }x_1\text{ and }x_2\},\]
where $\ell(\gamma)$ denotes the length of curve $\gamma$. For a bounded, connected $C^2$ domain, it is clear that $L_\Omega<\infty$. The coefficient matrix $A\in C^{0,1}(\overline\Omega;\mathbb{R}^{n\times n})$ is symmetric and uniformly elliptic: there exist constants $0<\lambda\le\Lambda<\infty$ such that $\lambda|\xi|^2\le A(x)\xi\cdot\xi\le\Lambda|\xi|^2$ for all $\xi\in\mathbb{R}^n$ and $x\in\overline\Omega$. Let $K_A$ denote the Lipschitz constant of $A$. Assume $c\in L^\infty(\Omega\times\mathbb{R})$ and set $c_0:=\|c\|_{L^\infty}<\infty$. Under these assumptions of bounded and measurable coefficients, any nonnegative, nontrivial weak entire solution (an entire solution is defined for all $t\in\mathbb{R}$; hereafter simply referred to as a nonnegative solution) $u\in W^{1,0}_2(\Omega\times\mathbb{R})\cap L^\infty_{\mathrm{loc}}(\mathbb{R};L^2(\Omega))$ of \eqref{eq3.6} is bounded strictly away from zero and infinity on every compact time interval.

\begin{theorem}\label{theorem3.14}
There exist constants $C'=C'(n,\lambda,\Lambda,c_0,K_A,\Omega)>0$ and $d_0=d_0(\Omega)<\min\{1,(\mathrm{diam}(\Omega))^2\}$, both independent of $\omega$ and $d$, such that for any $0<d\le d_0$ and any $\omega>0$, every nonnegative solution of \eqref{eq3.6} satisfies
\begin{equation}\label{eq3.7}
\sup_{x\in\Omega}u(x,t)\le
\exp\!\left(\frac{C'}{\sqrt d}\right)\inf_{x\in\Omega}u(x,t), \qquad\forall~t\in\mathbb{R}.
\end{equation}
\end{theorem}
The proof of Theorem \ref{theorem3.14} hinges on the following classical interior forward Harnack inequality.

\begin{lemma}\label{lemma3.15}
Let $w\ge0$ be a weak solution of
\[\partial_s w-\mathrm{div}(\tilde{A}(y)\nabla w)+q(y,s)w=0,\qquad (y,s)\in B_r(z)\times(s_0,s_0+\theta r^2),\]
where $B$ is symmetric and uniformly elliptic with constants $\tilde{\lambda},\tilde{\Lambda}$, and $q\in L^\infty$ with $\|q\|_\infty\le q_0$. Then there exists a constant $C_F=C_F(n,\tilde{\lambda},\tilde{\Lambda},q_0,\theta,r)>0$ such that
\[
\sup_{B_{r/2}(z)\times(s_0+\frac14\theta r^2,\,s_0+\frac12\theta r^2)}w
\le C_F
\inf_{B_{r/2}(z)\times(s_0+\frac34\theta r^2,\,s_0+\theta r^2)}w.
\]
\end{lemma}

The proof of this lemma can be accomplished via Moser iteration (combining local $L^\infty$ boundedness and the weak Harnack inequality), where the potential term $q$ can be absorbed during the iteration, yielding a constant depending on $q_0$. It holds for bounded measurable coefficients (allowing time dependence); for details, we refer to \cite{Moser1964A,Moser1971On,Fabes1986A,Fabes1997Behavior,Fabes1999Behavior,Lieberman1984The,Lieberman1987Local}.

The following lemma serves as the key tool for closing the argument at the same time. Its proof follows directly from the comparison principle via an explicit supersolution construction.

\begin{lemma}\label{lemma3.16}
Let $w\ge0$ be a nonnegative weak solution of
\begin{equation*}
  \left\{
  \begin{aligned}
    &\partial_s w-\mathrm{div}(\tilde{A}(y)\nabla w)+q(y,s)w=0,&&(y,s)\in D\times\mathbb{R},\\
    &\mathbf{n}\cdot \tilde{A}\nabla w=0, &&(y,s)\in\partial D\times\mathbb{R},
  \end{aligned}
  \right.
\end{equation*}
where $D\subset\mathbb{R}^n$ is a bounded domain, $\tilde{A}$ is symmetric and uniformly elliptic, $q\in L^\infty$ with $\|q\|_\infty\le q_0$. Then for any $s\in\mathbb{R}$ and any $\Delta s>0$,
 \[M(s+\Delta s)\le {\rm e}^{q_0\Delta s}\,M(s),\]
where $M(s):=\sup_{y\in D}w(y,s)$.
\end{lemma}

\begin{proof}
Fix any initial time $s_0\in\mathbb{R}$, and let $M_0:=\sup_{y\in D} w(y,s_0)$. Define the spatially constant function
\[
\Phi(y,s) := M_0 {\rm e}^{q_0 (s-s_0)},\qquad (y,s)\in D\times [s_0,\infty).
\]
A direct computation shows
\[
\partial_s \Phi - \mathrm{div}(\tilde{A}\nabla\Phi) + q(y,s)\Phi
= q_0 \Phi + q \Phi
= (q_0 + q(y,s))\Phi.
\]
Since $q \ge -\|q\|_{L^\infty} \ge -q_0$, we have $q_0+q\ge 0$, and hence
\[
\partial_s \Phi - \mathrm{div}(\tilde{A}\nabla\Phi) + q\Phi \ge 0
\quad\text{in } D\times [s_0,\infty),
\]
i.e., $\Phi$ is a supersolution of the equation.
On the boundary $\partial D$, $\nabla\Phi\equiv 0$, so $\Phi$ automatically satisfies the homogeneous Neumann boundary condition $\mathbf{n}\cdot \tilde{A}\nabla\Phi=0$.
At the initial time $s=s_0$, $\Phi(y,s_0)=M_0 \ge w(y,s_0)$ for all $y\in D$.

By the weak comparison principle for linear parabolic equations with Neumann boundary conditions, we have
\[
w(y,s) \le \Phi(y,s) = M_0 {\rm e}^{q_0(s-s_0)},\qquad \forall\,y\in D,\;\forall\,s\ge s_0.
\]
Taking the supremum over $y\in D$ on both sides yields
\[
M(s) \le M(s_0) {\rm e}^{q_0(s-s_0)}.
\]
Setting $s_0=s$ and replacing $s$ by $s+\Delta s$ gives $M(s+\Delta s)\le {\rm e}^{q_0\Delta s}M(s)$, as desired.
\end{proof}

\begin{proof}[Proof of Theorem \ref{theorem3.14}]
We take $d_0=\min\{1, \mathrm{diam}(\Omega)^2\}$ and assume $0<d\le d_0$. The proof proceeds through several steps.

\noindent\textbf{Step 1: Scaling transformation.}
Set
\[
y=\frac{x}{\sqrt d},\qquad s=\frac{t}{\omega},\qquad
v(y,s)=u(\sqrt d\,y,\omega s),\qquad
\Omega_d:=\frac{1}{\sqrt d}\Omega.
\]
Direct computation gives
\[
\partial_t u=\frac{1}{\omega}\partial_s v,\quad
\nabla_x u=\frac{1}{\sqrt d}\nabla_y v,\quad
\mathrm{div}_x(A\nabla_x u)=\frac{1}{d}\mathrm{div}_y(A(\sqrt d\,y)\nabla_y v).
\]
Substituting these into \eqref{eq3.6} yields
\[
\left\{
\begin{aligned}
  &\partial_s v-\mathrm{div}_y(\tilde A(y)\nabla_y v)+\tilde c(y,s)v=0
\quad&&(y,s)\in\Omega_d\times\mathbb{R},\\
&\tilde{\mathbf{n}}\cdot\tilde A\nabla_y v=0, &&(y,s)\in \partial\Omega_d\times\mathbb{R},
\end{aligned}
\right.
\]
where $\tilde A(y)=A(\sqrt d\,y)$ and $\tilde c(y,s)=c(\sqrt d\,y,\omega s)$.
After this change of variables, the structural constants of the problem remain unchanged and are independent of $\omega,d$:
\[
\lambda|\xi|^2\le\tilde A\xi\cdot\xi\le\Lambda|\xi|^2,\qquad
\|\tilde c\|_{L^\infty}=c_0,\qquad
\operatorname{Lip}(\tilde A)\le K_A.
\]
The geometric parameters of the domain change as follows: $\mathrm{diam}(\Omega_d)=\frac{\mathrm{diam}(\Omega)}{\sqrt d}$ and the intrinsic diameter $L_{\Omega_d}=L_\Omega/\sqrt d$.

\noindent\textbf{Step 2: Local even reflection and boundary treatment.}
Since $\partial\Omega$ is of class $C^2$, there exists a constant $r_0=r_0(\Omega)>0$, independent of $d$, such that for any boundary point $z_0\in\partial\Omega_d$, the ball $B_{4r_0}(z_0)$ can be mapped onto a half-ball via a $C^2$ diffeomorphism with uniformly bounded $C^2$ norm (depending only on $\Omega$).

Fix any $z\in\Omega_d$. We consider two cases.
\begin{enumerate}[{\rm (a)}]
  \item If $\operatorname{dist}(z,\partial\Omega_d)\ge r_0$, then $B_{r_0}(z)\subset\Omega_d$, and $v$ itself is a weak solution in $B_{r_0}(z)\times\mathbb{R}$.
  \item If $\operatorname{dist}(z,\partial\Omega_d)<r_0$, choose $z_0\in\partial\Omega_d$ such that $|z-z_0|<r_0$. Then $B_{2r_0}(z)\subset B_{4r_0}(z_0)$. We flatten the boundary in $B_{4r_0}(z_0)$, perform an even reflection across the flat boundary, and map back to the original coordinates. This yields an extended function defined on $B_{2r_0}(z)\times\mathbb{R}$, which is a nonnegative weak solution to the parabolic equation with extended coefficients. The extended matrix remains uniformly elliptic with constants $\hat\lambda,\hat\Lambda$ depending only on $\lambda,\Lambda,\Omega$, and the extended zero-order term satisfies $\|\hat c\|_{L^\infty}=c_0$.
\end{enumerate}
In both cases, the restriction of $v$ to $B_{r_0/2}(z)\times\mathbb{R}$ can be viewed as the restriction of an interior solution on a ball of radius $r_0$, with coefficients satisfying the same uniform bounds independent of $d$ and $z$. Therefore, the interior forward Harnack inequality (Lemma \ref{lemma3.15}) applies to every ball $B_{r_0/2}(z)$ with $z\in\Omega_d$, with a constant $C_F$ independent of $d$ and $z$.

\noindent\textbf{Step 3: Forward Harnack estimate on the fixed-scale ball.}
For any $z\in\Omega_d$ and any $\sigma\in\mathbb{R}$, apply Lemma \ref{lemma3.15} with $r=r_0$, $\theta=1$, $s_0=\sigma-3r_0^2/8$. This gives
\[
\sup_{B_{r_0/2}(z)\times(\sigma-\frac{r_0^2}{8},\sigma+\frac{r_0^2}{8})}v
\le C_F
\inf_{B_{r_0/2}(z)\times(\sigma+\frac{3r_0^2}{8},\sigma+\frac{5r_0^2}{8})}v,
\]
where $C_F=C_F(n,\lambda,\hat\Lambda,c_0)$. In particular, choosing $\sigma$ and $\sigma+r_0^2/2$ to lie in the two time intervals respectively, we obtain
\begin{equation}\label{eq3.8}
\sup_{B_{r_0/2}(z)}v(\cdot,\sigma)
\le C_F\inf_{B_{r_0/2}(z)}v\left(\cdot,\sigma+\frac{r_0^2}{2}\right).
\end{equation}
Note that $B_{r_0/2}(z)$ might not be entirely contained in $\Omega$ if $z$ close sufficiently to boundary $\partial\Omega$ in the inequalities above. Should this happen, we treat $v$ appearing therein as the extension of the solution. The reason we refrain from using more rigorous notation (namely, replacing $B_{r_0/2}(z)$ by $B_{r_0/2}(z) \cap \Omega_d$) is that it would greatly burden the notation.

\noindent\textbf{Step 4: Ball-chain covering.}
Fix any two points $y_1,y_2\in\Omega_d$. By the definition of the intrinsic diameter, there exists a rectifiable curve $\gamma\subset\Omega_d$ connecting $y_1$ and $y_2$ such that
\[\ell(\gamma)\le\frac{L_\Omega}{\sqrt d}+1.\]
Select points $z_1,z_2,\ldots,z_{N_d}$ along $\gamma$ with arc-length spacing $r_0/2$, such that: (a) $z_1=y_1$, $z_{N_d}=y_2$; (b) the arc-length along $\gamma$ from $z_i$ to $z_{i+1}$ is $r_0/2$, which implies the Euclidean distance $|z_i-z_{i+1}|\le r_0/2$. The number of balls satisfies
\[
N_d-1\le\frac{\ell(\gamma)}{r_0/2}=\frac{2\ell(\gamma)}{r_0}
\le\frac{2L_\Omega}{r_0\sqrt d}+\frac{2}{r_0},
\]
i.e.,
\begin{equation}\label{eq3.9}
N_d\le\frac{2L_\Omega}{r_0\sqrt d}+\frac{2}{r_0}+1.
\end{equation}
Since $|z_i-z_{i+1}|\le r_0/2$, the adjacent balls $B_{r_0/2}(z_i)$ and $B_{r_0/2}(z_{i+1})$ intersect. Moreover, by Step 2, the forward Harnack estimate \eqref{eq3.8} applies to each $B_{r_0/2}(z_i)$.

\noindent\textbf{Step 5: Spatio-temporal iteration.}
Fix the initial time $s\in\mathbb{R}$. For $i=1,2,\ldots,N_d$, apply \eqref{eq3.8} on the ball $B_{r_0/2}(z_i)$, advancing time from $s+\frac{(i-1)r_0^2}{2}$ to $s+\frac{i r_0^2}{2}$:
\begin{equation}\label{eq3.10}
\sup_{B_{r_0/2}(z_i)}v\left(\cdot,s+\frac{(i-1)r_0^2}{2}\right)
\le C_F\inf_{B_{r_0/2}(z_i)}v\left(\cdot,s+\frac{i r_0^2}{2}\right).
\end{equation}
We prove by induction that for each $k=1,\ldots,N_d$,
\begin{equation}\label{eq3.11}
v(y_1,s)\le C_F^k\inf_{B_{r_0/2}(z_k)}v\left(\cdot,s+\frac{k r_0^2}{2}\right).
\end{equation}
For $k=1$, this follows directly from \eqref{eq3.10} with $i=1$. Assume \eqref{eq3.11} holds for $k$. Pick an intersection point $\tilde{z}_k\in B_{r_0/2}(z_k)\cap B_{r_0/2}(z_{k+1})\cap\gamma$ (which lies in $\Omega_d$ by construction). Then
 \[\inf_{B_{r_0/2}(z_k)}v\left(\cdot,s+\frac{k r_0^2}{2}\right)\le v\left(\tilde{z}_k,s+\frac{k r_0^2}{2}\right) \le\sup_{B_{r_0/2}(z_{k+1})}v\left(\cdot,s+\frac{k r_0^2}{2}\right).\]
Combining this with \eqref{eq3.10} ($i=k+1$) and the induction hypothesis yields
 \[v(y_1,s)\le C_F^{k+1}\inf_{B_{r_0/2}(z_{k+1})}v\left(\cdot,s+\frac{(k+1)r_0^2}{2}\right).\]
The induction is complete. Taking $k=N_d$ and noting that $y_2\in B_{r_0/2}(z_{N_d})$, we get
 \[v(y_1,s)\le C_F^{N_d}\,v\left(y_2,s+\frac{N_d r_0^2}{2}\right).\]
Since $y_1,y_2\in\Omega_d$ are arbitrary,
\begin{equation}\label{eq3.12}
\sup_{y\in\Omega_d}v(y,s)
\le C_F^{N_d}\inf_{y\in\Omega_d}v(y,s+\Delta s),
\qquad\Delta s:=\frac{N_d r_0^2}{2}.
\end{equation}
By \eqref{eq3.9},
\begin{equation}\label{eq3.13}
\Delta s=\frac{N_d r_0^2}{2}\le\frac{L_\Omega r_0}{\sqrt d}+r_0+\frac{r_0^2}{2}.
\end{equation}

\noindent\textbf{Step 6: Closing at the same time.}
Let $M(s):=\sup_{y\in\Omega_d}v(y,s)$ and $m(s):=\inf_{y\in\Omega_d}v(y,s)$. By Lemma \ref{lemma3.16}, for any $\Delta s>0$,
\begin{equation}\label{eq3.14}
M(s)\le {\rm e}^{c_0\Delta s}M(s-\Delta s).
\end{equation}
We emphasize that the estimate in Lemma \ref{lemma3.16} depends only on the $L^\infty$ bound of the zero-order term and is independent of the domain size and geometry; thus it applies uniformly to $\Omega_d$ for all $0<d\le d_0$.

Replacing $s$ by $s-\Delta s$ in \eqref{eq3.12} gives
\[
M(s-\Delta s)\le C_F^{N_d}\,m(s).
\]
Substituting this into \eqref{eq3.14} yields
\[
M(s)\le {\rm e}^{c_0\Delta s}C_F^{N_d}\,m(s),
\]
that is,
\begin{equation}\label{eq3.15}
\sup_{y\in\Omega_d}v(y,s)\le
C_F^{N_d} {\rm e}^{c_0\Delta s}\inf_{y\in\Omega_d}v(y,s)
\qquad\forall\,s\in\mathbb{R}.
\end{equation}

\noindent\textbf{Step 7: Constant estimation and scaling back.}
From \eqref{eq3.9} and \eqref{eq3.13}, we have
\[
\begin{aligned}
C_F^{N_d} {\rm e}^{c_0\Delta s}
&\le C_F^{\frac{2L_\Omega}{r_0\sqrt d}+\frac{2}{r_0}+1}
   {\rm e}^{c_0\left(\frac{L_\Omega r_0}{\sqrt d}+r_0+\frac{r_0^2}{2}\right)} \\
&=\exp\!\left(
   \frac{2L_\Omega\ln C_F}{r_0\sqrt d}
   +\left(\frac{2}{r_0}+1\right)\ln C_F
   +\frac{c_0 L_\Omega r_0}{\sqrt d}
   +c_0\left(r_0+\frac{r_0^2}{2}\right)
   \right).
\end{aligned}
\]
Since $d\le d_0\le1$, we have $1/\sqrt d\ge1$, so the constant terms can be absorbed into the exponential:
\[
\left(\frac{2}{r_0}+1\right)\ln C_F + c_0\left(r_0+\frac{r_0^2}{2}\right)
\le\frac{\left(\frac{2}{r_0}+1\right)\ln C_F + c_0\left(r_0+\frac{r_0^2}{2}\right)}{\sqrt d}.
\]
Therefore,
\[
C_F^{N_d} {\rm e}^{c_0\Delta s}\le
\exp\!\left(\frac{C'}{\sqrt d}\right),
\]
where
\[
C':=\frac{2L_\Omega}{r_0}\ln C_F + c_0 L_\Omega r_0
+\left(\frac{2}{r_0}+1\right)\ln C_F + c_0\left(r_0+\frac{r_0^2}{2}\right).
\]
Substituting this into \eqref{eq3.15} gives
\[
\sup_{y\in\Omega_d}v(y,s)\le
\exp\!\left(\frac{C'}{\sqrt d}\right)\inf_{y\in\Omega_d}v(y,s)
\qquad\forall\,s\in\mathbb{R}.
\]
Since $v(y,s)=u(\sqrt d\,y,\omega s)$, letting $s=t/\omega$ for any $t\in\mathbb{R}$ leads to
\[
\sup_{x\in\Omega}u(x,t)
=\sup_{y\in\Omega_d}v(y,t/\omega)
\le {\rm e}^{C'/\sqrt d}\inf_{y\in\Omega_d}v(y,t/\omega)
={\rm e}^{C'/\sqrt d}\inf_{x\in\Omega}u(x,t).
\]
This is exactly \eqref{eq3.7}. The constant $C'$ depends on $n,\lambda,\Lambda,c_0,K_A,\Omega$, but is independent of $\omega$ and $d$. The proof is complete.
\end{proof}

Recall that the normalized principal Floquet bundle $(H,\varphi)$ of \eqref{eq1.1} is determined by
\[(H(t),\varphi(x,t))=\left(-\omega\frac{\mathrm{d}}{\mathrm{d}t}\left(\ln \int_{\Omega}{u}(x,t)\,\mathrm{d}x\right),u(x,t)\mathrm{e}^{\frac{1}{\omega}\int^{t}_{0}H(s) \,\mathrm{d}s}\right), \quad(x,t)\in\Omega\times\mathbb{R},\]
where $u$ is the unique solution of \eqref{eq3.1}. Now we improve the Harnack estimate in Proposition \ref{proposition3.3}(ii).
\begin{theorem}\label{theorem3.17}
  Let $(\varphi,\psi)$ be the normalized principal Floquet bundle of \eqref{eq1.1} and \eqref{eq1.2}. Then there is a constant $C>0$ independent of $\omega>0$ and $d\in(0,1)$, such that
         \[{\rm e}^{-\frac{C}{\sqrt{d}}}\le \varphi(x,t),\psi(x,t)\le {\rm e}^{\frac{C}{\sqrt{d}}},\qquad \forall~(x,t)\in\overline{\Omega}\times\mathbb{R}.\]
\end{theorem}
\begin{proof}
By Theorem \ref{theorem3.14}, we know there is some positive constant $C$, independent of $\omega>0$ and $d\in(0,1)$ such that
  \[\sup_{x\in\Omega}\varphi(x,t)\le\exp\!\left(\frac{C}{\sqrt d}\right)\inf_{x\in\Omega}\varphi(x,t), \qquad\forall~t\in\mathbb{R}.\]
Therefore, for any $x,y\in\overline{\Omega}$, one has
  \[\exp\!\left(-\frac{C}{\sqrt d}\right)\varphi(y,t)\le \varphi(x,t)\le\exp\!\left(\frac{C}{\sqrt d}\right)\varphi(y,t), \qquad\forall~t\in\mathbb{R}.\]
Integrating above inequalities with respect to $y$ and noting that $|\Omega|=1$ and $\int_\Omega\varphi(y,t)\,\mathrm{d}y=1$ for all $t\in\mathbb{R}$ yield
  \[\exp\!\left(-\frac{C}{\sqrt d}\right)\le \varphi(x,t)\le\exp\!\left(\frac{C}{\sqrt d}\right), \qquad\forall~t\in\mathbb{R}.\]
One can obtain the Harnack estimate for $\psi$ by a similar argument. This completes the proof.
\end{proof}

\chapter{Asymptotic behavior of $H$: $d\to0$ with fixed $\omega\in(0,\infty)$}\label{chp4}
From this chapter through Chapter 7, our basic approach to studying the limiting behavior of the principal Floquet exponent is to estimate its upper and lower bounds separately. On one hand, we estimate the upper bound by utilizing the principal eigenvalue of the time-averaged linear elliptic problem (see Lemma \ref{lemma3.12}, Corollary \ref{corollary3.13}, etc.), constructing sub- and super-solutions, and applying the comparison principle (see Section 2.2). On the other hand, the lower bound is obtained mainly by combining the Harnack principle (see Proposition \ref{proposition3.3}(ii)) to construct suitable sub- and super-solutions, and then employing the comparison principle (see Section 2.2). To clearly reflect this analytical strategy, we devote two sections in the current chapter to estimating the upper and lower bounds separately, thereby determining the limit. In the following three chapters, however, the estimates for the upper and lower bounds will be presented together within a single lemma, and some repeated analysis and computational details will be omitted for brevity.

This chapter focuses on the asymptotic behavior of the principal Floquet bundle $H$ of \eqref{eq1.1} for small diffusion rate $d$ with fixed $\omega>0$, and we will use $H(\cdot;d)$ to emphasize the dependence on $d$. Recall the notations $\mathscr{C}$ and $\mathscr{C}_*$: for a given function $c\in\mathcal{C}$,
 \[\mathscr{C}:=\left\{\hat{c}\in C (\overline{\Omega})~|~\text{there exists }T_n\to\infty\text{ such that }\|\hat{c}_{T_n}-\hat{c}\|_{C(\overline{\Omega})}\to 0\right\}\]
and for a given sequence $\{T_n\}_{n\in\mathbb{N}^+}$ of $T\to\infty$,
 \[\mathscr{C}_*:=\left\{\hat{c}\in C (\overline{\Omega})~|~\left\{\hat{c}_n \right\}_{n\in\mathbb{N}^+}\text{ admits a convergent subsequence converging to }\hat{c}\text{ in } C(\overline{\Omega})\right\}.\]
This chapter presents two principal results: one pertaining to the sequence $\{T_n\}_{n\in\mathbb{N}^+}$ of $T\to\infty$ (Theorem \ref{theorem4.1}), and the other dealing with the limit $T \to \infty$.
\begin{theorem}\label{theorem4.1}
Assume that $c\in\mathcal{C}\cap C^{2,1}(\overline{\Omega}\times\mathbb{R})$ satisfies \eqref{H2}. Let $(H(\cdot;d),\varphi)$ be the normalized principal Floquet bundle of \eqref{eq1.1}, and $\{T_n\}_{n\in\mathbb{N}^+}$ be a sequence satisfying $T_n\to\infty$ as $n\to\infty$. Then
  \[\lim\limits_{d\to 0}\liminf\limits_{n\to\infty}\fint^{T_n}_{0}H(s;d)\,\mathrm{d}s= \min\limits_{x\in\overline{\Omega}}\liminf_{n\to\infty}\fint^{T_n}_{0} c(x,s)\,\mathrm{d}s= \inf_{\hat{c}\in\mathscr{C}_*}\min_{x\in\overline{\Omega}}\hat{c}(x).\]
and
  \[\lim\limits_{d\to 0}\limsup\limits_{n\to\infty}\fint^{T_n}_{0}H(s;d)\,\mathrm{d}s= \limsup_{n\to\infty}\min\limits_{x\in\overline{\Omega}}\fint^{T_n}_{0} c(x,s)\,\mathrm{d}s= \sup_{\hat{c}\in\mathscr{C}_*}\min_{x\in\overline{\Omega}}\hat{c}(x).\]
Consequently, if the sequence $\left\{\fint^{T_n}_{0} c(\cdot,s)\,\mathrm{d}s\right\}_{n\in\mathbb{N}^+}$ of functions converges in $C(\overline{\Omega})$, then
  \[\lim\limits_{d\to 0}\liminf\limits_{n\to\infty}\fint^{T_n}_{0}H(s;d)\,\mathrm{d}s= \lim\limits_{d\to 0}\limsup\limits_{n\to\infty}\fint^{T_n}_{0}H(s;d)\,\mathrm{d}s = \min\limits_{x\in\overline{\Omega}}\lim_{n\to\infty}\fint^{T_n}_{0} c(x,s)\,\mathrm{d}s.\]
\end{theorem}

\section{Upper bound}
In this section, we investigate the upper bound of the principal Floquet exponent.
\begin{lemma}\label{lemma4.2}
Let $(H,\varphi)$ be the normalized principal Floquet bundle of \eqref{eq1.1} with potential $c\in\mathcal{C}$, and $\{T_n\}_{n\in\mathbb{N}^+}$ be a sequence satisfying $T_n\to\infty$ as $n\to\infty$. Then
\begin{equation}\label{eq4.1}
 \begin{aligned}
   &\limsup\limits_{d\to 0}\liminf\limits_{n\to\infty}\fint^{T_n}_{0} H(s;d)\,\mathrm{d}s\le \inf_{\hat{c}\in\mathscr{C}_*}\min\limits_{x\in\overline{\Omega}}\hat{c}(x) =\min\limits_{x\in\overline{\Omega}}\liminf\limits_{n\to\infty} \fint^{T_n}_{0}c(x,s)\,\mathrm{d}s\\
   \text{and}\qquad
   &\limsup\limits_{d\to 0}\limsup\limits_{n\to\infty}\fint^{T_n}_{0} H(s;d)\,\mathrm{d}s\le \sup_{\hat{c}\in\mathscr{C}_*}\min\limits_{x\in\overline{\Omega}}\hat{c}(x) =\limsup\limits_{n\to\infty} \min\limits_{x\in\overline{\Omega}}\fint^{T_n}_{0}c(x,s)\,\mathrm{d}s.
 \end{aligned}
\end{equation}
Consequently, if the sequence $\left\{\fint^{T_n}_{0} c(\cdot,s)\,\mathrm{d}s\right\}_{n\in\mathbb{N}^+}$ of functions converges to some function $\hat{c}(\cdot)$ in $C(\overline{\Omega})$, then
 \[\limsup\limits_{d\to 0}\limsup\limits_{n\to\infty}\fint^{T_n}_{0} H(s;d)\,\mathrm{d}s\le \min\limits_{x\in\overline{\Omega}}\hat{c}(x).\]
\end{lemma}
\begin{proof}
Using Lemma \ref{lemma3.12} gives
 \[\liminf\limits_{n\to\infty}\fint^{T_n}_{0} H(s;d)\,\mathrm{d}s\le \inf_{\hat{c}\in\mathscr{C}_*}\mu^\infty(d,\hat{c})\quad\text{and}\quad\limsup\limits_{n\to\infty}\fint^{T_n}_{0} H(s;d)\,\mathrm{d}s\le \sup_{\hat{c}\in\mathscr{C}_*}\mu^\infty(d,\hat{c}),\]
where $\mu^\infty(d,\hat{c})$ is the principal eigenvalue of problem \eqref{eq1.3} with potential $\hat{c}$. By letting $d\to0$ and applying Lemmas \ref{lemma2.3}, \ref{lemma2.4} and \ref{lemma2.18}, we obtain \eqref{eq4.1}. This completes the proof.
\end{proof}

\begin{corollary}\label{corollary4.3}
Let $(H,\varphi)$ be the principal Floquet bundle of \eqref{eq1.1} with potential $c\in\mathcal{C}$. Then
 \[\limsup\limits_{d\to0}\liminf\limits_{T\to\infty}\fint^{T}_{0}H(s;d)\,\mathrm{d}s\le \inf\limits_{\hat{c}\in\mathscr{C}}\min_{x\in\overline{\Omega}}\hat{c}(x)= \min\limits_{x\in\overline{\Omega}}\liminf\limits_{T\to\infty} \fint^{T}_{0}c(x,s)\,\mathrm{d}s.\]
\end{corollary}
\begin{proof}
In view of Proposition \ref{proposition1.1}(iii), we know that the point-wise defined function
 \[\hat{c}_\infty(x):=\liminf\limits_{T\to\infty}\fint^{T}_{0}c(x,s)\,\mathrm{d}s\]
is continuous on $\overline{\Omega}$. So $\hat{c}_\infty$ must attains its minimum at some point $x_0 \in \overline{\Omega}$. Then there is a sequence $\{T_n\}_{n\in\mathbb{N}^+}$ of $T\to\infty$ such that
 \[\hat{c}_\infty(x_0)=\lim\limits_{n\to\infty}\fint^{T_n}_{0}c(x_0,s)\,\mathrm{d}s=\min\limits_{x\in\overline{\Omega}}\liminf\limits_{T\to\infty} \fint^{T}_{0}c(x,s)\,\mathrm{d}s.\]
Using the assertion (i) in Proposition \ref{proposition1.1}, there is a subsequence $\{T_{n_k}\}_{k\in\mathbb{N}^+}$ of $\{T_n\}_{n\in\mathbb{N}^+}$ such that
 \[\left\{\fint^{T_{n_k}}_{0}c(\cdot,s)\,\mathrm{d}s\right\}_{k\in\mathbb{N}^+}\text{ converges to some function }\hat{c}(\cdot)\text{ in }C(\overline{\Omega}).\]
Obviously, $\hat{c}_\infty(x)\le \hat{c}(x)$ for all $x\in\overline{\Omega}$, and they share a common minimum at $x_0$. Using Lemma \ref{lemma4.2}, we have
\begin{align*}
  \limsup\limits_{d\to0}\liminf\limits_{T\to\infty}\fint^{T}_{0}H(s;d)\,\mathrm{d}s&\le \limsup\limits_{d\to0}\liminf\limits_{n\to\infty}\fint^{T_n}_{0}H(s;d)\,\mathrm{d}s \\
  &\le \hat{c}(x_0)=\hat{c}_\infty(x_0)=\min\limits_{x\in\overline{\Omega}}\liminf\limits_{T\to\infty} \fint^{T}_{0}c(x,s)\,\mathrm{d}s.
\end{align*}
Using Lemma \ref{lemma2.4}, we obtain the desired result. This completes the proof.
\end{proof}
\begin{lemma}\label{lemma4.4}
   Let $(H,\varphi)$ be the normalized principal Floquet bundle of \eqref{eq1.1}. Then
  \[\limsup\limits_{d\to0}\limsup\limits_{T\to\infty}\fint^{T}_{0}H(s;d)\,\mathrm{d}s\le \sup\limits_{\hat{c}\in\mathscr{C}}\min_{x\in\overline{\Omega}}\hat{c}(x)= \limsup\limits_{T\to\infty} \min\limits_{x\in\overline{\Omega}} \fint^{T}_{0}c(x,s)\,\mathrm{d}s.\]
\end{lemma}
\begin{proof}
By Corollary \ref{corollary3.13}, we have
 \[\limsup\limits_{T\to\infty}\fint^{T}_{0}H(s;d)\,\mathrm{d}s\le \sup\limits_{\hat{c}\in\mathscr{C}}\mu^\infty(d,\hat{c}),\]
where $\mu^\infty(d,\hat{c})$ is the principal eigenvalue of \eqref{eq1.3}. Letting $d\to0$ and applying Lemmas \ref{lemma2.18} and \ref{lemma2.3} yield
 \[\limsup_{d\to0}\limsup\limits_{T\to\infty}\fint^{T}_{0}H(s;d)\,\mathrm{d}s\le \limsup_{d\to0}\sup\limits_{\hat{c}\in\mathscr{C}}\mu^\infty(d,\hat{c}) =\sup\limits_{\hat{c}\in\mathscr{C}}\min_{x\in\overline{\Omega}}\hat{c}(x)= \limsup\limits_{T\to\infty} \min\limits_{x\in\overline{\Omega}} \fint^{T}_{0}c(x,s)\,\mathrm{d}s.\]
This completes the proof.
\end{proof}

\section{Lower bound}
We now study the lower bound.
\begin{lemma}\label{lemma4.5}
Let $(H,\varphi)$ be the normalized principal Floquet bundle of \eqref{eq1.1} with potential $c\in\mathcal{C}\cap C^{2,1}(\overline{\Omega}\times\mathbb{R})$, and $\{T_n\}_{n\in\mathbb{N}^+}$ be a sequence satisfying $T_n\to\infty$ as $n\to\infty$. Assume that $c$ satisfies assumption \eqref{H2}. Then
\begin{equation}\label{eq4.2}
\begin{aligned}
  &\liminf\limits_{d\to 0}\liminf\limits_{n\to\infty}\fint^{T_n}_{0}H(s;d)\,\mathrm{d}s \ge \min\limits_{x\in\overline{ \Omega}} \liminf_{n\to\infty} \fint^{T_n}_{0}c(x,s)\,\mathrm{d}s =\inf_{\hat{c}\in\mathscr{C}_*}\min_{x\in\overline{\Omega}}\hat{c}(x) \\
  \text{and}\quad&\liminf\limits_{d\to 0}\limsup\limits_{n\to\infty}\fint^{T_n}_{0}H(s;d)\,\mathrm{d}s \ge \limsup_{n\to\infty} \min\limits_{x\in\overline{ \Omega}}\fint^{T_n}_{0}c(x,s)\,\mathrm{d}s =\sup_{\hat{c}\in\mathscr{C}_*} \min\limits_{x\in\overline{\Omega}}\hat{c}(x).
\end{aligned}
\end{equation}
Consequently, if the sequence $\left\{\fint^{T_n}_{0} c(\cdot,s)\,\mathrm{d}s\right\}_{n\in\mathbb{N}^+}$ of functions converges to some function $\hat{c}(\cdot)$ in $C(\overline{\Omega})$, then
 \[\liminf\limits_{d\to 0}\limsup\limits_{n\to\infty}\fint^{T_n}_{0}H(s;d)\,\mathrm{d}s \ge \liminf\limits_{d\to 0}\liminf\limits_{n\to\infty}\fint^{T_n}_{0}H(s;d)\,\mathrm{d}s \ge \min\limits_{x\in\overline{\Omega}}\hat{c}(x).\]
\end{lemma}
\begin{proof}
Recall that $\mathcal{M}(x,t;T)=t\fint^{T}_{0}c(x,s)\,\mathrm{d}s-\int^{t}_{0}c(x,s)\,\mathrm{d}s,~ (x,t)\in\overline{\Omega}\times[0,T]$ for any given $T>0$. Let $\phi^*$ be the principal eigenfunction of \eqref{eq2.6}, which belongs to $C^2(\overline{\Omega})$. By assumption \eqref{H2} and Lemma \ref{lemma2.16}, there exist constants $T_0>0$ and $\kappa>0$ such that for any $T>T_0$,
\begin{equation}\label{eq4.3}
  \mathbf{n}\cdot (A(x)\nabla\mathcal{M}(x,t;T))-\kappa\mathbf{n}\cdot(A(x)\nabla\phi^*(x))\ge 0,\qquad (x,t)\in\partial\Omega\times[0,T].
\end{equation}
Moreover, by assumption \eqref{H2}, there is a positive constant $C^*>0$ independent of $T>T_0$ such that
\begin{equation}\label{eq4.4}
\begin{aligned}
  &\Big|\frac{1}{\omega}\operatorname{div}(A(x)\nabla\mathcal{M}(x,t;T_n))-\frac{\kappa}{\omega}\operatorname{div}(A(x)\nabla\phi^*(x))\\
  &\quad +\frac{1}{\omega^2}\nabla\left(\mathcal{M}(x,t;T_n)-\kappa\phi^*(x)\right)\cdot\left(A(x)\nabla\left(\mathcal{M}(x,t;T_n) -\kappa\phi^*(x)\right)\right)\Big|\le C^*,\quad (x,t)\in \Omega\times(0,T].
\end{aligned}
\end{equation}
Let $\varepsilon>0$ be any given constant and set $d_\varepsilon:=\frac{\varepsilon}{2 C^*}$. For each $n\in\mathbb{N}^+$, define
\[\overline{\varphi}_n(x,t):=\exp\left[\frac{1}{\omega}\left(\mathcal{M}(x,t;T_n)-\kappa\phi^*(x)\right)\right], \qquad (x,t)\in\overline{\Omega}\times[0,T_n]\]
and
\[\underline{\varphi}_{n}(x,t):=\varphi(x,t;d) \exp\left[\frac{1}{\omega}\int^{t}_{0}\left(\fint^{T_{n}}_{0}H(r;d) \,\mathrm{d}r -H(s;d)+\frac{\varepsilon}{2}\right)\,\mathrm{d}s\right],\qquad (x,t)\in \overline{\Omega}\times[0,T_{n}].\]
Direct computations yield that for each $n\in\mathbb{N}^+$, the function $\overline{\varphi}_n$ satisfies
\begin{equation}\label{eq4.5}
\overline{\varphi}_{n}(x,0)=\overline{\varphi}_{n}(x,T_n), \qquad x\in\Omega,
\end{equation}
\begin{equation}\label{eq4.6}
\mathbf{n}\cdot(A(x)\nabla\overline{\varphi}_n)=\frac{1}{\omega}\left[\mathbf{n}\cdot (A(x)\nabla\mathcal{M}(x,t;T_n))-\kappa \mathbf{n}\cdot(A(x)\nabla\varphi^{*}(x)) \right]\overline{\varphi}_n,~(x,t)\in\partial\Omega\times(0,T_n],
\end{equation}
and
\begin{equation}\label{eq4.7}
\begin{aligned}
  &\omega\partial_t\overline{\varphi}_{n}-d\operatorname{div}(A(x)\nabla\overline{\varphi}_{n})+c(x,t)\overline{\varphi}_{n} \\
  =&\left(\fint^{T_n}_{0}c(x,s)\,\mathrm{d}s- d\frac{\operatorname{div}(A(x)\nabla\overline{\varphi}_{n})}{\overline{\varphi}_{n}}\right) \overline{\varphi}_{n}\\
  \ge&\left(\fint^{T_n}_{0}c(x,s)\,\mathrm{d}s- d\left|\frac{\operatorname{div}(A(x)\nabla\overline{\varphi}_{n})}{\overline{\varphi}_{n}}\right|\right) \overline{\varphi}_{n}\\
  =&\Bigg[\fint^{T_n}_{0}c(x,s)\,\mathrm{d}s- d\Big(\frac{1}{\omega}\Big|\operatorname{div}(A(x)\nabla\mathcal{M}(x,t;T_n))-\frac{\kappa}{\omega}\operatorname{div}(A(x)\nabla\phi^*(x))\\
  &\hskip 2cm +\frac{1}{\omega^2}\nabla\left(\mathcal{M}(x,t;T_n)-\kappa\phi^*(x)\right)\cdot\left(A(x)\nabla\left(\mathcal{M}(x,t;T_n) -\kappa\phi^*(x)\right)\right)\Big|\Big)\Bigg]\overline{\varphi}_{n}\\
  \ge &\left[\fint^{T_n}_{0}c(x,s)\,\mathrm{d}s- dC^*\right]\overline{\varphi}_{n},\qquad\qquad (x,t)\in\Omega\times(0,T_n],
\end{aligned}
\end{equation}
and the function $\underline{\varphi}_n$ satisfies
\begin{equation}\label{eq4.8}
\mathbf{n}\cdot(A(x)\nabla\underline{\varphi}_{n})=0,\qquad(x,t)\in\partial\Omega\times(0,T_{n}],
\end{equation}
\begin{equation}\label{eq4.9}
\underline{\varphi}_{n}(x,0)=\left(\frac{\varphi(x,0)}{\varphi(x,T_n)} \mathrm{e}^{-\frac{\varepsilon}{2\omega}T_n}\right)\underline{\varphi}_{n}(x,T_n), \qquad x\in\overline{\Omega}
\end{equation}
and
\begin{equation}\label{eq4.10}
  \omega\partial_t\underline{\varphi}_{n}-d\operatorname{div}(A(x)\nabla\underline{\varphi}_{n})+c(x,t)\underline{\varphi}_{n}=\left(\fint^{T_{n}}_{0} H(s;d)\,\mathrm{d}s +\frac{\varepsilon}{2}\right) \underline{\varphi}_{n},\qquad (x,t)\in\Omega\times(0,T_{n}].
\end{equation}

Now, we fix $d\in(0,d_\varepsilon)$. By Harnack principle (Proposition \ref{proposition3.3}(ii)), there is a positive constant $C_d$ such that
 \[\frac{1}{C_d}\le \varphi(x,t)\le C_d,\quad\forall~(x,t)\in\overline{\Omega}\times\mathbb{R}.\]
So there is a positive integer $N_\varepsilon>0$ such that
\[T_n>T_0\quad\text{ and }\quad\frac{\varphi(x,0)}{\varphi(x,T_n)} \mathrm{e}^{-\frac{\varepsilon}{2\omega}T_n}<1,\qquad\forall~n>N_\varepsilon,~x\in\overline{\Omega}.\]
This, together with \eqref{eq4.3}-\eqref{eq4.10}, yields
\begin{equation*}
  \left\{
  \begin{aligned}
    &\omega\partial_t\overline{\varphi}_{n}-d\operatorname{div}(A(x)\nabla\overline{\varphi}_{n})+c(x,t)\overline{\varphi}_{n}\ge \left(\min\limits_{x\in\overline{ \Omega}}\fint^{T_n}_{0}c(x,s)\,\mathrm{d}s-\frac{\varepsilon}{2}\right)\overline{\varphi}_{n}, &&(x,t)\in\Omega\times(0,T_n],\\
    &\omega\partial_t\underline{\varphi}_{n}-d\operatorname{div}(A(x)\nabla\underline{\varphi}_{n})+c(x,t)\underline{\varphi}_{n}= \left(\fint^{T_{n}}_{0} H(s;d)\,\mathrm{d}s +\frac{\varepsilon}{2}\right) \underline{\varphi}_{n},&& (x,t)\in\Omega\times(0,T_{n}],\\
    &\mathbf{n}\cdot(A(x)\nabla\overline{\varphi}_{n})\ge 0 \ge \mathbf{n}\cdot(A(x)\nabla\underline{\varphi}_{n}),&&(x,t)\in\partial\Omega\times(0,T_n],\\
    &\overline{\varphi}_{n}(x,0)\ge\overline{\varphi}_{n}(x,T_n),~ \underline{\varphi}_{n}(x,0)\le\underline{\varphi}_{n}(x,T_n),&&x\in\overline{\Omega}.
  \end{aligned}
  \right.
\end{equation*}
Applying Lemma \ref{lemma2.13}, we have
 \[\liminf_{n\to\infty}\fint^{T_{n}}_{0} H(s;d)\,\mathrm{d}s +\frac{\varepsilon}{2}\ge \liminf_{n\to\infty} \min\limits_{x\in\overline{ \Omega}}\fint^{T_n}_{0}c(x,s)\,\mathrm{d}s-\frac{\varepsilon}{2}\]
and
 \[\limsup_{n\to\infty}\fint^{T_{n}}_{0} H(s;d)\,\mathrm{d}s +\frac{\varepsilon}{2}\ge \limsup_{n\to\infty} \min\limits_{x\in\overline{ \Omega}}\fint^{T_n}_{0}c(x,s)\,\mathrm{d}s-\frac{\varepsilon}{2}.\]
Then we use Lemmas \ref{lemma2.3} and \ref{lemma2.4} to obtain that
 \[\liminf_{n\to\infty}\fint^{T_{n}}_{0} H(s;d)\,\mathrm{d}s \ge \inf_{\hat{c}\in\mathscr{C}_*} \min\limits_{x\in\overline{ \Omega}}\hat{c}(x)-\varepsilon\]
and
 \[\limsup_{n\to\infty}\fint^{T_{n}}_{0} H(s;d)\,\mathrm{d}s \ge \sup_{\hat{c}\in\mathscr{C}_*} \min\limits_{x\in\overline{ \Omega}}\hat{c}(x)-\varepsilon.\]
Now letting $d\to0$, one gets
 \[\liminf_{d\to0}\liminf_{n\to\infty}\fint^{T_{n}}_{0} H(s;d)\,\mathrm{d}s \ge \inf_{\hat{c}\in\mathscr{C}_*} \min\limits_{x\in\overline{ \Omega}}\hat{c}(x)-\varepsilon\]
and
 \[\liminf_{d\to0}\limsup_{n\to\infty}\fint^{T_{n}}_{0} H(s;d)\,\mathrm{d}s \ge \sup_{\hat{c}\in\mathscr{C}_*} \min\limits_{x\in\overline{ \Omega}}\hat{c}(x)-\varepsilon.\]
Then by the arbitrariness of $\varepsilon>0$, we get \eqref{eq4.2}. This completes the proof.
\end{proof}

\medskip
\begin{lemma}\label{lemma4.6}
Let $c\in\mathcal{C}\cap C^{2,1}(\overline{\Omega}\times\mathbb{R})$ be given, and $(H,\varphi)$ be the normalized principal Floquet bundle of \eqref{eq1.1}. Suppose that assumption \eqref{H2} holds, then
    \[\liminf\limits_{d\to0}\liminf\limits_{T\to\infty}\fint^{T}_{0}H(s;d)\,\mathrm{d}s\ge \inf_{\hat{c}\in\mathscr{C}}\min_{x\in\overline{\Omega}}\hat{c}(x)= \min\limits_{x\in\overline{\Omega}}\liminf\limits_{T\to\infty} \fint^{T}_{0}c(x,s)\,\mathrm{d}s\]
and
  \[\liminf\limits_{d\to0}\limsup\limits_{T\to\infty}\fint^{T}_{0}H(s;d)\,\mathrm{d}s\ge \sup_{\hat{c}\in\mathscr{C}}\min_{x\in\overline{\Omega}}\hat{c}(x)= \limsup\limits_{T\to\infty}\min\limits_{x\in\overline{\Omega}} \fint^{T}_{0}c(x,s)\,\mathrm{d}s.\]
\end{lemma}
\begin{proof}
The main idea of the proof is similar to that of Lemma \ref{lemma4.5}. Let $\phi^*$ be the principal eigenfunction of \eqref{eq2.6}, $\kappa$, $T_0$ and $C^*$ be constants such that \eqref{eq4.3} and \eqref{eq4.4} hold. Let $\varepsilon>0$ be any given constant. We define two families $\{\overline{\varphi}_T\}_{T>0}$ and $\{\underline{\varphi}_T\}_{T>0}$ of functions as
   \[\overline{\varphi}_T(x,t):=\exp\left[\frac{1}{\omega}\left(\mathcal{M}(x,t;T)-\kappa\phi^*(x)\right)\right], \qquad (x,t)\in\overline{\Omega}\times[0,T]\]
and
   \[\underline{\varphi}_T(x,t):=\varphi(x,t) \exp \left[\frac{1}{\omega}\int^{t}_{0}\left(\fint^{T}_{0}H(r;d) \,\mathrm{d}r -H(s;d)+\frac{\varepsilon}{2}\right)\,\mathrm{d}s\right],\qquad (x,t)\in \overline{\Omega}\times[0,T].\]
   It is easy to verify that
   \[\left|\frac{\operatorname{div}(A(x)\nabla\overline{\varphi}_T)}{\overline{\varphi}_T}\right|\le C^*,\qquad (x,t)\in \Omega\times(0,T].\]
Then one can check that for each $T>T_0$, $\overline{\varphi}_T(x,t)$ and $\underline{\varphi}_T(x,t)$ satisfy
\begin{equation*}
  \left\{
  \begin{aligned}
    &\omega\partial_t\overline{\varphi}_T-d\operatorname{div}(A(x)\nabla\overline{\varphi}_T)+c(x,t)\overline{\varphi}_T\\
    &\quad\ge \left( \fint^{T}_{0}c(x,s)\,\mathrm{d}s -dC^*\right)\overline{\varphi}_T
    \ge \left(\min_{x\in\overline{\Omega}} \fint^{T}_{0}c(x,s)\,\mathrm{d}s -dC^*\right)\overline{\varphi}_T, &&(x,t) \in\Omega\times(0,T],\\
    &\omega\partial_t\underline{\varphi}_T-d\operatorname{div}(A(x)\nabla\underline{\varphi}_T)+c(x,t)\underline{\varphi}_T=\left(\fint^{T}_{0} H(s;d)\,\mathrm{d}s +\frac{\varepsilon}{2}\right) \underline{\varphi}_T,&&(x,t)\in\Omega\times(0,T],\\
    &\mathbf{n}\cdot(A(x)\nabla\overline{\varphi}_T)\ge 0 = \mathbf{n}\cdot(A(x)\nabla\underline{\varphi}_T),&&(x,t)\in\partial\Omega\times(0,T],\\
    &\overline{\varphi}_T(x,0)=\overline{\varphi}_T(x,T),&&x\in\overline{\Omega},\\
    &\underline{\varphi}_T(x,0)=\left(\frac{\varphi(x,0)}{\varphi(x,T)} \mathrm{e}^{-\frac{\varepsilon}{2\omega}T}\right)\underline{\varphi}_T(x,T), &&x\in\overline{\Omega}.
  \end{aligned}
  \right.
\end{equation*}
Then by the Harnack principle (Proposition \ref{proposition3.3}(ii)), for any $d>0$, there is a constant $C_d>1$ independent of $T>T_0$ such that
\[\frac{\varphi(x,0)}{\varphi(x,T)}<C_{d}^{2}.\]
So there is a positive constant, denoted still by $T_0$, such that
\[C_{d}^{2}\mathrm{e}^{-\frac{\varepsilon}{2\omega}T}<1,\qquad\forall~T>T_0.\]
Applying Corollary \ref{corollary2.15} yields
 \[\liminf\limits_{T\to\infty}\fint^{T}_{0}H(s;d)\,\mathrm{d}s+ \frac{\varepsilon}{2}\ge \liminf\limits_{T\to\infty}\min\limits_{x\in\overline{\Omega}} \fint^{T}_{0}c(x,s)\,\mathrm{d}s-dC^*\]
and
 \[\limsup\limits_{T\to\infty}\fint^{T}_{0}H(s;d)\,\mathrm{d}s+ \frac{\varepsilon}{2}\ge \limsup\limits_{T\to\infty}\min\limits_{x\in\overline{\Omega}} \fint^{T}_{0}c(x,s)\,\mathrm{d}s-dC^*.\]
Using Lemma \ref{lemma2.4}, we have
\[\liminf_{T\to\infty}\fint^{T}_{0} H(s;d)\,\mathrm{d}s +\frac{\varepsilon}{2}\ge \min\limits_{x\in\overline{\Omega}}\liminf_{T\to\infty}\fint^{T}_{0}c(x,s)\,\mathrm{d}s -dC^*.\]
Letting $d\to 0$ and by the arbitrariness of $\varepsilon>0$, we easily obtain the inequalities in the lemma. Finally, a direct application of Lemmas \ref{lemma2.3} and \ref{lemma2.4} yields the identities in the lemma. This finishes the proof.
\end{proof}

\medskip
Theorem \ref{theorem1.5} follows directly from Lemma \ref{lemma4.2}, Corollary \ref{corollary4.3}, and Lemmas \ref{lemma4.4}-\ref{lemma4.6}.

\chapter{Asymptotic behavior of $H$: $\omega\to0$ and $\omega\to\infty$ with fixed $d>0$}\label{chp5}
This chapter is concerned with the asymptotic analysis of the principal Floquet bundle under a fixed diffusion coefficient $d>0$, considering both the slow-frequency limit ($\omega \to 0$) and the fast-frequency limit ($\omega \to \infty$). Two main results are established: Theorem \ref{theorem5.1} (proved mainly with respect to a sequence $\{T_n\}_{n \in \mathbb{N}^+}$ with $T_n \to \infty$ as $n\to\infty$), and Theorem \ref{theorem1.6}. The arguments proceed by establishing a series of lemmas. To emphasize the dependence of $H$ on $\omega>0$ for fixed $d>0$, we denote by $H(\cdot;\omega)$ the normalized principal Floquet bundle. For a given function $c\in\mathcal{C}$, recall the set $\mathscr{C}$ of functions defined as
  \[\mathscr{C}:=\left\{\hat{c}\in C (\overline{\Omega})~|~\text{there exists }T_n\to\infty\text{ such that }\|\hat{c}_{T_n}-\hat{c}\|_{C(\overline{\Omega})}\to 0\right\},\]
where $\hat{c}_T(\cdot)=\fint^{T}_{0}c(\cdot,s)\,\mathrm{d}s$; and set $\mathscr{C}$ of functions defined as
  \[\mathscr{C}_*=\left\{\hat{c}\in C (\overline{\Omega})~|~\left\{\hat{c}_n\right\}_{n\in\mathbb{N}^+}\text{ admits a convergent subsequence converging to }\hat{c}\text{ in } C(\overline{\Omega})\right\}\]
for some sequence $\{T_n\}_{n\in\mathbb{N}^+}$ of $T\to\infty$, where $\hat{c}_n(\cdot)=\hat{c}_{T_n}(\cdot)$.

\begin{theorem}\label{theorem5.1}
Let $H(\cdot;\omega)$ be the normalized principal Floquet bundle of \eqref{eq1.1} with potential $c$ and $\{T_n\}_{n\in\mathbb{N}^+}$ be a sequence satisfying $T_n\to\infty$ as $n\to\infty$.
\begin{enumerate}[{\upshape (i)}]
  \item {\rm (}The case $\omega\to0${\rm )} Assume that $c\in \mathcal{C}\cap C^{0,1}(\overline{\Omega} \times\mathbb{R})$ satisfies $\sup_{t\in\mathbb{R}}\|\partial_t c(\cdot,t)\|_\infty<\infty$. For each $t\in\mathbb{R}$, denote by $\mu^{0}(t)$ the principal eigenvalue of \eqref{eq1.4} with $c(\cdot,t)$. Then
      \[\lim_{\omega\to0}\liminf_{n\to\infty}\fint^{T_n}_{0}H(s;\omega)\,\mathrm{d}s= \liminf_{n\to\infty}\fint^{T_n}_{0}\mu^0(s)\,\mathrm{d}s\]
      and
      \[\lim_{\omega\to0}\limsup_{n\to\infty}\fint^{T_n}_{0}H(s;\omega)\,\mathrm{d}s= \limsup_{n\to\infty}\fint^{T_n}_{0}\mu^0(s)\,\mathrm{d}s.\]
      Consequently, if the limit $\lim_{n\to\infty}\fint^{T_n}_{0}\mu^0(s)\,\mathrm{d}s$ exists, then
       \[\lim_{\omega\to0}\liminf_{n\to\infty}\fint^{T_n}_{0}H(s;\omega)\,\mathrm{d}s= \lim_{\omega\to0}\limsup_{n\to\infty}\fint^{T_n}_{0}H(s;\omega)\,\mathrm{d}s= \lim_{n\to\infty}\fint^{T_n}_{0}\mu^0(s)\,\mathrm{d}s.\]
  \item {\rm (}The case $\omega\to\infty${\rm )} Assume that $c\in \mathcal{C}\cap C^{2,1}(\overline{\Omega}\times\mathbb{R})$ satisfies assumption \eqref{H2}. Denote by $\mu^\infty$ the principal eigenvalue of \eqref{eq1.3} with potential $\hat{c}$. Then
       \[\lim_{\omega\to\infty}\liminf_{n\to\infty}\fint^{T_n}_{0}H(s;\omega)\,\mathrm{d}s= \inf_{\hat{c}\in\mathscr{C}_*}\mu^\infty(\hat{c})\quad\text{and}\quad \lim_{\omega\to\infty}\limsup_{n\to\infty}\fint^{T_n}_{0}H(s;\omega)\,\mathrm{d}s= \sup_{\hat{c}\in\mathscr{C}_*}\mu^\infty(\hat{c}).\]
      Consequently, if the sequence $\left\{\fint^{T_n}_{0}c(\cdot,s)\,\mathrm{d}s \right\}_{n\in\mathbb{N}^+}$ of functions converges to some function $\hat{c}(\cdot)$ in $C(\overline{\Omega})$, then
       \[\lim_{\omega\to\infty}\liminf_{n\to\infty}\fint^{T_n}_{0}H(s;\omega)\,\mathrm{d}s= \lim_{\omega\to\infty}\limsup_{n\to\infty}\fint^{T_n}_{0}H(s;\omega)\,\mathrm{d}s=\mu^\infty(\hat{c}).\]
\end{enumerate}
\end{theorem}
\section{The case $\omega\to0$ with fixed $d\in(0,\infty)$}
We will use the normalized principal eigenfunction of \eqref{eq2.7} to construct suitable sub- and super-solutions. That is, for each $t\in\mathbb{R}$, the positive function $\phi^{0}(\cdot,t)$ related to principal eigenvalue $\mu^0(t)$ satisfies
\begin{equation*}
  \left\{
  \begin{aligned}
    &-d\operatorname{div}(A(x)\nabla\phi^{0})+c(x,t)\phi^{0}=\mu^0(t)\phi^{0},&&x\in\Omega,\\
    &\mathbf{n}\cdot(A(x)\nabla\phi^{0})=0,&&x\in\partial\Omega,\\
    &\|\phi^{0}(\cdot,t)\|_\infty=1.
  \end{aligned}
  \right.
\end{equation*}
We divide both sides by $\phi^{0}$ and perform a direct calculation to obtain
\begin{equation}\label{eq5.1}
-d\operatorname{div}(A(x)\nabla\ln\phi^0)-d\nabla\ln\phi^0\cdot(A(x)\nabla\ln\phi^0)+c(x,t)=\mu^0(t),\qquad x\in\Omega.
\end{equation}
We further recall several useful results stated in Chapter \ref{chp2}. Proposition \ref{proposition3.3}(ii) states that the normalized principal Floquet bundle $\varphi$ of \eqref{eq1.1} is uniformly bounded with respect to $\omega>0$ for fixed $d>0$, and Lemma \ref{lemma2.21} gives that $\phi^{0}$ is bounded. Consequently, there exists a positive constant $C>1$, independent of $\omega>0$, such that
\begin{equation}\label{eq5.2}
|\ln\varphi|+|\ln \phi^{0}|<C,\qquad\forall~(x,t)\in\overline{\Omega}\times\mathbb{R}.
\end{equation}
Furthermore, by Lemma \ref{lemma2.22}, the function $|\partial_t\ln \phi^{0}|$ is also bounded. These facts will be useful in the discussion to follow.
\begin{lemma}\label{lemma5.2}
Let $H(\cdot;\omega)$ be the normalized principal Floquet bundle of \eqref{eq1.1} with potential $c\in \mathcal{C}\cap C^{0,1}(\overline{\Omega}\times\mathbb{R})$, and $\mu^{0}(t)$, $t\in\mathbb{R}$ be the principal eigenvalue of \eqref{eq2.7} with potential $c(\cdot,t)$. Then for any sequence $\{T_n\}_{n\in\mathbb{N}^+}$ of $T\to\infty$, if $\sup_{t\in\mathbb{R}}\|\partial_tc(\cdot,t)\|_\infty<\infty$, there hold
 \[\lim_{\omega\to0}\liminf_{n\to\infty}\fint^{T_n}_{0}H(s;\omega)\,\mathrm{d}s=\liminf_{n\to\infty}\fint^{T_n}_{0}\mu^0(s)\,\mathrm{d}s\]
and
 \[\lim_{\omega\to0}\limsup_{n\to\infty}\fint^{T_n}_{0}H(s;\omega)\,\mathrm{d}s=\limsup_{n\to\infty}\fint^{T_n}_{0}\mu^0(s)\,\mathrm{d}s.\]
Consequently, if the limit $\lim_{n\to\infty}\fint^{T_n}_{0}\mu^0(s)\,\mathrm{d}s$ exists, then
  \[\lim_{\omega\to0}\liminf_{n\to\infty}\fint^{T_n}_{0}H(s;\omega)\,\mathrm{d}s= \lim_{\omega\to0}\limsup_{n\to\infty}\fint^{T_n}_{0}H(s;\omega)\,\mathrm{d}s= \lim_{n\to\infty}\fint^{T_n}_{0}\mu^0(s)\,\mathrm{d}s.\]
\end{lemma}
\begin{proof}
Fix an arbitrary $\varepsilon>0$.

\noindent\textbf{Lower bound.} For each $n\in\mathbb{N}^+$, define
\[\overline{\varphi}_n(x,t):=\phi^{0}(x,t)\exp\left[\frac{1}{\omega }\int^{t}_{0}\left(\fint^{T_{n}}_{0} \mu^0(r)\,\mathrm{d}r-\mu^0(s) -\varepsilon \right)\,\mathrm{d}s\right],\quad(x,t) \in\overline{\Omega}\times[0,T_{n}],\]
where $\phi^{0}$ is the principal eigenfunction of \eqref{eq2.7}, and
\[\underline{\varphi}_n(x,t):=\varphi(x,t)\exp\left[\frac{1}{\omega }\int^{t}_{0}\left(\fint^{T_{n}}_{0} H(r;\omega )\,\mathrm{d}r -H(s;\omega ) +\varepsilon \right)\,\mathrm{d}s\right],\quad(x,t) \in\overline{\Omega}\times[0,T_{n}].\]
 A straightforward calculation yields that for each $n\in\mathbb{N}^+$ and all $(x,t)\in\Omega\times(0,T_{n}]$, there hold
\begin{align*}
   & \omega \partial_t\overline{\varphi}_n- d\operatorname{div}(A(x)\nabla  \overline{\varphi}_n)+c(x,t) \overline{\varphi}_n \\
 = & \left[\omega \partial_t\ln\overline{\varphi}_n-d\operatorname{div}(A(x)\nabla \ln\overline{\varphi}_n)- d\nabla \ln\overline{\varphi}_n\cdot(A(x)\nabla\ln\overline{\varphi}_n)+c(x,t) \right]\overline{\varphi}_n \\
 = & \Bigg[\fint^{T_{n}}_{0}\mu^{0}(s)\,\mathrm{d}s + \omega\partial_t\ln \phi^{0}-\varepsilon  \\
 &\qquad \underbrace{-d\operatorname{div}(A(x)\nabla\ln\phi^0)-d\nabla\ln\phi^0\cdot(A(x)\nabla\ln\phi^0) +c(x,t)-\mu^0(t)}_{=0\text{ by }\eqref{eq5.1}} \Bigg]\overline{\varphi}_n\\
 = & \left[\fint^{T_{n}}_{0}\mu^{0}(s)\,\mathrm{d}s +\omega\partial_t\ln \phi^{0}-\varepsilon \right]\overline{\varphi}_n \\
\ge & \left[\fint^{T_{n}}_{0}\mu^{0}(s)\,\mathrm{d}s -\omega\left\|\partial_t\ln \phi^{0}\right\|_\infty-\varepsilon \right]\overline{\varphi}_n
\end{align*}
and
\begin{align*}
   & \omega \partial_t\underline{\varphi}_n- d\operatorname{div}(A(x)\nabla\underline{\varphi}_n)+c(x,t) \underline{\varphi}_n \\
 = & \left[\omega \partial_t\ln\underline{\varphi}_n-d\operatorname{div}(A(x)\nabla \ln\underline{\varphi}_n)- d\nabla \ln\underline{\varphi}_n\cdot(A(x)\nabla\ln\underline{\varphi}_n) +c(x,t)\right]\underline{\varphi}_n \\
 = & \Bigg[\fint^{T_{n}}_{0}H(s;\omega )\,\mathrm{d}s+\varepsilon \\
 &\qquad+\underbrace{\omega\partial_t\ln\varphi-d\operatorname{div}(A(x)\nabla\ln\varphi-d\nabla\ln\varphi\cdot(A(x)\nabla \ln\varphi)+c(x,t)-H(t;\omega)}_{=0}\Bigg] \underline{\varphi}_n\\
 = & \left(\fint^{T_{n}}_{0}H(s;\omega )\,\mathrm{d}s+\varepsilon \right) \underline{\varphi}_n.
\end{align*}
Then for every $\omega>0$, there is an integer $N_\varepsilon>0$ such that
\[T_{n}>\frac{2C\omega}{\varepsilon},\qquad\forall~n>N_\varepsilon,\]
where $C>1$ is the constant given in \eqref{eq5.2}. Now, it is straightforward to verify that for all $n>N_\varepsilon$, $\overline{\varphi}_n$ and $\underline{\varphi}_n$ satisfy
\begin{equation*}
\left\{
  \begin{aligned}
   & \omega \partial_t\overline{\varphi}_n-d\operatorname{div}(A(x)\nabla \overline{\varphi}_n)+c(x,t) \overline{\varphi}_n \\
   &\hskip 4cm \ge  \left(\fint^{T_{n}}_{0}\mu^{0}(s)\,\mathrm{d}s -\omega\left\|\partial_t\ln \phi^{0}\right\|_\infty-\varepsilon \right)\overline{\varphi}_n, &&(x,t)\in\Omega\times(0,T_{n}], \\
   & \omega \partial_t\underline{\varphi}_n-d\operatorname{div}(A(x)\nabla \underline{\varphi}_n)+c(x,t) \underline{\varphi}_n = \left(\fint^{T_{n}}_{0}H(s;\omega )\,\mathrm{d}s+\varepsilon \right) \underline{\varphi}_n, &&(x,t)\in\Omega\times(0,T_{n}], \\
   &\mathbf{n}\cdot(A(x)\nabla\overline{\varphi}_n)=\mathbf{n}\cdot(A(x)\nabla\underline{\varphi}_n) = 0,&&(x,t)\in\partial\Omega\times(0,T_{n}],\\
   &\overline{\varphi}_n(x,0)=\phi^{0}(x,0)\ge\phi^{0}(x,T_{n})\mathrm{e}^{-\frac{\varepsilon T_{n}}{\omega }}=\overline{\varphi}_n(x,T_{n}),&&x\in\overline{\Omega},\\
   &\underline{\varphi}_n(x,0)=\varphi(x,0)\le\varphi(x,T_{n})\mathrm{e}^{\frac{\varepsilon T_{n}}{\omega }}=\underline{\varphi}_n(x,T_{n}),&&x\in\overline{\Omega}.
  \end{aligned}
\right.
\end{equation*}
Using Lemma \ref{lemma2.13}, we obtain
 \[\liminf_{n\to\infty}\fint^{T_{n}}_{0}H(s;\omega )\,\mathrm{d}s+\varepsilon\ge \liminf_{n\to\infty}\fint^{T_{n}}_{0}\mu^{0}(s)\,\mathrm{d}s -\omega\left\|\partial_t\ln \phi^{0}\right\|_\infty-\varepsilon\]
and
 \[\limsup_{n\to\infty}\fint^{T_{n}}_{0}H(s;\omega )\,\mathrm{d}s+\varepsilon\ge \limsup_{n\to\infty}\fint^{T_{n}}_{0}\mu^{0}(s)\,\mathrm{d}s -\omega\left\|\partial_t\ln \phi^{0}\right\|_\infty-\varepsilon.\]
By the arbitrariness of $\varepsilon>0$, we conclude that
 \[\liminf_{n\to\infty}\fint^{T_{n}}_{0}H(s;\omega )\,\mathrm{d}s\ge \liminf_{n\to\infty}\fint^{T_{n}}_{0}\mu^{0}(s)\,\mathrm{d}s -\omega\left\|\partial_t\ln \phi^{0}\right\|_\infty\]
and
 \[\limsup_{n\to\infty}\fint^{T_{n}}_{0}H(s;\omega )\,\mathrm{d}s\ge \limsup_{n\to\infty}\fint^{T_{n}}_{0}\mu^{0}(s)\,\mathrm{d}s -\omega\left\|\partial_t\ln \phi^{0}\right\|_\infty.\]
Finally, applying the boundedness of $\partial_t\ln\phi^0$ (see Lemma \ref{lemma2.22}) and letting $\omega\to 0$, we obtain
 \[\liminf_{\omega\to0}\liminf_{n\to\infty}\fint^{T_{n}}_{0}H(s;\omega )\,\mathrm{d}s\ge \liminf_{n\to\infty}\fint^{T_{n}}_{0}\mu^{0}(s)\,\mathrm{d}s\]
and
 \[\liminf_{\omega\to0}\limsup_{n\to\infty}\fint^{T_{n}}_{0}H(s;\omega )\,\mathrm{d}s\ge \limsup_{n\to\infty}\fint^{T_{n}}_{0}\mu^{0}(s)\,\mathrm{d}s.\]
This obtains the lower bounds.

\medskip
\noindent\textbf{Upper bound.} In a similar manner, for each $n\in\mathbb{N}^+$, define
\[\overline{\varphi}_n(x,t):=\varphi(x,t)\exp\left[\frac{1}{\omega}\int^{t}_{0}\left(\fint^{T_{n}}_{0} H(r;\omega)\,\mathrm{d}r -H(s;\omega) -\varepsilon \right)\,\mathrm{d}s\right],\quad(x,t) \in\overline{\Omega}\times[0,T_{n}]\]
and
\[\underline{\varphi}_n(x,t):=\phi^{0}(x,t)\exp\left[\frac{1}{\omega}\int^{t}_{0}\left(\fint^{T_{n}}_{0} \mu^0(r)\,\mathrm{d}r -\mu^0(s) +\varepsilon \right)\,\mathrm{d}s\right]\quad(x,t) \in\overline{\Omega}\times[0,T_{n}].\]
Proceeding similarly, for each $\omega>0$, there exists an integer $N_\varepsilon$ such that
\[T_{n}>\frac{2C\omega}{\varepsilon},\qquad\forall~n>N_\varepsilon,\]
where the constant $C>0$ is given in \eqref{eq5.2}. Then it can be verified that for each $n>N_\varepsilon$, the functions $\underline{\varphi}_n$ and $\overline{\varphi}_n$ satisfy
\begin{equation*}
\left\{
  \begin{aligned}
   & \omega \partial_t\overline{\varphi}_n-d\operatorname{div}(A(x)\nabla \overline{\varphi}_n)+c(x,t) \overline{\varphi}_n= \left(\fint^{T_{n}}_{0}H(s;\omega )\,\mathrm{d}s-\varepsilon \right)\overline{\varphi}_n, &&(x,t)\in\Omega\times(0,T_{n}], \\
   & \omega \partial_t\underline{\varphi}_n-d\operatorname{div}(A(x)\nabla \underline{\varphi}_n)+c(x,t) \underline{\varphi}_n\\
   &\hskip 4cm \le \left(\fint^{T_{n}}_{0}\mu^{0}(s)\,\mathrm{d}s +\omega\left\|\partial_t\ln \phi^{0}\right\|_\infty+\varepsilon \right) \underline{\varphi}_n, &&(x,t)\in\Omega\times(0,T_{n}], \\
   &\mathbf{n}\cdot(A(x)\nabla\overline{\varphi}_n)=\mathbf{n}\cdot(A(x)\nabla\underline{\varphi}_n) = 0,&&(x,t)\in\partial\Omega\times(0,T_{n}],\\
   &\overline{\varphi}_n(x,0)=\varphi(x,0)\ge\varphi(x,T_{n})\mathrm{e}^{-\frac{\varepsilon T_{n}}{\omega }}=\overline{\varphi}_n(x,T_{n}),&&x\in\overline{\Omega},\\
   &\underline{\varphi}_n(x,0)=\phi^{0}(x,0)\le\phi^{0}(x,T_{n})\mathrm{e}^{\frac{\varepsilon T_{n}}{\omega }}=\underline{\varphi}_n(x,T_{n}),&&x\in\overline{\Omega}.
  \end{aligned}
\right.
\end{equation*}
Applying Lemma \ref{lemma2.13} and using the arbitrariness of $\varepsilon>0$, we obtain
 \[\liminf_{n\to\infty}\fint^{T_{n}}_{0}H(s;\omega )\,\mathrm{d}s\le \liminf_{n\to\infty}\fint^{T_{n}}_{0}\mu^{0}(s)\,\mathrm{d}s +\omega\left\|\partial_t\ln \phi^{0}\right\|_\infty\]
and
 \[\limsup_{n\to\infty}\fint^{T_{n}}_{0}H(s;\omega )\,\mathrm{d}s\le \limsup_{n\to\infty}\fint^{T_{n}}_{0}\mu^{0}(s)\,\mathrm{d}s +\omega\left\|\partial_t\ln \phi^{0}\right\|_\infty.\]
Then applying Lemma \ref{lemma2.22} again and letting $\omega\to 0$, we obtain
 \[\limsup_{\omega\to0}\liminf_{n\to\infty}\fint^{T_{n}}_{0}H(s;\omega )\,\mathrm{d}s\le \liminf_{n\to\infty}\fint^{T_{n}}_{0}\mu^{0}(s)\,\mathrm{d}s\]
and
 \[\limsup_{\omega\to0}\limsup_{n\to\infty}\fint^{T_{n}}_{0}H(s;\omega )\,\mathrm{d}s\le \limsup_{n\to\infty}\fint^{T_{n}}_{0}\mu^{0}(s)\,\mathrm{d}s.\]
We obtain the upper bounds.

A combination of the lower bounds and the upper bounds yields the desired result. This completes the proof.
\end{proof}

\begin{lemma}\label{lemma5.3}
Let $c\in \mathcal{C}\cap C^{0,1}(\overline{\Omega}\times\mathbb{R})$ be a function satisfying $\sup_{t\in\mathbb{R}}\|\partial_tc(\cdot,t)\|_\infty <\infty$, $H(\cdot ;\omega)$ be the normalized principal Floquet bundle of \eqref{eq1.1} with potential $c$, and $\mu^{0}(t)$, $t\in\mathbb{R}$ be the principal eigenvalue of \eqref{eq2.7} with potential $c(\cdot,t)$. Then
\begin{equation}\label{eq5.3}
  \lim_{\omega\to0}\liminf_{T\to\infty}\fint^{T}_{0}H(s;\omega)\,\mathrm{d}s= \liminf_{T\to\infty}\fint^{T}_{0}\mu^0(s)\,\mathrm{d}s
\end{equation}
and
\begin{equation}\label{eq5.4}
  \lim_{\omega\to0}\limsup_{T\to\infty}\fint^{T}_{0}H(s;\omega)\,\mathrm{d}s= \limsup_{T\to\infty}\fint^{T}_{0}\mu^0(s)\,\mathrm{d}s.
\end{equation}
\end{lemma}
\begin{proof}
We divide the proof into two steps.

\medskip
\noindent\textbf{Step 1.} We prove \eqref{eq5.3}. Firstly, let $\{T_n\}_{n\in\mathbb{N}^+}$ be a sequence satisfying $T_n\to\infty$ as $n\to\infty$ such that
\[\lim_{n\to\infty}\fint^{T_n}_{0}\mu^0(s)\,\mathrm{d}s= \liminf_{T\to\infty}\fint^{T_n}_{0}\mu^0(s)\,\mathrm{d}s.\]
Then by Lemma \ref{lemma5.2}, we have
\begin{equation}\label{eq5.5}
\begin{aligned}
     \limsup_{\omega\to0}\liminf_{T\to\infty}\fint^{T}_{0}H(s;\omega)\,\mathrm{d}s \le & \limsup_{\omega\to0}\liminf_{n\to\infty}\fint^{T_n}_{0}H(s;\omega)\,\mathrm{d}s \\
     = & \lim_{n\to\infty}\fint^{T_n}_{0}\mu^0(s)\,\mathrm{d}s= \liminf_{T\to\infty}\fint^{T}_{0}\mu^0(s)\,\mathrm{d}s.
\end{aligned}
\end{equation}
This yields the upper bound of \eqref{eq5.3}.

We now establish the lower bound of \eqref{eq5.3}. For each $T>0$, define
 \[\overline{\varphi}_T (x,t):=\phi^{0}(x,t)\exp\left[\frac{1}{\omega }\int^{t}_{0}\left(\fint^{T}_{0}\mu^0(r)\,\mathrm{d}r -\mu^0(s) -\varepsilon \right)\,\mathrm{d}s\right],\quad(x,t) \in\overline{\Omega}\times[0,T],\]
and
 \[\underline{\varphi}_T (x,t):=\varphi(x,t)\exp\left[\frac{1}{\omega}\int^{t}_{0}\left(\fint^{T}_{0} H(r;\omega)\,\mathrm{d}r-H(s;\omega )+\varepsilon \right)\,\mathrm{d}s\right],\quad(x,t) \in\overline{\Omega}\times[0,T],\]
where $\phi^{0}$ is the principal eigenfunction of \eqref{eq2.7}, $\varphi$ is the normalized principal Floquet bundle of \eqref{eq1.1}. Recall the constant $C>1$ given in \eqref{eq5.2}, that is, $C$ is independent of $\omega>0$ such that
 \[|\ln\varphi|+|\ln \phi^{0}|<C,\qquad\forall~(x,t)\in\overline{\Omega}\times\mathbb{R}.\]
Then for each fixed $\omega>0$ and any
\begin{equation}\label{eq5.6}
 T>\frac{2C\omega}{\varepsilon},
\end{equation}
the functions $\overline{\varphi}_T $ and $\underline{\varphi}_T $ satisfy
\begin{equation*}
  \left\{
  \begin{aligned}
    &\omega \partial_t\overline{\varphi}_T -d\operatorname{div}(A(x)\nabla \overline{\varphi}_T )+c(x,t)\overline{\varphi}_T  \\
    &\hskip 3cm =\left(\fint^{T}_{0}\mu^0(s)\,\mathrm{d}s+\omega \frac{\partial_t\phi^{0}}{\phi^{0}}-\varepsilon \right) \overline{\varphi}_T \\
    &\hskip 3cm \ge\left(\fint^{T}_{0}\mu^0(s)\,\mathrm{d}s-\omega \|\partial_t\ln\phi^{0}\|_\infty-\varepsilon \right) \overline{\varphi}_T ,&&(x,t)\in\Omega\times(0,T],\\
    &\omega \partial_t\underline{\varphi}_T -d\operatorname{div}(A(x)\nabla \underline{\varphi}_T )+c(x,t)\underline{\varphi}_T  =\left(\fint^{T}_{0}H(s;\omega )\,\mathrm{d}s+\varepsilon \right) \underline{\varphi}_T ,&&(x,t)\in\Omega\times(0,T],\\
    &\mathbf{n}\cdot(A(x)\nabla\overline{\varphi}_T )=\mathbf{n}\cdot(A(x)\nabla\underline{\varphi}_T )= 0,&&(x,t)\in\partial\Omega\times(0,T],\\
    &\overline{\varphi}_T (x,0)=\phi^{0}(x,0)\ge \phi^{0}(x,T) \mathrm{e}^{-\frac{\varepsilon}{\omega }T}=\overline{\varphi}_T (x,T),&&x\in\Omega,\\
    &\underline{\varphi}_T (x,0)=\varphi(x,0)\le \varphi(x,T)\mathrm{e}^{\frac{\varepsilon}{\omega }T}=\underline{\varphi}_T (x,T),&&x\in\Omega.
  \end{aligned}
  \right.
\end{equation*}
Corollary \ref{corollary2.15} implies that
\[\liminf_{T\to\infty}\fint^{T}_{0}\mu^0(s)\,\mathrm{d}s-\omega \|\partial_t\ln\phi^{0}\|_\infty-\varepsilon
\le \liminf_{T\to\infty}\fint^{T}_{0}H(s;\omega )\,\mathrm{d}s+\varepsilon .\]
Now using Lemma \ref{lemma2.22}, letting $\omega\to0$ and by the arbitrariness of $\varepsilon>0$, we have
\begin{equation}\label{eq5.7}
   \liminf_{\omega\to0}\liminf_{T\to\infty}\fint^{T}_{0}H(s;\omega)\,\mathrm{d}s\ge \liminf_{T\to\infty}\fint^{T}_{0}\mu^0(s)\,\mathrm{d}s.
\end{equation}
Combining \eqref{eq5.5} and \eqref{eq5.7}, we have \eqref{eq5.3}.

\medskip
\noindent\textbf{Step 2.} The idea of proving \eqref{eq5.4} is similar to Step 1. Assume that $\{T_n\}_{n\in\mathbb{N}^+}$ is a sequence satisfying $T_n\to\infty$ as $n\to\infty$ such that
\[\lim_{n\to\infty}\fint^{T_n}_{0}\mu^0(s)\,\mathrm{d}s= \limsup_{T\to\infty}\fint^{T}_{0}\mu^0(s)\,\mathrm{d}s.\]
Then by Lemma \ref{lemma5.2}, we obtain the lower bound, namely,
\[\begin{aligned}
  \liminf_{\omega\to0}\limsup_{T\to\infty}\fint^{T}_{0}H(s;\omega)\,\mathrm{d}s \ge & \liminf_{\omega\to0}\limsup_{n\to\infty}\fint^{T_n}_{0}H(s;\omega)\,\mathrm{d}s \\
  = & \lim_{n\to\infty}\fint^{T_n}_{0}\mu^0(s)\,\mathrm{d}s= \limsup_{T\to\infty}\fint^{T}_{0}\mu^0(s)\,\mathrm{d}s.
\end{aligned}\]

To obtain the upper bound, we define
 \[\overline{\varphi}_T (x,t):=\varphi(x,t)\exp\left[\frac{1}{\omega }\int^{t}_{0}\left(\fint^{T}_{0} H(r;\omega )\,\mathrm{d}r -H(s;\omega ) -\varepsilon \right)\,\mathrm{d}s\right],\quad(x,t) \in\overline{\Omega}\times[0,T] \]
and
 \[\underline{\varphi}_T (x,t):=\phi^{0}(x,t)\exp\left[\frac{1}{\omega }\int^{t}_{0}\left(\fint^{T}_{0} \mu^0(s)\,\mathrm{d}s -\mu^0(s) +\varepsilon \right)\,\mathrm{d}s\right],\quad(x,t) \in\overline{\Omega}\times[0,T]\]
for each $T>0$, where $\phi^{0}$ is the principal eigenfunction of \eqref{eq2.7}, $\varphi$ is the normalized principal Floquet bundle of \eqref{eq1.1}. We fix $\omega>0$ and then compute that for each
$T>\frac{2C\omega}{\varepsilon}$ (where $C>1$ is defined by \eqref{eq5.2}), the functions $\overline{\varphi}_T $ and $\underline{\varphi}_T $ satisfy
\begin{equation*}
  \left\{
  \begin{aligned}
    &\omega \partial_t\overline{\varphi}_T -d\operatorname{div}(A(x)\nabla \overline{\varphi}_T )+c(x,t)\overline{\varphi}_T  =\left(\fint^{T}_{0}H(s;\omega )\,\mathrm{d}s-\varepsilon \right) \overline{\varphi}_T ,&&(x,t)\in\Omega\times(0,T],\\
    &\omega \partial_t\underline{\varphi}_T -d\operatorname{div}(A(x)\nabla \underline{\varphi}_T )+c(x,t)\underline{\varphi}_T  \\
    &\qquad\qquad \le\left(\fint^{T}_{0}\mu^0(s)\,\mathrm{d}s+\omega \|\partial_t\ln\phi^{0}\|_\infty+\varepsilon \right) \underline{\varphi}_T ,&&(x,t)\in\Omega\times(0,T],\\
    &\mathbf{n}\cdot(A(x)\nabla\overline{\varphi}_T )=\mathbf{n}\cdot(A(x)\nabla\underline{\varphi}_T )= 0,&&(x,t)\in\partial\Omega\times(0,T],\\
    &\overline{\varphi}_T (x,0)=\varphi(x,0)\ge \varphi(x,T) \mathrm{e}^{-\frac{\varepsilon}{\omega }T}=\overline{\varphi}_T (x,T),&&x\in\Omega,\\
    &\underline{\varphi}_T (x,0)=\phi^{0}(x,0)\le \phi^{0}(x,T)\mathrm{e}^{\frac{\varepsilon}{\omega }T}=\underline{\varphi}_T (x,T),&&x\in\Omega.
  \end{aligned}
  \right.
\end{equation*}
Now using Corollary \ref{corollary2.15} and Lemma \ref{lemma2.22}, letting $\omega\to0$ and by the arbitrariness of $\varepsilon>0$, we deduce that
\begin{equation*}
   \limsup_{\omega\to0}\limsup_{T\to\infty}\fint^{T}_{0}H(s;\omega)\,\mathrm{d}s\le \limsup_{T\to\infty}\fint^{T}_{0}\mu^0(s)\,\mathrm{d}s.
\end{equation*}
This, together with the lower bound, yields \eqref{eq5.4}.

The proof is complete.
\end{proof}

\medskip
The proof of Lemma \ref{lemma5.3} requires that, for any given $\omega > 0$, a time $T > \frac{2C\omega}{\varepsilon}$ is chosen to ensure $\underline{\varphi}(x,0) \le \underline{\varphi}(x,T)$; see \eqref{eq5.6}. Alternatively, for a fixed $T>0$, the inequality can be guaranteed by taking $\omega>0$ small enough. This observation enables the derivation of the following local asymptotic result concerning the normalized principal Floquet bundle.
\begin{lemma}\label{lemma5.4}
Let $c\in \mathcal{C}\cap C^{0,1}(\overline{\Omega}\times\mathbb{R})$ be a function satisfying $\sup_{t\in\mathbb{R}}\|\partial_tc(\cdot,t)\|_\infty<\infty$, $H(\cdot;\omega)$ be the normalized principal Floquet bundle of \eqref{eq1.1} with potential $c$, and $\mu^{0}(t)$, $t\in\mathbb{R}$ be the principal eigenvalue of \eqref{eq2.7} with potential $c(\cdot,t)$. For any $-\infty<T_1<T_2<\infty$, there holds
  \[ \lim_{\omega\to0}\fint^{T_2}_{T_1}H(s;\omega)\,\mathrm{d}s= \fint^{T_2}_{T_1}\mu^0(s)\,\mathrm{d}s.\]
\end{lemma}
\begin{proof}
  For the sake of convenience, we set $T_1=0$ and $T_2=T>0$. Let $\varphi$ be the normalized principal Floquet bundle of \eqref{eq1.1} and $(\mu^0(t),\phi^{0}(\cdot,t))$ be the normalized principal eigenpair of \eqref{eq2.7} with potential $c(\cdot,t)$ for each $t\in[0,T]$.

\medskip
\noindent\textbf{Upper bound.} Define
\begin{equation*}
\overline{\varphi}_T (x,t):=\varphi(x,t)\exp\left[\frac{1}{\omega }\int^{t}_{0}\left(\fint^{T}_{0} H(r;\omega )\,\mathrm{d}r -H(s;\omega ) -\varepsilon \right)\,\mathrm{d}s\right],\quad(x,t) \in\overline{\Omega}\times[0,T]
\end{equation*}
and
\[\underline{\varphi}_T (x,t):=\phi^{0}(x,t)\exp\left[\frac{1}{\omega }\int^{t}_{0} \left(\fint^{T}_{0} \mu^0(r)\,\mathrm{d}r -\mu^0(s) +\varepsilon \right)\,\mathrm{d}s \right],\quad(x,t) \in\overline{\Omega}\times[0,T].\]
Recall the constant $C>1$ defined by \eqref{eq5.2}. Then for any $\omega\in\left(0,\frac{T\varepsilon}{2C}\right)$, the functions $\overline{\varphi}_T $ and $\underline{\varphi}_T $ satisfy
\begin{equation*}
  \left\{
  \begin{aligned}
    &\omega \partial_t\overline{\varphi}_T -d\operatorname{div}(A(x)\nabla \overline{\varphi}_T )+c(x,t)\overline{\varphi}_T  =\left(\fint^{T}_{0}H(s;\omega )\,\mathrm{d}s-\varepsilon \right) \overline{\varphi}_T ,&&(x,t)\in\Omega\times(0,T],\\
    &\omega \partial_t\underline{\varphi}_T -d\operatorname{div}(A(x)\nabla \underline{\varphi}_T )+c(x,t)\underline{\varphi}_T \\
    &\qquad\qquad \le\left(\fint^{T}_{0}\mu^0(s)\,\mathrm{d}s+\omega \|\partial_t\ln\phi^{0}\|_\infty+\varepsilon \right) \underline{\varphi}_T ,&&(x,t)\in\Omega\times(0,T],\\
    &\mathbf{n}\cdot(A(x)\nabla\overline{\varphi}_T )=\mathbf{n}\cdot(A(x)\nabla\underline{\varphi}_T )= 0,&&(x,t)\in\partial\Omega\times(0,T],\\
    &\overline{\varphi}_T (x,0)=\varphi(x,0)\ge \varphi(x,T) \mathrm{e}^{-\frac{\varepsilon}{\omega }T}=\overline{\varphi}_T (x,T),&&x\in\Omega,\\
    &\underline{\varphi}_T (x,0)=\phi^{0}(x,0)\le  \phi^{0}(x,T)\mathrm{e}^{\frac{\varepsilon}{\omega }T}=\underline{\varphi}_T (x,T),&&x\in\Omega.
  \end{aligned}
  \right.
\end{equation*}
Now using Lemmas \ref{lemma2.12} and \ref{lemma2.22}, then letting $\omega\to0$ and by the arbitrariness of $\varepsilon>0$, we deduce that
  \[ \limsup_{\omega\to0}\fint^{T}_{0}H(s;\omega)\,\mathrm{d}s\le \fint^{T}_{0}\mu^0(s)\,\mathrm{d}s.\]

\medskip
\noindent\textbf{Lower bound.} Similarly, we easily show
  \[\liminf_{\omega\to 0}\fint^{T}_{0}H(s;\omega)\,\mathrm{d}s\ge \fint^{T}_{0} \mu^0(s)\,\mathrm{d}s.\]
In fact, one can construct the sub- and super-solutions as
 \[\overline{\varphi}_T (x,t):=\phi^{0}(x,t)\exp\left[\frac{1}{\omega }\int^{t}_{0} \left(\fint^{T}_{0} \mu^0(r)\,\mathrm{d}r -\mu^0(s) -\varepsilon \right)\,\mathrm{d}s\right],\quad(x,t) \in\overline{\Omega}\times[0,T]\]
and
 \[\underline{\varphi}_T (x,t):=\varphi(x,t)\exp\left[\frac{1}{\omega }\int^{t}_{0} \left(\fint^{T}_{0} H(r;\omega )\,\mathrm{d}r -H(s;\omega ) +\varepsilon \right)\,\mathrm{d}s \right],\quad(x,t) \in\overline{\Omega}\times[0,T],\]
respectively. We omit the details of arguments.

Combining the upper bound and the lower bound, we complete the proof of the lemma.
\end{proof}

\section{The case $\omega\to\infty$ with fixed $d\in(0,\infty)$}
Let $c\in\mathcal{C}$ and $\{T_n\}_{n\in\mathbb{N}^+}$ be a sequence satisfying $T_n\to\infty$ as $n\to\infty$; define $\hat{c}_n(x):=\fint^{T_n}_{0}c(x,s)\,\mathrm{d}s$, $x\in\overline{\Omega}$. Consider the Neumann eigenvalue problem
 \begin{equation}\label{eq5.8}
  \left\{
  \begin{aligned}
    &-d\operatorname{div}(A(x)\nabla \hat{\phi}) +\hat{c}_n(x)\hat{\phi}=\mu\hat{\phi},&&x\in\Omega,\\
    &\mathbf{n}\cdot(A(x)\nabla\hat{\phi})=0,&&x\in\partial\Omega,\\
    &\|\hat{\phi}\|_{\infty}=1.
  \end{aligned}
  \right.
\end{equation}
A direct corollary of Lemma \ref{lemma2.19} is as follows.
\begin{lemma}\label{lemma5.5}
Let $\{T_n\}_{n\in\mathbb{N}^+}$ be a sequence satisfying $T_n\to\infty$ as $n\to\infty$ and $(\mu^\infty_{n},\hat{\phi}_n)=(\mu^\infty_{T_n},\hat{\phi}_{T_n})$ be the principal eigenpair of \eqref{eq5.8}. If $\lim_{n\to\infty}\left\|\hat{c}_n(\cdot)-\hat{c}(\cdot)\right\|_{ C(\overline{\Omega})}=0$, then
  \[\lim_{n\to\infty}\mu^\infty_{n}=\mu^\infty\quad\text{ and } \quad\lim_{n\to\infty}\left\|\hat{\phi}_n-\hat{\phi}\right\|_{C^{1+\alpha}(\overline{\Omega})}=0,\]
where $\alpha\in(0,1)$ and $(\mu^\infty,\hat{\phi})$ is the principal eigenpair of \eqref{eq1.3} with potential $\hat{c}$.
\end{lemma}

\begin{lemma}\label{theorem5.6}
Let $H(\cdot;\omega)$ be the normalized principal Floquet bundle of \eqref{eq1.1} with potential $c\in \mathcal{C}\cap C^{2,1}(\overline{\Omega}\times\mathbb{R})$. If $c$ satisfies \eqref{H2}, then for any sequence $\{T_n\}_{n\in\mathbb{N}^+}$ of $T\to\infty$, there hold
 \[\lim_{\omega\to\infty}\liminf_{n\to\infty}\fint^{T_n}_{0}H(s;\omega)\,\mathrm{d}s= \inf_{\hat{c}\in\mathscr{C}_*}\mu^\infty(\hat{c})\quad\text{and}\quad \lim_{\omega\to\infty}\limsup_{n\to\infty}\fint^{T_n}_{0}H(s;\omega)\,\mathrm{d}s= \sup_{\hat{c}\in\mathscr{C}_*}\mu^\infty(\hat{c}).\]
Consequently, if the sequence $\left\{\fint^{T_n}_{0}c(\cdot,s)\,\mathrm{d}s\right\}_{n\in\mathbb{N}^+}$ converges to $\hat{c}(\cdot)$ in $C(\overline{\Omega})$, then
  \[\lim_{\omega\to\infty}\liminf_{n\to\infty}\fint^{T_n}_{0}H(s;\omega)\,\mathrm{d}s= \lim_{\omega\to\infty}\limsup_{n\to\infty}\fint^{T_n}_{0}H(s;\omega)\,\mathrm{d}s=\mu^\infty(\hat{c}),\]
where $\mu^\infty(\hat{c})$ is the principal eigenvalue of \eqref{eq1.3} with potential $\hat{c}$.
\end{lemma}
\begin{proof}
Firstly, we derive the upper bound. Using Lemma \ref{lemma3.12}, we have
  \[\liminf_{n\to\infty}\fint^{T_n}_{0}H(s;\omega)\,\mathrm{d}s\le\inf_{\hat{c}\in\mathscr{C}_*}\mu^\infty(\hat{c})\quad\text{and}\quad \limsup_{n\to\infty}\fint^{T_n}_{0}H(s;\omega)\,\mathrm{d}s\le\sup_{\hat{c}\in\mathscr{C}_*}\mu^\infty(\hat{c}).\]
Letting $\omega\to\infty$, we see that
  \[\limsup_{\omega\to\infty}\liminf_{n\to\infty}\fint^{T_n}_{0}H(s;\omega)\,\mathrm{d}s\le\inf_{\hat{c}\in\mathscr{C}_*}\mu^\infty(\hat{c})\quad\text{and}\quad \limsup_{\omega\to\infty}\limsup_{n\to\infty}\fint^{T_n}_{0}H(s;\omega)\,\mathrm{d}s\le\sup_{\hat{c}\in\mathscr{C}_*}\mu^\infty(\hat{c}).\]

Now we turn to derive the lower bound. Recall the function $\mathcal{M}(x,t;T)=t\fint^{T}_{0}c(x,s)\,\mathrm{d}s-\int^{t}_{0}c(x,s)\,\mathrm{d}s$ defined on $\overline{\Omega}\times[0,\infty)$ and assumption \eqref{H2}
\[\limsup\limits_{T\to\infty}\sup\limits_{(x,t)\in\Omega\times[0,T]}(|\operatorname{div}(A(x)\nabla\mathcal{M}(x,t;T))|+|\nabla \mathcal{M}(x,t;T)\cdot (A(x)\nabla \mathcal{M}(x,t;T))|)<\infty.\]
There exist positive constants $T_0$ and $V_0$ independent of $\omega$ such that
\[\sup_{(x,t)\in\Omega\times[0,T]}(|\operatorname{div}(A(x)\nabla\mathcal{M}(x,t;T))|+|\nabla \mathcal{M}(x,t;T)\cdot (A(x)\nabla \mathcal{M}(x,t;T))|)<V_0,\quad\forall~T>T_0.\]
Denote by $\phi^{*}$ the principal eigenfunction of \eqref{eq2.6}. By Lemma \ref{lemma2.16}, there exists constant $\kappa>0$ such that for any $T>T_0$,
\begin{equation}\label{eq5.9}
-\kappa\mathbf{n}\cdot(A(x)\nabla\phi^{*})\ge \sup\limits_{(x,t)\in\partial\Omega\times[0,T]} |\mathbf{n}\cdot(A(x)\nabla\mathcal{M}(x,t;T))|, \quad\forall~x\in\partial\Omega,~t\in[0,T].
\end{equation}
For each $n\in\mathbb{N}^+$, define $\hat{c}_n(x):=\fint^{T_n}_{0}c(x,s)\,\mathrm{d}s$, $x\in\overline{\Omega}$, and let $(\mu_{n},\hat{\phi}_n)$ be the principal eigenpairs of problem \eqref{eq5.8} with potential $\hat{c}_n$. The principal eigenfunctions, together with $\kappa$ and $\phi^{*}$, will be used to construct upper solutions in the following discussion. We assume that $\{\hat{c}_n\}_{n\in\mathbb{N}^+}$ converges. (If not, we may choose a convergent subsequence and proceed with the following analysis along it. Because this subsequence is selected arbitrarily, we will be able to invoke the arbitrariness of $\hat{c}\in\mathscr{C}_*$ in due course to derive the required result.) By Lemma \ref{lemma5.5}, $\hat{\phi}_n$ converges to some function $\hat{\phi}\in C^{1+\alpha}(\overline{\Omega})$ in $C^{1+\alpha}$ norm, and we may assume that
\[\|\ln\hat{\phi}_n\|_{ C^1(\overline{\Omega})}\le \|\ln\hat{\phi}\|_{ C^1(\overline{\Omega})}+1,\quad\forall ~n\in\mathbb{N}^+.\]

For each $n\in\mathbb{N}^+$, define
\[\overline{\varphi}_{n}(x,t):=\hat{\phi}_{n}(x)\exp\left[\frac{1}{\omega}\left(\mathcal{M}(x,t;T_{n})- \kappa\phi^{*}\right) \right],\quad(x,t) \in\overline{\Omega}\times[0,T_{n}]\]
and
\[\underline{\varphi}_{n}(x,t):=\varphi(x,t)\exp\left[\frac{1}{\omega}\int^{t}_{0}\left(\fint^{T_{n}}_{0} H(r;\omega)\,\mathrm{d}r - H(s;\omega) +\varepsilon \right)\,\mathrm{d}s\right],\quad(x,t) \in\overline{\Omega}\times[0,T_{n}].\]
A straightforward computation yields that for each $n\in\mathbb{N}^+$, and all $(x,t)\in\Omega\times(0,T_n]$, $\overline{\varphi}_{n}$ and $\underline{\varphi}_{n}$ satisfy
\begin{equation}\label{eq5.10}
\begin{aligned}
   & \omega\partial_t\overline{\varphi}_{n}-d\operatorname{div}(A(x)\nabla\overline{\varphi}_{n})+c(x,t)\overline{\varphi}_{n} \\
 = & \left[\omega\partial_t\ln\overline{\varphi}_{n}-d\operatorname{div}(A(x)\nabla \ln\overline{\varphi}_{n})-d\nabla \ln\overline{\varphi}_{n}\cdot(A(x)\nabla \ln\overline{\varphi}_{n})+c(x,t)\right]\overline{\varphi}_{n} \\
 = & \Bigg[\fint^{T_n}_{0}c(x,s)\,\mathrm{d}s-c(x,t)-d\operatorname{div}(A(x)\nabla \ln\hat{\phi}_{n})-\frac{d}{\omega}\operatorname{div}(A(x)\nabla(\mathcal{M}(x,t;T_n)-\kappa\phi^*))\\ &\quad-d\nabla \ln\hat{\phi}_{n}\cdot(A(x)\nabla \ln\hat{\phi}_{n}) -\frac{d}{\omega}\nabla \ln\hat{\phi}_{n}\cdot(A(x)\nabla (\mathcal{M}(x,t;T_n)-\kappa\phi^*))\\
 &\quad-\frac{d}{\omega}\nabla (\mathcal{M}(x,t;T_n)-\kappa\phi^*)\cdot(A(x)\nabla \ln\hat{\phi}_{n}) \\
 &\quad -\frac{d}{\omega^2}\nabla (\mathcal{M}(x,t;T_n)-\kappa\phi^*)\cdot(A(x)\nabla (\mathcal{M}(x,t;T_n)-\kappa\phi^*)) +c(x,t)\Bigg]\overline{\varphi}_{n} \\
 = & \Bigg[\mu_n-\frac{d}{\omega}\operatorname{div}(A(x)\nabla(\mathcal{M}(x,t;T_n)-\kappa\phi^*))-2\frac{d}{\omega}\nabla \ln\hat{\phi}_{n}\cdot(A(x)\nabla (\mathcal{M}(x,t;T_n)-\kappa\phi^*))\\
 &\quad -\frac{d}{\omega^2}\nabla (\mathcal{M}(x,t;T_n)-\kappa\phi^*)\cdot(A(x)\nabla (\mathcal{M}(x,t;T_n)-\kappa\phi^*))\Bigg]\overline{\varphi}_{n} \\
 = & \Bigg[\mu_{n}- \frac{d}{\omega}\operatorname{div}\left(A(x)\nabla\left(\mathcal{M}(x,t;T_{n})-\kappa\phi^{*}\right)\right) \\
 &~-\frac{d}{\omega}\nabla \left(\mathcal{M}(x,t;T_{n})-\kappa\phi^{*}\right)\cdot\left(A(x)\left(2\nabla\ln\hat{\phi}_{n}+\frac{1}{\omega} \nabla \left(\mathcal{M}(x,t;T_{n})-\kappa\phi^{*}\right)\right)\right)\Bigg]\overline{\varphi}_{n}\\
 \ge & \Bigg[\mu_{n}- \frac{d}{\omega}\left(V_0+\kappa\|\operatorname{div}(A(\cdot)\nabla\phi^*)\|_\infty\right) \\
 &~-\frac{d\Lambda}{\omega} \left(V_0+\kappa\|\phi^{*}\|_{C^2(\Omega)}\right)\cdot\left(2\|\ln\hat{\phi}\|_{ C^1(\overline{\Omega})}+2+\frac{1}{\omega} \left(V_0+\kappa\|\phi^{*}\|_{C^2(\Omega)}\right)\right)\Bigg]\overline{\varphi}_{n} \\
 \ge & \left[\mu_{n}- C\left(\frac{d}{\omega}+\frac{d}{\omega^2}\right)\right]\overline{\varphi}_{n}
\end{aligned}
\end{equation}
for some constant $C>1$ which is independent of $\omega>0$, and
\begin{equation}\label{eq5.11}
\begin{aligned}
   & \omega\partial_t\underline{\varphi}_{n}-d\operatorname{div}(A(x)\nabla \underline{\varphi}_{n})+c(x,t)\underline{\varphi}_{n}\\
 = & \left(\omega\partial_t\ln\underline{\varphi}_{n}-d\operatorname{div}(A(x)\nabla \ln\underline{\varphi}_{n})-d\nabla \ln\underline{\varphi}_{n}\cdot(A(x)\nabla\ln\underline{\varphi}_n) +c(x,t)\right)\underline{\varphi}_{n} \\
 = & \left(\fint^{T_{n}}_{0}H(s;\omega)\,\mathrm{d}s+\varepsilon  \right)\underline{\varphi}_{n},
\end{aligned}
\end{equation}
respectively. Now by Proposition \ref{proposition3.3}(ii), we know that for any fixed $d>0$, there exists a positive constant $C_d$, independent of $\omega>0$, such that
\[|\ln\varphi(x,t)| \le C_d,\qquad (x,t)\in\overline{\Omega}\times\mathbb{R}.\]
Then for each $\omega>0$, there is an integer $N_\varepsilon>0$ such that
\[T_n>\frac{2C_d\omega }{\varepsilon},\qquad\forall~n>N_\varepsilon.\]
This, combining with \eqref{eq5.9}, \eqref{eq5.10}, \eqref{eq5.11}, and boundary conditions for $\varphi$ and $\hat{\phi}_n$, yields that for fixed $\omega>0$ and any $n>N_\varepsilon$, the pair $(\underline{\varphi}_n,\overline{\varphi}_n)$ satisfies
\begin{equation*}
  \left\{
  \begin{aligned}
    &\omega \partial_t\overline{\varphi}_n-d\operatorname{div}(A(x)\nabla \overline{\varphi}_n)+c(x,t)\overline{\varphi}_n \ge \left[\mu_{n}- C\left(\frac{d}{\omega}+\frac{d}{\omega^2}\right)\right] \overline{\varphi}_{n},&&(x,t)\in\Omega\times(0,T_{n}],\\
    &\omega \partial_t\underline{\varphi}_n-d\operatorname{div}(A(x)\nabla \underline{\varphi}_n)+c(x,t)\underline{\varphi}_n = \left(\fint^{T_{n}}_{0}H(s;\omega)\,\mathrm{d}s+\varepsilon  \right)\underline{\varphi}_{n},&&(x,t)\in\Omega\times(0,T_{n}],\\
    &\mathbf{n}\cdot(A(x)\nabla\overline{\varphi}_n)\ge 0 =\mathbf{n}\cdot(A(x)\nabla\underline{\varphi}_n), &&(x,t)\in\partial\Omega\times(0,T_{n}],\\
    &\overline{\varphi}_n(x,0)=\overline{\varphi}_n(x,T_{n}),&&x\in\overline{\Omega},\\
    &\underline{\varphi}_{n}(x,0)=\varphi(x,0)\le\mathrm{e}^{\frac{ \varepsilon}{\omega}T_{n}-2C_d}\varphi(x,0)\le\varphi(x,T_{n})\mathrm{e}^{\frac{ \varepsilon}{\omega}T_{n}}=\underline{\varphi}_{n}(x,T_{n}),&&x\in\overline{\Omega}.
  \end{aligned}
  \right.
\end{equation*}
For any $\hat{c}\in\mathscr{C}_*$, there is a subsequence $\{T_{n'}\}_{n'\in\mathbb{N}^+}$ of $\{T_{n}\}_{n\in\mathbb{N}^+}$ such that $\hat{c}_{n'}$ converges to some $\hat{c}$ in $C(\overline{\Omega})$. Therefore, using Lemmas \ref{lemma2.13} and \ref{lemma5.5} yields that
\[\limsup_{n\to\infty}\fint^{T_{n}}_{0}H(s;\omega)\,\mathrm{d}s+\varepsilon \ge \limsup_{n'\to\infty}\fint^{T_{n'}}_{0}H(s;\omega)\,\mathrm{d}s+\varepsilon \ge \mu^\infty(\hat{c})- C\left(\frac{d}{\omega}+\frac{d}{\omega^2}\right).\]
By arbitrariness of $\hat{c}\in\mathscr{C}_*$, we get that
\[\limsup_{n\to\infty}\fint^{T_{n}}_{0}H(s;\omega)\,\mathrm{d}s+\varepsilon \ge \sup_{\hat{c}\in\mathscr{C}_*}\mu^\infty(\hat{c})- C\left(\frac{d}{\omega}+\frac{d}{\omega^2}\right).\]
Sending $\omega\to\infty$ and by the arbitrariness of $\varepsilon>0$, we obtain that
  \[\liminf_{\omega\to\infty}\limsup_{n\to\infty}\fint^{T_n}_{0}H(s;\omega)\,\mathrm{d}s\ge \sup_{\hat{c}\in\mathscr{C}_*}\mu^\infty(\hat{c}).\]
On the other hand, for each $\omega>0$, there is a subsequence of $\{T_n\}_{n\in\mathbb{N}^+}$, denoted still by $\{T_{n'}\}_{n'\in\mathbb{N}^+}$, such that
 \[\lim_{n'\to\infty}\fint^{T_{n'}}_{0}H(s;\omega)\,\mathrm{d}s=\liminf_{n\to\infty}\fint^{T_{n}}_{0}H(s;\omega)\,\mathrm{d}s.\]
Passing to a subsequence of $\{T_{n'}\}_{n'\in\mathbb{N}^+}$, we assume that $\{\hat{c}_{n'}\}_{n'\in\mathbb{N}^+}$ converges to some function in $C(\overline{\Omega})$, denoted still by $\hat{c}\in\mathscr{C}_*$. An application of Lemmas \ref{lemma2.13} and \ref{lemma5.5} and then letting $\omega\to\infty$ yield
 \[\liminf_{\omega\to\infty}\liminf_{n\to\infty}\fint^{T_n}_{0}H(s;\omega)\,\mathrm{d}s\ge \inf_{\hat{c}\in\mathscr{C}_*}\mu^\infty(\hat{c}).\]

This together with the estimate of upper bound, we can complete the proof of the lemma.
\end{proof}

\medskip
\begin{lemma}\label{lemma5.7}
Let $H(\cdot;\omega)$ be the normalized principal Floquet bundle of \eqref{eq1.1} with potential $c\in \mathcal{C}\cap C^{2,1}(\overline{\Omega}\times\mathbb{R})$, $\mu^\infty(\hat{c})$ be the principal eigenvalue of \eqref{eq1.3} with potential $\hat{c}\in\mathscr{C}$. If the assumption \eqref{H2} holds, then
  \[\lim_{\omega\to\infty}\liminf_{T\to\infty}\fint^{T}_{0}H(s;\omega)\,\mathrm{d}s= \inf_{\hat{c}\in\mathscr{C}}\mu^\infty(\hat{c})\quad\text{and}\quad \lim_{\omega\to\infty}\limsup_{T\to\infty}\fint^{T}_{0}H(s;\omega)\,\mathrm{d}s= \sup_{\hat{c}\in\mathscr{C}}\mu^\infty(\hat{c}).\]
\end{lemma}
\begin{proof}
We show the first identity only since the second one can be proven by a similar way with a little modification. First of all, for any $0<\varepsilon<<1$, there is a function $\hat{c}_{\varepsilon}\in\mathscr{C}$ and sequence $\{T_n\}_{n\in\mathbb{N}^+}$ of $T\to\infty$ such that
\[\lim_{n\to\infty}\fint^{T_n}_{0}c(\cdot,s)\,\mathrm{d}s=\hat{c}_{\varepsilon}\qquad\text{and}\qquad \mu^\infty(\hat{c}_{\varepsilon})\le \inf_{\hat{c}\in\mathscr{C}}\mu^\infty(\hat{c})+\varepsilon.\]
Thus
\begin{equation}\label{eq5.12}
\mu^\infty\left(\lim_{n\to\infty}\fint^{T_n}_{0}c(\cdot,s)\,\mathrm{d}s\right)= \mu^\infty(\hat{c}_{\varepsilon})\le \inf_{\hat{c}\in\mathscr{C}}\mu^\infty(\hat{c})+\varepsilon.
\end{equation}
Using Lemma \ref{theorem5.6}, we have
\[\limsup_{\omega\to\infty}\liminf_{T\to\infty}\fint^{T}_{0}H(s;\omega)\,\mathrm{d}s\le \limsup_{\omega\to\infty}\limsup_{n\to\infty}\fint^{T_n}_{0}H(s;\omega)\,\mathrm{d}s= \mu^\infty\left(\lim_{n\to\infty}\fint^{T_n}_{0}c(\cdot,s)\,\mathrm{d}s\right).\]
By arbitrariness of $\varepsilon>0$, we get
  \[\limsup_{\omega\to\infty}\liminf_{T\to\infty}\fint^{T}_{0}H(s;\omega)\,\mathrm{d}s\le \inf_{\hat{c}\in\mathscr{C}}\mu^\infty(\hat{c}).\]
Then we need only to show
 \[\liminf_{\omega\to\infty}\liminf_{T\to\infty}\fint^{T}_{0}H(s;\omega)\,\mathrm{d}s\ge \inf_{\hat{c}\in\mathscr{C}}\mu^\infty(\hat{c}).\]
Clearly, it is sufficiently to show the estimate
\begin{equation}\label{eq5.13}
  \liminf_{\omega\to\infty}\liminf_{T\to\infty}\fint^{T}_{0}H(s;\omega)\,\mathrm{d}s\ge \liminf_{T\to\infty} \mu^\infty\left(\fint^{T}_{0}c(\cdot,s)\,\mathrm{d}s\right)
\end{equation}
and the identity
\begin{equation}\label{eq5.14}
  \liminf_{T\to\infty}\mu^\infty\left(\fint^{T}_{0}c(\cdot,s)\,\mathrm{d}s\right)=\inf_{\hat{c}\in\mathscr{C}} \mu^\infty(\hat{c}).
\end{equation}

\medskip
We first show \eqref{eq5.14}. Obviously, \eqref{eq5.12} also implies that
 \[\liminf_{T\to\infty}\mu^\infty\left(\fint^{T}_{0}c(\cdot,s)\,\mathrm{d}s\right)\le \inf_{\hat{c}\in\mathscr{C}} \mu^\infty(\hat{c}).\]
We then aim to show the inverse inequality. There is a sequence $\{T_n\}_{n\in\mathbb{N}^+}$ of $T\to\infty$ such that
 \[\lim_{n\to\infty}\mu^\infty\left(\fint^{T_n}_{0}c(\cdot,s)\,\mathrm{d}s\right)= \liminf_{T\to\infty}\mu^\infty\left(\fint^{T}_{0}c(\cdot,s)\,\mathrm{d}s\right).\]
An application of Proposition \ref{proposition1.1} gives that there is a subsequence $\{T_{n'}\}_{n'\in\mathbb{N}^+}$ of $\{T_{n}\}_{n\in\mathbb{N}^+}$ such that $\left\{\fint^{T_{n'}}_{0}c(\cdot,s)\,\mathrm{d}s\right\}_{n'\in\mathbb{N}^+}$ converges to some $\hat{c}'$ in $C(\overline{\Omega})$. Here we note that $\hat{c}'$ also belongs to $\mathscr{C}$. By the continuous  dependence of $\mu^\infty(\cdot)$ on potential (see Lemma \ref{lemma5.5}), we have
  \begin{align*}
    \inf_{\hat{c}\in\mathscr{C}}\mu^\infty(\hat{c}) & \le\mu^\infty(\hat{c}')=  \mu^\infty\left(\lim_{n'\to\infty}\fint^{T_{n'}}_{0}c(\cdot,s)\,\mathrm{d}s\right)
    =  \lim_{n'\to\infty}\mu^\infty\left(\fint^{T_{n'}}_{0}c(\cdot,s)\,\mathrm{d}s\right) \\
    &\hskip 4cm  =  \lim_{n\to\infty}\mu^\infty\left(\fint^{T_n}_{0}c(\cdot,s)\,\mathrm{d}s\right)= \liminf_{T\to\infty}\mu^\infty\left(\fint^{T}_{0}c(\cdot,s)\,\mathrm{d}s\right).
  \end{align*}
Hence, limit \eqref{eq5.14} is proven.

\medskip
Now it remains to show \eqref{eq5.13}. Recall that the constants and functions appears in the proof of Lemma \ref{theorem5.6}, such as $V_0$, $\kappa$, $C$, $C_d$, $\phi^{*}$ and so on. For each $T>0$, set
\[\overline{\varphi}_T(x,t):=\hat{\phi}_T(x)\exp\left[\frac{1}{\omega}\left(\mathcal{M}(x,t;T)- \kappa\phi^{*}\right) \right],\quad(x,t) \in\overline{\Omega}\times[0,T],\]
where $\mathcal{M}(x,t;T)=t\fint^{T}_{0}c(x,s)\,\mathrm{d}s-\int^{t}_{0}c(x,s)\,\mathrm{d}s$ and $\hat{\phi}_T$ is the principal eigenfunctions of problem \eqref{eq5.8} with potential $\fint^{T}_{0}c(\cdot,s)\,\mathrm{d}s$, and
\[\underline{\varphi}_T(x,t):=\varphi(x,t)\exp\left[\frac{1}{\omega}\int^{t}_{0}\left(\fint^{T}_{0} H(r;\omega)\,\mathrm{d}r - H(s;\omega) +\varepsilon \right)\,\mathrm{d}s\right],\quad(x,t) \in\overline{\Omega}\times[0,T],\]
where $\varphi$ is the normalized principal Floquet bundle of \eqref{eq1.1}. Observe that the constant $C$ is independent of $\omega$. Then, arguing similarly as the proof in Lemma \ref{theorem5.6}, for each $\omega>0$ and $T>\frac{2C_d\omega}{\varepsilon}$, the pair $(\underline{\varphi}_T,\overline{\varphi}_T)$ satisfies
\begin{equation*}
  \left\{
  \begin{aligned}
    &\omega \partial_t\overline{\varphi}_T-d\operatorname{div}(A(x)\nabla \overline{\varphi}_T)+c(x,t)\overline{\varphi}_T \\
    &\hskip 4cm \ge \left[\mu^\infty\left(\fint^{T}_{0}c(\cdot,s)\,\mathrm{d}s\right)- C\left(\frac{d}{\omega}+\frac{d}{\omega^2}\right)\right] \overline{\varphi}_T,&&(x,t)\in\Omega\times(0,T],\\
    &\omega \partial_t\underline{\varphi}_T-d\operatorname{div}(A(x)\nabla \underline{\varphi}_T)+c(x,t)\underline{\varphi}_T = \left(\fint^{T}_{0}H(s;\omega)\,\mathrm{d}s+\varepsilon  \right)\underline{\varphi}_T,&&(x,t)\in\Omega\times(0,T],\\
    &\mathbf{n}\cdot(A(x)\nabla\overline{\varphi}_T)\ge 0 =\mathbf{n}\cdot(A(x)\nabla\underline{\varphi}_T), &&(x,t)\in\partial\Omega\times(0,T],\\
    &\overline{\varphi}_T(x,0)=\overline{\varphi}_T(x,T),&&x\in\overline{\Omega},\\
    &\underline{\varphi}_T(x,0)=\varphi(x,0)\le\mathrm{e}^{\frac{ \varepsilon}{\omega}T-2C_d}\varphi(x,0)\le\varphi(x,T)\mathrm{e}^{\frac{ \varepsilon}{\omega}T}=\underline{\varphi}_T(x,T),&&x\in\overline{\Omega}.
  \end{aligned}
  \right.
\end{equation*}
Using Corollary \ref{corollary2.15}, we conclude that for each $\omega>0$,
\[\liminf_{n\to\infty}\fint^{T}_{0}H(s;\omega)\,\mathrm{d}s+\varepsilon \ge \liminf_{T\to\infty}\mu^\infty\left(\fint^{T}_{0}c(\cdot,s)\,\mathrm{d}s\right)- C\left(\frac{d}{\omega}+\frac{d}{\omega^2}\right).\]
Sending $\omega\to\infty$ and by the arbitrariness of $\varepsilon>0$, we get \eqref{eq5.13}. This completes the proof.
\end{proof}

\section{The properties of the limiting values}
We note that the limiting values obtained in the previous two sections are determined by the principal eigenvalues of associated elliptic problems, which depend nontrivially on the diffusion rate $d>0$. Therefore, in order to better understand how these limiting values of the principal Floquet exponent relate to the diffusion rate, we wish to exploit known properties of elliptic principal eigenvalues with respect to $d$. Our proof relies crucially on two uniform convergence results: the uniform convergence of the principal eigenvalue $\mu^0(t;d)$ with respect to the parameter $t$ (see Lemma \ref{lemma2.24}), and the uniform convergence of the principal eigenvalue $\mu^\infty(d)$ with respect to the potential function (see Lemmas \ref{lemma2.25} and \ref{lemma2.26}). We prove the following result for the sequence $\{T_n\}_{n\in\mathbb{N}^+}$ of $T\to\infty$.
\begin{theorem}\label{theorem5.8}
Let $H(\cdot;\omega)$ be the normalized principal Floquet bundle of \eqref{eq1.1}, with potential $c\in \mathcal{C}$, and $\mu^{0}(t;d)$, $t\in\mathbb{R}$ be the principal eigenvalue of \eqref{eq1.4}. Let $\{T_n\}_{n\in\mathbb{N}^+}$ be a sequence satisfying $T_n\to\infty$ as $n\to\infty$.
\begin{enumerate}[{\rm (i)}]
  \item If $c\in C^{0,1}(\overline{\Omega}\times\mathbb{R})$ fulfills $\sup_{t\in\mathbb{R}}\|\partial_tc(\cdot,t)\|_\infty<\infty$, then the limits
      \[\lim_{\omega\to0}\liminf_{n\to\infty}\fint^{T_n}_{0}H(s;\omega,d)\,\mathrm{d}s\quad\text{and}\quad \lim_{\omega\to0}\limsup_{n\to\infty}\fint^{T_n}_{0}H(s;\omega,d)\,\mathrm{d}s\]
      are continuous and non-decreasing in $d>0$, and satisfy
      \begin{align*}
         & \lim_{d\to 0}\lim_{\omega\to0}\liminf_{n\to\infty}\fint^{T_n}_{0}H(s;\omega,d)\,\mathrm{d}s =\liminf_{n\to\infty}\fint^{T_n}_{0}\min_{x\in\overline{\Omega}}c(x,s)\,\mathrm{d}s, \\
         & \lim_{d\to 0}\lim_{\omega\to0}\limsup_{n\to\infty}\fint^{T_n}_{0}H(s;\omega,d)\,\mathrm{d}s =\limsup_{n\to\infty}\fint^{T_n}_{0} \min_{x\in\overline{\Omega}}c(x,s)\,\mathrm{d}s, \\
         & \lim_{d\to \infty}\lim_{\omega\to0}\liminf_{n\to\infty}\fint^{T_n}_{0} H(s;\omega,d)\,\mathrm{d}s=\liminf_{n\to\infty}\fint^{T_n}_{0}\int_\Omega c(x,s)\,\mathrm{d}x\mathrm{d}s, \\
         &\lim_{d\to \infty}\lim_{\omega\to0}\limsup_{n\to\infty}\fint^{T_n}_{0} H(s;\omega,d)\,\mathrm{d}s=\limsup_{n\to\infty}\fint^{T_n}_{0}\int_\Omega c(x,s)\,\mathrm{d}x\mathrm{d}s.
      \end{align*}
  \item Assume  in addition that $c\in\mathcal{C}\cap C^{2,1}(\overline{\Omega}\times\mathbb{R})$ satisfies \eqref{H2}. If the sequence $\left\{\fint^{T_n}_{0}c(\cdot,s)\,\mathrm{d}s \right\}_{n\in\mathbb{N}^+}$ of functions converges to some function $\hat{c}(\cdot)$ in $C(\overline{\Omega})$, then the limits
      \[\lim_{\omega\to\infty}\liminf_{n\to\infty}\fint^{T_n}_{0}H(s;\omega,d)\,\mathrm{d}s\quad\text{and}\quad \lim_{\omega\to\infty}\limsup_{n\to\infty}\fint^{T_n}_{0}H(s;\omega,d)\,\mathrm{d}s \]
are continuous and non-decreasing in $d>0$, and
 \[\lim_{d\to 0}\lim_{\omega\to\infty}\liminf_{n\to\infty}\fint^{T_n}_{0}H(s;\omega,d)\,\mathrm{d}s= \lim_{d\to 0}\lim_{\omega\to\infty}\limsup_{n\to\infty}\fint^{T_n}_{0}H(s;\omega,d)\,\mathrm{d}s =\min_{x\in\overline{\Omega}}\hat{c}(x),\]
and
 \[\lim_{d\to \infty}\lim_{\omega\to\infty}\liminf_{n\to\infty}\fint^{T_n}_{0}H(s;\omega,d)\,\mathrm{d}s = \lim_{d\to \infty}\lim_{\omega\to\infty}\limsup_{n\to\infty}\fint^{T_n}_{0}H(s;\omega,d)\,\mathrm{d}s =\fint_{\Omega}\hat{c}(x)\,\mathrm{d}x.\]
\end{enumerate}
\end{theorem}

We obtain Theorems \ref{theorem5.8} and \ref{corollary1.8} by examining the limiting behavior, as $d$ varies, of the quantities associated with the eigenvalues $\mu^0(\cdot;d)$ and $\mu^\infty(d)$.
\begin{lemma}\label{lemma5.9}
Let $c\in \mathcal{C}\cap C^{0,1}(\overline{\Omega}\times\mathbb{R})$ be a function such that $\sup_{t\in\mathbb{R}}\|\partial_tc(\cdot,t)\|_\infty<\infty$, $H(\cdot;\omega)$ be the normalized principal Floquet bundle of \eqref{eq1.1} with potential $c$, and $\mu^{0}(t;d)$, $t\in\mathbb{R}$ be the principal eigenvalue of \eqref{eq2.7}. Then for any sequence $\{T_n\}_{n\in\mathbb{N}^+}$ of $T\to\infty$, the quantities $\liminf_{n\to\infty}\fint^{T_n}_{0}\mu^0(s;d)\,\mathrm{d}s$ and $\limsup_{n\to\infty}\fint^{T_n}_{0}\mu^0(s;d)\,\mathrm{d}s$ are continuous and nondecreasing in $d>0$, and
\begin{align*}
 &\lim_{d\to 0}\liminf_{n\to\infty}\fint^{T_n}_{0}\mu^0(s;d)\,\mathrm{d}s=\liminf_{n\to\infty}\fint^{T_n}_{0}\min_{x\in\overline{\Omega}}c(x,s)\,\mathrm{d}s,\\
 &\lim_{d\to 0}\limsup_{n\to\infty}\fint^{T_n}_{0}\mu^0(s;d)\,\mathrm{d}s =\limsup_{n\to\infty}\fint^{T_n}_{0}\min_{x\in\overline{\Omega}}c(x,s)\,\mathrm{d}s,\\
 &\lim_{d\to \infty}\liminf_{n\to\infty}\fint^{T_n}_{0}\mu^0(s;d)\,\mathrm{d}s=\liminf_{n\to\infty}\fint^{T_n}_{0}\int_\Omega c(x,s)\,\mathrm{d}x\mathrm{d}s,\\
 &\lim_{d\to \infty}\limsup_{n\to\infty}\fint^{T_n}_{0}\mu^0(s;d)\,\mathrm{d}s =\limsup_{n\to\infty}\fint^{T_n}_{0}\int_\Omega c(x,s)\,\mathrm{d}x\mathrm{d}s.
\end{align*}
 Furthermore, $\liminf_{T\to\infty}\fint^{T}_{0}\mu^0(s;d)\,\mathrm{d}s$ and $\limsup_{T\to\infty}\fint^{T}_{0}\mu^0(s;d)\,\mathrm{d}s$ are continuous and increasing in $d>0$; and the following limits are true,
\begin{align*}
   &\lim_{d\to 0}\liminf_{T\to\infty}\fint^{T}_{0}\mu^0(s;d)\,\mathrm{d}s=\liminf_{T\to\infty}\fint^{T}_{0}\min_{x\in\overline{\Omega}}c(x,s)\,\mathrm{d}s,\\
 &\lim_{d\to 0}\limsup_{T\to\infty}\fint^{T}_{0}\mu^0(s;d)\,\mathrm{d}s =\limsup_{T\to\infty}\fint^{T}_{0}\min_{x\in\overline{\Omega}}c(x,s)\,\mathrm{d}s,\\
 &\lim_{d\to \infty}\liminf_{T\to\infty}\fint^{T}_{0}\mu^0(s;d)\,\mathrm{d}s=\liminf_{T\to\infty}\fint^{T}_{0}\int_\Omega c(x,s)\,\mathrm{d}x\mathrm{d}s,\\
 &\lim_{d\to \infty}\limsup_{T\to\infty}\fint^{T}_{0}\mu^0(s;d)\,\mathrm{d}s =\limsup_{T\to\infty}\fint^{T}_{0}\int_\Omega c(x,s)\,\mathrm{d}x\mathrm{d}s.
\end{align*}
\end{lemma}
\begin{proof}
Owing to the Rayleigh variational representation of the principal eigenvalue, $\mu(t;d)$ exhibits monotonic increase in the diffusion coefficient $d>0$ for every fixed time $t\in\mathbb{R}$. As a result, this monotonicity is inherited by both the upper and lower limits of its long-time average, rendering them non-decreasing in $d>0$. Furthermore, in conjunction with the uniform boundedness of the potential and the uniform ellipticity of the elliptic operator, uniform energy bounds for the principal eigenfunction can be established. These bounds entail that the Lipschitz continuity of $\mu(t;d)$ with respect to d is uniformly valid for all $t \in \mathbb{R}$. Consequently, both the upper and lower limits of the long-time average are locally Lipschitz continuous with respect to $d>0$.

Because of $c\in\mathcal{C}$, Lemma \ref{lemma2.24} implies that
 \[\lim_{d\to 0}\mu^0(t;d)=\min_{x\in\overline{\Omega}}c(x,t) \quad \text{holds uniformly in }t\in\mathbb{R}.\]
Therefore, for any $\varepsilon>0$ there exists $d_\varepsilon >0$ such that whenever $0<d<d_\varepsilon $,
 \[\sup_{t\in\mathbb{R}}\left|\mu^0(t;d)-\min_{x\in\overline{\Omega}}c(x,t)\right|<\varepsilon.\]
Consequently, for every $T>0$,
 \[\left|\fint^{T}_{0}\mu^0(s;d)\,\mathrm{d}s-\fint^{T}_{0}\min_{x\in\overline{\Omega}}c(x,s)\,\mathrm{d}s\right| \le\fint^{T}_{0}\left|\mu^0(s;d)-\min_{x\in\overline{\Omega}}c(x,s)\right|\,\mathrm{d}s<\varepsilon.\]
Taking $T=T_n$ in particular, we obtain
 \[\sup_{n\in\mathbb{N}^+}\left|\fint^{T_n}_{0}\mu^0(s;d)\,\mathrm{d}s-\fint^{T_n}_{0}\min_{x\in\overline{\Omega}}c(x,s)\,\mathrm{d}s\right|<\varepsilon, \qquad\forall~d\in(0,d_\varepsilon).\]
Consequently, we have
 \[\fint^{T_n}_{0}\min_{x\in\overline{\Omega}}c(x,s)\,\mathrm{d}s-\varepsilon<\fint^{T_n}_{0}\mu^0(s;d)\,\mathrm{d}s<\fint^{T_n}_{0}\min_{x\in\overline{\Omega}}c(x,s)\,\mathrm{d}s+\varepsilon \qquad\forall~n\in\mathbb{N}^+,~d\in(0,d_\varepsilon).\]
Taking the $\liminf$ with respect to $n$ and using the monotonicity of the $\liminf$, we get
 \[\liminf_{n\to\infty}\left(\fint^{T_n}_{0}\min_{x\in\overline{\Omega}}c(x,s)\,\mathrm{d}s-\varepsilon\right) \le\liminf_{n\to\infty}\fint^{T_n}_{0}\mu^0(s;d)\,\mathrm{d}s \le\liminf_{n\to\infty}\left(\fint^{T_n}_{0}\min_{x\in\overline{\Omega}}c(x,s)\,\mathrm{d}s+\varepsilon\right),\]
or equivalently,
\[\liminf_{n\to\infty}\fint^{T_n}_{0}\min_{x\in\overline{\Omega}}c(x,s)\,\mathrm{d}s-\varepsilon \le\liminf_{n\to\infty}\fint^{T_n}_{0}\mu^0(s;d)\,\mathrm{d}s \le\liminf_{n\to\infty}\fint^{T_n}_{0}\min_{x\in\overline{\Omega}}c(x,s)\,\mathrm{d}s+\varepsilon\]
for all $d\in(0,d_\varepsilon)$. Thus, by letting $d\to 0$ and the arbitrariness of $\varepsilon$, we conclude that
 \[\lim_{d\to0}  \liminf_{n\to\infty}\fint^{T_n}_{0}\mu^0(s;d)\,\mathrm{d}s = \liminf_{n\to\infty}\fint^{T_n}_{0}\min_{x\in\overline{\Omega}}c(x,s)\,\mathrm{d}s.\]
Similarly, we can show that
 \[\lim_{d\to0}  \limsup_{n\to\infty}\fint^{T_n}_{0}\mu^0(s;d)\,\mathrm{d}s =  \limsup_{n\to\infty}\fint^{T_n}_{0}\min_{x\in\overline{\Omega}}c(x,s)\,\mathrm{d}s.\]

For the large-diffusion case, Lemma \ref{lemma2.26} implies that
 \[\lim_{d\to \infty}\mu^0(t;d)=\fint_\Omega c(x,t)\,\mathrm{d}x \quad \text{holds uniformly in }t\in\mathbb{R}.\]
Then following a line of argument analogous to that for the small-diffusion limit gives

 \[\fint^{T_n}_{0}\fint_{\Omega}c(x,s)\,\mathrm{d}x\mathrm{d}s-\varepsilon <\fint^{T_n}_{0}\mu^0(s;d)\,\mathrm{d}s<\fint^{T_n}_{0}\fint_{\Omega}c(x,s) \,\mathrm{d}x\mathrm{d}s+\varepsilon, \qquad\forall~n\in\mathbb{N}^+,~d>d_\varepsilon\]
for any given $\varepsilon>0$ and all $d>d_\varepsilon$ for some $d_\varepsilon>0$. Then we can establish the following equalities,
 \[\lim_{d\to\infty}\liminf_{n\to\infty}\fint^{T_n}_{0}\mu^0(s;d)\,\mathrm{d}s =  \liminf_{n\to\infty}\fint^{T_n}_{0}\fint_\Omega c(x,s)\,\mathrm{d}x\mathrm{d}s\]
and
 \[\lim_{d\to\infty}\limsup_{n\to\infty}\fint^{T_n}_{0}\mu^0(s;d)\,\mathrm{d}s =  \limsup_{n\to\infty}\fint^{T_n}_{0}\fint_\Omega c(x,s)\,\mathrm{d}x\mathrm{d}s.\]

Similarly, by a few modifications one can check the last four limits in the lemma. This completes the proof.
\end{proof}

\begin{lemma}\label{lemma5.10}
Let $c\in \mathcal{C}$ be a function, $\{T_n\}_{n\in\mathbb{N}^+}$ be a sequence satisfying $T_n\to\infty$ as $n\to\infty$. Then the two quantities $\inf_{\hat{c}\in\mathscr{C}_*}\mu^\infty(d,\hat{c})$ and $\sup_{\hat{c}\in\mathscr{C}_*}\mu^\infty(d,\hat{c})$ are continuous and non-decreasing in $d>0$, and the following limits are valid,
\begin{equation}\label{eq5.15}
  \begin{aligned}
   & \lim_{d\to 0}\inf_{\hat{c}\in\mathscr{C}_*}\mu^\infty(d,\hat{c})=\min_{x\in\overline{\Omega}} \liminf_{n\to\infty}\fint^{T_n}_{0}c(x,s)\,\mathrm{d}s, \\
   & \lim_{d\to\infty}\inf_{\hat{c}\in\mathscr{C}_*}\mu^\infty(d,\hat{c})= \liminf_{n\to\infty}\fint^{T_n}_{0} \fint_{\Omega}c(x,s)\,\mathrm{d}x\mathrm{d}s, \\
   & \lim_{d\to 0}\sup_{\hat{c}\in\mathscr{C}_*}\mu^\infty(d,\hat{c})= \limsup_{n\to\infty}\min_{x\in\overline{\Omega}}\fint^{T_n}_{0}c(x,s)\,\mathrm{d}s, \\
   & \lim_{d\to\infty}\sup_{\hat{c}\in\mathscr{C}_*}\mu^\infty(d,\hat{c})= \limsup_{n\to\infty}\fint^{T_n}_{0} \fint_{\Omega}c(x,s)\,\mathrm{d}x\mathrm{d}s.
  \end{aligned}
\end{equation}
Furthermore, the two quantities $\inf_{\hat{c}\in\mathscr{C}}\mu^\infty(d,\hat{c})$ and $\sup_{\hat{c}\in\mathscr{C}}\mu^\infty(d,\hat{c})$ are continuous and nondecreasing in $d>0$; and the following quantities are valid,
\begin{equation}\label{eq5.16}
  \begin{aligned}
   & \lim_{d\to 0}\inf_{\hat{c}\in\mathscr{C}}\mu^\infty(d,\hat{c})=\min_{x\in\overline{\Omega}} \liminf_{T\to\infty}\fint^{T}_{0}c(x,s)\,\mathrm{d}s, \\
   & \lim_{d\to\infty}\inf_{\hat{c}\in\mathscr{C}}\mu^\infty(d,\hat{c})= \liminf_{T\to\infty}\fint^{T}_{0} \fint_{\Omega}c(x,s)\,\mathrm{d}x\mathrm{d}s, \\
   & \lim_{d\to 0}\sup_{\hat{c}\in\mathscr{C}}\mu^\infty(d,\hat{c})= \limsup_{T\to\infty}\min_{x\in\overline{\Omega}}\fint^{T}_{0}c(x,s)\,\mathrm{d}s, \\
   & \lim_{d\to\infty}\sup_{\hat{c}\in\mathscr{C}}\mu^\infty(d,\hat{c})= \limsup_{T\to\infty}\fint^{T}_{0} \fint_{\Omega}c(x,s)\,\mathrm{d}x\mathrm{d}s.
  \end{aligned}
\end{equation}
\end{lemma}
\begin{proof}
By compactness of the potential set and joint continuity of $\mu_\infty$, the family is equicontinuous in $d$ on every compact $d$-interval; hence the upper and lower envelops ( $\inf_{\hat{c}\in\mathscr{C}_*}\mu^\infty(d,\hat{c})$ and $\sup_{\hat{c}\in\mathscr{C}_*}\mu^\infty(d,\hat{c})$) are continuous in $d$. We now give a brief proof of the monotonicity of $\inf_{\hat{c}\in\mathscr{C}_*} \mu^\infty(d,\hat{c})$. Let $0<d_1<d_2$. Suppose, for contradiction, that
 \[\inf_{\hat{c}\in\mathscr{C}_*}\mu^\infty(d_1;\hat{c})> \inf_{\hat{c}\in\mathscr{C}_*}\mu^\infty(d_2;\hat{c}).\]
Set
\[\varepsilon = \frac{1}{2}\Bigl(\inf_{\hat{c}\in\mathscr{C}_*}\mu^\infty(d_1;\hat{c})-\inf_{\hat{c}\in\mathscr{C}_*}\mu^\infty(d_2;\hat{c})\Bigr) > 0.\]
By the definition of the infimum, there exists $\hat{c}_\varepsilon\in\mathscr{C}_*$ such that
 \[\mu^\infty(d_2;\hat{c}_\varepsilon) \le \inf_{\hat{c}\in\mathscr{C}_*}\mu^\infty(d_2;\hat{c}) + \varepsilon.\]
Using the monotonicity of the elliptic principal eigenvalue with respect to the diffusion rate (i.e., the map $d \mapsto \mu^\infty(d,\hat{c})$ is nondecreasing), we obtain
 \[\inf_{\hat{c}\in\mathscr{C}_*}\mu^\infty(d_1;\hat{c}) \le \mu^\infty(d_1;\hat{c}_\varepsilon) \le \mu^\infty(d_2;\hat{c}_\varepsilon) \le \inf_{\hat{c}\in\mathscr{C}_*}\mu^\infty(d_2;\hat{c}) + \varepsilon.\]
Substituting the expression for $\varepsilon$, the right-hand side equals
 \[\frac{1}{2}\Bigl(\inf_{\hat{c}\in\mathscr{C}_*}\mu^\infty(d_1;\hat{c}) + \inf_{\hat{c}\in\mathscr{C}_*}\mu^\infty(d_2;\hat{c})\Bigr) = \inf_{\hat{c}\in\mathscr{C}_*}\mu^\infty(d_1;\hat{c}) - \varepsilon,\]
which yields
 \[\inf_{\hat{c}\in\mathscr{C}_*}\mu^\infty(d_1;\hat{c}) \le \inf_{\hat{c}\in\mathscr{C}_*}\mu^\infty(d_1;\hat{c}) - \varepsilon,\]
a contradiction. Hence the monotonicity is proven.

We now consider the small-diffusion limits. By Lemma \ref{lemma2.18}, we have
 \[ \lim_{d\to0}\inf_{\hat{c}\in \mathscr{C}_*}\mu^{\infty}(d,\hat{c}) =\inf_{\hat{c}\in \mathscr{C}_*}\min_{x\in\overline{\Omega}}\hat{c}(x) \quad\text{and}\quad \lim_{d\to0}\sup_{\hat{c}\in \mathscr{C}_*}\mu^{\infty}(d,\hat{c})=\sup_{\hat{c}\in \mathscr{C}_*}\min_{x\in\overline{\Omega}}\hat{c}(x).\]
Using Lemmas \ref{lemma2.3} and \ref{lemma2.4} yields
 \[\lim_{d\to0}\inf_{\hat{c}\in\mathscr{C}_*}\mu^\infty(d,\hat{c})= \min_{x\in\overline{\Omega}}\liminf_{n\to0}\fint^{T_n}_{0}c(x,s)\,\mathrm{d}s\]
and
 \[ \lim_{d\to0}\sup_{\hat{c}\in\mathscr{C}_*}\mu^\infty(d,\hat{c})= \limsup_{n\to\infty}\min_{x\in\overline{\Omega}}\fint^{T_n}_{0}c(x,s)\,\mathrm{d}s.\]

We now turn to the large-diffusion limit. Indeed, Lemma \ref{lemma2.25} tells us that the limit
\[\lim_{d\to\infty}\mu^\infty(d,\hat{c})=\fint_\Omega \hat{c}(x)\,\mathrm{d}x\]
holds uniformly with respect to $\hat{c}\in\mathscr{C}_*$. In other words, for every $\varepsilon>0$ there exists $d_\varepsilon>0$ such that whenever $d>d_\varepsilon$,
 \[\fint_\Omega\hat{c}(x)\,\mathrm{d}x-\varepsilon < \mu^\infty(d,\hat{c}) < \fint_\Omega\hat{c}(x)\,\mathrm{d}x+\varepsilon, \qquad \forall\,\hat{c}\in\mathscr{C}_*.\]
Taking the infimum and supremum over $\hat{c}\in\mathscr{C}_*$ gives
 \[\inf_{\hat{c}\in\mathscr{C}_*}\fint_\Omega\hat{c}(x)\,\mathrm{d}x-\varepsilon \le \inf_{\hat{c}\in\mathscr{C}_*}\mu^\infty(d,\hat{c}) \le \inf_{\hat{c}\in\mathscr{C}_*}\fint_\Omega\hat{c}(x)\,\mathrm{d}x+\varepsilon,\qquad\forall~d>d_\varepsilon\]
and
 \[\sup_{\hat{c}\in\mathscr{C}_*}\fint_\Omega\hat{c}(x)\,\mathrm{d}x-\varepsilon \le \sup_{\hat{c}\in\mathscr{C}_*}\mu^\infty(d,\hat{c}) \le \sup_{\hat{c}\in\mathscr{C}_*}\fint_\Omega\hat{c}(x)\,\mathrm{d}x+\varepsilon,\qquad\forall~d>d_\varepsilon,\]
respectively. Letting $d\to\infty$ and noting that $\varepsilon>0$ is arbitrary, we obtain
 \[\lim_{d\to\infty}\inf_{\hat{c}\in\mathscr{C}_*}\mu^\infty(d,\hat{c})= \inf_{\hat{c}\in\mathscr{C}_*}\fint_\Omega\hat{c}(x)\,\mathrm{d}x =\liminf_{n\to\infty}\fint^{T_n}_{0}\fint_\Omega c(x,s)\,\mathrm{d}x\mathrm{d}s,\]
and
 \[\lim_{d\to\infty}\sup_{\hat{c}\in\mathscr{C}_*}\mu^\infty(d,\hat{c})= \sup_{\hat{c}\in\mathscr{C}_*}\fint_\Omega\hat{c}(x)\,\mathrm{d}x =\limsup_{n\to\infty}\fint^{T_n}_{0}\fint_\Omega c(x,s)\,\mathrm{d}x\mathrm{d}s,\]
respectively, where Lemma \ref{lemma2.7} has been used.

Following the same line of thought, \eqref{eq5.16} can be proven. This completes the proof.
\end{proof}

\medskip
Combining Theorem \ref{theorem5.1} with Lemmas \ref{lemma5.9} and \ref{lemma5.10} gives Theorem \ref{theorem5.8}. Moreover, Corollary \ref{corollary1.8} follows directly from Lemmas \ref{lemma5.9} and \ref{lemma5.10}.

\chapter{Asymptotic behavior of $H$: $d\to\infty$ and $(\omega,d)\to(\infty,0)$}\label{chp6}
The purpose of this chapter is to prove Theorem \ref{theorem1.9} and its version in the sense of sequence $\{T_n\}_{n\in\mathbb{N}^+}$ of $T\to\infty$.
\begin{theorem}\label{theorem6.1}
Let $H(\cdot;\omega,d)$ be the normalized principal Floquet bundle of \eqref{eq1.1} with potential $c\in\mathcal{C}$ and let $\{T_n\}_{n\in\mathbb{N}^+}$ be a sequence satisfying $T_n\to\infty$ as $n\to\infty$.
\begin{enumerate}[{\upshape (i)}]
  \item Let $\vartheta\in[0,\infty]$ be given. Assume that $\sup_{t\in\mathbb{R}}\|\partial_t c(\cdot,t)\|_{\infty}<\infty$. Suppose further that $c\in C^{2,1}(\overline{\Omega}\times\mathbb{R})$ satisfies assumption \eqref{H2} if $\frac{\omega}{d}\to\infty$. Then
       \[ \lim\limits_{\left(d,\frac{\omega}{d}\right)\to(\infty,\vartheta)} \liminf\limits_{n\to\infty}\fint^{T_n}_{0}H(s;\omega,d)\,\mathrm{d}s = \liminf_{n\to\infty}\fint^{T_n}_{0}\fint_\Omega c(x,s)\,\mathrm{d}x\mathrm{d}s\]
      and
       \[\lim\limits_{\left(d,\frac{\omega}{d}\right)\to(\infty,\vartheta)} \limsup\limits_{n\to\infty}\fint^{T_n}_{0}H(s;\omega,d)\,\mathrm{d}s = \limsup_{n\to\infty}\fint^{T_n}_{0}\fint_\Omega c(x,s)\,\mathrm{d}x\mathrm{d}s.\]
      Consequently, if the sequence $\left\{\fint^{T_n}_{0} c(\cdot,s)\,\mathrm{d}s\right\}_{n\in\mathbb{N}^+}$ of functions converges to some function $\hat{c}(\cdot)$ in $C(\overline{\Omega})$, then
       \[\lim\limits_{\left(d,\frac{\omega}{d}\right)\to(\infty,\vartheta)} \liminf\limits_{n\to\infty}\fint^{T_n}_{0}H(s;\omega,d)\,\mathrm{d}s= \lim\limits_{\left(d,\frac{\omega}{d}\right)\to(\infty,\vartheta)} \limsup\limits_{n\to\infty}\fint^{T_n}_{0}H(s;\omega,d)\,\mathrm{d}s = \fint_\Omega \hat{c}(x)\,\mathrm{d}x.\]
  \item Suppose that $c$ satisfies \eqref{H2}. Then
       \[\lim_{(\omega,d)\to(\infty,0)}\liminf_{n\to\infty}\fint^{T_n}_{0}H(s;\omega,d)\,\mathrm{d}s= \min_{x\in\overline{\Omega}}\liminf_{n\to\infty}\fint^{T_n}_{0}c(x,s)\,\mathrm{d}s= \inf_{\hat{c}\in\mathscr{C}_*}\min_{x\in\overline{\Omega}}\hat{c}(x)\]
      and
       \[\lim_{(\omega,d)\to(\infty,0)}\limsup_{n\to\infty}\fint^{T_n}_{0}H(s;\omega,d)\,\mathrm{d}s= \limsup_{n\to\infty}\min_{x\in\overline{\Omega}}\fint^{T_n}_{0}c(x,s)\,\mathrm{d}s= \sup_{\hat{c}\in\mathscr{C}_*}\min_{x\in\overline{\Omega}}\hat{c}(x).\]
      Consequently, if the sequence $\left\{\fint^{T_n}_{0}c(\cdot,s)\,\mathrm{d}s\right\}_{n\in\mathbb{N}^+}$ converges to $\hat{c}(\cdot)$ in $C(\overline{\Omega})$, then
       \[\lim_{ (\omega,d)\to(\infty,0)}\liminf_{n\to\infty}\fint^{T_n}_{0}H(s;\omega,d)\,\mathrm{d}s= \lim_{(\omega,d)\to(\infty,0)}\limsup_{n\to\infty}\fint^{T_n}_{0}H(s;\omega,d)\,\mathrm{d}s= \min_{x\in\overline{\Omega}}\hat{c}(x).\]
\end{enumerate}
\end{theorem}
The proof will be accomplished via a series of lemmas. Precisely, we analyze the asymptotic behavior of the normalized principal Floquet bundle under two extreme parameter regimes: (i) when the diffusion coefficient $d$ tends to infinity and $\frac{\omega}{d}$ approaches a certain limit, its time average converges to the spatio-temporal integral mean of the potential function $c$; (ii) whereas when $(\omega, d)$ tends to $(\infty, 0)$, its time average converges to the minimum value attained by the temporal average of $c$ over the spatial domain.
\section{The case $\left(d,\frac{\omega}{d}\right)\to(\infty,\vartheta)$ with $\vartheta\in[0,\infty]$}
This section investigates the asymptotics of the normalized principal Floquet exponent under large diffusion. We will see that the same limit is obtained whether $\omega/d$ tends to zero, tends to a finite positive number, or tends to infinity; nevertheless, the method of constructing sub- and super-solutions depends on the limiting value of the ratio.
\begin{lemma}\label{lemma6.2}
Let $\vartheta\in[0,\infty]$ be fixed and let $c\in\mathcal{C}\cap C^{2,1}(\overline{\Omega}\times\mathbb{R})$ be a given function satisfying $\sup_{t\in\mathbb{R}}\|\partial_t c(\cdot,t)\|_{\infty}<\infty$. Denote by $(H,\varphi)$ the normalized principal Floquet bundle of \eqref{eq1.1} with potential $c$. Suppose further that assumption \eqref{H2} holds for $c$ if $\vartheta=\infty$. Then for any sequence $\{T_n\}_{n\in\mathbb{N}^+}$ of $T\to\infty$, there hold
  \[\lim\limits_{\left(d,\frac{\omega}{d}\right)\to(\infty,\vartheta)} \liminf\limits_{n\to\infty}\fint^{T_n}_{0}H(s;\omega,d)\,\mathrm{d}s= \liminf_{n\to\infty} \fint^{T_n}_{0}\fint_\Omega c(x,s)\,\mathrm{d}x\mathrm{d}s\]
and
  \[\lim\limits_{\left(d,\frac{\omega}{d}\right)\to(\infty,\vartheta)} \limsup\limits_{n\to\infty}\fint^{T_n}_{0}H(s;\omega,d)\,\mathrm{d}s= \limsup_{n\to\infty} \fint^{T_n}_{0}\fint_\Omega c(x,s)\,\mathrm{d}x\mathrm{d}s.\]
Consequently, if the sequence $\left\{\fint^{T_n}_{0} c(\cdot,s)\,\mathrm{d}s\right\}_{n\in\mathbb{N}^+}$ of functions converges to some function $\hat{c}(\cdot)$ in $C(\overline{\Omega})$, then
   \[\lim\limits_{\left(d,\frac{\omega}{d}\right)\to(\infty,\vartheta)} \liminf\limits_{n\to\infty}\fint^{T_n}_{0}H(s;\omega,d)\,\mathrm{d}s= \lim\limits_{\left(d,\frac{\omega}{d}\right)\to(\infty,\vartheta)} \limsup\limits_{n\to\infty}\fint^{T_n}_{0}H(s;\omega,d)\,\mathrm{d}s = \fint_\Omega \hat{c}(x)\,\mathrm{d}x.\]
\end{lemma}
\begin{proof}
First of all, if $\left\{\fint^{T_n}_{0} c(\cdot,s)\,\mathrm{d}s\right\}_{n\in\mathbb{N}^+}$ converges, by Lebesgue's dominated theorem, we know
\begin{equation*}
 \fint_\Omega \hat{c}(x)\,\mathrm{d}x=\fint_\Omega\lim_{n\to\infty}\fint^{T_n}_{0}c(x,s)\,\mathrm{d}s\mathrm{d}x= \lim_{n\to\infty}\fint^{T_n}_{0}\fint_\Omega c(x,s)\,\mathrm{d}x\mathrm{d}s.
\end{equation*}
The identities in the last line in the lemma follows directly from the first two identities.

It remains to show the first two identities. By Lemma \ref{lemma3.6}, one has
   \[\liminf\limits_{n\to\infty}\fint^{T_n}_{0} H(s;\omega,d)\,\mathrm{d}s \le\liminf_{n\to\infty}\fint^{T_n}_{0}\fint_{\Omega}c(x,s)\,\mathrm{d}x\mathrm{d}s\]
   and
   \[\limsup\limits_{n\to\infty}\fint^{T_n}_{0} H(s;\omega,d)\,\mathrm{d}s \le\limsup_{n\to\infty}\fint^{T_n}_{0}\fint_{\Omega}c(x,s)\,\mathrm{d}x\mathrm{d}s.\]
By letting $\left(d,\frac{\omega}{d}\right)\to(\infty,\vartheta)$, one easily obtains
\begin{equation}\label{eq6.1}
\begin{aligned}
     &\limsup\limits_{\left(d,\frac{\omega}{d}\right)\to(\infty,\vartheta)}\liminf\limits_{n\to\infty}\fint^{T_n}_{0} H(s;\omega,d)\,\mathrm{d}s \le\liminf_{n\to\infty}\fint^{T_n}_{0}\fint_{\Omega}c(x,s)\,\mathrm{d}x\mathrm{d}s\qquad
   \text{and} \\
     &\limsup\limits_{\left(d,\frac{\omega}{d}\right)\to(\infty,\vartheta)}\limsup\limits_{n\to\infty}\fint^{T_n}_{0} H(s;\omega,d)\,\mathrm{d}s \le\limsup_{n\to\infty}\fint^{T_n}_{0}\fint_{\Omega}c(x,s)\,\mathrm{d}x\mathrm{d}s.
\end{aligned}
\end{equation}

We now turn to derive the lower bounds by constructing suitable sub- and super-solutions. In view of the Harnack principle of $\varphi$ (Proposition \ref{proposition3.3}(ii)), we see that for the normalized principal Floquet bundle of \eqref{eq1.1}, there is a constant $C>0$ independent of $\omega$ and large $d$ such that
\begin{equation}\label{eq6.2}
|\ln \varphi(x,t)|\le C,\qquad (x,t)\in\Omega\times\mathbb{R}.
\end{equation}
We will prove the lower bound in three cases: (i) $\left(d,\frac{\omega}{d}\right)\to(\infty,0)$, (ii) $\left(d,\frac{\omega}{d}\right)\to(\infty,\vartheta)$ with $\vartheta\in(0,\infty)$, and (iii) $\left(d,\frac{\omega}{d}\right)\to(\infty,\infty)$. Although the limiting behavior is same in three cases, the sub- and super-solutions constructed in the three cases are different.

\medskip
\noindent\textbf{Case (i).} $\left(d,\frac{\omega}{d}\right)\to(\infty,0)$.

For each $t\in\mathbb{R}$, let $\Gamma_0$ be the unique solution of
\begin{equation}\label{eq6.3}
  \left\{
  \begin{aligned}
    &-\operatorname{div}(A(x)\nabla\Gamma_0)=\fint_\Omega c(x,t)\,\mathrm{d}x-c(x,t),&&x\in\Omega,\\
    &\mathbf{n}\cdot(A(x)\nabla\Gamma_0)=0,&&x\in\partial\Omega,\\
    &\int_\Omega{\Gamma_0}(x,t)\,\mathrm{d}x=0.
  \end{aligned}
  \right.
\end{equation}
The existence and uniqueness of the solution of \eqref{eq6.3} are standard results of Fredholm's alternative theorem applying to elliptic problems, and we omit the details. We will use $\Gamma_0$ to construct suitable sub- and super-solutions in forthcoming arguments.

By the global Schauder regularity theory for elliptic equations with Neumann boundary conditions, since $\sup_{t\in\mathbb{R}} \|c(\cdot,t)\|_{C^\delta(\overline\Omega)} < \infty$ for $c\in\mathcal{C}$, the right-hand term $\fint_\Omega c(x,t)\,\mathrm{d}x - c(x,t)$ is uniformly bounded in $C^\delta(\overline\Omega)$ with respect to $t\in\mathbb{R}$, which yields $\Gamma_0(\cdot, t) \in C^{2+\delta}(\overline\Omega)$ for each $t \in \mathbb{R}$ and the uniform boundedness of $\|\Gamma_0(\cdot,t)\|_{C^2(\overline\Omega)}$ over $t\in\mathbb{R}$.

We now justify the continuous differentiability of $\Gamma_0(\cdot,t)$ with respect to $t$. Let
$$
X = \left\{ u\in C^{2+\delta}(\overline\Omega) : \mathbf{n}\cdot(A(x)\nabla u)=0 \text{ on } \partial\Omega,\ \int_\Omega u\,\mathrm{d}x=0 \right\},
\quad
Y = \left\{ f\in C^{\delta}(\overline\Omega) : \int_\Omega f\,\mathrm{d}x=0 \right\}.
$$
The operator $L = -\operatorname{div}(A(x)\nabla\cdot) : X \to Y$ is a linear isomorphism by Fredholm's alternative. Denote the right-hand side of \eqref{eq6.3} by $f(x,t):=\fint_\Omega c(x,t)\,\mathrm{d}x - c(x,t)$. Since $c(x,t)$ is continuously differentiable in $t$, the map $t\mapsto f(\cdot,t)\in Y$ is continuously differentiable, and thus the solution $\Gamma_0(\cdot,t) = L^{-1} f(\cdot,t)$ is continuously differentiable in $t\in\mathbb{R}$.

Differentiating system \eqref{eq6.3} with respect to $t$, we find that $\partial_t \Gamma_0$ satisfies the same type of elliptic Neumann problem with zero spatial average:
\begin{equation*}
  \left\{
  \begin{aligned}
    &-\operatorname{div}(A(x)\nabla\partial_t\Gamma_0)=\fint_\Omega \partial_t c(x,t)\,\mathrm{d}x-\partial_t c(x,t),&&x\in\Omega,\\
    &\mathbf{n}\cdot(A(x)\nabla\partial_t\Gamma_0)=0,&&x\in\partial\Omega,\\
    &\int_\Omega\partial_t\Gamma_0(x,t)\,\mathrm{d}x=0.
  \end{aligned}
  \right.
\end{equation*}
By the assumption $\sup_{t\in\mathbb{R}} \|\partial_t c(\cdot,t)\|_{C(\overline\Omega)} < \infty$, its right-hand side is uniformly bounded in $L^\infty(\Omega)$. Taking any $p>N$, where $N$ denotes the spatial dimension, the global $W^{2,p}$ estimate for elliptic equations together with the Sobolev embedding theorem implies that $\|\partial_t \Gamma_0(\cdot,t)\|_{C(\overline\Omega)}$ is also uniformly bounded for all $t\in\mathbb{R}$. So there is a positive constant, denoted still by $C$, such that
 \begin{equation}\label{eq6.4}
   \sup_{t\in\mathbb{R}}\left(\|{\Gamma_0}(\cdot,t)\|_{C^{2}(\overline\Omega)}+\|\partial_t {\Gamma_0}(\cdot,t)\|_{C(\overline\Omega)}\right)<C.
 \end{equation}

Fix $(\omega,d)$ first. Then for each $n\in\mathbb{N}^+$, we define
 \[\begin{aligned}
      &\overline{\varphi}_n(x,t):=\\
      &\quad\exp\left[\frac{1}{\omega}\int^{t}_{0}\left(\fint^{T_{n}}_{0}\fint_\Omega c(x,r)\,\mathrm{d}x\mathrm{d}r-\fint_\Omega c(x,s)\,\mathrm{d}x-\varepsilon \right)\,\mathrm{d}s+\frac{{\Gamma_0}(x,t)}{d}\right],\quad (x,t)\in\overline{\Omega}\times[0,T_{n}]
   \end{aligned}
 \]
and
 \[\underline{\varphi}_n(x,t):=\varphi(x,t) \exp\left[\frac{1}{\omega}\int^{t}_{0}\left(\fint^{T_{n}}_{0}H(r;\omega,d)\,\mathrm{d}r- H(s;\omega,d)+\varepsilon \right)\,\mathrm{d}s\right],\quad (x,t)\in\overline{\Omega}\times[0,T_{n}].\]
We compute directly to obtain that for each $n\in\mathbb{N}^+$ and all $(x,t)\in\Omega\times(0,T_n]$, there hold
\begin{align*}
  &\omega\partial_t\overline{\varphi}_{n}-d\operatorname{div}(A(x)\nabla\overline{\varphi}_{n}) +c(x,t)\overline{\varphi}_{n} \\
  = & \left[\omega\partial_t\ln\overline{\varphi}_{n}-d\operatorname{div}(A(x)\nabla\ln \overline{\varphi}_{n}) -d\nabla \ln\overline{\varphi}_{n}\cdot(A(x)\nabla \ln\overline{\varphi}_{n})+c(x,t)\right]\overline{\varphi}_{n} \\
  = & \Bigg[\fint^{T_{n}}_{0}\fint_{\Omega}c(x,s)\,\mathrm{d}x\mathrm{d}s-\fint_{\Omega}c(x,t)\,\mathrm{d}x -\varepsilon +\frac{\omega}{d}\partial_t{\Gamma_0}\\
   &\hskip 3cm -\operatorname{div}(A(x)\nabla\Gamma_0) -\frac{1}{d}\nabla{\Gamma_0}\cdot(A(x)\nabla{\Gamma_0})+c(x,t)\Bigg] \overline{\varphi}_{n} \\
  \ge & \left(\fint^{T_{n}}_{0}\fint_{\Omega}c(x,s)\,\mathrm{d}x\mathrm{d}s -\frac{\omega}{d}C-\frac{\Lambda}{d}C^2-\varepsilon \right) \overline{\varphi}_{n}\hskip 2cm \text{(by \eqref{eq6.4})}
\end{align*}
and
\begin{align*}
  &\omega\partial_t\underline{\varphi}_{n}-d\operatorname{div}(A(x)\nabla\underline{\varphi}_{n}) +c(x,t)\underline{\varphi}_{n} \\
  = & \left[\omega\partial_t\ln\underline{\varphi}_{n}-d\operatorname{div}(A(x)\nabla \ln\underline{\varphi}_{n}) -d\nabla \ln\underline{\varphi}_{n}\cdot(A(x)\nabla \ln\underline{\varphi}_{n})+c(x,t)\right]\underline{\varphi}_{n} \\
  = & \Bigg[\omega\partial_t\ln\varphi+\fint^{T_{n}}_{0}H(s;\omega,d)\,\mathrm{d}s- H(t;\omega,d) +\varepsilon  \\
  &\hskip 3cm -d\operatorname{div}(A(x)\nabla\ln\varphi)-d\nabla \ln\varphi\cdot(A(x)\nabla \ln\varphi)+c(x,t)\Bigg]\underline{\varphi}_{n} \\
  = & \left(\fint^{T_{n}}_{0}H(s;\omega,d)\,\mathrm{d}s+ \varepsilon \right)\underline{\varphi}_{n}.
\end{align*}
There is an integer $N_\varepsilon>0$ such that
\[T_n>\frac{2C\omega}{d\varepsilon},\qquad \forall~n>N_\varepsilon.\]
Then for each $n>N_\varepsilon$, the pair $(\underline{\varphi}_n,\overline{\varphi}_n)$ satisfies
\begin{equation*}
  \left\{
  \begin{aligned}
    &\omega \partial_t\overline{\varphi}_n-d\operatorname{div}(A(x)\nabla\overline{\varphi}_n) +c(x,t)\overline{\varphi}_n \\
    &\qquad\qquad \ge \left(\fint^{T_{n}}_{0}\fint_{\Omega}c(x,s)\,\mathrm{d}x\mathrm{d}s -\frac{\omega}{d}C-\frac{1}{d}C^2-\varepsilon \right) \underline{\varphi}_{n} ,&&(x,t)\in\Omega\times(0,T_{n}],\\
    &\omega \partial_t\underline{\varphi}_n-d\operatorname{div}(A(x)\nabla\underline{\varphi}_n)+c(x,t)\underline{\varphi}_n =\left(\fint^{T_{n}}_{0}H(s;\omega,d)\,\mathrm{d}s+\varepsilon  \right)\underline{\varphi}_{n},&&(x,t)\in\Omega\times(0,T_{n}],\\
    &\mathbf{n}\cdot(A(x)\nabla\overline{\varphi}_n)\ge 0 =\mathbf{n}\cdot(A(x)\nabla\underline{\varphi}_n), &&(x,t)\in\partial\Omega\times(0,T_{n}],\\
    &\overline{\varphi}_n(x,0)=\mathrm{e}^{\frac{{\Gamma_0}(x,0)}{d}} \ge \mathrm{e}^{-\frac{\varepsilon T_{n}}{\omega}+\frac{{\Gamma_0}(x,T_{n})}{d}}=\overline{\varphi}_n(x,T_{n}) ,&&x\in\overline{\Omega},\\
    &\underline{\varphi}_n(x,0)=\varphi(x,0)\le \varphi(x,0)\mathrm{e}^{\frac{\varepsilon T_n}{\omega}-2C} \le \varphi(x,T_{n})\mathrm{e}^{\frac{\varepsilon T_{n}}{\omega}}=\underline{\varphi}_n(x,T_{n}),&&x\in\overline{\Omega}.
  \end{aligned}
  \right.
\end{equation*}
We have used \eqref{eq6.2} to obtain the second inequality in the last line. Using Lemma \ref{lemma2.13}, we have
 \[\liminf_{n\to\infty}\fint^{T_{n}}_{0}H(s;\omega,d)\,\mathrm{d}s+\varepsilon \ge \liminf_{n\to\infty}\fint^{T_{n}}_{0}\fint_{\Omega}c(x,s)\,\mathrm{d}x\mathrm{d}s -\frac{\omega}{d}C-\frac{1}{d}C^2-\varepsilon \]
and
 \[\limsup_{n\to\infty}\fint^{T_{n}}_{0}H(s;\omega,d)\,\mathrm{d}s+\varepsilon \ge \limsup_{n\to\infty}\fint^{T_{n}}_{0}\fint_{\Omega}c(x,s)\,\mathrm{d}x\mathrm{d}s -\frac{\omega}{d}C-\frac{1}{d}C^2-\varepsilon .\]
Then sending $\left(d,\frac{\omega}{d}\right)\to (\infty,0)$ combined with the arbitrariness of $\varepsilon>0$, we obtain
\begin{equation}\label{eq6.5}
 \begin{aligned}
   &\liminf\limits_{\left(d,\frac{\omega}{d}\right)\to(\infty,0)}\liminf_{n\to\infty}\fint^{T_{n}}_{0}H(s;\omega,d)\,\mathrm{d}s\ge \liminf_{n\to\infty}\fint^{T_{n}}_{0}\fint_{\Omega}c(x,s)\,\mathrm{d}x\mathrm{d}s\\ \text{and}\quad
   &\liminf\limits_{\left(d,\frac{\omega}{d}\right)\to(\infty,0)}\limsup_{n\to\infty}\fint^{T_{n}}_{0}H(s;\omega,d)\,\mathrm{d}s\ge \limsup_{n\to\infty}\fint^{T_{n}}_{0}\fint_{\Omega}c(x,s)\,\mathrm{d}x\mathrm{d}s.
 \end{aligned}
\end{equation}

\medskip
\noindent\textbf{Case (ii).} $\left(d,\frac{\omega}{d}\right)\to(\infty,\vartheta)$ with $\vartheta\in(0,\infty)$.

For this case, we will use the solution of the following problem to construct a super-solution,
\begin{equation}\label{eq6.6}
  \left\{
  \begin{aligned}
    &\vartheta\partial_t{\Gamma_\vartheta}-\operatorname{div}(A(x)\nabla \Gamma_\vartheta) =\fint_\Omega c(x,t)\,\mathrm{d}x-c(x,t),&&(x,t)\in\Omega\times\mathbb{R},\\
    &\mathbf{n}\cdot(A(x)\nabla\Gamma_\vartheta)=0,&&(x,t)\in\partial\Omega\times\mathbb{R},\\
    &\limsup_{t\to-\infty}\fint_{\Omega}{\Gamma_\vartheta}^2(x,t)\,\mathrm{d}x<\infty.
  \end{aligned}
  \right.
\end{equation}
Problem \eqref{eq6.6} admits a unique (up to an addictive constant) solution, denoted by ${\Gamma_\vartheta}$. Moreover, there is a positive constant, denoted still by $C>0$, such that (see \cite[Appendix A.]{Lam2024The})
 \[\sup_{t\in\mathbb{R}}(\|{\Gamma_\vartheta}(\cdot,t)\|_\infty+\|\partial_t{\Gamma_\vartheta}(\cdot,t)\|_\infty)<C.\]
We fix the pair $(\omega,d)$ first, fulfilling $d>2C$. For each $n\in\mathbb{N}^+$, we define
 \[\overline{\varphi}_n(x,t):=\left(1+\frac{1}{d}{\Gamma_\vartheta}(x,t)\right)\exp\left[ \frac{1}{\omega}\int^{t}_{0}\left(\fint^{T_{n}}_{0}\fint_\Omega c(x,r)\,\mathrm{d}x\mathrm{d}r -\fint_\Omega c(x,s)\,\mathrm{d}x-\varepsilon  \right)\,\mathrm{d}s\right]\]
and
 \[\underline{\varphi}_n(x,t):=\varphi(x,t)\exp\left[\frac{1}{\omega}\int^{t}_{0}\left( \fint^{T_{n}}_{0}H(r;\omega,d)\,\mathrm{d}r- H(s;\omega,d)+\varepsilon \right)\,\mathrm{d}s\right]\]
on $(x,t)\in\overline{\Omega}\times[0,T_{n}]$. Direct computations yield that
\begin{align*}
  &\omega\partial_t\overline{\varphi}_{n}-d\operatorname{div}(A(x)\nabla\overline{\varphi}_{n})+c(x,t)\overline{\varphi}_{n} \\
  = & \left[\frac{\omega}{d}\frac{\partial_t{\Gamma_\vartheta}}{1+{\Gamma_\vartheta}/d} +\fint^{T_{n}}_{0}\fint_\Omega c(x,s)\,\mathrm{d}x\mathrm{d}s-\fint_\Omega c(x,s)\,\mathrm{d}x-\varepsilon  -\frac{\operatorname{div}(A(x)\nabla\Gamma_\vartheta)}{1+{\Gamma_\vartheta}/d}+c(x,t)\right]\overline{\varphi}_{n}\\
  = & \Bigg[\frac{\omega}{d}\left(1-\frac{{\Gamma_\vartheta}}{d+{\Gamma_\vartheta}}\right)\partial_t{\Gamma_\vartheta}+\fint^{T_{n}}_{0}\fint_\Omega c(x,s)\,\mathrm{d}x\mathrm{d}s-\fint_\Omega c(x,s)\,\mathrm{d}x-\varepsilon \\
  &~~+\left(1-\frac{{\Gamma_\vartheta}}{d+{\Gamma_\vartheta}}\right)\left(\fint_\Omega c(x,s)\,\mathrm{d}x-c(x,t)-\vartheta\partial_t{\Gamma_\vartheta}\right)+c(x,t)\Bigg]\overline{\varphi}_{n}\\
  = & \Bigg[\fint^{T_{n}}_{0}\fint_\Omega c(x,s)\,\mathrm{d}x\mathrm{d}s+ \left(\frac{\omega}{d}-\vartheta\right)\partial_t{\Gamma_\vartheta}-\varepsilon \\
  &~~-\frac{{\Gamma_\vartheta}}{d+{\Gamma_\vartheta}}\left(\fint_\Omega c(x,s)\,\mathrm{d}x-c(x,t)+\left(\frac{\omega}{d}-\vartheta\right)\partial_t{\Gamma_\vartheta}\right)\Bigg]\overline{\varphi}_{n}\\
   \ge& \left(\fint^{T_{n}}_{0}\fint_\Omega c(x,s)\,\mathrm{d}x\mathrm{d}s-\left|\frac{\omega}{d}-\vartheta\right| C-\varepsilon -\frac{C}{d-C}\left(2\|c\|_\infty+\left|\frac{\omega}{d}-\vartheta\right| C\right)\right)\overline{\varphi}_{n}.
\end{align*}
There is an integer $N_\varepsilon>0$ such that
\[T_n>\max\left\{\frac{2C\omega}{\varepsilon},\frac{\omega\ln 3}{\varepsilon}\right\},\qquad \forall~n>N_\varepsilon.\]
Then for each $n>N_\varepsilon$, the pair $(\underline{\varphi}_n,\overline{\varphi}_n)$ satisfies
\begin{equation*}
  \left\{
  \begin{aligned}
    &\omega \partial_t\overline{\varphi}_n-d\operatorname{div}(A(x)\nabla\overline{\varphi}_n) +c(x,t)\overline{\varphi}_n
    \ge \Bigg[\fint^{T_{n}}_{0}\fint_\Omega c(x,s)\,\mathrm{d}x\mathrm{d}s \\ &\qquad-\left|\frac{\omega}{d}-\vartheta\right| C-\varepsilon -\frac{C}{d-C}\left(2\|c\|_\infty+\left|\frac{\omega}{d}-\vartheta\right| C\right)\Bigg] \underline{\varphi}_{n} ,&&(x,t)\in\Omega\times(0,T_{n}],\\
    &\omega \partial_t\underline{\varphi}_n-d\operatorname{div}(A(x)\nabla\underline{\varphi}_n) +c(x,t)\underline{\varphi}_n = \left(\fint^{T_{n}}_{0}H(s;\omega,d)\,\mathrm{d}s +\varepsilon  \right)\underline{\varphi}_{n},&&(x,t)\in\Omega\times(0,T_{n}],\\
    &\mathbf{n}\cdot(A(x)\nabla\overline{\varphi}_n)= 0 =\mathbf{n}\cdot(A(x)\nabla \underline{\varphi}_n), &&(x,t)\in\partial\Omega\times(0,T_{n}],\\
    &\overline{\varphi}_n(x,0)=1+\frac{{\Gamma_\vartheta}(x,0)}{d} \ge \left(1+\frac{{\Gamma_\vartheta}(x,T_{n})}{d}\right)\mathrm{e}^{-\frac{\varepsilon T_{n}}{\omega}}=\overline{\varphi}_n(x,T_{n}),&&x\in\overline{\Omega},\\
    &\underline{\varphi}_n(x,0)=\varphi(x,0)\le\underbrace{\varphi(x,0)\mathrm{e}^{\frac{\varepsilon T_{n}}{\omega}-2C}\le \varphi(x,T_{n})\mathrm{e}^{\frac{\varepsilon T_{n}}{\omega}}}_{\text{by \eqref{eq6.2}}}=\underline{\varphi}_n(x,T_{n}),&&x\in\overline{\Omega}.
  \end{aligned}
  \right.
\end{equation*}
Using Lemma \ref{lemma2.13}, then sending $\left(d,\frac{\omega}{d}\right)\to(\infty, \vartheta)$ and combining with the arbitrariness of $\varepsilon>0$, we have that
\begin{equation}\label{eq6.7}
 \begin{aligned}
   &\liminf\limits_{\left(d,\frac{\omega}{d}\right)\to(\infty,\vartheta)}\liminf_{n\to\infty}\fint^{T_{n}}_{0}H(s;\omega,d)\,\mathrm{d}s\ge \liminf_{n\to\infty}\fint^{T_{n}}_{0}\fint_{\Omega}c(x,s)\,\mathrm{d}x\mathrm{d}s \\
   \text{and}\quad
   &\liminf\limits_{\left(d,\frac{\omega}{d}\right)\to(\infty,\vartheta)}\limsup_{n\to\infty}\fint^{T_{n}}_{0}H(s;\omega,d)\,\mathrm{d}s\ge \limsup_{n\to\infty}\fint^{T_{n}}_{0}\fint_{\Omega}c(x,s)\,\mathrm{d}x\mathrm{d}s.
 \end{aligned}
\end{equation}

\medskip
\noindent\textbf{Case (iii).} $\left(d,\frac{\omega}{d}\right)\to(\infty,\infty)$.

For any fixed $T>0$, denote by ${\Gamma_{\infty,T}} $ the unique solution of
\begin{equation}\label{eq6.8}
  \left\{
  \begin{aligned}
    &-\operatorname{div}(A(x)\nabla \Gamma_{\infty,T}) +\fint^{T}_{0}c(x,s)\,\mathrm{d}s=\fint^{T}_{0}\int_\Omega c(x,s)\,\mathrm{d}x\mathrm{d}s,&&x\in\Omega,\\
    &\mathbf{n}\cdot(A(x)\nabla\Gamma_{\infty,T}) =0,&&x\in\partial\Omega,\\
    &\int_\Omega{\Gamma_{\infty,T}} (x)\,\mathrm{d}x=0.
  \end{aligned}
  \right.
\end{equation}
The standard regularity theory for elliptic equation implies that there is a constant, denoted still by $C>0$, such that
 \[\|{\Gamma_{\infty,T}} \|_{C^{2}(\overline{\Omega})}<C.\]
We remark that the aforementioned estimate is uniform with respect to $T>0$. The reason is that, owing to $c \in \mathcal{C}$, the nonlinear term in the equation possesses a uniformly bounded $C^{\delta}(\overline{\Omega})$ norm. Applying the Schauder's estimates then gives a uniform bound on $\|\Gamma_{\infty,T}\|_{C^{2,\delta}(\overline{\Omega})}$, whence the uniformity follows. On the other hand, assumption \eqref{H2} implies that there is a positive constant, denoted still by $C$, such that
\begin{equation}\label{eq6.9}
\sup_{T>0}\sup_{(x,t)\in\Omega\times[0,T]}(|\operatorname{div}(A(x)\nabla\mathcal{M}(x,t;T))|+|\nabla \mathcal{M}(x,t;T)\cdot (A(x)\nabla \mathcal{M}(x,t;T))|)<C.
\end{equation}
By Lemma \ref{lemma2.16}, there exist constants $\kappa,T_0>0$ such that for any $T>T_0$,
\[-\kappa\mathbf{n}\cdot(A(x)\nabla\phi^*)\ge \sup\limits_{(x,t)\in\partial\Omega\times[0,T]} |\mathbf{n}\cdot(A(x)\nabla\mathcal{M}(x,t;T))|, \quad\forall~x\in\partial\Omega,~t\in[0,T],\]
where $\phi^*$ is the principal eigenfunction of \eqref{eq2.6}.

For each $n\in\mathbb{N}^+$, we define
 \[\overline{\varphi}_n(x,t):=\left(1+\frac{1}{d}{\Gamma_{\infty,n}} (x)\right)\exp\left[\frac{1}{\omega} \left(t\fint^{T_{n}}_{0} c(x,s)\,\mathrm{d}s-\int^{t}_{0} c(x,s)\,\mathrm{d}s-\kappa\phi^*(x)\right)\right],\]
where $\Gamma_{\infty,n}=\Gamma_{\infty,T_n}$ is the unique solution of \eqref{eq6.8} for each $n\in\mathbb{N}^+$, and
 \[\underline{\varphi}_n(x,t):=\varphi(x,t) \exp\left[\frac{1}{\omega}\int^{t}_{0}\left(\fint^{T_{n}}_{0}H(r;\omega,d)\,\mathrm{d}r- H(s;\omega,d)+\varepsilon \right)\,\mathrm{d}s\right]\]
on $(x,t)\in\overline{\Omega}\times[0,T_{n}]$. By \eqref{eq6.8}, \eqref{eq6.9} and some direct computations, one checks that for all $d>2\|{\Gamma_{\infty,n}}\|_{C^2(\overline{\Omega})}$, it holds
\begin{align*}
  &\omega\partial_t\overline{\varphi}_{n}-d\operatorname{div}(A(x)\nabla\overline{\varphi}_{n})+c(x,t)\overline{\varphi}_{n} \\
  = & \left[\omega\partial_t\ln\overline{\varphi}_{n}- d\operatorname{div}(A(x)\nabla\ln\overline{\varphi}_{n}) -d\nabla \ln\overline{\varphi}_{n}\cdot(A(x)\nabla \ln\overline{\varphi}_{n})+c(x,t)\right]\overline{\varphi}_{n} \\
  = &\Bigg[  \fint^{T_{n}}_{0}c(x,s)\,\mathrm{d}s-c(x,t) \\
    &~~-d\operatorname{div}\left(A(x)\nabla\left( \ln\left(1+\frac{{\Gamma_{\infty,n}} }{d}\right)+\frac{1}{\omega} \left(\mathcal{M}(x,t;T_{n})-\kappa\phi^*\right)\right) \right)\\
    &~~-d\nabla\left( \ln\left(1+\frac{{\Gamma_{\infty,n}} }{d}\right)+ \frac{1}{\omega}\left(\mathcal{M}(x,t;T_{n})-\kappa\phi^*\right)\right)\\
    &~~~~~~\cdot\left(A(x)\nabla\left( \ln\left(1+\frac{{\Gamma_{\infty,n}} }{d}\right)+ \frac{1}{\omega}\left(\mathcal{M}(x,t;T_{n})-\kappa\phi^*\right)\right)\right)+c(x,t)\Bigg]\overline{\varphi}_{n}\\
  = &\Bigg[  \fint^{T_{n}}_{0}c(x,s)\,\mathrm{d}s-d\operatorname{div}\left(A(x)\nabla \ln\left(1+\frac{{\Gamma_{\infty,n}} }{d}\right)\right)\\
  &~~-\frac{d}{\omega}\operatorname{div}\left(A(x)\nabla\left(\mathcal{M}(x,t;T_{n})-\kappa\phi^*\right)\right) -d\nabla \ln\left(1+\frac{{\Gamma_{\infty,n}} }{d}\right)\cdot\left(A(x)\nabla \ln\left(1+\frac{{\Gamma_{\infty,n}} }{d}\right)\right) \\
    &~~-\frac{2d}{\omega}\nabla\ln\left(1+\frac{{\Gamma_{\infty,n}} }{d}\right)\cdot \left(A(x)\nabla\left(\mathcal{M}(x,t;T_{n})-\kappa\phi^*\right)\right)\\
    &~~-\frac{d}{\omega^{2}} \nabla\left(\mathcal{M}(x,t;T_{n})-\kappa\phi^*\right)\cdot\left(A(x) \nabla\left(\mathcal{M}(x,t;T_{n})-\kappa\phi^*\right)\right)\Bigg] \overline{\varphi}_{n}\\
  = &\Bigg[\fint^{T_{n}}_{0}c(x,s)\,\mathrm{d}s +\frac{d}{d+\Gamma_{\infty,n}} \left(\fint^{T_n}_{0}\fint_\Omega c(x,s)\,\mathrm{d}x\mathrm{d}s-\fint^{T_n}_{0} c(x,s)\,\mathrm{d}s\right) \\ &~~-\frac{d}{\omega}\operatorname{div}\left(A(x)\nabla\left(\mathcal{M}(x,t;T_{n})-\kappa\phi^*\right)\right)-\frac{2d}{\omega} \frac{1}{d+{\Gamma_{\infty,n}} }\nabla{\Gamma_{\infty,n}}\cdot\left(A(x) \nabla\left(\mathcal{M}(x,t;T_{n})-\kappa\phi^*\right)\right) \\
  &~~ -\frac{d}{\omega^{2}}\nabla\left(\mathcal{M}(x,t;T_{n})-\kappa\phi^*\right)\cdot\left(A(x)\nabla\left(\mathcal{M}(x,t;T_{n})-\kappa\phi^*\right)\right) \Bigg] \overline{\varphi}_{n}\\
  = &\Bigg[\fint^{T_{n}}_{0}\fint_\Omega c(x,s)\,\mathrm{d}x\mathrm{d}s -\frac{\Gamma_{\infty,n}}{d+\Gamma_{\infty,n}} \left(\fint^{T_n}_{0}\fint_\Omega c(x,s)\,\mathrm{d}x\mathrm{d}s-\fint^{T_n}_{0} c(x,s)\,\mathrm{d}s\right) \\ &~~-\frac{d}{\omega}\operatorname{div}\left(A(x)\nabla\left(\mathcal{M}(x,t;T_{n})-\kappa\phi^*\right)\right)-\frac{2d}{\omega} \frac{1}{d+{\Gamma_{\infty,n}} }\nabla{\Gamma_{\infty,n}}\cdot\left(A(x) \nabla\left(\mathcal{M}(x,t;T_{n})-\kappa\phi^*\right)\right) \\
  &~~ -\frac{d}{\omega^{2}}\nabla\left(\mathcal{M}(x,t;T_{n})-\kappa\phi^*\right)\cdot\left(A(x)\nabla\left(\mathcal{M}(x,t;T_{n})-\kappa\phi^*\right)\right) \Bigg] \overline{\varphi}_{n}\\
\ge&\Bigg[\fint^{T_{n}}_{0}\fint_\Omega c(x,s)\,\mathrm{d}x\mathrm{d}s -\frac{2\|c\|_\infty\|{\Gamma_{\infty,n}} \|^{2}_{C^2(\overline{\Omega})}}{d-\|{\Gamma_{\infty,n}} \|_{C^2(\overline{\Omega})}} -\frac{d}{\omega}\left(C+\kappa\|\phi^*\|_{C^2(\overline{\Omega})}\right)\\
    &~~-\frac{2d}{\omega} \frac{\|{\Gamma_{\infty,n}} \|_{C^2(\overline{\Omega})}
    \left(C+\kappa\|\phi^*\|_{C^2(\overline{\Omega})}\right) }{d-\|{\Gamma_{\infty,n}} \|_{C^2(\overline{\Omega})}} -\frac{d}{\omega^{2}}\left(C+\kappa\|\phi^*\|_{C^2(\overline{\Omega})}\right)^2 \Bigg] \overline{\varphi}_{n}
\end{align*}
in $\Omega\times(0,T_n]$ and
 \[\mathbf{n}\cdot(A(x)\nabla\overline{\varphi}_n)=\frac{1}{\omega} \left(1+\frac{1}{d}\Gamma_{\infty,n}\right)[\mathbf{n}\cdot(A(x)\nabla\mathcal{M}(x,t;T_{n}))- \kappa\mathbf{n}\cdot(A(x)\nabla\phi^*)]\overline{\varphi}_n\ge 0,\quad \forall~d>C.\]
We can find an integer $N_\varepsilon>0$ such that
\[T_n>\max\left\{\frac{2C\omega}{\varepsilon},\frac{\omega\ln 3}{\varepsilon}\right\},\qquad \forall~n>N_\varepsilon.\]
Then for each $n>N_\varepsilon$, the pair $(\underline{\varphi}_n,\overline{\varphi}_n)$ satisfies
\begin{equation*}
  \left\{
  \begin{aligned}
    &\omega \partial_t\overline{\varphi}_n-d\operatorname{div}(A(x)\nabla\overline{\varphi}_n)+c(x,t)\overline{\varphi}_n \\
    &~~\ge \Bigg[\fint^{T_{n}}_{0}\fint_\Omega c(x,s)\,\mathrm{d}x\mathrm{d}s -\frac{2\|c\|_\infty\|{\Gamma_{\infty,n}} \|^{2}_{C^2(\overline{\Omega})}}{d-\|{\Gamma_{\infty,n}} \|_{C^2(\overline{\Omega})}} \\
    &\qquad~-\frac{d}{\omega}\left(C+\kappa\|\phi^*\|_{C^2(\overline{\Omega})}\right)-\frac{2d}{\omega} \frac{\|{\Gamma_{\infty,n}} \|_{C^2(\overline{\Omega})}
    \left(C+\kappa\|\phi^*\|_{C^2(\overline{\Omega})}\right) }{d-\|{\Gamma_{\infty,n}} \|_{C^2(\overline{\Omega})}}\\
    &\qquad~ -\frac{d}{\omega^{2}}\left(C+\kappa\|\phi^*\|_{C^2(\overline{\Omega})}\right)^2\Bigg] \overline{\varphi}_{n} ,&&(x,t)\in\Omega\times(0,T_{n}],\\
    &\omega \partial_t\underline{\varphi}_n-d\operatorname{div}(A(x)\nabla\underline{\varphi}_n) +c(x,t)\underline{\varphi}_n = \left(\fint^{T_{n}}_{0}H(s;\omega,d)\,\mathrm{d}s+\varepsilon  \right)\underline{\varphi}_{n},&&(x,t)\in\Omega\times(0,T_{n}],\\
    &\mathbf{n}\cdot(A(x)\nabla\overline{\varphi}_n)\ge0=\mathbf{n}\cdot (A(x)\nabla\underline{\varphi}_n), &&(x,t)\in\partial\Omega\times(0,T_{n}],\\
    &\overline{\varphi}_n(x,0)= \left(1+\frac{{\Gamma_{\infty,n}} (x)}{d}\right)\mathrm{e}^{-\frac{\kappa\phi^*}{\omega}}=\overline{\varphi}_n(x,T_{n}),&&x\in\overline{\Omega},\\
    &\underline{\varphi}_n(x,0)=\varphi(x,0) \le \underbrace{\varphi(x,0)\mathrm{e}^{\frac{\varepsilon T_{n}}{\omega}-2C}\le \varphi(x,T_{n})\mathrm{e}^{\frac{\varepsilon T_{n}}{\omega}}}_{\text{by \eqref{eq6.2}}}=\underline{\varphi}_n(x,T_{n}),&&x\in\overline{\Omega}.
  \end{aligned}
  \right.
\end{equation*}
Using Lemma \ref{lemma2.13}, then sending $\left(d,\frac{\omega}{d}\right)\to(\infty, \infty)$ combined with the arbitrariness of $\varepsilon>0$, we obtain that
\begin{equation}\label{eq6.10}
  \begin{aligned}
   &\liminf\limits_{\left(d,\frac{\omega}{d}\right)\to(\infty,\infty)}\liminf_{n\to\infty}\fint^{T_{n}}_{0}H(s;\omega,d)\,\mathrm{d}s\ge \liminf_{n\to\infty}\fint^{T_{n}}_{0}\fint_{\Omega}c(x,s)\,\mathrm{d}x\mathrm{d}s \\
   \text{and}\quad
   &\liminf\limits_{\left(d,\frac{\omega}{d}\right)\to(\infty,\infty)}\limsup_{n\to\infty}\fint^{T_{n}}_{0}H(s;\omega,d)\,\mathrm{d}s\ge \limsup_{n\to\infty}\fint^{T_{n}}_{0}\fint_{\Omega}c(x,s)\,\mathrm{d}x\mathrm{d}s.
 \end{aligned}
\end{equation}

Combining \eqref{eq6.1}, \eqref{eq6.5}, \eqref{eq6.7} and \eqref{eq6.10}, we complete the proof.
\end{proof}

\begin{lemma}\label{lemma6.3}
Let $\vartheta\in[0,\infty]$ be fixed and let $c\in\mathcal{C}\cap C^{2,1}(\overline{\Omega}\times\mathbb{R})$ be a given function satisfying $\sup_{t\in\mathbb{R}}\|\partial_t c(\cdot,t)\|_{\infty}<\infty$. Denote by $(H,\varphi)$ the normalized principal Floquet bundle of \eqref{eq1.1} with potential $c$. Suppose further that assumption \eqref{H2} holds if $\vartheta=\infty$, then
   \[ \lim\limits_{\left(d,\frac{\omega}{d}\right)\to(\infty,\vartheta)} \liminf\limits_{T\to\infty}\fint^{T}_{0}H(s;\omega,d)\,\mathrm{d}s = \liminf_{T\to\infty}\fint^{T}_{0}\fint_\Omega c(x,s)\,\mathrm{d}x\mathrm{d}s\]
 and
   \[ \lim\limits_{\left(d,\frac{\omega}{d}\right)\to(\infty,\vartheta)} \limsup\limits_{T\to\infty}\fint^{T}_{0}H(s;\omega,d)\,\mathrm{d}s = \limsup_{T\to\infty}\fint^{T}_{0}\fint_\Omega c(x,s)\,\mathrm{d}x\mathrm{d}s.\]
\end{lemma}
\begin{proof}
  An application of Lemma \ref{lemma3.6} yields
   \[ \limsup\limits_{\left(d,\frac{\omega}{d}\right)\to(\infty,\vartheta)} \liminf\limits_{T\to\infty}\fint^{T}_{0}H(s;\omega,d)\,\mathrm{d}s \le \liminf_{T\to\infty}\fint^{T}_{0}\fint_\Omega c(x,s)\,\mathrm{d}x\mathrm{d}s\]
 and
   \[ \limsup\limits_{\left(d,\frac{\omega}{d}\right)\to(\infty,\vartheta)} \limsup\limits_{T\to\infty}\fint^{T}_{0}H(s;\omega,d)\,\mathrm{d}s \le \limsup_{T\to\infty}\fint^{T}_{0}\fint_\Omega c(x,s)\,\mathrm{d}x\mathrm{d}s.\]
By choosing suitable sequence $\{T_n\}_{n\in\mathbb{N}^+}$ of $T\to\infty$ such that
 \[\lim_{n\to\infty}\fint^{T_n}_{0}\fint_\Omega c(x,s)\,\mathrm{d}x\mathrm{d}s=\limsup_{T\to\infty}\fint^{T}_{0}\fint_\Omega c(x,s)\,\mathrm{d}x\mathrm{d}s\]
and then using Lemma \ref{lemma6.2}, we further have
   \[ \liminf\limits_{\left(d,\frac{\omega}{d}\right)\to(\infty,\vartheta)} \limsup\limits_{T\to\infty}\fint^{T}_{0}H(s;\omega,d)\,\mathrm{d}s \ge \limsup_{T\to\infty}\fint^{T}_{0}\fint_\Omega c(x,s)\,\mathrm{d}x\mathrm{d}s.\]
So it needs only to show
\begin{equation}\label{eq6.11}
  \liminf\limits_{\left(d,\frac{\omega}{d}\right)\to(\infty,\vartheta)} \liminf\limits_{T\to\infty}\fint^{T}_{0}H(s;\omega,d)\,\mathrm{d}s \ge \liminf_{T\to\infty}\fint^{T}_{0}\fint_\Omega c(x,s)\,\mathrm{d}x\mathrm{d}s.
\end{equation}

For any given $T>0$, we define following axillary functions on $\overline{\Omega}\times[0,T]$,
 \[\begin{aligned}
   \overline{\varphi}_T&(x,t):=\\
   &\left\{
 \begin{aligned}
   &\exp\left[\frac{1}{\omega}\int^{t}_{0}\left(\fint^{T}_{0}\fint_\Omega c(x,r)\,\mathrm{d}x\mathrm{d}r-\fint_\Omega c(x,s)\,\mathrm{d}x-\varepsilon \right)\,\mathrm{d}s+\frac{{\Gamma_0}(x,t)}{d}\right],&&\text{ if }\frac{\omega}{d}\to0,\\
 &\left(1+\frac{1}{d}{\Gamma_\vartheta}(x,t)\right)\exp\left[\frac{1}{\omega}\int^{t}_{0}\left(t\fint^{T}_{0}\fint_\Omega c(x,r)\,\mathrm{d}x\mathrm{d}r-\fint_\Omega c(x,s)\,\mathrm{d}x-\varepsilon \right)\,\mathrm{d}s\right],&&\text{ if }\frac{\omega}{d}\to\vartheta,\\
 &\left(1+\frac{1}{d}{\Gamma_{\infty,T}} (x)\right)\exp\left[\frac{1}{\omega} \left(t\fint^{T}_{0} c(x,s)\,\mathrm{d}s-\int^{t}_{0} c(x,s)\,\mathrm{d}s-\kappa\phi^*(x)\right)\right],&&\text{ if }\frac{\omega}{d}\to\infty
 \end{aligned}
 \right.
 \end{aligned}
\]
and
 \[\underline{\varphi}_T(x,t):=\varphi(x,t) \exp\left[\frac{1}{\omega}\int^{t}_{0}\left(\fint^{T}_{0}H(r;\omega,d)\,\mathrm{d}r- H(s;\omega,d)+\varepsilon \right)\,\mathrm{d}s\right],\]
where $\varphi$ is the normalized principal Floquet bundle of \eqref{eq1.1}, ${\Gamma_0}$, $\Gamma_\vartheta$ and $\Gamma_{\infty,T}$ are the solutions of \eqref{eq6.3}, \eqref{eq6.6} and \eqref{eq6.8}, respectively. Proceeding along the lines as in the proof of Lemma \ref{lemma6.2}, we have \eqref{eq6.11}. This completes the proof.
\end{proof}

\section{The case $(\omega,d)\to(\infty,0)$}
We now examine the limit $(\omega, d) \to (\infty, 0)$. The method for constructing sub- and super-solutions, and the results thereby obtained, are quite similar to the situation where only $d \to 0$; see Chapter \ref{chp4}.
\begin{lemma}\label{lemma6.4}
Let $c\in \mathcal{C}\cap C^{2,1}(\overline{\Omega}\times\mathbb{R})$ be a function satisfying \eqref{H2}, $H(\cdot;\omega,d)$ be the normalized principal Floquet bundle of \eqref{eq1.1} with potential $c$, $\{T_n\}_{n\in\mathbb{N}^+}$ be a sequence satisfying $T_n\to\infty$ as $n\to\infty$. One has
  \[\lim_{(\omega,d)\to(\infty,0)}\liminf_{n\to\infty}\fint^{T_n}_{0}H(s;\omega,d)\,\mathrm{d}s= \min_{x\in\overline{\Omega}}\liminf_{n\to\infty}\fint^{T_n}_{0}c(x,s)\,\mathrm{d}s=\inf_{\hat{c}\in\mathscr{C}_*} \min\limits_{x\in\overline{\Omega}}\hat{c}(x)\]
and
  \[\lim_{(\omega,d)\to(\infty,0)}\limsup_{n\to\infty}\fint^{T_n}_{0}H(s;\omega,d)\,\mathrm{d}s= \limsup_{n\to\infty}\min_{x\in\overline{\Omega}}\fint^{T_n}_{0}c(x,s)\,\mathrm{d}s=\sup_{\hat{c}\in\mathscr{C}_*} \min\limits_{x\in\overline{\Omega}}\hat{c}(x).\]
Consequently, if the sequence $\left\{\fint^{T_n}_{0}c(\cdot,s)\,\mathrm{d}s\right\}_{n\in\mathbb{N}^+}$ of functions converges to $\hat{c}(\cdot)$ in $C(\overline{\Omega})$, then
  \[\lim_{(\omega,d)\to(\infty,0)}\liminf_{n\to\infty}\fint^{T_n}_{0}H(s;\omega,d)\,\mathrm{d}s= \lim_{(\omega,d)\to(\infty,0)}\limsup_{n\to\infty}\fint^{T_n}_{0}H(s;\omega,d)\,\mathrm{d}s= \min_{x\in\overline{\Omega}}\hat{c}(x).\]
\end{lemma}
\begin{proof}
Using Lemma \ref{lemma3.12}, one has
 \[\liminf\limits_{n\to\infty}\fint^{T_n}_{0} H(s;\omega,d)\,\mathrm{d}s\le\inf\limits_{\hat{c}\in \mathscr{C}_*}\mu^\infty (d,\hat{c})\quad\text{and}\quad\limsup\limits_{n\to\infty}\fint^{T_n}_{0} H(s;\omega,d)\,\mathrm{d}s\le\sup\limits_{\hat{c}\in \mathscr{C}_*}\mu^\infty (d,\hat{c}),\]
where $\mu^\infty(d,\hat{c})$ is the principal eigenvalue of \eqref{eq1.3} with potential $\hat{c}$. By letting $(\omega,d)\to(\infty, 0)$ and applying \eqref{eq5.15} in Lemma \ref{lemma5.10}, we get
 \[\limsup_{(\omega,d)\to(\infty,0)}\liminf_{n\to\infty}\fint^{T_n}_{0}H(s;\omega,d)\,\mathrm{d}s\le\lim_{d\to 0}\inf_{\hat{c}\in\mathscr{C}_*}\mu^\infty(d,\hat{c})=\min_{x\in\overline{\Omega}}\liminf_{n\to\infty}\fint^{T_n}_{0}c(x,s)\,\mathrm{d}s\]
and
 \[\limsup_{(\omega,d)\to(\infty,0)}\limsup_{n\to\infty}\fint^{T_n}_{0}H(s;\omega,d)\,\mathrm{d}s\le\lim_{d\to 0}\sup_{\hat{c}\in\mathscr{C}_*}\mu^\infty(d,\hat{c})=\limsup_{n\to\infty}\min_{x\in\overline{\Omega}}\fint^{T_n}_{0}c(x,s)\,\mathrm{d}s.\]

We turn to derive the lower bounds. Recall \[\mathcal{M}(x,t;T)=t\fint^{T}_{0}c(x,s)\,\mathrm{d}s-\int^{t}_{0}c(x,s)\,\mathrm{d}s\]
and assumption \eqref{H2}
\[\limsup\limits_{T\to\infty}\sup\limits_{(x,t)\in\Omega\times[0,T]}(|\operatorname{div}(A(x)\nabla\mathcal{M}(x,t;T))|+|\nabla \mathcal{M}(x,t;T)\cdot (A(x)\nabla \mathcal{M}(x,t;T))|)<\infty.\]
There exists a positive constant $C$ such that
\[\sup_{T>0}~\sup_{(x,t)\in\Omega\times[0,T]}(|\operatorname{div}(A(x)\nabla\mathcal{M}(x,t;T))|_{\infty}+|\nabla \mathcal{M}(x,t;T)\cdot (A(x)\nabla \mathcal{M}(x,t;T))|)<C.\]
Let $\phi^*$ be the principal eigenfunction of \eqref{eq2.6}. By Lemma \ref{lemma2.16}, we can find a constant $\kappa>0$ such that
\[-\kappa\mathbf{n}\cdot(A(x)\nabla\phi^*)\ge \sup_{T>0}~\sup\limits_{(x,t)\in\partial\Omega\times[0,T]} |\mathbf{n}\cdot(A(x)\nabla\mathcal{M}(x,t;T))|.\]

Let $\varepsilon>0$ be given. For each $n\in\mathbb{N}^+$, define
\[\overline{\varphi}_n(x,t):=\exp\left[\frac{1}{\omega}\left(\mathcal{M}(x,t;T_{n})- \kappa\phi^*\right)\right],\quad(x,t) \in\overline{\Omega}\times[0,T_{n}]\]
and
\[\underline{\varphi}_n(x,t):=\varphi(x,t)\exp\left[\frac{1}{\omega}\int^{t}_{0} \left(\fint^{T_{n}}_{0} H(r;\omega,d)\,\mathrm{d}r -H(s;\omega,d) +\varepsilon \right)\,\mathrm{d}s\right],\quad(x,t) \in\overline{\Omega}\times[0,T_{n}].\]
By direct computations, one can check that for each $n\in\mathbb{N}^+$, $(\underline{\varphi}_{n},\overline{\varphi}_{n})$ satisfies
\begin{align*}
   & \omega\partial_t\overline{\varphi}_{n}-d\operatorname{div}(A(x)\nabla\overline{\varphi}_{n})+c(x,t)\overline{\varphi}_{n} \\
 = & \left(\omega\partial_t\ln\overline{\varphi}_{n}-d\operatorname{div}(A(x)\nabla\ln\overline{\varphi}_{n})-d\nabla \ln\overline{\varphi}_{n}\cdot(A(x)\nabla \ln\overline{\varphi}_{n}) +c(x,t)\right)\overline{\varphi}_{n} \\
 = & \Bigg[\fint^{T_n}_{0}c(x,s)\,\mathrm{d}s- \frac{d}{\omega}\operatorname{div}\left(A(x)\nabla \left(\mathcal{M}(x,t;T_{n})-\kappa\phi^*\right)\right) \\
 &~-\frac{d}{\omega^2}\nabla \left(\mathcal{M}(x,t;T_{n})-\kappa\phi^*\right)\cdot\left(A(x)\nabla \left(\mathcal{M}(x,t;T_{n})-\kappa\phi^*\right)\right) \Bigg]\overline{\varphi}_{n}  \\
 \ge & \Bigg[\min_{x\in\overline{\Omega}}\fint^{T_n}_{0}c(x,s)\,\mathrm{d}s- \frac{d}{\omega} \left(C+\kappa\|\phi^*\|_{C^2(\overline{\Omega})} \right)-\frac{d}{\omega^2}\left(C+\kappa\|\phi^*\|_{C^2(\overline{\Omega})} \right)^2 \Bigg]\overline{\varphi}_{n}
\end{align*}
and
\begin{align*}
   & \omega\partial_t\underline{\varphi}_{n}-d\operatorname{div}(A(x)\nabla\underline{\varphi}_{n}) +c(x,t)\underline{\varphi}_{n}\\
 = & \left(\omega\partial_t\ln\underline{\varphi}_{n}-d\operatorname{div}(A(x)\nabla\ln\underline{\varphi}_{n}) -d(\nabla \ln\underline{\varphi}_{n}\cdot(A(x)\nabla \ln\underline{\varphi}_{n})+c(x,t)\right)\underline{\varphi}_{n} \\
 = & \left(\fint^{T_{n}}_{0}H(s;\omega,d)\,\mathrm{d}s+\varepsilon  \right)\underline{\varphi}_{n}
\end{align*}
in $\Omega\times(0,T_n]$. By virtue of the Harnack principle for the normalized principal Floquet bundle $\varphi$, we know that for fixed $(\omega,d)$, there is a constant $C_d>0$ such that
\begin{equation}\label{eq6.12}
|\ln\varphi(x,t)|\le C_d\qquad\forall~(x,t)\in\overline{\Omega}\times\mathbb{R}.
\end{equation}
There is an integer $N_\varepsilon>0$ such that
\[T_n>\frac{2C_d\omega}{\varepsilon},\qquad \forall~n>N_\varepsilon.\]
Then one can check that for each $n>N_\varepsilon$, the pair $(\underline{\varphi}_n,\overline{\varphi}_n)$ satisfies
\begin{equation*}
  \left\{
  \begin{aligned}
    &\omega \partial_t\overline{\varphi}_n-d\operatorname{div}(A(x)\nabla\overline{\varphi}_n)+c(x,t)\overline{\varphi}_n \ge\Bigg[\min_{x\in\overline{\Omega}}\fint^{T_n}_{0}c(x,s)\,\mathrm{d}s \\
 &\qquad\qquad\quad - \frac{d}{\omega} \left(C+\kappa\|\phi^*\|_{C^2(\overline{\Omega})} \right)-\frac{d}{\omega^2}\left(C+\kappa\|\phi^*\|_{C^2(\overline{\Omega})} \right)^2 \Bigg]\overline{\varphi}_{n},&&(x,t)\in\Omega\times(0,T_{n}],\\
    &\omega \partial_t\underline{\varphi}_n-d\operatorname{div}(A(x)\nabla\underline{\varphi}_n)+c(x,t)\underline{\varphi}_n =\left(\fint^{T_{n}}_{0}H(s;\omega,d)\,\mathrm{d}s+\varepsilon  \right)\underline{\varphi}_{n},&&(x,t)\in\Omega\times(0,T_{n}],\\
    &\mathbf{n}\cdot(A(x)\nabla\overline{\varphi}_n)\ge 0 =\mathbf{n}\cdot(A(x)\nabla \underline{\varphi}_n), &&(x,t)\in\partial\Omega\times(0,T_{n}],\\
    &\overline{\varphi}_n(x,0)=\overline{\varphi}_n(x,T_{n}) ,&&x\in\overline{\Omega},\\
    &\underline{\varphi}_n(x,0)=\varphi(x,0)\le\underbrace{\varphi(x,0)\mathrm{e}^{\frac{ \varepsilon T_{n}}{\omega}-2C_d}\le\varphi(x,T_{n})\mathrm{e}^{\frac{ \varepsilon T_{n}}{\omega}}}_{\text{by \eqref{eq6.12}}}=\underline{\varphi}_n(x,T_{n}),&&x\in\overline{\Omega}.
  \end{aligned}
  \right.
\end{equation*}
Using Lemma \ref{lemma2.13}, we have
\[
\begin{aligned}
   &\liminf\limits_{n\to\infty}\fint^{T_n}_{0}H(s;\omega,d)\,\mathrm{d}s+\varepsilon \\
   &\qquad\qquad\ge \liminf\limits_{n\to\infty}\min_{x\in\overline{\Omega}}\fint^{T_n}_{0}c(x,s)\,\mathrm{d}s- \frac{d}{\omega} \left(C+\kappa\|\phi^*\|_{C^2(\overline{\Omega})} \right)-\frac{d}{\omega^2} \left(C+\kappa\|\phi^*\|_{C^2(\overline{\Omega})} \right)^2\\
\text{and}&\\
   &\limsup\limits_{n\to\infty}\fint^{T_n}_{0}H(s;\omega,d)\,\mathrm{d}s+\varepsilon \\
   &\qquad\qquad\ge \limsup\limits_{n\to\infty}\min_{x\in\overline{\Omega}}\fint^{T_n}_{0}c(x,s)\,\mathrm{d}s- \frac{d}{\omega} \left(C+\kappa\|\phi^*\|_{C^2(\overline{\Omega})} \right)-\frac{d}{\omega^2} \left(C+\kappa\|\phi^*\|_{C^2(\overline{\Omega})} \right)^2.
\end{aligned}
\]
Then sending $(\omega,d)\to(\infty,0)$, and combining Lemma \ref{lemma2.4} and the arbitrariness of $\varepsilon>0$, we have
\[
\begin{aligned}
  & &&\liminf_{(\omega,d)\to(\infty,0)}\liminf\limits_{n\to\infty}\fint^{T_n}_{0}H(s;\omega,d)\,\mathrm{d}s\ge \min_{x\in\overline{\Omega}}\liminf\limits_{n\to\infty}\fint^{T_n}_{0}c(x,s)\,\mathrm{d}s\\
  &\text{and} &&\liminf_{(\omega,d)\to(\infty,0)}\limsup\limits_{n\to\infty}\fint^{T_n}_{0}H(s;\omega,d)\,\mathrm{d}s\ge \limsup\limits_{n\to\infty}\min_{x\in\overline{\Omega}}\fint^{T_n}_{0}c(x,s)\,\mathrm{d}s.
\end{aligned}
\]
This, together with the upper bounds, yields the desired result. The proof is complete.
\end{proof}

\begin{lemma}\label{lemma6.5}
Let $c\in \mathcal{C}\cap C^{2,1}(\overline{\Omega}\times\mathbb{R})$ be a function satisfying \eqref{H2}, $H(\cdot;\omega,d)$ be the normalized principal Floquet bundle of \eqref{eq1.1}. Then
\begin{equation}\label{eq6.13}
\lim_{(\omega,d)\to(\infty,0)}\liminf_{T\to\infty}\fint^{T}_{0}H(s;\omega,d)\,\mathrm{d}s= \min_{x\in\overline{\Omega}}\liminf_{T\to\infty}\fint^{T}_{0}c(x,s)\,\mathrm{d}s
\end{equation}
and
\begin{equation}\label{eq6.14}
\lim_{(\omega,d)\to(\infty,0)}\limsup_{T\to\infty}\fint^{T}_{0}H(s;\omega,d)\,\mathrm{d}s= \limsup_{T\to\infty}\min_{x\in\overline{\Omega}}\fint^{T}_{0}c(x,s)\,\mathrm{d}s.
\end{equation}
\end{lemma}
The proof follows the same lines as that of Corollary \ref{corollary4.3}, Lemmas \ref{lemma4.4} and \ref{lemma4.6}. However, for the convenience of reading, we provide a complete proof.

\begin{proof}[Proof of Lemma \ref{lemma6.5}]
Let $C$, $C_d$ and $\phi^*$ be the notations introduced in the proof of Lemma \ref{lemma6.4}. Fix $\varepsilon>0$. For a given $T>0$, define
\[\overline{\varphi}_T(x,t):=\exp\left[\frac{1}{\omega}\left(\mathcal{M}(x,t;T)- \kappa\phi^*\right) \right],\quad(x,t) \in\overline{\Omega}\times[0,T]\]
and
\[\underline{\varphi}_T(x,t):=\varphi(x,t)\exp\left[\frac{1}{\omega}\int^{t}_{0}\left(\fint^{T}_{0} H(r;\omega,d)\,\mathrm{d}r - H(s;\omega,d) +\varepsilon \right)\,\mathrm{d}s \right],\quad(x,t) \in\overline{\Omega}\times[0,T],\]
where $\varphi$ is the normalized principal Floquet bundle of \eqref{eq1.1}. A direct calculation yields
\begin{equation*}
  \left\{
  \begin{aligned}
    &\omega \partial_t\overline{\varphi}_T-d\operatorname{div}(A(x)\nabla\overline{\varphi}_T)+c(x,t)\overline{\varphi}_T \\
 &\ge\Bigg[\fint^{T}_{0}c(x,s)\,\mathrm{d}s- \frac{d}{\omega} \left(C+\kappa\|\phi^*\|_{C^2(\overline{\Omega})} \right)-\frac{d}{\omega^2}\left(C+\kappa\|\phi^*\|_{C^2(\overline{\Omega})} \right)^2 \Bigg]\overline{\varphi}_T\\
 &\ge\Bigg[\min_{x\in\overline{\Omega}}\fint^{T}_{0}c(x,s)\,\mathrm{d}s- \frac{d}{\omega} \left(C+\kappa\|\phi^*\|_{C^2(\overline{\Omega})} \right)-\frac{d}{\omega^2}\left(C+\kappa\|\phi^*\|_{C^2(\overline{\Omega})} \right)^2 \Bigg]\overline{\varphi}_T,&&(x,t)\in\Omega\times(0,T],\\
    &\omega \partial_t\underline{\varphi}_T-d\operatorname{div}(A(x)\nabla\underline{\varphi}_T)+c(x,t)\underline{\varphi}_T =\left(\fint^{T}_{0}H(s;\omega,d)\,\mathrm{d}s+\varepsilon  \right)\underline{\varphi}_T,&&(x,t)\in\Omega\times(0,T],\\
    &\mathbf{n}\cdot(A(x)\nabla\overline{\varphi}_T)\ge 0 =\mathbf{n}\cdot(A(x)\nabla\underline{\varphi}_T), &&(x,t)\in\partial\Omega\times(0,T],\\
    &\overline{\varphi}_T(x,0)=\overline{\varphi}_T(x,T) ,&&x\in\overline{\Omega},\\
    &\underline{\varphi}_T(x,0)=\varphi(x,0)\le\varphi(x,0)\mathrm{e}^{\frac{ \varepsilon T}{\omega}-2C_d}\le\varphi(x,T)\mathrm{e}^{\frac{ \varepsilon T}{\omega}}=\underline{\varphi}_T(x,T),&&x\in\overline{\Omega},
  \end{aligned}
  \right.
\end{equation*}
if $T>2C_d\omega/\varepsilon$. An application of Corollary \ref{corollary2.15} yields
\[
\begin{aligned}
   &\liminf\limits_{T\to\infty}\fint^{T}_{0}H(s;\omega,d)\,\mathrm{d}s+\varepsilon  \\
    &\qquad\ge\liminf_{T\to\infty}\min_{x\in\overline{\Omega}}\fint^{T}_{0}c(x,s)\,\mathrm{d}s- \frac{d}{\omega} \left(C+\kappa\|\phi^*\|_{C^2(\overline{\Omega})} \right)-\frac{d}{\omega^2} \left(C+\kappa\|\phi^*\|_{C^2(\overline{\Omega})} \right)^2
\end{aligned}
\]
and
\[
\begin{aligned}
   &\limsup\limits_{T\to\infty}\fint^{T}_{0}H(s;\omega,d)\,\mathrm{d}s+\varepsilon  \\
    &\qquad\ge\limsup_{T\to\infty}\min_{x\in\overline{\Omega}}\fint^{T}_{0}c(x,s)\,\mathrm{d}s- \frac{d}{\omega} \left(C+\kappa\|\phi^*\|_{C^2(\overline{\Omega})} \right)-\frac{d}{\omega^2} \left(C+\kappa\|\phi^*\|_{C^2(\overline{\Omega})} \right)^2.
\end{aligned}
\]
By letting $(\omega,d)\to(\infty,0)$ and observing the arbitrariness of $\varepsilon>0$, we obtain
 \[\liminf_{(\omega,d)\to(\infty,0)}\liminf_{T\to\infty}\fint^{T}_{0}H(s;\omega,d)\,\mathrm{d}s\ge \liminf_{T\to\infty}\min_{x\in\overline{\Omega}}\fint^{T}_{0}c(x,s)\,\mathrm{d}s\]
and
 \[\liminf_{(\omega,d)\to(\infty,0)}\limsup_{T\to\infty}\fint^{T}_{0}H(s;\omega,d)\,\mathrm{d}s\ge \limsup_{T\to\infty}\min_{x\in\overline{\Omega}}\fint^{T}_{0}c(x,s)\,\mathrm{d}s.\]
Since
 \[\min_{x\in\overline{\Omega}}\liminf_{T\to\infty}\fint^{T}_{0}c(x,s)\,\mathrm{d}s =\liminf_{T\to\infty}\min_{x\in\overline{\Omega}}\fint^{T}_{0}c(x,s)\,\mathrm{d}s\]
(see Lemma \ref{lemma2.4}), we get
 \[\liminf_{(\omega,d)\to(\infty,0)}\liminf_{T\to\infty}\fint^{T}_{0}H(s;\omega,d)\,\mathrm{d}s\ge \min_{x\in\overline{\Omega}}\liminf_{T\to\infty}\fint^{T}_{0}c(x,s)\,\mathrm{d}s.\]

Let $x_0\in\overline{\Omega}$ be a point and $\{T_n\}_{n\in\mathbb{N}^+}$ be a sequence satisfying $T_n\to\infty$ as $n\to\infty$ such that
 \[\lim_{n\to\infty}\fint^{T_n}_{0}c(x_0,s)\,\mathrm{d}s=\liminf_{T\to\infty}\fint^{T}_{0}c(x_0,s)\,\mathrm{d}s =\min_{x\in\overline{\Omega}}\liminf_{T\to\infty}\fint^{T}_{0}c(x,s)\,\mathrm{d}s.\]
Now let $\{T_{n'}\}_{n'\in\mathbb{N}^+}$ be a subsequence of $\{T_n\}_{n\in\mathbb{N}^+}$ such that $\left\{\fint^{T_{n'}}_{0}c(\cdot,s)\,\mathrm{d}s\right\}_{n'\in\mathbb{N}^+}$ converges to some $\hat{c}$ in $C(\overline{\Omega})$. By the selection, one has
 \[\hat{c}(x)\ge\lim_{n\to\infty}\fint^{T_n}_{0}c(x_0,s)\,\mathrm{d}s,\quad\forall~x\in\overline{\Omega},\]
and the minimum of $\hat{c}$ is achievable at $x_0$. By Lemma \ref{lemma6.4}, we have
 \[\begin{aligned}
     \limsup_{(\omega,d)\to(\infty,0)}\liminf_{T\to\infty}\fint^{T}_{0}H(s;\omega,d)\,\mathrm{d}s &\le \limsup_{(\omega,d)\to(\infty,0)}\limsup_{n'\to\infty}\fint^{T_{n'}}_{0}H(s;\omega,d)\,\mathrm{d}s \\
     &\le\min_{x\in\overline{\Omega}}\hat{c}(x) \\
     & = \min_{x\in\overline{\Omega}}\liminf_{T\to\infty}\fint^{T}_{0}c(x,s)\,\mathrm{d}s.
   \end{aligned}
 \]
This, together with the lower bound, verifies \eqref{eq6.13}.

Using Corollary \ref{corollary3.13}, we have
 \[\limsup\limits_{T\to\infty}\fint^{T}_{0}H(s;\omega,d)\,\mathrm{d}s\le \sup\limits_{\hat{c}\in\mathscr{C}}\mu^\infty(d,\hat{c}),\]
where $\mu^\infty(d,\hat{c})$ is the principal eigenvalue of \eqref{eq1.3}.
By letting $(\omega,d)\to(\infty,0)$, we get
 \[\limsup_{(\omega,d)\to(\infty,0)}\limsup\limits_{T\to\infty}\fint^{T}_{0}H(s;\omega,d)\,\mathrm{d}s\le \limsup_{d\to0}\sup\limits_{\hat{c}\in\mathscr{C}}\mu^\infty(d,\hat{c}).\]
Using Lemma \ref{lemma2.18}, we have
 \[\limsup_{d\to0}\sup\limits_{\hat{c}\in\mathscr{C}}\mu^\infty(d,\hat{c}) =\sup\limits_{\hat{c}\in\mathscr{C}}\min_{x\in\overline{\Omega}}\hat{c}(x).\]
By Lemma \ref{lemma2.3}, we obtain
 \[ \limsup_{(\omega,d)\to(\infty,0)}\limsup_{T\to\infty}\fint^{T}_{0}H(s;\omega,d)\,\mathrm{d}s\le \limsup_{T\to\infty}\min_{x\in\overline{\Omega}}\fint^{T}_{0}c(x,s)\,\mathrm{d}s.\]
A combination of the upper and the lower bounds gives \eqref{eq6.14}. This completes the proof.
\end{proof}

\chapter{Asymptotic behavior of $H$: joint effects of small $\omega$ and small $d$}\label{chp7}
We prove Theorem \ref{theorem1.10} in this chapter, splitting the proof into several lemmas. This chapter establishes the following result with respect to the sequence $\{T_n\}_{n\in\mathbb{N}^+}$ of $T\to\infty$, thereby concluding the demonstration of Theorem \ref{theorem1.10}.
\begin{theorem}\label{theorem7.1}
  Let $H(\cdot;\omega,d)$ be the normalized principal Floquet bundle of \eqref{eq1.1} with potential $c\in \mathcal{C}\cap C^{2,1}(\overline{\Omega}\times\mathbb{R})$ and let $\{T_n\}_{n\in\mathbb{N}^+}$ be a sequence satisfying $T_n\to\infty$ as $n\to\infty$.
\begin{enumerate}[{\rm (i)}]
  \item Assume that the function $c$ satisfies uniform H\"{o}lder condition \eqref{H3}. Then
      \[\lim_{\left(d,\frac{\omega}{\sqrt{d}}\right)\to(0,0)}\liminf_{n\to\infty}\fint^{T_n}_{0}H(s;\omega,d)\,\mathrm{d}s= \liminf_{n\to\infty}\fint^{T_n}_{0}\min_{x\in\overline{\Omega}}c(x,s)\,\mathrm{d}s\]
     and
      \[\lim_{\left(d,\frac{\omega}{\sqrt{d}}\right)\to(0,0)}\limsup_{n\to\infty}\fint^{T_n}_{0}H(s;\omega,d)\,\mathrm{d}s= \limsup_{n\to\infty}\fint^{T_n}_{0}\min_{x\in\overline{\Omega}}c(x,s)\,\mathrm{d}s.\]
     Consequently, if the sequence $\left\{\fint^{T_n}_{0}\min_{x\in\overline{\Omega}}c(x,s)\,\mathrm{d}s\right\}_{n\in\mathbb{N}^+}$ converges, then
       \[\begin{aligned}
         \lim_{\left(d,\frac{\omega}{\sqrt{d}}\right)\to(0,0)}\liminf_{n\to\infty}\fint^{T_n}_{0}H(s;\omega,d) \,\mathrm{d}s
        =& \lim_{\left(d,\frac{\omega}{\sqrt{d}}\right)\to(0,0)}\limsup_{n\to\infty}\fint^{T_n}_{0} H(s;\omega,d) \,\mathrm{d}s\\
        =&\lim_{n\to\infty}\fint^{T_n}_{0}\min_{x\in\overline{\Omega}}c(x,s)\,\mathrm{d}s.
        \end{aligned}\]
  \item Assume that the function $c$ satisfies assumption \eqref{H2}. Then
       \[\lim\limits_{\left(\omega,\frac{\omega}{\sqrt{d}}\right)\to(0,\infty)} \liminf\limits_{n\to\infty}\fint^{T_n}_{0}H(s;\omega,d)\,\mathrm{d}s = \min\limits_{x\in\overline{\Omega}}\liminf_{n\to\infty}\fint^{T_n}_{0}c(x,s)\,\mathrm{d}s= \inf_{\hat{c}\in\mathscr{C}_*}\min_{x\in\overline{\Omega}}\hat{c}(x)\]
      and
       \[\lim\limits_{\left(\omega,\frac{\omega}{\sqrt{d}}\right)\to(0,\infty)} \limsup\limits_{n\to\infty}\fint^{T_n}_{0}H(s;\omega,d)\,\mathrm{d}s = \limsup_{n\to\infty}\min\limits_{x\in\overline{\Omega}}\fint^{T_n}_{0}c(x,s)\,\mathrm{d}s= \sup_{\hat{c}\in\mathscr{C}_*}\min_{x\in\overline{\Omega}}\hat{c}(x).\]
     Consequently, if the sequence $\left\{\fint^{T_n}_{0} c(\cdot,s)\,\mathrm{d}s\right\}_{n\in\mathbb{N}^+}$ of functions converges to $\hat{c}(\cdot)$ in $C(\overline{\Omega})$, then
       \[\lim\limits_{\left(\omega,\frac{\omega}{\sqrt{d}}\right)\to(0,\infty)} \liminf\limits_{n\to\infty}\fint^{T_n}_{0}H(s;\omega,d)\,\mathrm{d}s= \lim\limits_{\left(\omega,\frac{\omega}{\sqrt{d}}\right)\to(0,\infty)} \limsup\limits_{n\to\infty}\fint^{T_n}_{0}H(s;\omega,d)\,\mathrm{d}s = \min\limits_{x\in\overline{\Omega}}\hat{c}(x).\]
\end{enumerate}
\end{theorem}
\section{The case $\left(d,\frac{\omega}{\sqrt{d}}\right)\to(0,0)$}
We first show the following asymptotic result on a finite interval.
\begin{lemma}\label{lemma7.2}
  Let $H(\cdot;\omega,d)$ be the normalized principal Floquet bundle of \eqref{eq1.1} with potential $c\in \mathcal{C}\cap C^{0,1}(\overline{\Omega}\times\mathbb{R})$ fulfilling \eqref{H3}. For any $T_1,T_2\in\mathbb{R}$ with $T_1<T_2$, there holds
 \[\lim_{\left(d,\frac{\omega}{\sqrt{d}}\right)\to(0,0)}\fint^{T_2}_{T_1}H(s;\omega,d)\,\mathrm{d}s= \fint^{T_2}_{T_1}\min_{x\in\overline{\Omega}}c(x,s)\,\mathrm{d}s.\]
\end{lemma}
\begin{proof}
Without loss of generality, we assume that $T_1=0$ and $T_2=T$. By Lemma \ref{lemma3.5}(i), we know
 \[H(t)\ge\min_{x\in\overline{\Omega}} c(x,t),\qquad t\in\mathbb{R},\]
which follow that
 \[\liminf_{\left(d,\frac{\omega}{\sqrt{d}}\right)\to(0,0)}\fint^{T}_{0}H(s;\omega,d)\,\mathrm{d}s\ge \fint^{T}_{0}\min_{x\in\overline{\Omega}}c(x,s)\,\mathrm{d}s.\]
Therefore, we need only to show that
\begin{equation}\label{eq7.1}
  \limsup_{\left(d,\frac{\omega}{\sqrt{d}}\right)\to(0,0)}\fint^{T}_{0}H(s;\omega,d)\,\mathrm{d}s\le \fint^{T}_{0}\min_{x\in\overline{\Omega}}c(x,s)\,\mathrm{d}s.
\end{equation}

Note that $c\in \mathcal{C}\cap C^{0,1}(\overline{\Omega}\times\mathbb{R})$ satisfy the uniform H\"{o}lder condition \eqref{H3}. Then for any $t\in\mathbb{R}$, we have
 \[|c(x,s)-c(x,t)|\le L(s-t)^{\frac{\delta}{2}},\quad\forall~s>t,~x\in\overline{\Omega},\]
which follows that
 \[c(x,t)-L(s-t)^{\frac{\delta}{2}}\le c(x,s)\le c(x,t)+L(s-t)^{\frac{\delta}{2}},\quad\forall~s>t,~x\in\overline{\Omega}.\]
Fix $\tau>0$. Integrating it over $[t,t+\tau]$, we have
 \[c(x,t)-L\tau^{1+\frac{\delta}{2}}\le \int^{t+\tau}_{t}c(x,s)\,\mathrm{d}s\le c(x,t)+L\frac{\tau^{1+\frac{\delta}{2}}}{1+{\frac{\delta}{2}}} \le c(x,t)+L\tau^{1+\frac{\delta}{2}}, \quad\forall~x\in\overline{\Omega}.\]
Taking the minimum in $x$ over $\overline{\Omega}$ in the inequality above, we obtain
 \[  \min_{x\in\overline{\Omega}}c(x,t)-L\tau^{1+{\frac{\delta}{2}}}\le \min_{x\in\overline{\Omega}}\int^{t+\tau}_{t}c(x,s)\,\mathrm{d}s\le \min_{x\in\overline{\Omega}}c(x,t)+L\tau^{1+\frac{\delta}{2}},\]
or equivalently,
\begin{equation}\label{eq7.2}
  \left|\min_{x\in\overline{\Omega}}c(x,t)- \min_{x\in\overline{\Omega}}\int^{t+\tau}_{t}c(x,s)\,\mathrm{d}s\right|\le L\tau^{1+\frac{\delta}{2}}.
\end{equation}
Similarly, one has
\begin{equation}\label{eq7.3}
  \left|\min_{x\in\overline{\Omega}}c(x,t)- \int^{t+\tau}_{t}\min_{x\in\overline{\Omega}}c(x,s)\,\mathrm{d}s\right|\le L\tau^{1+\frac{\delta}{2}}.
\end{equation}
Combining \eqref{eq7.2} and \eqref{eq7.3} yields
\begin{equation}\label{eq7.4}
  \left|\min_{x\in\overline{\Omega}}\int^{t+\tau}_{t}c(x,s)\,\mathrm{d}s- \int^{t+\tau}_{t}\min_{x\in\overline{\Omega}}c(x,s)\,\mathrm{d}s\right|\le2\tau^{1+\frac{\delta}{2}}L.
\end{equation}

Notice from Theorem \ref{theorem3.17} that $-\sqrt{d}\ln\varphi$ is uniformly bounded in $(\omega,d)\in (0,1)\times(0,1)$, that is, there is a constant $C>0$ independent of $(\omega,d)$ such that
\begin{equation}\label{eq7.5}
  \sup_{(x,t)\in\Omega\times\mathbb{R}}\left|-\sqrt{d}\ln\varphi(x,t)\right|<C.
\end{equation}
For any given $\tau>$ and $t\in\mathbb{R}$, define
 \[u(x):=\exp\left(\fint^{t+\tau}_{t}\ln\varphi(x,s)\,\mathrm{d}s\right),\qquad x\in\overline{\Omega}.\]
By direct computations, we have
  \begin{align*}
    &\operatorname{div}(A(x)\nabla u)\\
  = & \Bigg[\fint^{t+\tau}_{t}\left(\frac{\operatorname{div}(A(x)\nabla\varphi)}{\varphi}- \frac{\nabla\varphi\cdot(A(x)\nabla\varphi)}{\varphi^2}\right)\,\mathrm{d}s \\
  &\hskip 3cm+\nabla\left(\fint^{t+\tau}_{t}\ln\varphi \,\mathrm{d}s\right)\cdot\left(A(x)\nabla\left(\fint^{t+\tau}_{t}\ln\varphi \,\mathrm{d}s\right)\right)\Bigg]u \\
    \le&\fint^{t+\tau}_{t}\frac{\operatorname{div}(A(x)\nabla\varphi)}{\varphi}\,\mathrm{d}s u,\qquad\qquad \forall~x\in\Omega,
  \end{align*}
where Lemma \ref{lemma3.11} has been used in obtaining the last inequality. Substituting the equation satisfied by $\varphi$ into the above inequality, one gets that
\[
    \left\{
    \begin{aligned}
      &-d\operatorname{div}(A(x)\nabla u) +\left(\fint^{t+\tau}_{t}c(x,s)\,\mathrm{d}s\right)u\\
      &\hskip 2 cm \ge \left(\fint^{t+\tau}_{t}H(s;\omega,d)\,\mathrm{d}s+ \frac{\omega }{\sqrt{d}\tau}\left.(-\sqrt{d}\ln\varphi(x,t))\right|^{t+\tau}_{t}\right)u\\
      &\hskip 2 cm \ge \left(\fint^{t+\tau}_{t}H(s;\omega,d)\,\mathrm{d}s- \frac{2\omega C }{\sqrt{d}\tau}\right)u,&&x\in\Omega,\\
      &\mathbf{n}\cdot(A(x)\nabla u)=0&&x\in\partial\Omega,
    \end{aligned}
    \right.
\]
where we have used \eqref{eq7.5} in the second inequality. Let $\mu\left(d,\fint^{t+\tau}_{t}c(x,s)\,\mathrm{d}s\right)$ be the principal eigenvalue of \eqref{eq1.3} with potential $\fint^{t+\tau}_{t}c(x,s)\,\mathrm{d}s$. Then
by comparison principle for elliptic eigenvalue problem, one easily get that
\begin{equation}\label{eq7.6}
 \fint^{t+\tau}_{t}H(s;\omega,d)\,\mathrm{d}s\le \mu\left(d,\fint^{t+\tau}_{t}c(x,s)\,\mathrm{d}s\right)+\frac{\omega}{\sqrt{d}}\frac{2C}{\tau}.
\end{equation}
Using the estimate \eqref{eq2.9} in the proof of Lemma \ref{lemma2.23}, for any $\varepsilon>0$, there is a positive constant $\tilde{r}_\varepsilon$, independent of $(t,\tau)\in\mathbb{R}\times(0,1)$, such that
\begin{equation}\label{eq7.7}
  \mu\left(d,\fint^{t+\tau}_{t}c(x,s)\,\mathrm{d}s\right)\le \min_{x\in\overline{\Omega}}\fint^{t+\tau}_{t}c(x,s)\,\mathrm{d}s+2^{N+4}\Lambda \frac{d}{r^2}+\varepsilon,\quad\forall~r\in(0,\tilde{r}_\varepsilon).
\end{equation}

We now use the above inequalities to estimate $ \fint^{T}_{0} H(s;\omega,d)\,\mathrm{d}s$. For any given $0<\hbar<<1$, denote by $\ell=\lfloor T/\hbar\rfloor$ and $t_i=i\hbar$, $i=0,1,\cdots,\ell$. A direct computation yields
\begin{align*}
   & \fint^{T}_{0} H(s;\omega,d)\,\mathrm{d}s \\
 = & \frac{1}{T}\left(\hbar\sum^{\ell-1}_{i=0}\fint^{t_{i+1}}_{t_i} H(s;\omega,d)\,\mathrm{d}s +(T-t_\ell)\fint^{T}_{t_\ell} H(s;\omega,d)\,\mathrm{d}s \right) \\
 \le & \frac{1}{T}\left[\hbar\sum^{\ell-1}_{i=0}\mu\left(d,\fint^{t_{i+1}}_{t_{i}}c(x,s)\,\mathrm{d}s\right) +(T-t_\ell)\mu\left(d,\fint^{T}_{t_\ell}c(x,s)\,\mathrm{d}s\right)\right] +\frac{2C}{\hbar}\left(1+\frac{1}{\ell}\right)\frac{\omega}{\sqrt{d}} &&\text{(by \eqref{eq7.6})}\\
  \le & \frac{1}{T}\left[\sum^{\ell-1}_{i=0}  \min_{x\in\overline{\Omega}}\int^{t_{i+1}}_{t_{i}}c(x,s)\,\mathrm{d}s +\min_{x\in\overline{\Omega}}\int^{T}_{t_\ell}c(x,s)\,\mathrm{d}s\right]+ 2^{N+4}\Lambda \frac{d}{r^2}+\varepsilon+\frac{2C}{\hbar}\left(1+\frac{1}{\ell}\right)\frac{\omega}{\sqrt{d}}  &&\text{(by \eqref{eq7.7})}\\
  \le & \frac{1}{T}\left[\sum^{\ell-1}_{i=0}\int^{t_{i+1}}_{t_{i}}\min_{x\in\overline{\Omega}}c(x,s)\,\mathrm{d}s +\int^{T}_{t_\ell}\min_{x\in\overline{\Omega}}c(x,s)\,\mathrm{d}s\right] \\
  &\hskip 2cm +2L\hbar^{1+\frac{\delta}{2}}+ 2^{N+4}\Lambda \frac{d}{r^2}+\varepsilon+\frac{2C}{\hbar}\left(1+\frac{1}{\ell}\right)\frac{\omega}{\sqrt{d}}  &&\text{(by \eqref{eq7.4})}\\
  \le & \fint^{T}_{0}\min_{x\in\overline{\Omega}} c(x,s)\,\mathrm{d}s+2L\hbar^{1+\frac{\delta}{2}}+ 2^{N+4}\Lambda \frac{d}{r^2}+\varepsilon+\frac{2C}{\hbar}\left(1+\frac{1}{\ell}\right)\frac{\omega}{\sqrt{d}},\qquad r\in(0,\tilde{r}_\varepsilon).
\end{align*}
This estimate is uniform in $T>0$ but depends on the step size $\hbar$. Hence, for each fixed $\hbar>0$ we first let $\left(d,\frac{\omega}{\sqrt{d}}\right)\to (0,0)$ and then use the arbitrariness of $\hbar>0$ and $\varepsilon>0$ to obtain estimate \eqref{eq7.1}. This completes the proof.
\end{proof}

\medskip
Theorem \ref{theorem1.11} is a direct corollary of Lemmas \ref{lemma7.2} and \ref{lemma2.9}. We shall establish the asymptotics of the principal Floquet exponent by means of Lemma \ref{lemma7.2}. The argument crucially relies on the fact that the limiting behavior in Lemma \ref{lemma7.2} is uniform with respect to $T>0$.

\begin{lemma}\label{lemma7.3}
Let $H(\cdot;\omega,d)$ be the normalized principal Floquet bundle of \eqref{eq1.1} with potential $c\in \mathcal{C}\cap C^{0,1}(\overline{\Omega}\times\mathbb{R})$, and let $\mu^{0}(t)$, $t\in\mathbb{R}$ be the principal eigenvalue of \eqref{eq2.7} with potential $c(\cdot,t)$. If the function $c$ satisfies uniform H\"{o}lder condition \eqref{H3}, then for any sequence $\{T_n\}_{n\in\mathbb{N}^+}$ of $T\to\infty$,
 \[\lim_{\left(d,\frac{\omega}{\sqrt{d}}\right)\to(0,0)}\liminf_{n\to\infty}\fint^{T_n}_{0}H(s;\omega,d)\,\mathrm{d}s= \liminf_{n\to\infty}\fint^{T_n}_{0}\min_{x\in\overline{\Omega}}c(x,s)\,\mathrm{d}s\]
and
 \[\lim_{\left(d,\frac{\omega}{\sqrt{d}}\right)\to(0,0)}\limsup_{n\to\infty}\fint^{T_n}_{0}H(s;\omega,d)\,\mathrm{d}s= \limsup_{n\to\infty}\fint^{T_n}_{0}\min_{x\in\overline{\Omega}}c(x,s)\,\mathrm{d}s.\]
Consequently, if the limit $\lim_{n\to\infty}\fint^{T_n}_{0}\min_{x\in\overline{\Omega}}c(x,s) \,\mathrm{d}s$ exists, there hold
 \begin{align*}
   \lim_{\left(d,\frac{\omega}{\sqrt{d}}\right)\to(0,0)}\liminf_{n\to\infty}\fint^{T_n}_{0}H(s;\omega,d)\,\mathrm{d}s
   =&\lim_{\left(d,\frac{\omega}{\sqrt{d}}\right)\to(0,0)}\limsup_{n\to\infty}\fint^{T_n}_{0}H(s;\omega,d)\,\mathrm{d}s\\
   =&\lim_{n\to\infty}\fint^{T_n}_{0}\min_{x\in\overline{\Omega}}c(x,s)\,\mathrm{d}s.
 \end{align*}
\end{lemma}
\begin{proof}
It follows directly from
 \[H(t)\ge\min_{x\in\overline{\Omega}} c(x,t),\qquad t\in\mathbb{R},\]
that
 \[\liminf_{\left(d,\frac{\omega}{\sqrt{d}}\right)\to(0,0)}\liminf_{n\to\infty}\fint^{T_n}_{0}H(s;\omega,d)\,\mathrm{d}s\ge \liminf_{n\to\infty}\fint^{T_n}_{0}\min_{x\in\overline{\Omega}}c(x,s)\,\mathrm{d}s\]
and
 \[\liminf_{\left(d,\frac{\omega}{\sqrt{d}}\right)\to(0,0)}\limsup_{n\to\infty}\fint^{T_n}_{0}H(s;\omega,d)\,\mathrm{d}s\ge \limsup_{n\to\infty}\fint^{T_n}_{0}\min_{x\in\overline{\Omega}}c(x,s)\,\mathrm{d}s.\]
We need to show
\begin{equation}\label{eq7.8}
  \begin{aligned}
     & \limsup_{\left(d,\frac{\omega}{\sqrt{d}}\right)\to(0,0)}\liminf_{n\to\infty}\fint^{T_n}_{0}H(s;\omega,d)\,\mathrm{d}s\le \liminf_{n\to\infty}\fint^{T_n}_{0} \min_{x\in\overline{\Omega}}c(x,s)\,\mathrm{d}s \\
     \text{and}\quad& \limsup_{\left(d,\frac{\omega}{\sqrt{d}}\right)\to(0,0)}\limsup_{n\to\infty}\fint^{T_n}_{0}H(s;\omega,d)\,\mathrm{d}s\le \limsup_{n\to\infty}\fint^{T_n}_{0} \min_{x\in\overline{\Omega}}c(x,s)\,\mathrm{d}s.
  \end{aligned}
\end{equation}
We recall that the last inequality appearing in the proof of Lemma \ref{lemma7.2} is uniform in $T>0$. Substitution of the sequence $\{T_n\}_{n\in\mathbb{N}^+}$ for $T$ then yields the following estimate
 \[\fint^{T_n}_{0} H(s;\omega,d)\,\mathrm{d}s
  \le \fint^{T_n}_{0}\min_{x\in\overline{\Omega}} c(x,s)\,\mathrm{d}s+2L\hbar^{1+\frac{\delta}{2}}+ 2^{N+4}\Lambda \frac{d}{r^2}+\varepsilon+\frac{2C}{\hbar}\left(1+\frac{1}{\ell}\right)\frac{\omega}{\sqrt{d}}\]
for all $n\in\mathbb{N}^+,~r\in(0,\tilde{r}_\varepsilon)$, where $\hbar,\varepsilon>0$ are any given constant and $\tilde{r}_\varepsilon$ relies on $\varepsilon$ by independent of $n$ and $d>0$. By letting $n\to\infty$, we have
 \[\liminf_{n\to\infty}\fint^{T_n}_{0} H(s;\omega,d)\,\mathrm{d}s\le \liminf_{n\to\infty}\fint^{T_n}_{0}\min_{x\in\overline{\Omega}} c(x,s)\,\mathrm{d}s+ 2L\hbar^{1+\frac{\delta}{2}}+ 2^{N+4}\Lambda \frac{d}{r^2}+\varepsilon+\frac{2C}{\hbar}\left(1+\frac{1}{\ell}\right)\frac{\omega}{\sqrt{d}}\]
and
 \[\limsup_{n\to\infty}\fint^{T_n}_{0} H(s;\omega,d)\,\mathrm{d}s\le \limsup_{n\to\infty}\fint^{T_n}_{0}\min_{x\in\overline{\Omega}} c(x,s)\,\mathrm{d}s+ 2L\hbar^{1+\frac{\delta}{2}}+ 2^{N+4}\Lambda \frac{d}{r^2}+\varepsilon+\frac{2C}{\hbar}\left(1+\frac{1}{\ell}\right)\frac{\omega}{\sqrt{d}}.\]
Now letting $\left(d,\frac{\omega}{\sqrt{d}}\right)\to (0,0)$ and then by the arbitrariness of $\hbar,\varepsilon>0$, we arrive at \eqref{eq7.8}. The proof is complete.
\end{proof}

\medskip
Furthermore, using Lemma \ref{lemma7.3} and the arguments therein, we have the following assertions.
\begin{lemma}\label{lemma7.4}
Let $H(\cdot;\omega,d)$ be the normalized principal Floquet bundle of \eqref{eq1.1} with potential $c\in \mathcal{C}\cap C^{0,1}(\overline{\Omega}\times\mathbb{R})$, and let $\mu^{0}(t)$, $t\in\mathbb{R}$ be the principal eigenvalue of \eqref{eq2.7} with potential $c(\cdot,t)$. If the function $c$ satisfies uniform H\"{o}lder condition \eqref{H3}, then
 \[\lim_{\left(d,\frac{\omega}{\sqrt{d}}\right)\to(0,0)}\liminf_{T\to\infty}\fint^{T}_{0}H(s;\omega,d)\,\mathrm{d}s= \liminf_{T\to\infty}\fint^{T}_{0}\min_{x\in\overline{\Omega}}c(x,s)\,\mathrm{d}s\]
and
 \[\lim_{\left(d,\frac{\omega}{\sqrt{d}}\right)\to(0,0)}\limsup_{T\to\infty}\fint^{T}_{0}H(s;\omega,d)\,\mathrm{d}s= \limsup_{T\to\infty}\fint^{T}_{0}\min_{x\in\overline{\Omega}}c(x,s)\,\mathrm{d}s.\]
\end{lemma}
\begin{proof}
Let $\{T_n\}_{n\in\mathbb{N}^+}$ be a sequence satisfying $T_n\to\infty$ as $n\to\infty$ such that
 \[ \lim_{n\to\infty}\fint^{T_n}_{0}\min_{x\in\overline{\Omega}}c(x,s)\,\mathrm{d}s= \liminf_{T\to\infty}\fint^{T}_{0}\min_{x\in\overline{\Omega}}c(x,s)\,\mathrm{d}s.\]
Then by Lemma \ref{lemma7.3}, we conclude that
  \begin{align*}
    \limsup_{\left(d,\frac{\omega}{\sqrt{d}}\right)\to(0,0)}\liminf_{T\to\infty}\fint^{T}_{0}H(s;\omega,d)\,\mathrm{d}s \le & \lim_{\left(d,\frac{\omega}{\sqrt{d}}\right)\to(0,0)}\limsup_{n\to\infty}\fint^{T_n}_{0}H(s;\omega,d)\,\mathrm{d}s \\
    = & \lim_{n\to\infty}\fint^{T_n}_{0}\min_{x\in\overline{\Omega}}c(x,s)\,\mathrm{d}s=  \liminf_{T\to\infty}\fint^{T}_{0}\min_{x\in\overline{\Omega}}c(x,s)\,\mathrm{d}s.
  \end{align*}
On the other hand, by virtue of $H(t;\omega,d)\ge \min_{\overline{\Omega}}c(\cdot,t)$ for all $t\in\mathbb{R}$ (see Lemma \ref{lemma3.5}), we have
\begin{align*}
   \inf_{T>T_1}\fint^{T}_{0}H(s;\omega,d)\,\mathrm{d}s - &\inf_{T>T_1}\fint^{T}_{0}\min_{x\in \overline{\Omega}} c(x,s)\,\mathrm{d}s \\
   \ge & \inf_{T>T_1}\left(\fint^{T}_{0}H(s;\omega,d)\,\mathrm{d}s - \fint^{T}_{0}\min_{x\in\overline{\Omega}}c(x,s)\,\mathrm{d}s\right)\ge 0
\end{align*}
for all $T_1>0$. By letting $T_1\to\infty$ first and then $d\to 0$ with $\frac{\omega}{\sqrt{d}}\to 0$, one easily gets that
 \[\liminf_{\left(d,\frac{\omega}{\sqrt{d}}\right)\to(0,0)}\liminf_{T\to\infty}\fint^{T}_{0}H(s;\omega,d)\,\mathrm{d}s\ge \liminf_{T\to\infty}\fint^{T}_{0}\min_{x\in\overline{\Omega}}c(x,s)\,\mathrm{d}s.\]
Together this with the upper bound, we conclude the first identity.

For the second one, choosing suitable sequence of $T\to\infty$ and applying Lemma \ref{lemma7.3} yield that
 \[\liminf_{\left(d,\frac{\omega}{\sqrt{d}}\right)\to(0,0)}\limsup_{T\to\infty}\fint^{T}_{0}H(s;\omega,d)\,\mathrm{d}s\ge \limsup_{T\to\infty}\fint^{T}_{0}\min_{x\in\overline{\Omega}}c(x,s)\,\mathrm{d}s.\]
So we need only to show the inverse inequality. Indeed, for any $T>0$, repeating the arguments in the proof of Lemma \ref{lemma7.3}, one easily proves that
 \[\fint^{T}_{0}H(s;\omega,d)\,\mathrm{d}s\le\fint^{T}_{0}\min_{x\in\overline{\Omega}}c(x,s) \,\mathrm{d}s +2L\hbar^{1+\frac{\delta}{2}}+ 2^{N+4}\Lambda \frac{d}{r^2}+\varepsilon+\frac{2C}{\hbar}\left(1+\frac{1}{\ell}\right)\frac{\omega}{\sqrt{d}},\quad~r\in(0,\tilde{r}_\varepsilon),\]
where $\hbar,\varepsilon>0$ are any given constant and $\tilde{r}_\varepsilon$ relies on $\varepsilon$ by independent of $T>0$ and $d>0$. Letting $T\to\infty$ first and then $d\to 0$ with $\frac{\omega}{\sqrt{d}}\to 0$, and by arbitrariness of $\hbar,\varepsilon>0$, we obtain
 \[\limsup_{\left(d,\frac{\omega}{\sqrt{d}}\right)\to(0,0)}\limsup_{T\to\infty}\fint^{T}_{0}H(s;\omega,d)\,\mathrm{d}s\le \limsup_{T\to\infty}\fint^{T}_{0}\min_{x\in\overline{\Omega}}c(x,s)\,\mathrm{d}s.\]
The second assertion follows from this inequality and the lower bound. This ends the proof.
\end{proof}

\section{The case $\left(\omega,\frac{\omega}{\sqrt{d}}\right)\to(0,\infty)$}
This section investigates the asymptotic behavior of the normalized principal Floquet exponent in the limit where $\left(\omega,\frac{\omega}{\sqrt{d}}\right)\to (0,\infty)$. The proof strategy proceeds as follows: first, an upper bound is obtained by using the principal eigenvalue of the long-time averaged elliptic problem; next, a lower bound is established by constructing appropriate sub- and super-solutions and applying a comparison principle. Combining these bounds yields the exact limit.
\begin{lemma}\label{lemma7.5}
Let $H(\cdot;\omega,d)$ be the normalized principal Floquet bundle of \eqref{eq1.1} with potential $c\in\mathcal{C}\cap C^{2,1}(\overline{\Omega}\times\mathbb{R})$ satisfying \eqref{H2}, and $\{T_n\}_{n\in\mathbb{N}^+}$ be a sequence satisfying $T_n\to\infty$ as $n\to\infty$. Then
 \[\lim\limits_{\left(\omega,\frac{\omega}{\sqrt{d}}\right)\to(0,\infty)} \liminf\limits_{n\to\infty}\fint^{T_n}_{0}H(s;\omega,d)\,\mathrm{d}s = \min\limits_{x\in\overline{\Omega}}\liminf_{n\to\infty}\fint^{T_n}_{0}c(x,s)\,\mathrm{d}s =\inf_{\hat{c}\in\mathscr{C}_*}\min_{x\in\overline{\Omega}}\hat{c}(x)\]
and
 \[\lim\limits_{\left(\omega,\frac{\omega}{\sqrt{d}}\right)\to(0,\infty)} \limsup\limits_{n\to\infty}\fint^{T_n}_{0}H(s;\omega,d)\,\mathrm{d}s = \limsup_{n\to\infty}\min\limits_{x\in\overline{\Omega}}\fint^{T_n}_{0}c(x,s)\,\mathrm{d}s =\sup_{\hat{c}\in\mathscr{C}_*}\min_{x\in\overline{\Omega}}\hat{c}(x).\]
Consequently, if the sequence $\left\{\fint^{T_n}_{0} c(\cdot,s)\,\mathrm{d}s\right\}_{n\in\mathbb{N}^+}$ of functions converges to some function $\hat{c}(\cdot)$ in $C(\overline{\Omega})$, then
 \[\lim\limits_{\left(\omega,\frac{\omega}{\sqrt{d}}\right)\to(0,\infty)} \liminf\limits_{n\to\infty}\fint^{T_n}_{0}H(s;\omega,d)\,\mathrm{d}s= \lim\limits_{\left(\omega,\frac{\omega}{\sqrt{d}}\right)\to(0,\infty)} \limsup\limits_{n\to\infty}\fint^{T_n}_{0}H(s;\omega,d)\,\mathrm{d}s = \min\limits_{x\in\overline{\Omega}}\hat{c}(x).\]
\end{lemma}
\begin{proof}
The upper and lower bounds are estimated separately.

\medskip
\noindent\textbf{Upper bound.} Firstly, it follows from Lemma \ref{lemma3.12} that
\begin{equation}\label{eq7.9}
  \liminf\limits_{n\to\infty}\fint^{T_n}_{0} H(s;\omega,d)\,\mathrm{d}s \le\inf_{\hat{c}\in\mathscr{C}_*}\mu^{\infty}(d,\hat{c})\quad\text{and}\quad \limsup\limits_{n\to\infty}\fint^{T_n}_{0} H(s;\omega,d)\,\mathrm{d}s \le\sup_{\hat{c}\in\mathscr{C}_*}\mu^{\infty}(d,\hat{c}),
\end{equation}
where $\mu^{\infty}(d,\hat{c})$ is the principal eigenvalue of \eqref{eq1.3} with potential $\hat{c}$. By letting $\left(\omega,\frac{\omega}{\sqrt{d}}\right)\to (0,\infty)$ in \eqref{eq7.9} and applying Lemma \ref{lemma5.10} we obtain the upper bound
 \[\limsup\limits_{\left(\omega,\frac{\omega}{\sqrt{d}}\right)\to(0,\infty)} \liminf\limits_{n\to\infty}\fint^{T_n}_{0}H(s;\omega,d)\,\mathrm{d}s \le\min\limits_{x\in\overline{\Omega}}\liminf_{n\to\infty}\fint^{T_n}_{0}c(x,s)\,\mathrm{d}s\]
and
 \[\limsup\limits_{\left(\omega,\frac{\omega}{\sqrt{d}}\right)\to(0,\infty)} \limsup\limits_{n\to\infty}\fint^{T_n}_{0}H(s;\omega,d)\,\mathrm{d}s \le\limsup_{n\to\infty}\min\limits_{x\in\overline{\Omega}}\fint^{T_n}_{0}c(x,s)\,\mathrm{d}s.\]
\medskip
\noindent\textbf{Lower bound.} For a given $c\in\mathcal{C}$, recall the notation
 \[\mathcal{M}(x,t;T)=t\fint^{T}_{0}c(x,s)\,\mathrm{d}s-\int^{t}_{0}c(x,s)\,\mathrm{d}s, \qquad (x,t)\in\overline{\Omega}\times[0,T].\]
Assumption \eqref{H2} implies that there exist positive constants $T_0$ and $V_0$ such that
 \[\sup_{(x,t)\in\Omega\times[0,T]}(|\operatorname{div}(A(x)\nabla\mathcal{M}(x,t;T))|+|\nabla \mathcal{M}(x,t;T)\cdot (A(x)\nabla \mathcal{M}(x,t;T))|)<V_0,\quad\forall~T>T_0.\]
Without loss of generality, we assume that $T_n>T_0$ for all $n\in\mathbb{N}^+$. Let $\phi^*$ be the principal eigenfunction of \eqref{eq2.6}. Set
\begin{equation}\label{eq7.10}
  w(x)=1-\frac{1}{M}\phi^*(x),\qquad x\in\overline{\Omega},
\end{equation}
where $M>2$ is a constant which can be selected sufficiently large.

Let $\varepsilon\in(0,1)$. For each $n\in\mathbb{N}^+$, define
 \[\overline{\varphi}_n(x,t):=\exp\left[\frac{1}{\sqrt{d }}{w}(x)+\frac{1}{\omega }\mathcal{M}(x,t;T_{n})\right],\quad(x,t) \in\overline{\Omega}\times[0,T_{n}]\]
and
 \[\underline{\varphi}_n(x,t):=\varphi(x,t)\exp\left[\frac{1}{\omega }\int^{t}_{0} \left(\fint^{T_{n}}_{0} H(r;\omega ,d )\,\mathrm{d}r -H(s;\omega ,d ) +\varepsilon \right)\,\mathrm{d}s\right],\quad(x,t) \in\overline{\Omega}\times[0,T_{n}],\]
where $\varphi$ is the normalized principal Floquet bundle of \eqref{eq1.1}. By direct computations, we have
\begin{align*}
   & \omega \partial_t\overline{\varphi}_{n}-d \operatorname{div}(A(x)\nabla\overline{\varphi}_{n})+c(x,t)\overline{\varphi}_{n} \\
 = & \left(\omega \partial_t\ln\overline{\varphi}_{n}-d \operatorname{div}(A(x)\nabla\ln\overline{\varphi}_{n})-d\nabla \ln\overline{\varphi}_{n}\cdot(A(x)\nabla \ln\overline{\varphi}_{n})+c(x,t)\right)\overline{\varphi}_{n} \\
 = & \Bigg[\fint^{T_{n}}_{0}c(x,s)\,\mathrm{d}s -\sqrt{d}\operatorname{div}(A(x)\nabla w) -\frac{d }{\omega } \operatorname{div}(A(x)\nabla\mathcal{M}(x,t;T_{n}))\\
 &~~- \nabla{w}\cdot(A(x)\nabla{w})-\frac{2\sqrt{d }}{\omega }\nabla{w}\cdot(A(x)\nabla \mathcal{M}(x,t;T_{n})) \\
 &~~-\frac{d}{\omega^{2}}\nabla \mathcal{M}(x,t;T_{n})\cdot(A(x)\nabla \mathcal{M}(x,t;T_{n})) \Bigg]\overline{\varphi}_{n} \\
  \ge &  \Bigg[\min_{x\in\overline{\Omega}}\fint^{T_{n}}_{0}c(x,s)\,\mathrm{d}s -\frac{\sqrt{d }}{M}\Lambda\|\phi^*\|_{C^2(\overline{\Omega})}-\frac{d }{\omega }V_0- \frac{\|\phi^*\|^{2}_{C^2(\overline{\Omega})}}{M^2} -\frac{2\sqrt{d }}{M\omega } V_0\|\phi^*\|_{C^2(\overline{\Omega})}-\frac{d }{\omega^{2}}V^{2}_{0} \Bigg]\overline{\varphi}_{n}
 \end{align*}
in $(x,t)\in\Omega\times(0,T_n]$. We note that $\min_{x\in\partial\Omega}(-\mathbf{n}\cdot(A(x)\nabla\phi^*))>0$ and select the pair $(\omega,d)$ satisfying
 \[\frac{\sqrt{d}}{\omega}<\frac{\min\limits_{x\in\partial\Omega}(-\mathbf{n}\cdot(A(x)\nabla\phi^*))}{MV_0}.\]
Then, by the Harnack principle (Proposition \ref{proposition3.3}(ii)) for the normalized principal Floquet bundle $\varphi$, there is a constant $C_d>0$ such that
\begin{equation}\label{eq7.11}
|\ln\varphi(x,t)|<C_d,\qquad\forall~(x,t)\in\overline{\Omega}\times\mathbb{R}.
\end{equation}
Now one can find an integer $N_\varepsilon>0$ such that
 \[T_n>\frac{2C_d\omega}{\varepsilon},\qquad \forall~n>N_\varepsilon.\]
Then one can check that for each $n>N_\varepsilon$, the pair $(\underline{\varphi}_n,\overline{\varphi}_n)$ satisfies
\begin{equation*}
  \left\{
  \begin{aligned}
    &\omega \partial_t\overline{\varphi}_n-d\operatorname{div}(A(x)\nabla\overline{\varphi}_n) +c(x,t)\overline{\varphi}_n \ge\Bigg[\min_{x\in\overline{\Omega}}\fint^{T_{n}}_{0}c(x,s)\,\mathrm{d}s-\frac{d }{\omega }V_0 \\
    &\quad-\frac{\sqrt{d }}{M}\|\phi^*\|_{C^2(\overline{\Omega})}- \frac{\|\phi^*\|^{2}_{C^2(\overline{\Omega})}}{M^2} -\frac{2\sqrt{d }}{M\omega } V_0\|\phi^*\|_{C^2(\overline{\Omega})}-\frac{d }{\omega^{2}}V^{2}_{0} \Bigg]\overline{\varphi}_{n},&&(x,t)\in\Omega\times(0,T_{n}],\\
    &\omega \partial_t\underline{\varphi}_n-d\operatorname{div}(A(x)\nabla\underline{\varphi}_n) +c(x,t)\underline{\varphi}_n = \left(\fint^{T_{n}}_{0}H(s;\omega,d)\,\mathrm{d}s+\varepsilon  \right)\underline{\varphi}_{n},&&(x,t)\in\Omega\times(0,T_{n}],\\
    &\mathbf{n}\cdot(A(x)\nabla\overline{\varphi}_n) \\
    &~=\frac{1}{\sqrt{d}}\left(-\frac{1}{M}\mathbf{n}\cdot(A(x)\nabla\phi^*) +\frac{\sqrt{d}}{\omega}\mathbf{n}\cdot(A(x)\nabla\mathcal{M}(x,t;T_n))\right)\overline{\varphi}_n\ge 0 , &&(x,t)\in\partial\Omega\times(0,T_{n}],\\
    &\mathbf{n}\cdot(A(x)\nabla\underline{\varphi}_n)=0, &&(x,t)\in\partial\Omega\times(0,T_{n}],\\
    &\overline{\varphi}_n(x,0)=\overline{\varphi}_n(x,T_{n}) ,&&x\in\overline{\Omega},\\
    &\underline{\varphi}_n(x,0)=\varphi(x,0)\le\mathrm{e}^{\varepsilon T_{n}-2C_d}\varphi(x,0)\le\varphi(x,T_{n})\mathrm{e}^{\varepsilon T_{n}}=\underline{\varphi}_n(x,T_{n}),&&x\in\overline{\Omega}.
  \end{aligned}
  \right.
\end{equation*}
Using Lemma \ref{lemma2.13}, we have
 \[
\begin{aligned}
   \liminf\limits_{n\to\infty}\fint^{T_n}_{0}H(s;\omega,d)\,\mathrm{d}s+\varepsilon \ge &\liminf\limits_{n\to\infty}\min_{x\in\overline{\Omega}}\fint^{T_{n}}_{0}c(x,s)\,\mathrm{d}s \\
    &-\frac{\sqrt{d }}{M}\|\phi^*\|_{C^2(\overline{\Omega})}-\frac{d }{\omega }V_0- \frac{\|\phi^*\|^{2}_{C^2(\overline{\Omega})}}{M^2} -\frac{2\sqrt{d }}{M\omega } V_0\|\phi^*\|_{C^2(\overline{\Omega})}-\frac{d }{\omega^{2}}V^{2}_{0}
\end{aligned}
\]
and
 \[
\begin{aligned}
   \limsup\limits_{n\to\infty}\fint^{T_n}_{0}H(s;\omega,d)\,\mathrm{d}s+\varepsilon \ge &\limsup\limits_{n\to\infty}\min_{x\in\overline{\Omega}}\fint^{T_{n}}_{0}c(x,s)\,\mathrm{d}s \\
    &-\frac{\sqrt{d }}{M}\|\phi^*\|_{C^2(\overline{\Omega})}-\frac{d }{\omega }V_0- \frac{\|\phi^*\|^{2}_{C^2(\overline{\Omega})}}{M^2} -\frac{2\sqrt{d }}{M\omega } V_0\|\phi^*\|_{C^2(\overline{\Omega})}-\frac{d }{\omega^{2}}V^{2}_{0}.
\end{aligned}
\]
Applying Lemma \ref{lemma2.4} and sending $\left(\omega,\frac{\omega}{\sqrt{d}}\right)\to (0,\infty)$, we have
 \[
\begin{aligned}
   & &&\liminf\limits_{\left(\omega,\frac{\omega}{\sqrt{d}}\right)\to(0,\infty)}\liminf\limits_{n\to\infty}\fint^{T_n}_{0}H(s;\omega,d)\,\mathrm{d}s+\varepsilon \ge \min_{x\in\overline{\Omega}}\liminf\limits_{n\to\infty}\fint^{T_{n}}_{0}c(x,s)\,\mathrm{d}s- \frac{\|\phi^*\|^{2}_{C^2(\overline{\Omega})}}{M^2}\\
   &\text{and}&&\liminf\limits_{\left(\omega,\frac{\omega}{\sqrt{d}}\right)\to(0,\infty)}\limsup\limits_{n\to\infty}\fint^{T_n}_{0}H(s;\omega,d)\,\mathrm{d}s+\varepsilon \ge \limsup\limits_{n\to\infty}\min_{x\in\overline{\Omega}}\fint^{T_{n}}_{0}c(x,s)\,\mathrm{d}s- \frac{\|\phi^*\|^{2}_{C^2(\overline{\Omega})}}{M^2}.
\end{aligned}
\]
Lastly, by the arbitrariness of $\varepsilon\in(0,1)$ and $M>2$, we show that
 \[
\begin{aligned}
   & &&\liminf\limits_{\left(\omega,\frac{\omega}{\sqrt{d}}\right)\to(0,\infty)}\liminf\limits_{n\to\infty}\fint^{T_n}_{0}H(s;\omega,d)\,\mathrm{d}s\ge \min_{x\in\overline{\Omega}}\liminf\limits_{n\to\infty}\fint^{T_{n}}_{0}c(x,s)\,\mathrm{d}s  \\
   &\text{and} \quad &&\liminf\limits_{\left(\omega,\frac{\omega}{\sqrt{d}}\right)\to(0,\infty)}\limsup\limits_{n\to\infty}\fint^{T_n}_{0}H(s;\omega,d)\,\mathrm{d}s\ge \limsup\limits_{n\to\infty}\min_{x\in\overline{\Omega}}\fint^{T_{n}}_{0}c(x,s)\,\mathrm{d}s.
\end{aligned}
\]

Using Lemmas \ref{lemma2.3} and \ref{lemma2.4}, we have
\[\min\limits_{x\in\overline{\Omega}}\liminf_{n\to\infty}\fint^{T_n}_{0}c(x,s)\,\mathrm{d}s =\inf_{\hat{c}\in\mathscr{C}_*}\min_{x\in\overline{\Omega}}\hat{c}(x) \quad\text{and}\quad\limsup_{n\to\infty}\min\limits_{x\in\overline{\Omega}}\fint^{T_n}_{0}c(x,s)\,\mathrm{d}s =\sup_{\hat{c}\in\mathscr{C}_*}\min_{x\in\overline{\Omega}}\hat{c}(x).\]
Combining the upper bounds and the lower bounds, we complete the proof of the lemma.
\end{proof}

\begin{lemma}\label{lemma7.6}
 Let $(H,\varphi)$ be the normalized principal Floquet bundle of \eqref{eq1.1} with potential $c\in\mathcal{C}\cap C^{2,1}(\overline{\Omega}\times\mathbb{R})$. Suppose that assumption \eqref{H2} holds, then
 \[\lim\limits_{\left(\omega,\frac{\omega}{\sqrt{d}}\right)\to(0,\infty)} \liminf\limits_{T\to\infty}\fint^{T}_{0}H(s;\omega,d)\,\mathrm{d}s = \min\limits_{x\in\overline{\Omega}}\liminf_{T\to\infty}\fint^{T}_{0}c(x,s)\,\mathrm{d}s =\inf_{\hat{c}\in\mathscr{C}}\min_{x\in\overline{\Omega}}\hat{c}(x)\]
 and
 \[\lim\limits_{\left(\omega,\frac{\omega}{\sqrt{d}}\right)\to(0,\infty)} \limsup\limits_{T\to\infty}\fint^{T}_{0}H(s;\omega,d)\,\mathrm{d}s = \limsup_{T\to\infty}\min\limits_{x\in\overline{\Omega}}\fint^{T}_{0}c(x,s)\,\mathrm{d}s =\sup_{\hat{c}\in\mathscr{C}}\min_{x\in\overline{\Omega}}\hat{c}(x).\]
\end{lemma}
\begin{proof}
We first derive the lower bound. The proof proceeds analogously to that of the lower bound in Lemma \ref{lemma7.5}. Therefore, in the interest of brevity, we confine ourselves to presenting the key steps and estimates, omitting the detailed calculations. We continue to use the positive constants such as $V_0$, $C$ and $M$ and the functions $\phi^*$ and $w$ that appeared in the proof of the lower bound of Lemma \ref{lemma7.5}. Let $\varepsilon\in(0,1)$ and $T>0$ be given. Set
 \[\overline{\varphi}_T(x,t):=\exp\left[\frac{1}{\sqrt{d }}{w}(x)+\frac{1}{\omega }\mathcal{M}(x,t;T)\right],\quad(x,t) \in\overline{\Omega}\times[0,T]\]
and
 \[\underline{\varphi}_T(x,t):=\varphi(x,t)\exp\left[\frac{1}{\omega }\int^{t}_{0}\left(\fint^{T}_{0} H(r;\omega ,d )\,\mathrm{d}r - H(s;\omega ,d ) +\varepsilon \right)\,\mathrm{d}s \right],\quad(x,t) \in\overline{\Omega}\times[0,T],\]
where $w$ is given in \eqref{eq7.10}. By direct computations, one easily checks that for any $T>2C_d\varepsilon^{-1}\omega$ ($C_d$ is the constant defined in \eqref{eq7.11}), the pair $(\underline{\varphi}_T,\overline{\varphi}_T)$ satisfies
\begin{equation*}
  \left\{
  \begin{aligned}
    &\omega \partial_t\overline{\varphi}_T-d\operatorname{div}(A(x)\nabla\overline{\varphi}_T) +c(x,t)\overline{\varphi}_T \ge\Bigg[\min_{x\in\overline{\Omega}}\fint^{T}_{0}c(x,s)\,\mathrm{d}s-\frac{d }{\omega }V_0 \\
    &\quad-\frac{\sqrt{d }}{M}\|\phi^*\|_{C^2(\overline{\Omega})}- \frac{\|\phi^*\|^{2}_{C^2(\overline{\Omega})}}{M^2} -\frac{2\sqrt{d }}{M\omega } V_0\|\phi^*\|_{C^2(\overline{\Omega})}-\frac{d }{\omega^{2}}V^{2}_{0} \Bigg]\overline{\varphi}_T,&&(x,t)\in\Omega\times(0,T],\\
    &\omega \partial_t\underline{\varphi}_T-d\operatorname{div}(A(x)\nabla\underline{\varphi}_T) +c(x,t)\underline{\varphi}_T = \left(\fint^{T}_{0}H(s;\omega,d)\,\mathrm{d}s+\varepsilon  \right)\underline{\varphi}_T,&&(x,t)\in\Omega\times(0,T],\\
    &\mathbf{n}\cdot(A(x)\nabla\overline{\varphi}_T)\\
    &~~=\frac{1}{\sqrt{d}}\left(-\frac{1}{M}\mathbf{n}\cdot(A(x)\nabla\phi^*) +\frac{\sqrt{d}}{\omega}\mathbf{n}\cdot(A(x)\nabla\mathcal{M}(x,t;T))\right)\overline{\varphi}_T\ge 0, &&(x,t)\in\partial\Omega\times(0,T],\\
    &\mathbf{n}\cdot(A(x)\nabla \underline{\varphi}_T)=0, &&(x,t)\in\partial\Omega\times(0,T],\\
    &\overline{\varphi}_T(x,0)=\overline{\varphi}_T(x,T) ,&&x\in\overline{\Omega},\\
    &\underline{\varphi}_T(x,0)=\varphi(x,0)\le\mathrm{e}^{\frac{ \varepsilon T}{\omega}-2C_d}\varphi(x,0)\le\varphi(x,T)\mathrm{e}^{\frac{ \varepsilon T}{\omega}}=\underline{\varphi}_T(x,T),&&x\in\overline{\Omega}.
  \end{aligned}
  \right.
\end{equation*}
Using Corollary \ref{corollary2.15}, we have
 \[
 \begin{aligned}
     & \liminf\limits_{T\to\infty}\fint^{T}_{0}H(s;\omega,d)\,\mathrm{d}s+\varepsilon \\
 \ge & \liminf_{T\to\infty}\min_{x\in\overline{\Omega}}\fint^{T}_{0}c(x,s)\,\mathrm{d}s-\frac{d }{\omega }V_0 -\frac{\sqrt{d }}{M}\|\phi^*\|_{C^2(\overline{\Omega})}- \frac{\|\phi^*\|^{2}_{C^2(\overline{\Omega})}}{M^2} -\frac{2\sqrt{d }}{M\omega } V_0\|\phi^*\|_{C^2(\overline{\Omega})}-\frac{d }{\omega^{2}}V^{2}_{0}
 \end{aligned}
 \]
and
 \[
 \begin{aligned}
     & \limsup_{T\to\infty}\fint^{T}_{0}H(s;\omega,d)\,\mathrm{d}s+\varepsilon \\
 \ge & \limsup_{T\to\infty}\min_{x\in\overline{\Omega}}\fint^{T}_{0}c(x,s)\,\mathrm{d}s-\frac{d }{\omega }V_0 -\frac{\sqrt{d }}{M}\|\phi^*\|_{C^2(\overline{\Omega})}- \frac{\|\phi^*\|^{2}_{C^2(\overline{\Omega})}}{M^2} -\frac{2\sqrt{d }}{M\omega } V_0\|\phi^*\|_{C^2(\overline{\Omega})}-\frac{d }{\omega^{2}}V^{2}_{0}.
 \end{aligned}
 \]
Lastly, by sending $\left(\omega,\frac{\omega}{\sqrt{d}}\right)\to (0,\infty)$ and the arbitrariness of $\varepsilon\in(0,1)$ and $M>2$, we show that
 \[\liminf\limits_{\left(\omega,\frac{\omega}{\sqrt{d}}\right)\to(0,\infty)} \liminf\limits_{T\to\infty}\fint^{T}_{0}H(s;\omega,d)\,\mathrm{d}s \ge\min_{x\in\overline{\Omega}}\liminf_{T\to\infty} \fint^{T}_{0}c(x,s)\,\mathrm{d}s\]
and
 \[\liminf\limits_{\left(\omega,\frac{\omega}{\sqrt{d}}\right)\to(0,\infty)} \limsup\limits_{T\to\infty}\fint^{T}_{0}H(s;\omega,d)\,\mathrm{d}s \ge \limsup_{T\to\infty} \min\limits_{x\in\overline{\Omega}}\fint^{T}_{0}c(x,s)\,\mathrm{d}s= \sup_{\hat{c}\in\mathscr{C}}\min_{x\in\overline{\Omega}}\hat{c}(x),\]
where we have used Lemma \ref{lemma2.3} in the last identity.

There exists a sequence $\{T_n\}_{n\in\mathbb{N}^+}$ of $T\to\infty$ such that
 \[\lim_{n\to\infty}\min_{x\in\overline{\Omega}}\fint^{T_n}_{0}c(x,s)\,\mathrm{d}s =\liminf_{T\to\infty}\min_{x\in\overline{\Omega}}\fint^{T}_{0}c(x,s)\,\mathrm{d}s.\]
Passing to a subsequence, we may assume that $\left\{\fint^{T_n}_{0}c(\cdot,s)\,\mathrm{d}s\right\}_{n\in\mathbb{N}^+}$ converges in $C(\overline{\Omega})$. By Corollary \ref{corollary2.6}, we have
 \[\lim_{n\to\infty}\min_{x\in\overline{\Omega}}\fint^{T_n}_{0}c(x,s)\,\mathrm{d}s =\min_{x\in\overline{\Omega}}\lim_{n\to\infty}\fint^{T_n}_{0}c(x,s)\,\mathrm{d}s.\]
Then one use Lemma \ref{lemma7.5} to obtain
 \begin{align*}
   \limsup\limits_{\left(\omega,\frac{\omega}{\sqrt{d}}\right)\to(0,\infty)} \liminf\limits_{T\to\infty}\fint^{T}_{0}H(s;\omega,d)\,\mathrm{d}s\le & \limsup\limits_{\left(\omega,\frac{\omega}{\sqrt{d}}\right)\to(0,\infty)} \liminf\limits_{n\to\infty}\fint^{T_n}_{0}H(s;\omega,d)\,\mathrm{d}s \\
    = & \min_{x\in\overline{\Omega}}\lim_{n\to\infty}\fint^{T_n}_{0}c(x,s)\,\mathrm{d}s = \min\limits_{x\in\overline{\Omega}}\liminf_{T\to\infty}\fint^{T}_{0}c(x,s)\,\mathrm{d}s.
 \end{align*}
This, combining with the lower bound, yields the first two identities in the lemma.

Let $\mu^\infty(d,\hat{c})$ be the principal eigenvalue of \eqref{eq1.3}. By Corollary \ref{corollary3.13}, one has
 \[\limsup\limits_{T\to\infty}\fint^{T}_{0}H(s;\omega,d)\,\mathrm{d}s\le \sup\limits_{\hat{c}\in\mathscr{C}}\mu^\infty(d,\hat{c}).\]
Letting $\left(\omega,\frac{\omega}{\sqrt{d}}\right)\to(0,\infty)$ in the inequality above yields
 \[\limsup_{\left(\omega,\frac{\omega}{\sqrt{d}}\right)\to(0,\infty)}\limsup\limits_{T\to\infty}\fint^{T}_{0}H(s;\omega,d)\,\mathrm{d}s\le \limsup_{d\to0}\sup\limits_{\hat{c}\in\mathscr{C}}\mu^\infty(d,\hat{c}).\]
It follows directly from Lemmas \ref{lemma2.18} and \ref{lemma2.3} that
 \[ \limsup_{\left(\omega,\frac{\omega}{\sqrt{d}}\right)\to(0,\infty)}\limsup_{T\to\infty} \fint^{T}_{0}H(s;\omega,d)\,\mathrm{d}s\le \limsup_{T\to\infty}\min_{x\in\overline{\Omega}} \fint^{T}_{0}c(x,s)\,\mathrm{d}s.\]
Combining with the lower bound yields the last two identities in the lemma. This completes the proof.
\end{proof}

\section{Open problems}
In their study of the asymptotic behavior of the principal eigenvalue for periodic parabolic operators as $(\omega, d) \to (0,0)$, Liu and Lou \cite{Liu2022Classifying} conducted a tripartite case analysis ((i) $(\omega,\omega/\sqrt{d})\to(0,0)$; (ii) $(\omega,\omega/\sqrt{d})\to(0,\vartheta)$ for some $\vartheta\in(0,\infty)$; (iii) $(\omega,\omega/\sqrt{d})\to(0,\infty)$), and obtained complete limiting results. We therefore conjecture that such a classification remains applicable in the general setting. In particular, in the second case, via formal computation, we observe that the limit of $-\sqrt{d}\ln \varphi$ appears to converge to a viscosity solution of the Hamilton-Jacobi equation
\begin{equation}\label{eq7.12}
  \left\{
  \begin{aligned}
    &\vartheta\partial_t U+\nabla U\cdot(A(x)\nabla U)=c(x,t)-H^*(t),\quad &&(x,t)\in\Omega\times\mathbb{R},\\
    &\mathbf{n}\cdot(A(x)\nabla U)=0,&&(x,t)\in\partial\Omega\times\mathbb{R}.
  \end{aligned}
  \right.
\end{equation}
Naturally, this leads to the following questions.
\begin{question}
  Under what conditions does the Hamilton-Jacobi equation \eqref{eq7.12} possess a unique Lipschitz continuous viscosity solution $U$ related to some (or a unique) function $H^*(t)$ defined on $\mathbb{R}$? Moreover, does the limit
  \[\lim_{\left(\omega,\frac{\omega}{\sqrt{d}}\right)\to(0,\vartheta)}H(t;\omega,d)= H^*(t;\vartheta),\qquad t\in\mathbb{R}\]
   remain valid?
\end{question}

\chapter{Almost periodic problem}\label{chp8}
Our results can be used to derive asymptotic properties of the spectrum for almost periodic parabolic operators. Specifically, if the potential function in equation \eqref{eq1.1} is uniformly almost periodic, then the long-time integral average limit exists. Moreover, we will show that by approximating the uniformly almost periodic function with a finite sum of periodic functions, assumptions \eqref{H2} and \eqref{H3} can be removed.

\begin{lemma}\label{lemma8.1}
  Let $c \in C^{2,1}(\overline{\Omega}\times\mathbb{R})$. If $c$ is a finite sum of time-periodic functions from this space, then $c$ satisfies hypotheses \eqref{H2} and \eqref{H3}, even if the periods of these functions differ.
\end{lemma}
\begin{proof}
It is evident that a single time-periodic function meets hypothesis \eqref{H3}, and the same holds for any finite sum of time-periodic functions. Consequently, it remains only to verify that such a finite sum also satisfies hypothesis \eqref{H2}.

First, the hypothesis \eqref{H2} holds for a single periodic function. Indeed, assume that $c(x,t)\in C^{2,1}(\overline{\Omega}\times\mathbb{R})$ is periodic in $t$ with period $p>0$. A direct calculation gives
\begin{align*}
  \mathcal{M}(x,t;T)=&t\fint^{T}_{0}c(x,s)\,\mathrm{d}s-\int^{t}_{0}c(x,s)\,\mathrm{d}s \\
  =& \frac{t}{T}\int^{\left\lfloor \frac{T}{p}\right\rfloor p}_{0}c(x,s)\,\mathrm{d}s +\frac{t}{T}\int^{T-\left\lfloor \frac{T}{p}\right\rfloor p}_{0}c(x,s)\,\mathrm{d}s -\int^{\left\lfloor \frac{t}{p}\right\rfloor p}_{0}c(x,s)\,\mathrm{d}s -\int^{t-\left\lfloor \frac{t}{p}\right\rfloor p}_{0}c(x,s)\,\mathrm{d}s \\
  =& \left(\frac{t}{T}\left\lfloor \frac{T}{p}\right\rfloor -\left\lfloor \frac{t}{p}\right\rfloor  \right)\int^{p}_{0}c(x,s)\,\mathrm{d}s +\frac{t}{T}\int^{T-\left\lfloor \frac{T}{p} \right\rfloor p}_{0}c(x,s)\,\mathrm{d}s -\int^{t-\left\lfloor \frac{t}{p}\right\rfloor p}_{0}c(x,s)\,\mathrm{d}s,
\end{align*}
where $\lfloor\cdot\rfloor$ stands the floor function. Observing that
 \[\left|\frac{t}{p}-\left\lfloor \frac{t}{p}\right\rfloor \right|<1,\quad \frac{t}{T}<1,\quad \left|\frac{T}{p}-\left\lfloor \frac{T}{p}\right\rfloor \right|<1\]
and
\[\frac{t}{T}\left\lfloor \frac{T}{p}\right\rfloor -\left\lfloor \frac{t}{p}\right\rfloor = \frac{t}{T}\cdot\frac{T}{p}+\frac{t}{T}\left(\left\lfloor \frac{T}{p}\right\rfloor -\frac{T}{p}\right)-\left\lfloor \frac{t}{p}\right\rfloor  =\left(\frac{t}{p}-\left\lfloor \frac{t}{p}\right\rfloor \right)-\frac{t}{T}\left(\frac{T}{p}-\left\lfloor \frac{T}{p}\right\rfloor \right),\]
we have
 \[-1<\frac{t}{T}\left\lfloor \frac{T}{p}\right\rfloor -\left\lfloor \frac{t}{p}\right\rfloor <1.\]
Therefore, one easily checks that $\mathcal{M}(x,t;T)$ satisfies the hypothesis \eqref{H2}.

Now, if $c\in C^{2,1}(\overline{\Omega}\times\mathbb{R})$ is the sum of finitely many periodic functions $c_1,c_2,\cdots, c_K$ with periods $p_1,p_2,\cdots, p_K$, respectively, where $K\ge 1$ is an integer, then
\[\mathcal{M}(x,t;T)=\sum^{K}_{k=1}\left(t\fint^{T}_{0}c_k(x,s)\,\mathrm{d}s-\int^{t}_{0}c_k(x,s)\,\mathrm{d}s\right) =\sum^{K}_{k=1}\mathcal{M}_k(x,t;T).\]
Obviously, every $\mathcal{M}_k(x,t;T)$ satisfies the hypothesis \eqref{H2}, $k=1,2,\cdots,K$. As a result, $\mathcal{M}(x,t;T)$ satisfies the hypothesis \eqref{H2} due to the finiteness of $K$. This finishes the proof.
\end{proof}

\medskip
In order to prove Theorem 1.10, we need further to prove two lemmas that serve as preparations.
\begin{lemma}\label{lemma8.2}
  For each $t\in\mathbb{R}$, let $\mu^0(t)$ be the principal eigenvalue of \eqref{eq1.4}. If $\{c(x,\cdot)\}_{x\in\overline{\Omega}}$ is a uniformly almost periodic family, then $\mu^0(\cdot)$ is almost periodic.
\end{lemma}
\begin{proof}
  For each $t\in\mathbb{R}$, let $(\mu^0(t),\phi^0(\cdot,t))$ be the principal eigenpair of
\begin{equation*}
  \left\{
  \begin{aligned}
    &-d\operatorname{div}(A(x)\nabla\phi^{0})+c(x,t)\phi^{0}=\mu^0(t)\phi^{0},&&x\in\Omega,\\
    &\mathbf{n}\cdot(A(x)\nabla\phi^{0})=0,&&x\in\partial\Omega,\\
    &\|\phi^{0}(\cdot,t)\|_{L^2(\Omega)}=1.
  \end{aligned}
  \right.
\end{equation*}
By the variation formula for the principal eigenvalue $\mu^0$, for any $t_1,t_2\in\mathbb{R}$ we have
\[\begin{aligned}
    \mu^0(t_1)=&\int_{\Omega}\left\{d\nabla\phi^0(x,t_1)\cdot(A(x)\nabla\phi^0(x,t_1))+c(x,t_1)[\phi^0(x,t_1)]^2\right\}\,\mathrm{d}x \\
    = &\int_{\Omega}\left\{d\nabla\phi^0(x,t_1)\cdot(A(x)\nabla\phi^0(x,t_1))+c(x,t_2)[\phi^0(x,t_1)]^2\right\}\,\mathrm{d}x \\
    &\quad +\int_{\Omega}\left\{[c(x,t_2)-c(x,t_1)][\phi^0(x,t_1)]^2\right\}\,\mathrm{d}x \\
    \le & \mu^0(t_2)+\|c(\cdot,t_2)-c(\cdot,t_1)\|_\infty,
  \end{aligned}
\]
as well as $\mu^0(t_2)\le \mu^0(t_1)+\|c(\cdot,t_2)-c(\cdot,t_1)\|_\infty$. Therefore, one has
\[|\mu^0(t_1)- \mu^0(t_2)|\le\|c(\cdot,t_2)-c(\cdot,t_1)\|_\infty.\]
Based on the estimate above, one can deduce the almost periodicity of $\mu^0(t)$ from the almost periodicity of $c(\cdot,t)$. This completes the proof.
\end{proof}
\begin{lemma}\label{lemma8.3}
 Let $c\in\mathcal{C}$ be a given function. If $\{c(x,\cdot)\}_{x\in\overline{\Omega}}$ is a uniformly almost periodic family, then $\min_{x\in\overline{\Omega}}c(x,\cdot)$ is almost periodic in $t\in\mathbb{R}$, and thus the limit $\lim_{T\to\infty}\fint^{T}_{0}\min_{x\in\overline{\Omega}}c(x,s)\,\mathrm{d}s$ exists.
\end{lemma}
\begin{proof}
Since the family of functions $\{c(x, \cdot)\}_{x \in \Omega}$ is uniformly almost periodic, that is, for any $\varepsilon > 0$, there exists a relatively dense set $\mathcal{I} \subset \mathbb{R}$ such that
  \[|c(x, t+\tau) - c(x,t)| < \varepsilon,\qquad \forall~\tau\in\mathcal{I},~x \in \overline{\Omega},~t\in\mathbb{R}.\]
Consider the function $c_m(t):= \min_{x\in\overline{\Omega}}c(x,t)$, $t\in\mathbb{R}$. From the continuity of $c$ and the compactness of $\overline{\Omega}$, we know that $c_m(t)$ is continuous in $t\in\mathbb{R}$. For the above $\tau$ and any $t$, let $x_m(t)$ satisfy $c_m(t)=c(x_m(t),t)$, $t\in\mathbb{R}$. Then
  \[c_m(t+\tau)\le c(x_m(t),t+\tau)<c(x_m(t),t)+\varepsilon=c_m(t)+\varepsilon, \qquad t\in\mathbb{R}.\]
Similarly, we have $c_m(t)<c_m(t+\tau)+\varepsilon$, $t\in\mathbb{R}$ and therefore
 \[|c_m(t+\tau) - c_m(t)| < \varepsilon,\qquad\forall~t\in\mathbb{R}.\]
That is, $c_m(t)$ is also almost periodic in $t\in\mathbb{R}$. Since an almost periodic function has a mean value (see \cite[Theorem 3.1]{Fink2006Almost}), then the limit
  \[\lim_{T\to\infty}\fint^{T}_{0}c_m(s)\,\mathrm{d}s\quad\text{exists}.\]
This finishes the proof.
\end{proof}

\medskip
\begin{proof}[Proof of Theorem \ref{theorem1.14}]
Firstly, note that if the function $c(x,t)$ is a uniformly almost periodic function, then the limit
\begin{equation}\label{eq8.1}
 \lim_{T\to\infty}\fint^{T}_{0}c(x,s)\,\mathrm{d}s\quad\text{ exists pointwise in }x\in\overline{\Omega};
\end{equation}
see \cite[Theorem 3.1]{Fink2006Almost}. Using the Bohr's approximation theorem (see, for example, \cite[Theorem 3.17]{Fink2006Almost}), we know that for any $\epsilon>0$, there is a trigonometric polynomial $\tilde{c}(x,t)$ such that for each $x\in\overline{\Omega}$,
  \[\|c(x,\cdot)-\tilde{c}(x,\cdot)\|_\infty<\epsilon,\]
or equivalently,
  \[\tilde{c}(x,t)-\epsilon<c(x,t)<\tilde{c}(x,t)+\epsilon,\quad\forall~t\in\mathbb{R}.\]
By arguments of approximation, we assume that the trigonometric polynomial $\tilde{c}$ belongs to $\mathcal{C}\cap C^{2,1}(\overline{\Omega}\times\mathbb{R})$; see, for example, \cite[Lemma 4.1]{Bai2020Asymptotic}. From \eqref{eq8.1}, we know that for each $x\in\overline{\Omega}$, it holds that the limit
 \[\lim_{T\to\infty}\fint^{T}_{0}\tilde{c}(x,s)\,\mathrm{d}s\quad\text{ exists pointwise in }x\in\overline{\Omega},\]
which is also continuous in $x\in\overline{\Omega}$ by Proposition \ref{proposition1.1}. Applying Corollary \ref{corollary3.8}, we obtain that
\begin{align*}
  \liminf\limits_{T\to\infty}\fint^{T}_{0}  H(s;d,\tilde{c})\,\mathrm{d}s- \epsilon \le & \liminf\limits_{T\to\infty}\fint^{T}_{0}H(s;d,c)\,\mathrm{d}s \\
  \le & \limsup\limits_{T\to\infty}\fint^{T}_{0}H(s;d,c)\,\mathrm{d}s= \limsup\limits_{T\to\infty}\fint^{T}_{0}H(s;d,\tilde{c})\,\mathrm{d}s+\epsilon.
\end{align*}
On the other hand, the trigonometric polynomial $\tilde{c}$, which is the sum of finitely many periodic functions, satisfies assumption \eqref{H2}; see Lemma \ref{lemma8.1}. By letting $d\to 0$ for fixed $\omega>0$ and using Theorem \ref{theorem1.5}, we have
 \[\lim\limits_{d\to0}\liminf\limits_{T\to\infty}\fint^{T}_{0}H(s;d,\tilde{c})\,\mathrm{d}s= \lim\limits_{d\to0}\limsup\limits_{T\to\infty}\fint^{T}_{0}H(s;d,\tilde{c})\,\mathrm{d}s= \min\limits_{x\in\overline{\Omega}}\lim\limits_{T\to\infty} \fint^{T}_{0}\tilde{c}(x,s)\,\mathrm{d}s.\]
Consequently, by the arbitrariness of $\epsilon>0$, we conclude that
 \[\lim\limits_{d\to0}\liminf\limits_{T\to\infty}\fint^{T}_{0}H(s;d,c)\,\mathrm{d}s= \lim\limits_{d\to0}\limsup\limits_{T\to\infty}\fint^{T}_{0}H(s;d,c)\,\mathrm{d}s= \min\limits_{x\in\overline{\Omega}}\lim\limits_{T\to\infty} \fint^{T}_{0}c(x,s)\,\mathrm{d}s.\]
This proves Theorem \ref{theorem1.14}(i), and similarly, one can verify Theorem \ref{theorem1.14}(iv)-(v) and (vii). Theorem \ref{theorem1.14}(ii) follows from Lemma \ref{lemma8.2}, Theorem \ref{theorem1.14}(vi) follows from Lemma \ref{lemma8.3} and Theorem \ref{theorem1.14}(iii) is obvious. This finishes the proof.
\end{proof}

\chapter{A diffusive SIS model in non-autonomous environments}\label{chp9}
In 2008, Allen et al. \cite{Allen2008Asymptotic} proposed an SIS epidemic model with standard incidence on a bounded domain $\Omega\subset\mathbb{R}^N~(N\ge 1)$, where the authors investigated the existence and uniqueness of the endemic equilibrium and its asymptotic behavior as the diffusion rate $d_S$ of susceptible individuals converges to zero. Their model reads
\begin{equation*}
  \left\{
  \begin{aligned}
    &\partial_tS-d_{S}\Delta S=-\frac{\beta(x)SI}{S+I} +\gamma(x)I,\qquad &&x\in\Omega,t>0,\\
    &\partial_tI-d_{I}\Delta I=\frac{\beta(x)SI}{S+I} -\gamma(x)I,\qquad &&x\in\Omega,t>0,\\
    &\partial_{\mathbf{n}}S=\partial_{\mathbf{n}}I=0,&&x\in\partial\Omega,t>0,\\
    &S(x,0)=S_{0}(x),I(x,0)=I_{0}(x),&&x\in\overline{\Omega},
  \end{aligned}
  \right.
\end{equation*}
where $S(x,t)$ and $I(x,t)$ represent the densities of susceptible and infected populations at location $x\in\overline{\Omega}$ and time $t\in[0,\infty)$ respectively. The positive constants $d_{S}$ and $d_{I}$ are diffusion rates of susceptible and infected individuals respectively, positive functions $\beta(x)$ and $\gamma(x)$ are H${\rm \ddot{o}}$lder continuous on $\overline{\Omega}$, which account for the rates of disease transmission and recovery at $x\in\overline{\Omega}$ respectively. Peng \cite{Peng2009Asymptotic} examined the effects of large and small diffusion rates of susceptible and infected population on the persistence and extinction of the disease, and analyzed the asymptotic behavior of the endemic equilibrium when the diffusion rates $d_S,d_I$ approach infinity or zero. Due to the seasonal fluctuation and periodic availability of vaccination strategies, Peng and Zhao \cite{Peng2012A} proposed an SIS model with time-periodic coefficients $\beta(\cdot,t)$ and $\gamma(\cdot,t)$, in which the qualitative properties of the basic reproduction number were analyzed and the epidemic dynamics were established. More recently, there have been a number of studies concerning the SIS system in many different settings: SIS models with standard/saturated/mass-action incidence\cite{Peng2009Asymptotic, Gao2024A, Wu2016Asymptotic, Peng2025Spatial, Peng2025Spatial2, Castellano2024Multiplicity}, concentration phenomena in advective environments\cite{Cui2017Dynamics, Cui2021Concentration, Kuto2017Concentration, Li2025Asymptotic}, SIS models in spatially-temporally heterogeneous environment\cite{Jiang2018A, Liu2022Classifying,Wang2015A}, SIS epidemic patch models \cite{Allen2007Asymptotic,Li2019Dynamics}. In all of these works, the basic reproduction number plays an essential role and there has been extensive research concerned with the properties of the basic reproduction number \cite{Cui2016A,Liu2022Classifying,Zhang2021Asymptotic,Wang2012Basic}. Obviously, the arguments do not work in the setting of a non-autonomous environment, since it seems that there is no corresponding eigenvalue problem if the coefficients depend on $t$ but they are not periodic in $t$.

Let us consider the following SIS epidemic model in non-autonomous environment
\begin{equation}\label{eq9.1}
 \left\{
  \begin{aligned}
    &\omega\partial_t S-d_S\Delta S=-\frac{\beta(x,t)SI}{S+I}+\gamma(x,t)I,&&(x,t)\in\Omega\times(0,\infty),\\
    &\omega\partial_t I-d_I\Delta I=\frac{\beta(x,t)SI}{S+I}-\gamma(x,t)I,&&(x,t)\in\Omega\times(0,\infty),\\
    &\partial_{\mathbf{n}}S=\partial_{\mathbf{n}}I=0,&&(x,t)\in\partial\Omega\times(0,\infty),\\
    &S(x,0)=S_0(x)\ge0,~I(x,0)=I_0(x)\ge,\not\equiv0,&&x\in\overline{\Omega}.
  \end{aligned}
 \right.
\end{equation}
In the current situation, the assumptions on parameters are the same or similar to those mentioned above, except that the transmission rate $\beta(x,t)$ and the recovery rate $\gamma(x,t)$ depend nontrivially on both $x\in\overline{\Omega}$ and $t\in[0,\infty)$, but may not be periodic in $t\in[0,\infty)$; and in addition, $0<\beta_*<\beta<\beta^*<\infty$ and $0<\gamma_*<\gamma<\gamma^*<\infty$ on $\overline{\Omega}\times[0,\infty)$ for some constant $\beta_*,\beta^*,\gamma_*$ and $\gamma^*$. The total population is constant, that is
\[\int_{\Omega}[S(x,t)+I(x,t)]\,\mathrm{d}x=\int_{\Omega}[S_0(x)+I_0(x)]\,\mathrm{d}x=N_{total}=\text{a positive constant},\qquad\forall ~t\in[0,\infty).\]
The solution of system \eqref{eq9.1} is uniformly bounded in $t\in[0,\infty)$.
\begin{lemma}[{\cite[Lemma 3.2]{Peng2012A}}]\label{lemma9.1}
  Let $(S,I)$ be the unique positive solution of \eqref{eq9.1}. Then there exists a positive constant $C$ independent of the initial data $(S_0,I_0)$ such that
   \[\limsup\limits_{t\to\infty}(\|S(\cdot,t)\|_{L^\infty(\Omega)}+\|I(\cdot,t)\|_{L^\infty(\Omega)})\le C .\]
\end{lemma}

Define two critical values $\underline{\mathcal{R}}_0$ and $\overline{\mathcal{R}}_0$ via the normalized principal Floquet bundle,
 \[\limsup\limits_{T\to\infty}\fint^{T}_{0}H\left(s;\omega,d_I,\gamma-\beta/\underline{\mathcal{R}}_0\right)\,\mathrm{d}s=0 \quad \text{ and } \quad \liminf\limits_{T\to\infty}\fint^{T}_{0}H\left(s;\omega,d_I,\gamma-\beta/\overline{\mathcal{R}}_0\right)\,\mathrm{d}s=0.\]
By virtue of monotonicity of $H$ with potential (see Theorem \ref{theorem1.3} and Lemma \ref{lemma3.5}), $\underline{\mathcal{R}}_0$ and $\overline{\mathcal{R}}_0$ are well-defined and uniquely determined. These two critical values are natural extensions of the basic reproduction number. Specifically, if these two numbers are equal, we can define the basic reproduction number of system \eqref{eq9.1} as $\mathcal{R}_0:=\underline{\mathcal{R}}_0=\overline{\mathcal{R}}_0$; for example, the situation where the potential is periodic in $t\in\mathbb{R}$. The normalized principal Floquet bundle $H(\cdot;\gamma-\beta)$ also plays an important role in studying the dynamics of system \eqref{eq9.1}, and the next assertion follows from the definitions of $\underline{\mathcal{R}}_0$ and $\overline{\mathcal{R}}_0$.
\begin{theorem}\label{theorem9.2}
The number $1-\underline{\mathcal{R}}_0$ has the same sign as $\limsup\limits_{T\to\infty}\fint^{T}_{0}H(s;\gamma-\beta)\,\mathrm{d}s$; the number $1-\overline{\mathcal{R}}_0$ has the same sign as $\liminf\limits_{T\to\infty}\fint^{T}_{0}H(s;\gamma-\beta)\,\mathrm{d}s$. Furthermore, if $\beta,\gamma$ satisfy uniform H\"{o}lder condition \eqref{H3}, then $\underline{\mathcal{R}}_0$ and $\overline{\mathcal{R}}_0$ are estimated by
  \[\liminf\limits_{T\to\infty}\frac{\int^{T}_{0}\int_{\Omega}\beta(x,s)\,\mathrm{d}x\mathrm{d}s }{\int^{T}_{0}\int_{\Omega}\gamma(x,s)\,\mathrm{d}x\mathrm{d}s}\le\underline{\mathcal{R}}_0 \le \sup_{x(\cdot)\in C(\mathbb{R};\overline{\Omega})}\liminf_{T\to\infty} \frac{\fint^{T}_{0} \beta(x(s),s)\,\mathrm{d}s}{\fint^{T}_{0}\gamma(x(s),s)\,\mathrm{d}s}\]
  and
   \[\limsup\limits_{T\to\infty}\frac{\int^{T}_{0}\int_{\Omega}\beta(x,s)\,\mathrm{d}x\mathrm{d}s }{\int^{T}_{0}\int_{\Omega}\gamma(x,s)\,\mathrm{d}x\mathrm{d}s}\le \overline{\mathcal{R}}_0\le \sup_{x(\cdot)\in C(\mathbb{R};\overline{\Omega})}\limsup_{T\to\infty} \frac{\fint^{T}_{0} \beta(x(s),s)\,\mathrm{d}s}{\fint^{T}_{0}\gamma(x(s),s)\,\mathrm{d}s}.\]
\end{theorem}

The lower bounds and upper bounds of $\underline{\mathcal{R}}_0$ and $\overline{\mathcal{R}}_0$ are achievable; see Theorem \ref{theorem9.4}. We now use the two critical values $\underline{\mathcal{R}}_0$ and $\overline{\mathcal{R}}_0$ to study the dynamics of system \eqref{eq9.1}.
\begin{theorem}\label{theorem9.3}
Let $(S,I)$ be the unique positive solution of \eqref{eq9.1}.
\begin{enumerate}[{\rm (i)}]
  \item If $\overline{\mathcal{R}}_0<1$, then
   \[(S,I)\to\left(\frac{N_{total}}{|\Omega|},0\right)\quad\text{in }C(\overline{\Omega})\text{ as }t\to\infty.\]
  \item Suppose that $d_S=d_I$. If $\underline{\mathcal{R}}_0>1$, then
   \[\liminf\limits_{T\to\infty}\fint^{T}_{0} \int_{\Omega}I(x,s)\,\mathrm{d}x\mathrm{d}s>0.\]
\end{enumerate}
\end{theorem}

In Theorem \ref{theorem9.3}(ii) we assume that $d_S = d_I$. For the general case, it is difficult to determine the persistence of the system; see Remark \ref{remark9.11}. In other special cases, however, such as (a) $\beta/\gamma = r(t)\in C([0,\infty);(1,\infty))$, or (b) fixed $d_I > 0$ with $d_S$ sufficiently large, one can prove that the persistence conclusion stated above continues to hold.

Furthermore, applying our theoretical results established in the previous chapters, we can obtain the asymptotic behavior of $\underline{\mathcal{R}}_0$ and $\overline{\mathcal{R}}_0$ with respect to the diffusion rate $d_I$ and the frequency $\omega$.
\begin{theorem}\label{theorem9.4}
  Assume that $\gamma,\beta\in\mathcal{C}$. The following statements are valid.
  \begin{enumerate}[{\rm (i)}]
  \item Let $\vartheta\in[0,\infty]$ be given. Assume in addition that $\beta,\gamma\in C^{2,1}(\overline{\Omega}\times\mathbb{R})$, with uniformly bounded time derivatives $(\sup_{t\in\mathbb{R}}(\|\partial_t \beta(\cdot,t)\|_{\infty}+\|\partial_t \gamma(\cdot,t)\|_{\infty})<\infty)$; impose \eqref{H2} in the subregime required by Theorem \ref{theorem1.9}(i). Then \[\lim\limits_{\left(d_I,\frac{\omega}{d_I}\right)\to(\infty,\vartheta)}\underline{\mathcal{R}}_0(\omega,d_I) =\liminf_{T\to\infty} \frac{\int^{T}_{0} \int_{\Omega} \beta(x,s)\,\mathrm{d}x\mathrm{d}s}{\int^{T}_{0}\int_{\Omega}\gamma(x,s) \,\mathrm{d}x\mathrm{d}s}\]
      and
      \[ \lim\limits_{\left(d_I,\frac{\omega}{d_I}\right)\to(\infty,\vartheta)} \overline{\mathcal{R}}_0(\omega,d_I)=\limsup_{T\to\infty} \frac{\int^{T}_{0} \int_{\Omega} \beta(x,s)\,\mathrm{d}x\mathrm{d}s}{\int^{T}_{0}\int_{\Omega}\gamma(x,s) \,\mathrm{d}x\mathrm{d}s}.\]
  \item Assume that $\beta,\gamma\in \mathscr{X}$, where $\mathscr{X}$ is defined in \eqref{eq1.5}. If $\omega\in(0,\infty)$ is given, then
      \[\lim\limits_{d_I\to0}\underline{\mathcal{R}}_0(d_I) =\liminf_{T\to\infty}\max\limits_{x\in\overline{\Omega}} \frac{\int^{T}_{0}  \beta(x,s)\,\mathrm{d}s}{\int^{T}_{0}\gamma(x,s)\,\mathrm{d}s}~\text{and}~ \lim\limits_{d_I\to0}\overline{\mathcal{R}}_0(d_I) =\limsup_{T\to\infty}\max\limits_{x\in\overline{\Omega}} \frac{\int^{T}_{0}  \beta(x,s)\,\mathrm{d}s}{\int^{T}_{0}\gamma(x,s)\,\mathrm{d}s}.\]
  \item Assume that $\beta,\gamma\in C^{0,1}(\overline{\Omega}\times\mathbb{R})$ satisfies $\sup_{t\in\mathbb{R}}(\|\partial_t \beta(\cdot,t)\|_{\infty}+\|\partial_t \gamma(\cdot,t)\|_{\infty})<\infty$. Let $r^{-}_{0}$ and $r^{+}_{0}$ be positive constants determined by
       \[\limsup\limits_{T\to\infty}\fint^{T}_{0}\mu^0\left(s;d_I,\gamma-\beta/r^{-}_{0}\right)\,\mathrm{d}s=0 \quad \text{ and } \quad \liminf\limits_{T\to\infty}\fint^{T}_{0}\mu^0\left(s;d_I,\gamma-\beta/r^{+}_{0}\right)\,\mathrm{d}s=0.\]
       Then
      \[\lim\limits_{\omega\to0}\underline{\mathcal{R}}_0(\omega)=r^{-}_{0}\quad\text{and}\quad \lim\limits_{\omega\to0}\overline{\mathcal{R}}_0(\omega)=r^{+}_{0}.\]
  \item Assume that $\beta,\gamma\in \mathscr{X}$. For a given $r>0$, set $\hat{c}_{T,r}(\cdot):=\fint^{T}_{0}\left[\gamma(\cdot,s)-\beta(\cdot,s)/r\right]\,\mathrm{d}s$ and define the set $\mathscr{C}_r$ of functions as
     \[\mathscr{C}_r:=\left\{\hat{c}\in C (\overline{\Omega})~|~\text{there exists }T_n\to\infty\text{ such that }\|\hat{c}_{T_n,r}-\hat{c}_r\|_{C(\overline{\Omega})}\to 0\right\}.\]
      Let $r^{-}_{\infty}$ and $r^{+}_{\infty}$ be positive constants such that
       \[\sup_{\hat{c}\in\mathscr{C}_{r^{-}_{\infty}}}\mu^\infty(\hat{c})=0\quad \text{ and } \quad \inf_{\hat{c}\in\mathscr{C}_{r^{+}_{\infty}}}\mu^\infty(\hat{c})=0 .\]
       Then
      \[\lim\limits_{\omega\to\infty}\underline{\mathcal{R}}_0(\omega)=r^{-}_{\infty}\quad\text{and}\quad \lim\limits_{\omega\to\infty}\overline{\mathcal{R}}_0(\omega)=r^{+}_{\infty}.\]
  \item Assume that $\beta,\gamma\in \mathscr{X}$, where $\mathscr{X}$ is defined in \eqref{eq1.5}. Then
      \[\lim\limits_{(\omega,d_I)\to(\infty,0)}\underline{\mathcal{R}}_0(\omega,d_I) =\liminf_{T\to\infty}\max\limits_{x\in\overline{\Omega}} \frac{\fint^{T}_{0}  \beta(x,s)\,\mathrm{d}s}{\fint^{T}_{0}\gamma(x,s)\,\mathrm{d}s}\]
      and
      \[\lim\limits_{(\omega,d_I)\to(\infty,0)}\overline{\mathcal{R}}_0(\omega,d_I) =\limsup_{T\to\infty}\max\limits_{x\in\overline{\Omega}} \frac{\fint^{T}_{0}  \beta(x,s)\,\mathrm{d}s}{\fint^{T}_{0}\gamma(x,s)\,\mathrm{d}s}.\]
  \item Assume that the functions $\beta,\gamma\in C^{2,1}(\overline{\Omega}\times\mathbb{R})$ and that both satisfy uniform H\"{o}lder condition \eqref{H3}.
      Then
      \[\lim\limits_{\left(d_I,\frac{\omega}{\sqrt{d_I}}\right)\to(0,0)}\underline{\mathcal{R}}_0(\omega,d_I)=\sup\limits_{x(\cdot)\in C(\mathbb{R};\overline{\Omega})}\liminf_{T\to\infty} \frac{\fint^{T}_{0}  \beta(x(s),s)\,\mathrm{d}s}{\fint^{T}_{0}\gamma(x(s),s)\,\mathrm{d}s}\]
      and
      \[\lim\limits_{\left(d_I,\frac{\omega}{\sqrt{d_I}}\right)\to(0,0)}\overline{\mathcal{R}}_0(\omega,d_I)=\sup\limits_{x(\cdot)\in C(\mathbb{R};\overline{\Omega})}\limsup_{T\to\infty} \frac{\fint^{T}_{0}  \beta(x(s),s)\,\mathrm{d}s}{\fint^{T}_{0}\gamma(x(s),s)\,\mathrm{d}s}.\]
  \item If $\beta,\gamma\in \mathscr{X}$, then
      \[\lim\limits_{\left(\omega,\frac{\omega}{\sqrt{d_I}}\right)\to(0,\infty)}\underline{\mathcal{R}}_0(\omega,d_I) =\liminf_{T\to\infty}\max\limits_{x\in\overline{\Omega}} \frac{\int^{T}_{0}\beta(x,s)\,\mathrm{d}s}{\int^{T}_{0}\gamma(x,s)\,\mathrm{d}s}\]
      and
      \[\lim\limits_{\left(\omega,\frac{\omega}{\sqrt{d_I}}\right)\to(0,\infty)} \overline{\mathcal{R}}_0(\omega,d_I)=\limsup_{T\to\infty}\max\limits_{x\in\overline{\Omega}} \frac{\int^{T}_{0}  \beta(x,s)\,\mathrm{d}s}{\int^{T}_{0}\gamma(x,s)\,\mathrm{d}s}.\]
\end{enumerate}
\end{theorem}

\section{Dynamics}
Extend the domains of functions $\beta(\cdot,t)$ and $\gamma(\cdot,t)$ to whole $\mathbb{R}$ by setting $\beta(\cdot,t)=\beta(\cdot,0)$ and $\gamma(\cdot,t)=\gamma(\cdot,0)$ if $t<0$, respectively. By  mollifying the extended functions $\beta$ and $\gamma$ near the $t=0$, we may assume that the extended functions $\beta,\gamma\in\mathcal{C}$. While this operation might slightly modify the function values of $\beta$ and $\gamma$ in the vicinity of $0^+$, Lemma \ref{lemma3.9} informs us that such an alteration has no influence on the principal Floquet exponent. Moreover, drawing upon Lemma \ref{lemma3.9}, we shall briefly elaborate in Remark \ref{remark9.5} on why distinct ways of extending backward in time likewise do not impact the forward long-time average quantities considered in this work. Under this view, from now on, we regard $\beta(\cdot,t)$ and $\gamma(\cdot,t)$ as the functions defined on whole axis $\mathbb{R}$ without further explicit mention. For any positive number $r$, let $(H(\cdot;\omega,d_I,r),\varphi(\cdot,\cdot;r))$ be the normalized principal Floquet bundle of
\begin{equation}\label{eq9.2}
  \left\{
  \begin{aligned}
    &\omega\partial_t \varphi-d_I\Delta\varphi+\left(\gamma(x,t)-\frac{\beta(x,t)}{r}\right)\varphi =H(t;\omega,d_I,r)\varphi,&&(x,t)\in\Omega\times\mathbb{R},\\
    &\partial_{\mathbf{n}}\varphi=0,&&(x,t)\in\partial\Omega\times\mathbb{R},\\
    &\varphi(x,t)>0,&&(x,t)\in\overline{\Omega}\times\mathbb{R},\\
    &\int_{\Omega}\varphi(x,t)\,\mathrm{d}x=1,&&t\in\mathbb{R}.
  \end{aligned}
  \right.
\end{equation}
By virtue of Lemma \ref{lemma3.5}, $H(t;\omega,d_I,r)\le0$ for each $t\in\mathbb{R}$ if $r\le\frac{\inf_{\Omega\times\mathbb{R}}\beta(x,t)}{\sup_{\Omega\times\mathbb{R}}\gamma(x,t)}$ and $H(t;\omega,d_I,r)\ge0$ for each $t\in\mathbb{R}$ if $r\ge\frac{\sup_{\Omega\times\mathbb{R}}\beta(x,t)}{\inf_{\Omega\times\mathbb{R}} \gamma(x,t)}$. Recall the definitions of $\underline{\mathcal{R}}_0$ and $\overline{\mathcal{R}}_0$ in Section \ref{sec1.1},
 \[\limsup\limits_{T\to\infty}\fint^{T}_{0}H(s;\underline{\mathcal{R}}_0)\,\mathrm{d}s=0 \quad \text{ and } \quad \liminf\limits_{T\to\infty}\fint^{T}_{0}H(s;\overline{\mathcal{R}}_0)\,\mathrm{d}s=0.\]
The quantities $\underline{\mathcal{R}}_0$ and $\overline{\mathcal{R}}_0$ are well-defined. In fact, their well-definedness is based on the following facts:
\begin{enumerate}[{\rm (a)}]
  \item $H$ depends smoothly on parameters $r\in(0,\infty)$ (Proposition \ref{proposition3.3});
  \item Both
        \[\overline{\mathcal{H}}(r;\omega,d_I):=\limsup_{T\to\infty}\fint^{T}_{0}H(s;\omega,d_I,r)\,\mathrm{d}s\quad\text{ and }\quad\underline{\mathcal{H}}(r;\omega,d_I):=\liminf_{T\to\infty}\fint^{T}_{0}H(s;\omega,d_I,r)\,\mathrm{d}s\]
        are increasing in $r$ (see Corollary \ref{corollary3.8});
  \item In addition, by \eqref{eq3.4}, for any two distinct numbers $r_1,r_2>0$, there hold
       \[|\overline{\mathcal{H}}(r_1;\omega,d_I)-\overline{\mathcal{H}}(r_2;\omega,d_I)|\le \|\beta\|_\infty\left|\frac{1}{r_1}-\frac{1}{r_2}\right|,\]
       and
       \[ |\underline{\mathcal{H}}(r_1;\omega,d_I)-\underline{\mathcal{H}}(r_2;\omega,d_I)| \le \|\beta\|_\infty\left|\frac{1}{r_1}-\frac{1}{r_2}\right|;\]
  \item By the strict positivity of $\beta$ and $\gamma$, one has $\overline{\mathcal{H}}(0+;\omega,d_I)=-\infty$ and $\overline{\mathcal{H}}(\infty;\omega,d_I)>\inf_{\Omega\times\mathbb{R}}\gamma$, and similarly $\underline{\mathcal{H}}(0+;\omega,d_I)=-\infty$ and $\underline{\mathcal{H}}(\infty;\omega,d_I)>\inf_{\Omega\times\mathbb{R}}\gamma$.
\end{enumerate}
Furthermore, one can derive the upper bounds of $\underline{\mathcal{R}}_0$ and $\overline{\mathcal{R}}_0$ immediately, that is,
\[\underline{\mathcal{R}}_0\le\overline{\mathcal{R}}_0\le \frac{\sup_{\Omega\times\mathbb{R}}\beta(x,t)}{\inf_{\Omega\times\mathbb{R}} \gamma(x,t)}.\]
Indeed, the upper bound will be improved; see Lemma \ref{lemma9.7}.

\begin{remark}\label{remark9.5}{\rm
The normalized principal Floquet bundle $(H,\varphi)$ of \eqref{eq9.2} may be different for different extensions of $\beta$ and $\gamma$ to the whole axis $\mathbb{R}$. However, the values $\underline{\mathcal{R}}_0$ and $\overline{\mathcal{R}}_0$ are independent of the ways of extending $\beta$ and $\gamma$ to $\mathbb{R}$. Denote $c=\gamma-\frac{\beta}{r}$ for a given $r>0$. Then we use notations $c_1$ and $c_2$ to represent two different extensions of $\beta$ and $\gamma$ to $\mathbb{R}$. Obviously, $c_1(\cdot,t)=c_2(\cdot,t)$ for all $t>0$ over $\overline{\Omega}$. By arguments of approximation for functions, we may assume $c_1,c_2\in\mathcal{C}$ and $\lim_{t\to\infty}\left\|c_1(\cdot,t)-c_2(\cdot,t)\right\|_{C(\overline{\Omega})}=0$. Then using Lemma \ref{lemma3.9} (also, Theorem \ref{theorem1.3}-(iii)), we see
 \[\lim\limits_{T\to\infty}\fint^{T}_{0}[H_1(s;r)-H_2(s;r)]\,\mathrm{d}s= 0,\]
where $H_1(t;r)$ and $H_2(t;r)$ are the normalized principal Floquet bundles of \eqref{eq9.2} corresponding to $c_1$ and $c_2$ respectively. As a result, the $\underline{\mathcal{R}}_0$ and $\overline{\mathcal{R}}_0$ are constants for different ways to extend $\beta$ and $\gamma$.
}
\end{remark}

Our first result gives the lower bounds of $\underline{\mathcal{R}}_0$ and $\overline{\mathcal{R}}_0$.
\begin{lemma}\label{lemma9.6}
The following estimates hold:
  \[\underline{\mathcal{R}}_0\ge\liminf\limits_{T\to\infty}\frac{\int^{T}_{0}\int_{\Omega}\beta(x,s)\,\mathrm{d}x\mathrm{d}s }{\int^{T}_{0}\int_{\Omega}\gamma(x,s)\,\mathrm{d}x\mathrm{d}s}\quad\text{ and }\quad \overline{\mathcal{R}}_0 \ge\limsup\limits_{T\to\infty}\frac{\int^{T}_{0}\int_{\Omega}\beta(x,s)\,\mathrm{d}x\mathrm{d}s }{\int^{T}_{0}\int_{\Omega}\gamma(x,s)\,\mathrm{d}x\mathrm{d}s}.\]
\end{lemma}
\begin{proof}
Let $(H(\cdot;\omega,d_I,r),\varphi)$ be the normalized Floquet bundle of \eqref{eq9.2}. Divide both sides of the identity by $\varphi$ satisfied by $(H(\cdot;\omega,d_I,r),\varphi)$ and then integrate over $\Omega$ by parts to obtain
  \[\omega\frac{\mathrm{d}}{\mathrm{d}t}\int_{\Omega} \ln\varphi\,\mathrm{d}x-d_I\int_{\Omega}|\nabla\ln\varphi|^2\,\mathrm{d}x+\int_{\Omega}\left(\gamma(x,t) -\frac{\beta(x,t)}{r}\right)\,\mathrm{d}x=|\Omega|H(t;\omega,d_I,r),\qquad t\in[0,\infty).\]
Since $-d_I\int_{\Omega}|\nabla\ln\varphi|^2\,\mathrm{d}x\le0$, it follows that
  \[\omega\frac{\mathrm{d}}{\mathrm{d}t}\int_{\Omega} \ln\varphi\,\mathrm{d}x+\int_{\Omega}\left(\gamma(x,t) -\frac{\beta(x,t)}{r}\right)\,\mathrm{d}x\ge|\Omega|H(t;\omega,d_I,r),\qquad t\in[0,\infty).\]
Integrating the above inequality over $[0,T]$, one has
 \[r\ge \frac{\int^{T}_{0}\int_{\Omega}\beta(x,s)\,\mathrm{d}x\mathrm{d}s}{ \omega\int_{\Omega}\ln\varphi\,\mathrm{d}x\Big|^{T}_{0} +\int^{T}_{0}\int_{\Omega}\gamma(x,s)\,\mathrm{d}x\mathrm{d}s- |\Omega|\int^{T}_{0}H(s;\omega,d_I,r)\,\mathrm{d}s}.\]
Substituting $r=\underline{\mathcal{R}}_0$ and $r=\overline{\mathcal{R}}_0$ respectively, and then letting $T\to\infty$, one easily obtains the lower bounds of $\underline{\mathcal{R}}_0$ and $\overline{\mathcal{R}}_0$. The proof is complete.
\end{proof}

\begin{lemma}\label{lemma9.7}
Assume that $\beta,\gamma$ satisfy uniform H\"{o}lder condition \eqref{H3}. The following estimates hold:
  \[\underline{\mathcal{R}}_0\le\sup_{x(\cdot)\in C(\mathbb{R};\overline{\Omega})}\liminf\limits_{T\to\infty}\frac{\int^{T}_{0}\beta(x(s),s) \,\mathrm{d}s }{\int^{T}_{0}\gamma(x(s),s)\,\mathrm{d}s}\quad\text{and}\quad \overline{\mathcal{R}}_0 \le\sup_{x(\cdot)\in C(\mathbb{R};\overline{\Omega})}\limsup\limits_{T\to\infty}\frac{\int^{T}_{0}\beta(x(s),s)\,\mathrm{d}s }{\int^{T}_{0}\gamma(x(s),s)\,\mathrm{d}s}.\]
\end{lemma}
\begin{proof}
By Lemma \ref{lemma3.5}, we have
 \[\min_{x\in\overline{\Omega}}\left(\gamma(x,t)-\frac{\beta(x,t)}{\underline{\mathcal{R}}_0}\right)\le H(t;\omega,d_I,\underline{\mathcal{R}}_0),\quad\forall~t\in\mathbb{R}\]
and
 \[\min_{x\in\overline{\Omega}}\left(\gamma(x,t)-\frac{\beta(x,t)}{\overline{\mathcal{R}}_0}\right)\le H(t;\omega,d_I,\overline{\mathcal{R}}_0),\quad\forall~t\in\mathbb{R}.\]
Averaging the above inequalities over $[0,T]$ yields
 \[\fint^{T}_{0}\min_{x\in\overline{\Omega}}\left(\gamma(x,s)-\frac{\beta(x,s)}{\underline{\mathcal{R}}_0}\right)\,\mathrm{d}s \le \fint^{T}_{0}H(s;\omega,d_I,\underline{\mathcal{R}}_0)\,\mathrm{d}s\]
and
 \[\fint^{T}_{0}\min_{x\in\overline{\Omega}}\left(\gamma(x,s)-\frac{\beta(x,s)}{\overline{\mathcal{R}}_0}\right)\,\mathrm{d}s \le \fint^{T}_{0}H(s;\omega,d_I,\overline{\mathcal{R}}_0)\,\mathrm{d}s.\]
Letting $T\to\infty$, we have
 \[\limsup_{T\to\infty}\fint^{T}_{0}\min_{x\in\overline{\Omega}}\left(\gamma(x,s)-\frac{\beta(x,s)}{\underline{\mathcal{R}}_0}\right)\,\mathrm{d}s \le0\quad\text{and}\quad \liminf_{T\to\infty}\fint^{T}_{0}\min_{x\in\overline{\Omega}}\left(\gamma(x,s)-\frac{\beta(x,s)}{\overline{\mathcal{R}}_0}\right)\,\mathrm{d}s \le0.\]
Using Lemma \ref{lemma2.10}, we have
 \[\inf_{x(\cdot)\in C(\mathbb{R};\overline{\Omega})}\limsup_{T\to\infty}\fint^{T}_{0}\left(\gamma(x(s),s)-\frac{\beta(x(s),s)}{\underline{\mathcal{R}}_0}\right)\,\mathrm{d}s \le0\]
and
 \[ \inf_{x(\cdot)\in C(\mathbb{R};\overline{\Omega})}\liminf_{T\to\infty}\fint^{T}_{0}\left(\gamma(x(s),s)-\frac{\beta(x(s),s)}{\overline{\mathcal{R}}_0}\right)\,\mathrm{d}s \le0.\]
For any $\epsilon>0$, there exists $x_\epsilon(\cdot)\in C(\mathbb{R};\overline{\Omega})$ such that
  \[\limsup_{T\to\infty}\fint^{T}_{0} \left[\gamma(x_\epsilon(s),s) -\frac{\beta(x_\epsilon(s),s)}{\underline{\mathcal{R}}_0}\right]\,\mathrm{d}s<\frac{\epsilon}{2}.\]
Moreover, we can find a $T_\epsilon>0$ such that
 \[\fint^{T}_{0}\gamma(x_\epsilon(s),s)\,\mathrm{d}s -\frac{1}{\underline{\mathcal{R}}_0}\fint^{T}_{0}\beta(x_\epsilon(s),s)\,\mathrm{d}s<\epsilon,\qquad\forall~T>T_\epsilon,\]
which implies that
 \[\frac{\fint^{T}_{0}\beta(x_\epsilon(s),s)\,\mathrm{d}s}{\fint^{T}_{0}\gamma(x_\epsilon(s),s)\,\mathrm{d}s }>\left(1-\frac{\epsilon}{\fint^{T}_{0}\beta(x_\epsilon(s),s)\,\mathrm{d}s}\right)\underline{\mathcal{R}}_0\ge \left(1-\frac{\epsilon}{\inf_{\overline{\Omega}\times\mathbb{R}}\beta}\right)\underline{\mathcal{R}}_0.\]
By sending $T\to\infty$, we get
\[\liminf_{T\to\infty}\frac{\fint^{T}_{0}\beta(x_\epsilon(s),s)\,\mathrm{d}s}{\fint^{T}_{0} \gamma(x_\epsilon(s),s)\,\mathrm{d}s }\ge \left(1-\frac{\epsilon}{\inf_{\overline{\Omega}\times\mathbb{R}}\beta}\right)\underline{\mathcal{R}}_0.\]
Since $\epsilon$ is arbitrary, we see that
 \[\underline{\mathcal{R}}_0\le \sup_{x(\cdot)\in C(\mathbb{R};\overline{\Omega})}\liminf_{T\to\infty} \frac{\fint^{T}_{0} \beta(x(s),s)\,\mathrm{d}s}{\fint^{T}_{0}\gamma(x(s),s)\,\mathrm{d}s}.\]
In the same way, we can show that
 \[\overline{\mathcal{R}}_0\le \sup_{x(\cdot)\in C(\mathbb{R};\overline{\Omega})}\limsup_{T\to\infty} \frac{\fint^{T}_{0} \beta(x(s),s)\,\mathrm{d}s}{\fint^{T}_{0}\gamma(x(s),s)\,\mathrm{d}s}.\]
This completes the proof.
\end{proof}

\medskip
Theorem \ref{theorem9.2} follows from Lemmas \ref{lemma9.6} and \ref{lemma9.7}. We now explore the long-time behavior of the positive solution $(S,I)$ of \eqref{eq9.1}.
\begin{lemma}\label{lemma9.8}
Let $(S,I)$ be the solution of \eqref{eq9.1}. Then
  \[\liminf\limits_{t\to\infty}\int_{\Omega}S(x,t)\,\mathrm{d}x\ge\frac{\gamma_*}{\gamma^*+\beta^*}N_{total},\]
  where $\gamma^*:=\sup\limits_{\Omega\times[0,\infty)}\gamma$, $\beta^*:=\sup\limits_{\Omega\times[0,\infty)}\beta$ and $\gamma_*:=\inf\limits_{\Omega\times[0,\infty)}\gamma$.
\end{lemma}
\begin{proof}
Add $\gamma S$ to both sides of the $S$-equation and integrate the resulting identity over $\Omega$ to obtain
  \[\omega\frac{\mathrm{d}}{\mathrm{d}t} \int_{\Omega}S\,\mathrm{d}x+\int_{\Omega}\left(\frac{\beta(x,t)I}{S+I}+\gamma(x,t)\right)S \,\mathrm{d}x= \int_{\Omega}\gamma(x,t)(S+I)\,\mathrm{d}x,\qquad t\in(0,\infty).\]
Then
  \[\omega\frac{\mathrm{d}}{\mathrm{d}t} \int_{\Omega}S(x,t)\,\mathrm{d}x+\left(\gamma^*+\beta^* \right)\int_{\Omega}S(x,t) \,\mathrm{d}x\ge\gamma_*N_{total},\qquad t\in(0,\infty),\]
which implies
  \[\omega\frac{\mathrm{d}}{\mathrm{d}t}\left(\mathrm{e}^{\frac{\gamma^*+\beta^*}{\omega}t}\int_{\Omega}S(x,t)\,\mathrm{d}x\right)\ge N_{total}\gamma_*\mathrm{e}^{\frac{\gamma^*+\beta^*}{\omega}t},\qquad t\in(0,\infty).\]
Integrating the above inequality over $[0,T]$ and by some simple computations, we have
  \[\int_{\Omega}S(x,T)\,\mathrm{d}x\ge\frac{\gamma_*}{\gamma^*+\beta^*}N_{total}+\mathrm{e}^{-\frac{\gamma^*+\beta^*}{\omega}T}\left(\int_{\Omega} S(x,0)\,\mathrm{d}x-\frac{\gamma_*}{\gamma^*+\beta^*}N_{total}\right).\]
Consequently, sending $T\to\infty$, we obtain
  \[\liminf\limits_{T\to\infty}\int_{\Omega}S(x,T)\,\mathrm{d}x\ge\frac{\gamma_*}{\gamma^*+\beta^*}N_{total}.\]
This proves the lemma.
\end{proof}

\begin{lemma}[Extinction]\label{lemma9.9}
Let $(S,I)$ be the unique positive solution of \eqref{eq9.1}. If $\overline{\mathcal{R}}_0<1$, then
   \[(S,I)\to\left(\frac{N_{total}}{|\Omega|},0\right)\quad\text{in }C(\overline{\Omega})\text{ as }t\to\infty,\]
where $N_{total}>0$ is the total population.
\end{lemma}
\begin{proof}
  Let $(H_1,\varphi)$ be the normalized principal Floquet bundle of \eqref{eq9.2} with $r=1$. Set $u(x,t)=M\mathrm{e}^{-\int^{t}_{0}H_1(s)\,\mathrm{d}s}\varphi(x,t)$ for $(x,t)\in\overline{\Omega}\times[0,\infty)$, where $M>0$ is a positive number determined later. Then $u$ satisfies
  \begin{equation*}
 \left\{
  \begin{aligned}
    &\omega\partial_t u-d_I\Delta u+\left(\gamma(x,t)-\beta(x,t)\right)u=0,&&(x,t)\in\Omega\times(0,\infty),\\
    &\partial_{\mathbf{n}}u=0,&&(x,t)\in\partial\Omega\times(0,\infty),\\
    &u(x,0)=M\varphi(x,0)\ge,\not\equiv0,&&x\in\overline{\Omega}.
  \end{aligned}
 \right.
\end{equation*}
  On the other hand, $I$ satisfies
  \begin{equation*}
 \left\{
  \begin{aligned}
    &\omega\partial_t I-d_I\Delta I+\left(\gamma(x,t)-\beta(x,t)\right)I\le0,&&(x,t)\in\Omega\times(0,\infty),\\
    &\partial_{\mathbf{n}}I=0,&&(x,t)\in\partial\Omega\times(0,\infty),\\
    &I(x,0)=I_0(x)\ge,\not\equiv0,&&x\in\overline{\Omega}.
  \end{aligned}
 \right.
\end{equation*}
Hence, for any given initial data $I_0\ge 0$, one can choose $M$ large enough such that $u(x,0)\ge I_0(x)$ on $\overline{\Omega}$. By the classical comparison principle for parabolic equation, we conclude that $I(x,t)\le u(x,t)$ on $\overline{\Omega}\times[0,\infty)$, which implies
 \[\lim\limits_{t\to\infty}I(\cdot,t)\le\lim\limits_{t\to\infty}M\mathrm{e}^{-\int^{t}_{0}H_1(s)\,\mathrm{d}s} \varphi(\cdot,t)=0\qquad\text{uniformly on }\overline{\Omega},\]
where the fact that $\int^{t}_{0}H_1(s)\,\mathrm{d}s\to\infty$ as $t\to\infty$ if $\overline{\mathcal{R}}_0<1$ has been used. Then, with a little modification to the proof of \cite[Lemma 2.5]{Allen2008Asymptotic}, one can show $S\to\frac{N_{total}}{|\Omega|}$ in $C(\overline{\Omega})$. This ends the proof.
\end{proof}

\medskip
Next, we turn to investigate the persistence of system \eqref{eq9.1}.

\begin{lemma}[Weak persistence with $d_S=d_I$]\label{lemma9.10}
Let $(S,I)$ be the unique positive solution of \eqref{eq9.1} and $d:=d_S=d_I$. If $\underline{\mathcal{R}}_0>1$, then
   \[\liminf\limits_{T\to\infty}\fint^{T}_{0} \int_{\Omega}I(x,s)\,\mathrm{d}x\mathrm{d}s>0.\]
\end{lemma}
\begin{proof}
Let $(H_1,\varphi)$ be the normalized principal Floquet bundle of \eqref{eq9.2} with $r=1$. By Theorem \ref{theorem9.2}, we know $\limsup\limits_{T\to\infty}\fint^{T}_{0}H_1(s)\,\mathrm{d}s<0$ if $\underline{\mathcal{R}}_0>1$. One can check that
\begin{equation*}
 \left\{
  \begin{aligned}
    &\omega\partial_t (S+I)-d\Delta (S+I)=0,&&(x,t)\in\Omega\times(0,\infty),\\
    &\partial_{\mathbf{n}}(S+I)=0,&&(x,t)\in\partial\Omega\times(0,\infty).
  \end{aligned}
 \right.
\end{equation*}
Since $S+I$ is a nonnegative solution to the heat equation with homogeneous Neumann boundary conditions, the parabolic Harnack's inequality \cite[Theorem 1.2.6]{Lam2022Introduction} applies, giving
\[\sup\limits_{x\in\Omega}[S(x,t)+I(x,t)]\le C\inf\limits_{x\in\Omega}[S(x,t)+I(x,t)],\quad\text{for each }t>1,\]
where $C$ is a positive constant independent of $t$ and may vary from line to line. As a consequence,
\begin{equation}\label{eq9.3}
  \int_{\Omega}\frac{I}{S+I}\,\mathrm{d}x=\frac{1}{N_{total}}\int_\Omega \frac{\int_\Omega(S+I)\,\mathrm{d}x}{S+I}I\,\mathrm{d}x\le C\int_{\Omega}I\,\mathrm{d}x.
\end{equation}

Let $(H_1,\psi)$ be the adjoint normalized principal Floquet bundle satisfying
\begin{equation}\label{eq9.4}
  \left\{
  \begin{aligned}
    &-\omega\partial_t \psi-d_I\Delta\psi+\left(\gamma(x,t)-\beta(x,t)\right)\psi=H_1(t)\psi,&&(x,t)\in\Omega\times\mathbb{R},\\
    &\partial_{\mathbf{n}}\psi=0,&&(x,t)\in\partial\Omega\times\mathbb{R}.
  \end{aligned}
  \right.
\end{equation}
Multiply the $I$-equation by $\psi$ firstly and then integrate the resulting equation over $\Omega$ by parts to obtain
\[\omega\frac{\mathrm{d}}{\mathrm{d}t} \int_{\Omega}I\psi\,\mathrm{d}x+\int_{\Omega}\left(-\omega\partial_t \psi-d_I\Delta\psi+(\gamma(x,t)-\beta(x,t)) \psi\right)I\,\mathrm{d}x=-\int_{\Omega}\frac{\beta(x,t)I}{S+I}\psi I\,\mathrm{d}x\]
for all $t\in(0,\infty)$. Dividing both sides by $\int_{\Omega}\psi I\,\mathrm{d}x$ and combining with \eqref{eq9.3} and \eqref{eq9.4}, we have the following estimate
\begin{align*}
  \omega\frac{\mathrm{d}}{\mathrm{d}t} \left(\ln\int_{\Omega}\psi I\,\mathrm{d}x\right)+H_1(t)=-\int_{\Omega}\frac{\beta(x,t)\psi  I}{\int_{\Omega}I\psi\,\mathrm{d}x}\frac{I}{S+I}\,\mathrm{d}x\ge -C\int_{\Omega}\frac{I}{S+I}\,\mathrm{d}x
  \ge-C\int_{\Omega}I\,\mathrm{d}x,
\end{align*}
where the Harnack's inequality has been used for $I$ and $\psi$ in the first inequality and $C$ is some positive constant independent of $t$ which may differ at each occurrence. Then taking the integral average over $[0,T]$ yields
\begin{equation}\label{eq9.5}
\begin{aligned}
  \fint^{T}_{0}H_1(s)\,\mathrm{d}s
  \ge&-\frac{\omega}{T}\left.\ln\int_{\Omega}\psi I\,\mathrm{d}x\right|^{T}_{t=0}-C\fint^{T}_{0}\left(\int_{\Omega}I\,\mathrm{d}x\right)^{1/2}\,\mathrm{d}s\\
  \ge&-\frac{\omega}{T}\left. \ln\int_{\Omega}\psi I\,\mathrm{d}x\right|^{T}_{t=0}-C\left(\fint^{T}_{0}\int_{\Omega}I\,\mathrm{d}x\mathrm{d}s\right)^{1/2}.
\end{aligned}
\end{equation}
Using the Harnack principle for $\psi$ and $I$, we know that $\ln\int_{\Omega}\psi(x,T) I(x,T)\,\mathrm{d}x$ is uniformly bounded from above for all $T>0$. Therefore,
\[\liminf\limits_{T\to\infty}\left. \left(-\frac{\omega}{T}\ln\int_{\Omega}\psi I\,\mathrm{d}x\right)\right|^{T}_{t=0}\ge 0.\]
On the other hand, if $\liminf\limits_{T\to\infty}\fint^{T}_{0}\int_{\Omega}I\,\mathrm{d}x\mathrm{d}s=0$, then there exists a sequence $\{T_n\}_{n\in\mathbb{N}^+}$ of $T\to\infty$ such that
\[\lim\limits_{n\to\infty}\fint^{T_n}_{0}\int_{\Omega}I(x,s)\,\mathrm{d}x\mathrm{d}s=0.\]
Letting $n\to\infty$ with $T=T_n$ in \eqref{eq9.5}, one gets
\[\liminf\limits_{n\to\infty}\fint^{T_n}_{0}H_1(s)\,\mathrm{d}s\ge0,\]
which leads at a contradiction to $\limsup\limits_{T\to\infty}\fint^{T}_{0}H_1(s)\,\mathrm{d}s<0$.
  This finishes the proof.
\end{proof}

\begin{remark}\label{remark9.11}{\rm
  A key step in the proof of Lemma \ref{lemma9.10} is the estimate \eqref{eq9.3}. Precisely, if \[\liminf\limits_{T\to\infty}\fint^{T}_{0} \int_{\Omega}I(x,s)\,\mathrm{d}x\mathrm{d}s=0 \text{ implies }\liminf\limits_{T\to\infty}\fint^{T}_{0} \int_{\Omega}\frac{I}{S+I}\,\mathrm{d}x\mathrm{d}s=0,\]
  then the weak persistence follows. Unfortunately, we ran into trouble when proving it when the condition $d_S=d_I$ is removed.
  }
\end{remark}

\medskip
Theorem \ref{theorem9.3} follows from Lemmas \ref{lemma9.9} and \ref{lemma9.10}.

\section{Asymptotic behavior of the two critical numbers}
We conclude by investigating the asymptotics of $\underline{\mathcal{R}}_0$ and $\overline{\mathcal{R}}_0$. First, the solutions to a number of limiting equations are studied (Lemmas \ref{lemma9.12}-\ref{lemma9.14}). It is subsequently proven that the limits of $\underline{\mathcal{R}}_0$ and $\overline{\mathcal{R}}_0$ in various scaling regimes are precisely given by these solutions.
\begin{lemma}\label{lemma9.12}
The positive numbers
 \[\underline{r}:=\liminf_{T\to\infty} \frac{\int^{T}_{0} \int_{\Omega} \beta(x,s)\,\mathrm{d}x\mathrm{d}s }{\int^{T}_{0}\int_{\Omega}\gamma(x,s)\,\mathrm{d}x\mathrm{d}s}\quad\text{and}\quad \overline{r}:= \limsup_{T\to\infty} \frac{\int^{T}_{0} \int_{\Omega} \beta(x,s)\,\mathrm{d}x\mathrm{d}s}{\int^{T}_{0} \int_{\Omega}\gamma(x,s)\,\mathrm{d}x\mathrm{d}s}\]
are unique solutions of problems
 \[\limsup\limits_{T\to\infty} \fint^{T}_{0}\fint_{\Omega}\left(\gamma(x,s)- \frac{\beta(x,s)}{r}\right)\,\mathrm{d}x\mathrm{d}s=0 \quad\text{and}\quad \liminf\limits_{T\to\infty} \fint^{T}_{0}\fint_{\Omega}\left(\gamma(x,s)- \frac{\beta(x,s)}{r}\right)\,\mathrm{d}x\mathrm{d}s=0,\]
respectively.
\end{lemma}
\begin{proof}
  For $r\in (0,\infty)$, we define
 \[\overline{f}(r):=\limsup\limits_{T\to\infty} \fint^{T}_{0}\fint_{\Omega}\left(\gamma(x,s)- \frac{\beta(x,s)}{r}\right)\,\mathrm{d}x\mathrm{d}s\]
and
 \[\underline{f}(r):=\liminf\limits_{T\to\infty} \fint^{T}_{0}\fint_{\Omega}\left(\gamma(x,s)- \frac{\beta(x,s)}{r}\right)\,\mathrm{d}x\mathrm{d}s,\]
respectively. Apparently, both $\overline{f}(r)$ and $\underline{f}(r)$ are strictly increasing in $r\in(0,\infty)$. Direct computations yield that
\[\overline{f}(r)\to-\infty\quad\text{and}\quad\underline{f}(r)\to-\infty\quad\text{ as }r\to0,\]
while
\[\overline{f}(r)\to\limsup_{T\to\infty}\fint^{T}_{0}\fint_\Omega\gamma(x,s)\,\mathrm{d}s\quad \text{and}\quad \underline{f}(r)\to\liminf_{T\to\infty}\fint^{T}_{0}\fint_\Omega\gamma(x,s)\,\mathrm{d}s\quad\text{ as }r\to\infty.\]
So the solutions of $\overline{f}(r)=0$ and $\underline{f}(r)=0$ exist uniquely, respectively. It remains to check that
 \[\limsup\limits_{T\to\infty} \fint^{T}_{0}\fint_{\Omega}\left(\gamma(x,s)- \frac{\beta(x,s)}{\underline{r}}\right)\,\mathrm{d}x\mathrm{d}s=0\quad\text{and}\quad \liminf\limits_{T\to\infty} \fint^{T}_{0}\fint_{\Omega}\left(\gamma(x,s)- \frac{\beta(x,s)}{\overline{r}}\right)\,\mathrm{d}x\mathrm{d}s=0.\]
 Indeed, by the definition of $\underline{r}$, we know that there is a sequence $\{T_n\}_{n\in\mathbb{N}^+}$ of $T\to\infty$ such that
  \[\underline{r}=\lim_{n\to\infty} \frac{\int^{T_n}_{0} \int_{\Omega} \beta(x,s)\,\mathrm{d}x\mathrm{d}s }{\int^{T_n}_{0}\int_{\Omega}\gamma(x,s)\,\mathrm{d}x\mathrm{d}s},\quad\text{and }\lim\limits_{n\to\infty} \fint^{T_n}_{0}\fint_{\Omega}\gamma(x,s)\,\mathrm{d}x\mathrm{d}s,~ \lim\limits_{n\to\infty}\fint^{T_n}_{0}\fint_{\Omega}\beta(x,s)\,\mathrm{d}x\mathrm{d}s\text{ exist}.\]
 Therefore, we easily check that
  \[\limsup\limits_{T\to\infty} \fint^{T}_{0}\fint_{\Omega}\left(\gamma(x,s)-\frac{\beta(x,s)}{ \underline{r}}\right)\,\mathrm{d}x\mathrm{d}s\ge \lim\limits_{n\to\infty} \fint^{T_n}_{0}\fint_{\Omega}\left(\gamma(x,s)-\frac{\beta(x,s)}{ \underline{r}}\right)\,\mathrm{d}x\mathrm{d}s=0.\]
 Arguing by a contradiction, we assume that the inequality holds strictly. Then there is a sequence $\{\tilde{T}_n\}_{n\in\mathbb{N}^+}$ of $T\to\infty$ such that the limits $\lim\limits_{n\to\infty} \fint^{\tilde{T}_n}_{0}\fint_{\Omega}\gamma(x,s)\,\mathrm{d}x\mathrm{d}s$ and $\lim\limits_{n\to\infty}\fint^{\tilde{T}_n}_{0}\fint_{\Omega}\beta(x,s)\,\mathrm{d}x\mathrm{d}s$ exist, and
\[\lim\limits_{n\to\infty} \fint^{\tilde{T}_n}_{0}\fint_{\Omega}\gamma(x,s)\,\mathrm{d}x\mathrm{d}s -\frac{1}{ \underline{r}}\lim\limits_{n\to\infty}\fint^{\tilde{T}_n}_{0} \fint_{\Omega}\beta(x,s)\,\mathrm{d}x\mathrm{d}s>0.\]
Equivalently, one has
\[\underline{r}>\lim\limits_{n\to\infty}\frac{\int^{\tilde{T}_n}_{0}\int_{\Omega}\beta(x,s)\,\mathrm{d}x\mathrm{d}s}{ \int^{\tilde{T}_n}_{0}\int_{\Omega}\gamma(x,s)\,\mathrm{d}x\mathrm{d}s},\]
which contradicts the definition of $\underline{r}$. Similarly, we conclude
\[\liminf\limits_{T\to\infty} \fint^{T}_{0}\fint_{\Omega}\left(\gamma(x,s)- \frac{\beta(x,s)}{ \overline{r}}\right)\,\mathrm{d}x\mathrm{d}s=0.\]
This completes the proof.
\end{proof}

\begin{lemma}\label{lemma9.13}
Let $\beta,\gamma\in C(\overline{\Omega}\times[0,\infty))$ be uniformly bounded and uniformly positive. Define
 \[\underline{R}:=\liminf_{T\to\infty}\max_{x\in\overline{\Omega}} \frac{\fint_{0}^{T}\beta(x,s)\,\mathrm{d}s}{\fint_{0}^{T}\gamma(x,s)\,\mathrm{d}s},\qquad  \overline{R}:=\limsup_{T\to\infty}\max_{x\in\overline{\Omega}} \frac{\fint_{0}^{T}\beta(x,s)\,\mathrm{d}s}{\fint_{0}^{T}\gamma(x,s)\,\mathrm{d}s}.\]
\begin{enumerate}[{\rm (i)}]
 \item If a constant $\underline{r}>0$ satisfies
   \[\limsup_{T\to\infty}\min_{x\in\overline{\Omega}}    \fint^{T}_{0}\left(\gamma(x,s)-\frac{\beta(x,s)}{\underline{r}}\right)\,\mathrm{d}s=0,\]
  then $\underline{r}=\underline{R}$.
 \item If a constant $\overline{r}>0$ satisfies
   \[\liminf_{T\to\infty}\min_{x\in\overline{\Omega}}    \fint^{T}_{0}\left(\gamma(x,s)-\frac{\beta(x,s)}{\overline{r}}\right)\,\mathrm{d}s=0,\]
  then $\overline{r}=\overline{R}$.
\end{enumerate}
\end{lemma}
\begin{proof}
 Recall the notation
  \[\gamma_*=\inf_{(x,t)\in\overline{\Omega}\times\mathbb{R}}\gamma(x,t)>0.\]
Define functions $\overline{f},\underline{f}:~(0,\infty)\to\mathbb{R}$ by
\begin{align*} \overline{f}(r)&:=\limsup_{T\to\infty}\min_{x\in\overline{\Omega}} \fint^{T}_{0}\left(\gamma(x,s)-\frac{\beta(x,s)}{r}\right)\,\mathrm{d}s,\\ \underline{f}(r)&:=\liminf_{T\to\infty}\min_{x\in\overline{\Omega}} \fint^{T}_{0}\left(\gamma(x,s)-\frac{\beta(x,s)}{r}\right)\,\mathrm{d}s.
\end{align*}
Both $\overline{f}(r)$ and $\underline{f}(r)$ are strictly increasing, and
\begin{align*}
&\lim_{r\to0^+}\overline{f}(r)=\lim_{r\to0^+}\underline{f}(r)=-\infty,\\
&\lim_{r\to\infty}\overline{f}(r)=\limsup_{T\to\infty}\min_{x\in\overline{\Omega}} \fint^{T}_{0}\gamma(x,s)\,\mathrm{d}s,\\
&\lim_{r\to\infty}\underline{f}(r)=\liminf_{T\to\infty}\min_{x\in\overline{\Omega}} \fint^{T}_{0}\gamma(x,s)\,\mathrm{d}s .
\end{align*}
Consequently, the equations $\overline{f}(r)=0$ and $\underline{f}(r)=0$ possess unique positive solutions, which we denote by $\underline{r}$ and $\overline{r}$ respectively. It remains to prove that $\underline{r}=\underline{R}$ and $\overline{r}=\overline{R}$.

\medskip\noindent
\textbf{Step 1: $\underline{r}\ge \underline{R}$.}
Suppose, for contradiction, that $\underline{r}<\underline{R}$ and put $\varepsilon_1=(\underline{R}-\underline{r})/2>0$. By the definition of $\underline{R}$, for any sequence $\{T_n\}_{n\in\mathbb{N}^+}$, it holds
 \[\max_{x\in\overline{\Omega}} \frac{\fint_{0}^{T_n}\beta(x,s)\,\mathrm{d}s}      {\fint_{0}^{T_n}\gamma(x,s)\,\mathrm{d}s} >\underline{R}-\varepsilon_1 =\frac{\underline{r}+\underline{R}}{2}.\]
For each $n\in\mathbb{N}^+$, choose $x_n\in\overline{\Omega}$ attaining the above maximum. Then
 \[\frac{\fint_{0}^{T_n}\beta(x_n,s)\,\mathrm{d}s}{\fint_{0}^{T_n}\gamma(x_n,s)\,\mathrm{d}s} >\frac{\underline{r}+\underline{R}}{2}>\underline{r},\]
equivalently,
 \[\fint_{0}^{T_n}\beta(x_n,s)\,\mathrm{d}s >\underline{r}\fint_{0}^{T_n}\gamma(x_n,s)\,\mathrm{d}s.\]
Dividing by $\underline{r}$ and rearranging give
 \[\fint^{T_n}_{0} \left(\gamma(x_n,s)-\frac{\beta(x_n,s)}{\underline{r}}\right)\,\mathrm{d}s<0 .\]
Using the lower bound $(\fint^{T_n}_{0}\beta(x_n,s)\,\mathrm{d}s>\frac{\underline{r}+\underline{R}}{2}\fint^{T_n}_{0}\gamma(x_n,s)\,\mathrm{d}s)$, we obtain
 \[\fint^{T_n}_{0} \left(\gamma(x_n,s)-\frac{\beta(x_n,s)}{\underline{r}}\right)\,\mathrm{d}s <\fint^{T_n}_{0}\gamma(x_n,s)\,\mathrm{d}s\; \frac{\underline{r}-\underline{R}}{2\underline{r}}.\]
Because of $\gamma\ge\gamma_*>0$, we have $\fint^{T_n}_{0}\gamma(x_n,s)\,\mathrm{d}s\ge\gamma_*$, and since $\underline{r}-\underline{R}<0$,
 \[\fint^{T_n}_{0} \left(\gamma(x_n,s)-\frac{\beta(x_n,s)}{\underline{r}}\right)\,\mathrm{d}s \le\gamma_*\frac{\underline{r}-\underline{R}}{2\underline{r}}<0.\]
Hence
 \[\min_{x\in\overline{\Omega}}\fint^{T_n}_{0} \left(\gamma(x,s)-\frac{\beta(x,s)}{\underline{r}}\right)\,\mathrm{d}s \le\gamma_*\frac{\underline{r}-\underline{R}}{2\underline{r}}<0 .\]
Letting $n\to\infty$ yields
 \[\overline{f}(\underline{r}) =\limsup_{T\to\infty}\min_{x\in\overline{\Omega}} \fint^{T}_{0}\left(\gamma(x,s)-\frac{\beta(x,s)}{\underline{r}}\right)\,\mathrm{d}s \le\gamma_*\frac{\underline{r}-\underline{R}}{2\underline{r}}<0,\]
which contradicts $\overline{f}(\underline{r})=0$. Thus $\underline{r}\ge \underline{R}$.

\medskip\noindent
\textbf{Step 2: $\underline{r}\le \underline{R}$.} Assume that $\underline{r}>\underline{R}$ and set $\varepsilon_2=(\underline{r}-\underline{R})/2>0$. By the definition of $\underline{R}$, there exists a sequence $T_n\to\infty$ with
 \[\max_{x\in\overline{\Omega}} \frac{\fint_{0}^{T_n}\beta(x,s)\,\mathrm{d}s}      {\fint_{0}^{T_n}\gamma(x,s)\,\mathrm{d}s} <\underline{R}+\varepsilon_2 =\frac{\underline{r}+\underline{R}}{2}.\]
Consequently, for every $x\in\overline{\Omega}$,
 \[\fint_{0}^{T_n}\beta(x,s)\,\mathrm{d}s <\frac{\underline{r}+\underline{R}}{2}\fint_{0}^{T_n}\gamma(x,s)\,\mathrm{d}s.\]
Dividing by $\underline{r}$ and rearranging give
 \[\fint^{T_n}_{0} \left(\gamma(x,s)-\frac{\beta(x,s)}{\underline{r}}\right)\,\mathrm{d}s >\fint^{T_n}_{0}\gamma(x,s)\,\mathrm{d}s\; \frac{\underline{r}-\underline{R}}{2\underline{r}}.\]
Since $\gamma\ge\gamma_*>0$ and $\underline{r}-\underline{R}>0$, one has
 \[\fint^{T_n}_{0} \left(\gamma(x,s)-\frac{\beta(x,s)}{\underline{r}}\right)\,\mathrm{d}s \ge\gamma_*\frac{\underline{r}-\underline{R}}{2\underline{r}}>0,\qquad \forall x\in\overline{\Omega}.\]
Taking the minimum over $x\in\overline{\Omega}$ and then the limit as $n\to\infty$ gives
 \[\overline{f}(\underline{r}) \ge\gamma_*\frac{\underline{r}-\underline{R}}{2\underline{r}}>0,\]
again contradicting $\overline{f}(\underline{r})=0$. Hence $\underline{r}\le \underline{R}$.

\medskip\noindent
\textbf{Step 3: $\overline{r}\le \overline{R}$.} Suppose that $\overline{r}>\overline{R}$ and put $\varepsilon_1=(\overline{r}-\overline{R})/2>0$. By the definition of $\overline{R}$, there exists $T_0>0$ such that for all $T>T_0$,
 \[\max_{x\in\overline{\Omega}} \frac{\fint_{0}^{T}\beta(x,s)\,\mathrm{d}s}      {\fint_{0}^{T}\gamma(x,s)\,\mathrm{d}s} <\overline{R}+\varepsilon_1 =\frac{\overline{r}+\overline{R}}{2}.\]
Thus for every $x\in\overline{\Omega}$ and every $T>T_0$,
 \[\fint_{0}^{T}\beta(x,s)\,\mathrm{d}s <\frac{\overline{r}+\overline{R}}{2}\fint_{0}^{T}\gamma(x,s)\,\mathrm{d}s.\]
Dividing by $\overline{r}$ and rearranging yield
 \[\fint^{T}_{0} \left(\gamma(x,s)-\frac{\beta(x,s)}{\overline{r}}\right)\,\mathrm{d}s >\fint^{T}_{0}\gamma(x,s)\,\mathrm{d}s\; \frac{\overline{r}-\overline{R}}{2\overline{r}} .\]
Because $\gamma\ge\gamma_*>0$ and $\overline{r}-\overline{R}>0$,
 \[\fint^{T}_{0} \left(\gamma(x,s)-\frac{\beta(x,s)}{\overline{r}}\right)\,\mathrm{d}s \ge\gamma_*\frac{\overline{r}-\overline{R}}{2\overline{r}}>0,\qquad \forall x\in\overline{\Omega},\;T>T_0.\]
Taking the minimum over x and then the limit as $T\to\infty$ gives
\[\underline{f}(\overline{r}) \ge\gamma_*\frac{\overline{r}-\overline{R}}{2\overline{r}}>0,\]
which contradicts $\underline{f}(\overline{r})=0$. Therefore $\overline{r}\le \overline{R}$.

\medskip\noindent
\textbf{Step 4: $\overline{r}\ge \overline{R}$.} Assume that $\overline{r}<\overline{R}$ and set $\varepsilon_2=(\overline{R}-\overline{r})/2>0$. By the definition of $\overline{R}$, there exists a sequence $\{T_n\}_{n\in\mathbb{N}^+}$ such that
 \[\max_{x\in\overline{\Omega}} \frac{\int_0^{T_n}\beta(x,s)\,\mathrm{d}s}      {\int_0^{T_n}\gamma(x,s)\,\mathrm{d}s} >\overline{R}-\varepsilon_2 =\frac{\overline{r}+\overline{R}}{2}.\]
For each $n\in\mathbb{N}^+$, choose $x_n\in\overline{\Omega}$ attaining the above maximum. Then
 \[\frac{\fint_{0}^{T_n}\beta(x_n,s)\,\mathrm{d}s}{\fint_{0}^{T_n}\gamma(x_n,s)\,\mathrm{d}s} >\frac{\overline{r}+\overline{R}}{2}>\overline{r},\]
equivalently,
 \[\fint_{0}^{T_n}\beta(x_n,s)\,\mathrm{d}s >\overline{r}\fint_{0}^{T_n}\gamma(x_n,s)\,\mathrm{d}s .\]
Dividing by $\overline{r}$ and rearranging give
 \[\fint^{T_n}_{0} \left(\gamma(x_n,s)-\frac{\beta(x_n,s)}{\overline{r}}\right)\,\mathrm{d}s<0 .\]
Using the lower bound $(\int_0^{T_n}\beta(x_n,s)\,\mathrm{d}s> \frac{\overline{r}+\overline{R}}{2}\int_0^{T_n}\gamma(x_n,s)\,\mathrm{d}s)$, we obtain
 \[\fint^{T_n}_{0} \left(\gamma(x_n,s)-\frac{\beta(x_n,s)}{\overline{r}}\right)\,\mathrm{d}s <\fint^{T_n}_{0}\gamma(x_n,s)\,\mathrm{d}s\; \frac{\overline{r}-\overline{R}}{2\overline{r}} .\]
Since $\gamma\ge\gamma_*>0$, we have $\fint^{T_n}_{0}\gamma(x_n,s)\,\mathrm{d}s\ge\gamma_*$, and because of $\overline{r}-\overline{R}<0$,
 \[\fint^{T_n}_{0} \left(\gamma(x_n,s)-\frac{\beta(x_n,s)}{\overline{r}}\right)\,\mathrm{d}s \le\gamma_*\frac{\overline{r}-\overline{R}}{2\overline{r}}<0 .\]
Hence
 \[\min_{x\in\overline{\Omega}}\fint^{T_n}_{0} \left(\gamma(x,s)-\frac{\beta(x,s)}{\overline{r}}\right)\,\mathrm{d}s \le\gamma_*\frac{\overline{r}-\overline{R}}{2\overline{r}}<0 .\]
Letting $n\to\infty$ yields
 \[\underline{f}(\overline{r}) \le\gamma_*\frac{\overline{r}-\overline{R}}{2\overline{r}}<0,\]
which contradicts $\underline{f}(\overline{r})=0$. Thus $\overline{r}\ge \overline{R}$.

Combining Steps 1 and 2 we obtain $\underline{r}=\underline{R}$. Combining Steps 3 and 4 we obtain $\overline{r}=\overline{R}$. This completes the proof.
\end{proof}

\begin{lemma}\label{lemma9.14}
Let $\beta,\gamma\in\mathcal{C}$ be given functions. If $\underline{r}$ and $\overline{r}$ are the positive constants such that
  \[\inf_{x(\cdot)\in C(\mathbb{R};\overline{\Omega})}\limsup_{T\to\infty}\fint^{T}_{0} \left(\gamma(x(s),s) -\frac{\beta(x(s),s)}{\underline{r}}\right)\,\mathrm{d}s=0\]
and
  \[\inf_{x(\cdot)\in C(\mathbb{R};\overline{\Omega})}\liminf_{T\to\infty}\fint^{T}_{0} \left(\gamma(x(s),s)-\frac{\beta(x(s),s)}{\overline{r}}\right)\,\mathrm{d}s=0,\]
then
  \[\underline{r}=\sup_{x(\cdot)\in C(\mathbb{R};\overline{\Omega})}\liminf_{T\to\infty} \frac{\fint^{T}_{0} \beta(x(s),s)\,\mathrm{d}s}{\fint^{T}_{0}\gamma(x(s),s)\,\mathrm{d}s} \quad{and}\quad \overline{r}=\sup_{x(\cdot)\in C(\mathbb{R};\overline{\Omega})} \limsup_{T\to\infty}\frac{\fint^{T}_{0} \beta(x(s),s)\,\mathrm{d}s}{\fint^{T}_{0} \gamma(x(s),s)\,\mathrm{d}s}.\]
\end{lemma}
\begin{proof}
We first show that
\begin{equation}\label{eq9.6}
 \underline{r}\ge \sup_{x(\cdot)\in C(\mathbb{R};\overline{\Omega})}\liminf_{T\to\infty} \frac{\fint^{T}_{0} \beta(x(s),s)\,\mathrm{d}s}{\fint^{T}_{0}\gamma(x(s),s)\,\mathrm{d}s}.
\end{equation}
In fact, if
   \[\sup_{x(\cdot)\in C(\mathbb{R};\overline{\Omega})}\liminf_{T\to\infty} \frac{\fint^{T}_{0} \beta(x(s),s)\,\mathrm{d}s}{\fint^{T}_{0}\gamma(x(s),s)\,\mathrm{d}s}>\underline{r},\]
then there is a curve $x_0(\cdot)\in C(\mathbb{R};\overline{\Omega})$ such that
   \[\underline{R}:=\liminf_{T\to\infty} \frac{\fint^{T}_{0} \beta(x_0(s),s)\,\mathrm{d}s}{\fint^{T}_{0} \gamma(x_0(s),s)\,\mathrm{d}s}>\underline{r}.\]
We set $\varepsilon:=\frac{\underline{R}-\underline{r}}{2}$. Then there exists a constant $T_0>0$ such that
 \[\frac{\fint^{T}_{0} \beta(x_0(s),s)\,\mathrm{d}s}{\fint^{T}_{0} \gamma(x_0(s),s)\,\mathrm{d}s}>\underline{R}-\varepsilon=\frac{\underline{R}+\underline{r}}{2},\qquad\forall~T>T_0.\]
It follows that
 \[\fint^{T}_{0}\left[\gamma(x_0(s),s)-\frac{\beta(x_0(s),s)}{\underline{r}}\right]\,\mathrm{d}s< \frac{\underline{r}-\underline{R}}{2\underline{r}}\fint^{T}_{0}\gamma(x_0(s),s)\,\mathrm{d}s\le- \frac{\varepsilon}{\underline{r}}\inf_{\overline{\Omega}\times\mathbb{R}}\gamma.\]
By letting $T\to\infty$, we have
 \[\limsup_{T\to\infty}\fint^{T}_{0}\left[\gamma(x_0(s),s)-\frac{\beta(x_0(s),s)}{\underline{r}} \right]\,\mathrm{d}s \le-\frac{\varepsilon}{\underline{r}}\inf_{\overline{\Omega}\times\mathbb{R}}\gamma<0.\]
This is impossible, due to the definition of $\underline{r}$. We obtain \eqref{eq9.6}.

Now, by the definition of $\underline{r}$, we know that for any $\epsilon>0$, there is a curve $x_\epsilon(\cdot)\in C(\mathbb{R};\overline{\Omega})$ such that
  \[\limsup_{T\to\infty}\fint^{T}_{0} \left[\gamma(x_\epsilon(s),s) -\frac{\beta(x_\epsilon(s),s)}{\underline{r}}\right]\,\mathrm{d}s<\frac{\epsilon}{2}.\]
Hence, there is a constant $T_\epsilon>0$ such that
 \[\fint^{T}_{0}\gamma(x_\epsilon(s),s)\,\mathrm{d}s -\frac{1}{\underline{r}}\fint^{T}_{0}\beta(x_\epsilon(s),s)\,\mathrm{d}s<\epsilon,\qquad\forall~T>T_\epsilon,\]
or equivalently,
 \[\frac{\fint^{T}_{0}\beta(x_\epsilon(s),s)\,\mathrm{d}s}{\fint^{T}_{0}\gamma(x_\epsilon(s),s)\,\mathrm{d}s }>\left(1-\frac{\epsilon}{\fint^{T}_{0}\beta(x_\epsilon(s),s)\,\mathrm{d}s}\right)\underline{r}\ge \left(1-\frac{\epsilon}{\inf_{\overline{\Omega}\times\mathbb{R}}\beta}\right)\underline{r}.\]
Letting $T\to\infty$ yields
\[\liminf_{T\to\infty}\frac{\fint^{T}_{0}\beta(x_\epsilon(s),s)\,\mathrm{d}s}{\fint^{T}_{0} \gamma(x_\epsilon(s),s)\,\mathrm{d}s }\ge \left(1-\frac{\epsilon}{\inf_{\overline{\Omega}\times\mathbb{R}}\beta}\right)\underline{r}.\]
Since $\epsilon$ is arbitrary, we see that
 \[\underline{r}\le \sup_{x(\cdot)\in C(\mathbb{R};\overline{\Omega})}\liminf_{T\to\infty} \frac{\fint^{T}_{0} \beta(x(s),s)\,\mathrm{d}s}{\fint^{T}_{0}\gamma(x(s),s)\,\mathrm{d}s}.\]
This, together with \eqref{eq9.6}, yields
 \[ \underline{r}= \sup_{x(\cdot)\in C(\mathbb{R};\overline{\Omega})}\liminf_{T\to\infty} \frac{\fint^{T}_{0} \beta(x(s),s)\,\mathrm{d}s}{\fint^{T}_{0}\gamma(x(s),s)\,\mathrm{d}s}.\]

Similarly, one can verify the second identity in the lemma. This finishes the proof.
\end{proof}

\begin{lemma}\label{lemma9.15}
Let $\vartheta\in[0,\infty]$ be given. Assume $\beta,\gamma\in\mathcal{C}\cap C^{2,1}(\overline{\Omega}\times\mathbb{R})$, with uniformly bounded time derivatives, and impose \eqref{H2} whenever required by Theorem \ref{theorem1.9}(i). There hold
 \[\lim\limits_{{\left(d_I,\frac{\omega}{d_I}\right)\to(\infty,\vartheta)}}\underline{\mathcal{R}}_0(\omega,d_I) =\liminf_{T\to\infty} \frac{\int^{T}_{0} \int_{\Omega} \beta(x,s)\,\mathrm{d}x\mathrm{d}s}{\int^{T}_{0}\int_{\Omega}\gamma(x,s)\,\mathrm{d}x\mathrm{d}s}\]
      and
 \[ \lim\limits_{{\left(d_I,\frac{\omega}{d_I}\right)\to(\infty,\vartheta)}} \overline{\mathcal{R}}_0(\omega,d_I)=\limsup_{T\to\infty} \frac{\int^{T}_{0} \int_{\Omega} \beta(x,s)\,\mathrm{d}x\mathrm{d}s}{\int^{T}_{0}\int_{\Omega}\gamma(x,s)\,\mathrm{d}x\mathrm{d}s}.\]
\end{lemma}
\begin{proof}
Denote by $H(\cdot;\omega,d_I,\underline{r})$ and $H(\cdot;\omega,d_I,\overline{r})$ the normalized principal Floquet bundles of \eqref{eq9.2} respectively, where the constants $\underline{r}$ and $\overline{r}$ are given by
 \[\underline{r}=\liminf_{T\to\infty} \frac{\int^{T}_{0} \int_{\Omega} \beta(x,s)\,\mathrm{d}x\mathrm{d}s}{\int^{T}_{0}\int_{\Omega}\gamma(x,s)\,\mathrm{d}x\mathrm{d}s} \quad\text{and}\quad \overline{r}=\limsup_{T\to\infty} \frac{\int^{T}_{0} \int_{\Omega} \beta(x,s)\,\mathrm{d}x\mathrm{d}s}{\int^{T}_{0}\int_{\Omega}\gamma(x,s)\,\mathrm{d}x\mathrm{d}s}.\]
By Theorem \ref{theorem1.9}(i) and Lemma \ref{lemma9.12}, one has
 \[\lim\limits_{{\left(d_I,\frac{\omega}{d_I}\right)\to(\infty,\vartheta)}} \limsup\limits_{T\to\infty}\fint^{T}_{0}H(s;\omega,d_I,\underline{r})\,\mathrm{d}s=\limsup\limits_{T\to\infty} \fint^{T}_{0}\fint_{\Omega}\left(\gamma(x,s)-\frac{\beta(x,s)}{ \underline{r}}\right)\,\mathrm{d}x\mathrm{d}s=0\]
and
 \[\lim\limits_{{\left(d_I,\frac{\omega}{d_I}\right)\to(\infty,\vartheta)}} \liminf\limits_{T\to\infty}\fint^{T}_{0}H(s;\omega,d_I,\overline{r})\,\mathrm{d}s=\liminf\limits_{T\to\infty} \fint^{T}_{0}\fint_{\Omega}\left(\gamma(x,s)-\frac{\beta(x,s)}{ \overline{r}}\right)\,\mathrm{d}x\mathrm{d}s=0.\]

Now we show that
\begin{equation}\label{eq9.7}
 \lim\limits_{{\left(d_I,\frac{\omega}{d_I}\right)\to(\infty,\vartheta)}} \underline{\mathcal{R}}_0(\omega,d_I)=\underline{r}\quad\text{ and }\quad \lim\limits_{{\left(d_I,\frac{\omega}{d_I}\right)\to(\infty,\vartheta)}} \overline{\mathcal{R}}_0(\omega,d_I)=\overline{r}.
\end{equation}
Suppose to the contrary that there exists a sequence $\{(\omega_k,d_k)\}$ satisfying $d_k\to\infty$ and $\frac{\omega_k}{d_k}\to\vartheta\in[0,\infty]$ as $k\to\infty$ such that
\[\lim_{k\to\infty}\underline{\mathcal{R}}_0(\omega_k,d_k)=\underline{r}_0\neq\underline{r}\qquad\text{ for some }\underline{r}_0>0.\]
We claim that
\begin{equation}\label{eq9.8}
  \lim_{k\to\infty}\limsup_{T\to\infty}\fint^{T}_{0}H(s;\omega_k,d_k,\underline{r}_0)\,\mathrm{d}s=0.
\end{equation}
Indeed, for any given $\varepsilon>0$, there is an integer $K_\varepsilon$ such that
\[\left|\frac{1}{\underline{\mathcal{R}}_0(\omega_k,d_k)}-\frac{1}{\underline{r}_0}\right| <\frac{\varepsilon}{\|\beta\|_\infty},\qquad\forall~k>K_\varepsilon.\]
Thus, for all $k>K_\varepsilon$, we have
\begin{align*}
    & \limsup_{T\to\infty}\fint^{T}_{0}H\left(s;\omega_k,d_k,\left(\gamma-\frac{\beta}{\underline{r}_0}\right)\right)\,\mathrm{d}s \\
  = & \limsup_{T\to\infty}\fint^{T}_{0}H\left(s;\omega_k,d_k,\left(\gamma-\frac{\beta}{\underline{\mathcal{R}}_0}+ \frac{\beta}{\underline{\mathcal{R}}_0}-\frac{\beta}{\underline{r}_0}\right)\right)\,\mathrm{d}s\\
  \le& \limsup_{T\to\infty}\fint^{T}_{0}H\left(s;\omega_k,d_k,\left(\gamma-\frac{\beta}{\underline{\mathcal{R}}_0}+ \varepsilon\right)\right)\,\mathrm{d}s \\
  =& \limsup_{T\to\infty}\fint^{T}_{0}H\left(s;\omega_k,d_k,\left(\gamma-\frac{\beta}{\underline{\mathcal{R}}_0}\right)\right)\,\mathrm{d}s + \varepsilon\\
  = &\varepsilon.
\end{align*}
Similarly, we can show that
\[\limsup_{T\to\infty}\fint^{T}_{0}H\left(s;\omega_k,d_k,\left(\gamma-\frac{\beta}{\underline{r}_0}\right)\right)\,\mathrm{d}s\ge-\varepsilon.\]
This proves \eqref{eq9.8}. On the other hand, by Lemma \ref{lemma9.6}, $r_0\ge r$, and thus $r_0\neq \underline{r}$ implies that $r_0>\underline{r}$. Direct calculation yields
\begin{align*}
    & \lim_{k\to\infty}\limsup_{T\to\infty}\fint^{T}_{0}H(s;\omega_k,d_k,\underline{r}_0)\,\mathrm{d}s \\
  = & \limsup_{T\to\infty}\fint^{T}_{0}\fint_\Omega\left(\gamma-\frac{\beta}{\underline{r}_0}\right)\,\mathrm{d}x\mathrm{d}s \\
  = & \limsup_{T\to\infty}\fint^{T}_{0}\fint_\Omega\left[\left(\gamma-\frac{\beta}{\underline{r}}\right) +\beta\left(\frac{1}{\underline{r}}-\frac{1}{\underline{r}}_0\right)\right]\,\mathrm{d}x\mathrm{d}s\\
  \ge & \limsup_{T\to\infty}\fint^{T}_{0}\fint_\Omega\left(\gamma-\frac{\beta}{\underline{r}}\right) \,\mathrm{d}x\mathrm{d}s+\left(\frac{1}{\underline{r}}-\frac{1}{\underline{r}}_0\right)\inf_{\Omega\times\mathbb{R}}\beta \\
  = &\left(\frac{1}{\underline{r}}-\frac{1}{\underline{r}}_0\right)\inf_{\Omega\times\mathbb{R}}\beta>0.
\end{align*}
This is impossible, and the first identity in \eqref{eq9.7} is true. Similarly, one can show the second one in \eqref{eq9.7}. The proof is complete.
\end{proof}

\begin{lemma}\label{lemma9.16}
Assume that $\beta,\gamma\in \mathscr{X}$, where $\mathscr{X}$ is defined in \eqref{eq1.5}. Then
  \[\lim\limits_{d_I\to0}\underline{\mathcal{R}}_0(d_I)=\liminf_{T\to\infty} \max\limits_{x\in\overline{\Omega}}\frac{\int^{T}_{0}  \beta(x,s)\,\mathrm{d}s}{\int^{T}_{0}\gamma(x,s)\,\mathrm{d}s}\quad\text{and}\quad \lim\limits_{d_I\to0}\overline{\mathcal{R}}_0(d_I)=\limsup_{T\to\infty} \max\limits_{x\in\overline{\Omega}}\frac{\int^{T}_{0}  \beta(x,s)\,\mathrm{d}s}{\int^{T}_{0}\gamma(x,s)\,\mathrm{d}s}.\]
\end{lemma}
\begin{proof}
We define
 \[\underline{r}=\liminf_{T\to\infty}\max\limits_{x\in\overline{\Omega}} \frac{\int^{T}_{0}  \beta(x,s)\,\mathrm{d}s}{\int^{T}_{0}\gamma(x,s)\,\mathrm{d}s}\quad\text{and}\quad \overline{r}=\limsup_{T\to\infty}\max\limits_{x\in\overline{\Omega}} \frac{\int^{T}_{0}  \beta(x,s)\,\mathrm{d}s}{\int^{T}_{0}\gamma(x,s)\,\mathrm{d}s}.\]
It follows from Corollary \ref{corollary1.12}, Lemma \ref{lemma2.4} and Lemma \ref{lemma9.13} that
 \[\lim_{d_I\to 0} \limsup_{T\to\infty}\fint^{T}_{0}H(s;d_I,\underline{r})\,\mathrm{d}s=\limsup_{T\to\infty}\min_{x\in\overline{\Omega}} \fint^{T}_{0}\left(\gamma(x,s)-\frac{\beta(x,s)}{\underline{r}}\right)\,\mathrm{d}s=0\]
and
 \[\lim_{d_I\to 0} \liminf_{T\to\infty}\fint^{T}_{0}H(s;d_I,\overline{r})\,\mathrm{d}s=\liminf_{T\to\infty}\min_{x\in\overline{\Omega}} \fint^{T}_{0}\left(\gamma(x,s)-\frac{\beta(x,s)}{\overline{r}}\right)\,\mathrm{d}s=0.\]

The results follow by adapting the proof of Lemma \ref{lemma9.15} with slight adaptations. This completes the proof.
\end{proof}

\medskip
With the same technique as in Lemma \ref{lemma9.16}, we have the following lemma.
\begin{lemma}\label{lemma9.17}
Assume that $\beta,\gamma\in \mathscr{X}$, where $\mathscr{X}$ is defined in \eqref{eq1.5}.
\begin{enumerate}[{\rm (i)}]
  \item For the case $(\omega,d_I)\to(\infty,0)$, we have
   \[\lim\limits_{(\omega,d_I)\to(\infty,0)}\underline{\mathcal{R}}_0(\omega,d_I) =\liminf_{T\to\infty}\max\limits_{x\in\overline{\Omega}} \frac{\int^{T}_{0}  \beta(x,s)\,\mathrm{d}s}{\int^{T}_{0}\gamma(x,s)\,\mathrm{d}s} \]
and
 \[\lim\limits_{(\omega,d_I)\to(\infty,0)}\overline{\mathcal{R}}_0(\omega,d_I)= \limsup_{T\to\infty}\max\limits_{x\in\overline{\Omega}} \frac{\int^{T}_{0}  \beta(x,s)\,\mathrm{d}s}{\int^{T}_{0}\gamma(x,s)\,\mathrm{d}s}.\]
  \item  For the case $(\omega,\frac{\omega}{\sqrt{d_I}})\to(0,\infty)$, we have \[\lim\limits_{\left(\omega,\frac{\omega}{\sqrt{d_I}}\right)\to(0,\infty)}\underline{\mathcal{R}}_0(\omega,d_I) =\liminf_{T\to\infty}\max\limits_{x\in\overline{\Omega}} \frac{\int^{T}_{0}  \beta(x,s)\,\mathrm{d}s}{\int^{T}_{0}\gamma(x,s)\,\mathrm{d}s} \]
and
 \[\lim\limits_{\left(\omega,\frac{\omega}{\sqrt{d_I}}\right)\to(0,\infty)}\overline{\mathcal{R}}_0(\omega,d_I)= \limsup_{T\to\infty}\max\limits_{x\in\overline{\Omega}} \frac{\int^{T}_{0}  \beta(x,s)\,\mathrm{d}s}{\int^{T}_{0}\gamma(x,s)\,\mathrm{d}s}.\]
\end{enumerate}
\end{lemma}

\begin{lemma}\label{lemma9.18}
Assume that $\beta,\gamma\in \mathcal{C}\cap C^{2,1}(\overline{\Omega}\times\mathbb{R})$  and that both satisfy uniform H\"{o}lder condition \eqref{H3}. Then
  \[\lim\limits_{\left(d_I,\frac{\omega}{\sqrt{d_I}}\right)\to(0,0)}\underline{\mathcal{R}}_0(\omega,d_I)=\sup_{x(\cdot)\in C(\mathbb{R};\overline{\Omega})}\liminf_{T\to\infty} \frac{\fint^{T}_{0} \beta(x(s),s)\,\mathrm{d}s}{\fint^{T}_{0}\gamma(x(s),s)\,\mathrm{d}s}\]
and
  \[\lim\limits_{\left(d_I,\frac{\omega}{\sqrt{d_I}}\right)\to(0,0)}\overline{\mathcal{R}}_0(\omega,d_I)=\sup_{x(\cdot)\in C(\mathbb{R};\overline{\Omega})} \limsup_{T\to\infty}\frac{\fint^{T}_{0} \beta(x(s),s)\,\mathrm{d}s}{\fint^{T}_{0} \gamma(x(s),s)\,\mathrm{d}s}.\]
\end{lemma}
\begin{proof}
Denote by
  \[\underline{r}=\sup_{x(\cdot)\in C(\mathbb{R};\overline{\Omega})}\liminf_{T\to\infty} \frac{\fint^{T}_{0} \beta(x(s),s)\,\mathrm{d}s}{\fint^{T}_{0}\gamma(x(s),s)\,\mathrm{d}s} \quad{and}\quad \overline{r}=\sup_{x(\cdot)\in C(\mathbb{R};\overline{\Omega})} \limsup_{T\to\infty}\frac{\fint^{T}_{0} \beta(x(s),s)\,\mathrm{d}s}{\fint^{T}_{0} \gamma(x(s),s)\,\mathrm{d}s}.\]
Using Theorem \ref{theorem1.10} and Lemma \ref{lemma9.14}, we obtain
 \[\lim_{\left(d_I,\frac{\omega}{\sqrt{d_I}}\right)\to(0,0)} \limsup_{T\to\infty}\fint^{T}_{0}H(s;d_I,\underline{r})\,\mathrm{d}s=\inf_{x(\cdot)\in C(\mathbb{R};\overline{\Omega})}\limsup_{T\to\infty}\fint^{T}_{0} \left(\gamma(x(s),s) -\frac{\beta(x(s),s)}{\underline{r}}\right)\,\mathrm{d}s=0\]
and
 \[\lim_{\left(d_I,\frac{\omega}{\sqrt{d_I}}\right)\to(0,0)} \liminf_{T\to\infty}\fint^{T}_{0}H(s;d_I,\overline{r})\,\mathrm{d}s=\inf_{x(\cdot)\in C(\mathbb{R};\overline{\Omega})}\liminf_{T\to\infty}\fint^{T}_{0} \left(\gamma(x(s),s) -\frac{\beta(x(s),s)}{\overline{r}}\right)\,\mathrm{d}s=0.\]
Following the proof of Lemma \ref{lemma9.15}, the result can be proven with minor modifications. This finishes the proof.
\end{proof}

\medskip
For given $d_I>0$ and $r>0$, let $\mu^0\left(t;d_I,\gamma-\beta/r\right)$ be the principal eigenvalue of elliptic eigenvalue problem \eqref{eq1.4} with $d=d_I$ and $c=\gamma-\beta/r$. We define
 \[\overline{\mathcal{U}}(r;d_I):=\limsup\limits_{T\to\infty}\fint^{T}_{0}\mu^0\left(s;d_I,\gamma-\beta/r\right)\,\mathrm{d}s \quad\text{and}\quad \underline{\mathcal{U}}(r;d_I):=\liminf\limits_{T\to\infty}\fint^{T}_{0}\mu^0\left(s;d_I,\gamma-\beta/r\right)\,\mathrm{d}s.\]
For a given $r>0$, define
\[\overline{\mathcal{V}}(r;d_I):=\sup_{\hat{c}_{r}\in\mathscr{C}_{r}}\mu^\infty(d_I;\hat{c}_{r})\quad \text{ and } \quad \underline{\mathcal{V}}(r;d_I):=\inf_{\hat{c}_{r}\in\mathscr{C}_{r}}\mu^\infty(d_I;\hat{c}_{r}),\]
where $\mu^\infty(d_I,\hat{c}_{r})$ is the principal eigenvalue of \eqref{eq1.3} with $d=d_I$ and $m=\hat{c}_{r}$, the set $\mathscr{C}_r$ of functions is defined as
   \[\mathscr{C}_r:=\left\{\hat{c}\in C (\overline{\Omega})~|~\text{there exists }T_n\to\infty\text{ such that }\|\hat{c}_{T_n,r}-\hat{c}_r\|_{C(\overline{\Omega})}\to 0\right\}.\]
and $\hat{c}_{T,r}(\cdot):=\fint^{T}_{0}\left[\gamma(\cdot,s)-\beta(\cdot,s)/r\right]\,\mathrm{d}s$.

\begin{lemma}\label{lemma9.19}
 For a given $d_I>0$, the quantities $\overline{\mathcal{U}}(r;d_I)$, $\underline{\mathcal{U}}(r;d_I)$, $\overline{\mathcal{V}}(r;d_I)$ and $\underline{\mathcal{V}}(r;d_I)$ are continuous and strictly increasing in $r\in(0,\infty)$. Furthermore, there is a pair $(r_{-},r_{+})$ of constants such that
  \[\overline{\mathcal{U}}(r;d_I), \underline{\mathcal{U}}(r;d_I), \overline{\mathcal{V}}(r;d_I), \underline{\mathcal{V}}(r;d_I)<0\quad\text{ if }r<r_{-}\]
and
  \[\overline{\mathcal{U}}(r;d_I), \underline{\mathcal{U}}(r;d_I), \overline{\mathcal{V}}(r;d_I), \underline{\mathcal{V}}(r;d_I)>0\quad\text{ if }r>r_{+}.\]
Consequently, each of the equations $\overline{\mathcal{U}}(r;d_I)=0$, $\underline{\mathcal{U}}(r;d_I)=0$, $\overline{\mathcal{V}}(r;d_I)=0$ and $\underline{\mathcal{V}}(r;d_I)=0$ has a unique positive solution.
\end{lemma}
\begin{proof}
These results are consequences of the basic properties of the principal eigenvalue for elliptic problems. For completeness, we give the details here.

Recall two fundamental properties of the elliptic principal eigenvalue: (a) Translation invariance: For any constant $c$, $\mu(d,m+c)=\mu(d,m)+c$; (b) Strict monotonicity: If $m_1\leq m_2$ and $m_1\not\equiv m_2$, then $\mu(d,m_1)<\mu(d,m_2)$.

First of all, we choose
  \[r_{-}=\frac{\inf_{\Omega\times\mathbb{R}}\beta}{\sup_{\Omega\times\mathbb{R}}\gamma}\quad\text{ and }\quad r_{+}=\frac{\sup_{\Omega\times\mathbb{R}}\beta}{\inf_{\Omega\times\mathbb{R}}\gamma}.\]
Using Lemma \ref{lemma2.17}, one gets
  \[\mu^0\left(t;d_I,\gamma-\beta/r\right)<0\text{ if }r<r_{-}\quad\text{and}\quad \mu^0\left(t;d_I,\gamma-\beta/r\right)>0\text{ if }r>r_{+},\qquad\forall~t\in\mathbb{R}.\]
In fact, when $r<r_-$, for all $(x,t)\in\Omega\times\mathbb{R}$ we have $\gamma(x,t)-\frac{\beta(x,t)}{r}\leq\sup_{\Omega\times\mathbb{R}}\gamma-\frac{\inf_{\Omega\times\mathbb{R}}\beta}{r}<0$, hence $\hat{c}_{T,r}(x)=f_0^T\left(\gamma-\frac{\beta}{r}\right)ds<0$ for all $T>0$, so all functions in $\mathscr{C}_r$ are strictly negative. Similarly, when $r>r_+$, all functions in $\mathscr{C}_r$ are strictly positive. As a consequence, we see that
  \[\overline{\mathcal{U}}(r;d_I), \underline{\mathcal{U}}(r;d_I), \overline{\mathcal{V}}(r;d_I), \underline{\mathcal{V}}(r;d_I)\le0\quad\text{ if }r<r_{-}\]
and
  \[\overline{\mathcal{U}}(r;d_I), \underline{\mathcal{U}}(r;d_I), \overline{\mathcal{V}}(r;d_I), \underline{\mathcal{V}}(r;d_I)\ge0\quad\text{ if }r>r_{+}.\]

We now turn to show the continuity and monotonicity, which imply that the inequalities hold strictly in the above inequalities. Let $r_1,r_2$ be constants satisfying $r_1<r_2$. Then
\begin{align*}
  \overline{\mathcal{U}}(r_1;d_I)= & \limsup\limits_{T\to\infty}\fint^{T}_{0}\mu^0\left(s;d_I,\gamma-\frac{\beta}{r_1}\right)\,\mathrm{d}s  \\
  = & \limsup\limits_{T\to\infty}\fint^{T}_{0}\mu^0\left(s;d_I,\gamma-\frac{\beta}{r_2}+ \beta\left(\frac{1}{r_2}-\frac{1}{r_1}\right)\right)\,\mathrm{d}s \\
  \le & \limsup\limits_{T\to\infty}\fint^{T}_{0}\mu^0\left(s;d_I,\gamma-\frac{\beta}{r_2}+ \left(\frac{1}{r_2}-\frac{1}{r_1}\right)\inf_{\Omega\times\mathbb{R}}\beta\right)\,\mathrm{d}s \\
  \le & \limsup\limits_{T\to\infty}\fint^{T}_{0}\mu^0\left(s;d_I,\gamma-\frac{\beta}{r_2}\right)\,\mathrm{d}s+ \left(\frac{1}{r_2}-\frac{1}{r_1}\right)\inf_{\Omega\times\mathbb{R}}\beta \\
  = & \overline{\mathcal{U}}(r_2;d_I)+\left(\frac{1}{r_2}-\frac{1}{r_1}\right)\inf_{\Omega\times\mathbb{R}}\beta.
\end{align*}
Note that $\frac{1}{r_2}-\frac{1}{r_1}<0$ and $\beta(x,s)\geq\inf_{\Omega\times\mathbb{R}}\beta$, so the inequality
\[\beta(x,s)\cdot\left(\frac{1}{r_2}-\frac{1}{r_1}\right)\leq\inf_{\Omega\times\mathbb{R}}\beta\cdot\left(\frac{1}{r_2}-\frac{1}{r_1}\right)\]
holds, which justifies the above step by the translation invariance and strict monotonicity. Similarly, one can verify that
\[\overline{\mathcal{U}}(r_1;d_I)\ge\overline{\mathcal{U}}(r_2;d_I) +\left(\frac{1}{r_2}-\frac{1}{r_1}\right)\sup_{\Omega\times\mathbb{R}}\beta.\]
Combining the two inequalities above, for any $r_0>0$, as $r\to r_0$ we have $\frac{1}{r}-\frac{1}{r_0}\to0$. By the squeeze theorem, $\lim_{r\to r_0}\overline{\mathcal{U}}(r;d_I)=\overline{\mathcal{U}}(r_0;d_I)$, so $\overline{\mathcal{U}}(r;d_I)$ is continuous. Moreover, since $\frac{1}{r_2}-\frac{1}{r_1}<0$ and $\inf_{\Omega\times\mathbb{R}}\beta>0$, we have $\overline{\mathcal{U}}(r_1;d_I)<\overline{\mathcal{U}}(r_2;d_I)$, so $\overline{\mathcal{U}}(r;d_I)$ is strictly increasing. Similarly, one can show that $\underline{\mathcal{U}}(r;d_I)$ is continuous and strictly increasing in $r\in(0,\infty)$.

On the other hand, for any $0<\varepsilon_1<\left(\frac{1}{r_1}-\frac{1}{r_2}\right)\inf_{\Omega\times\mathbb{R}}\beta$, there is a sequence $\{T_n\}_{n\in\mathbb{N}^+}$ of $T\to\infty$ such that
\begin{equation}\label{eq9.9}
\lim_{n\to\infty}\hat{c}_{T_n,r_1}= \hat{c}_{r_1,\varepsilon_1}\in\mathscr{C}_{r_1}\quad {and}\quad\overline{\mathcal{V}}(r_1;d_I)=\sup_{\hat{c}_{r_1}\in\mathscr{C}_{r_1}}\mu^\infty(d_I;\hat{c}_{r_1}) \le \mu^\infty(d_I;\hat{c}_{r_1,\varepsilon_1})+\varepsilon_1.
\end{equation}
Then there is a subsequence $\{T_{n'}\}$ of $\{T_n\}_{n\in\mathbb{N}^+}$ such that
\[\lim_{n'\to\infty}\hat{c}_{T_{n'},r_2}= \hat{c}_{r_2,\varepsilon_1}\in\mathscr{C}_{r_2}.\]
Observing that $\hat{c}_{T_{n'},r_2}-\hat{c}_{T_{n'},r_1}\ge\left(\frac{1}{r_1}-\frac{1}{r_2}\right)\inf_{\Omega\times\mathbb{R}}\beta$, we have $\hat{c}_{r_2,\varepsilon_1}-\hat{c}_{r_1,\varepsilon_1}\ge\left(\frac{1}{r_1}-\frac{1}{r_2}\right)\inf_{\Omega\times\mathbb{R}}\beta$.
Consequently, we get
\[\mu^\infty(d_I;\hat{c}_{r_2,\varepsilon_1})-\mu^\infty(d_I;\hat{c}_{r_1,\varepsilon_1})\ge\left(\frac{1}{r_1}-\frac{1}{r_2}\right)\inf_{\Omega\times\mathbb{R}}\beta.\]
This, together with \eqref{eq9.9}, yields
\[\begin{aligned}
  \overline{\mathcal{V}}(r_1;d_I)\le & \mu^\infty(d_I;\hat{c}_{r_1,\varepsilon_1})+\varepsilon_1 \\
  \le & \mu^\infty(d_I;\hat{c}_{r_2,\varepsilon_1})+\varepsilon_1-\left(\frac{1}{r_1} -\frac{1}{r_2}\right) \inf_{\Omega\times\mathbb{R}}\beta \\
  \le & \overline{\mathcal{V}}(r_2;d_I)+\varepsilon_1-\left(\frac{1}{r_1} -\frac{1}{r_2}\right) \inf_{\Omega\times\mathbb{R}}\beta\\
  <&\overline{\mathcal{V}}(r_2;d_I).
\end{aligned}\]
This yields the monotonicity of $\overline{\mathcal{V}}(r;d_I)$. Similarly, select $\varepsilon_2=\left(\frac{1}{r_1}-\frac{1}{r_2}\right)\sup_{\Omega\times\mathbb{R}}\beta$. Then there is a sequence $\{\tilde{T}_n\}_{n\in\mathbb{N}^+}$ of $\tilde{T}\to\infty$ such that
\begin{equation}\label{eq9.10}
\lim_{n\to\infty}\hat{c}_{\tilde{T}_n,r_2}= \hat{c}_{r_2,\varepsilon_2}\in\mathscr{C}_{r_2}\quad {and}\quad\overline{\mathcal{V}}(r_2;d_I)=\sup_{\hat{c}_{r_2}\in\mathscr{C}_{r_2}}\mu^\infty(d_I;\hat{c}_{r_2}) \le \mu^\infty(d_I;\hat{c}_{r_2,\varepsilon_2})+\varepsilon_2.
\end{equation}
Then there is a subsequence $\{\tilde{T}_{n'}\}_{n'\in\mathbb{N}^+}$ of $\{\tilde{T}_n\}_{n\in\mathbb{N}^+}$ such that
\[\lim_{n'\to\infty}\hat{c}_{\tilde{T}_{n'},r_1}= \hat{c}_{r_1,\varepsilon_2}\in\mathscr{C}_{r_1}.\]
By $n'\to\infty$ in $\hat{c}_{\tilde{T}_{n'},r_2}-\hat{c}_{\tilde{T}_{n'},r_1}\le\left(\frac{1}{r_1}-\frac{1}{r_2}\right) \sup_{\Omega\times\mathbb{R}}\beta$, we have \[\hat{c}_{r_2,\varepsilon_2}-\hat{c}_{r_1,\varepsilon_2}\le \left(\frac{1}{r_1}-\frac{1}{r_2}\right)\sup_{\Omega\times\mathbb{R}}\beta.\]
Consequently, we get
\[\mu^\infty(d_I;\hat{c}_{r_2,\varepsilon_2})-\mu^\infty(d_I;\hat{c}_{r_1,\varepsilon_2}) \le\left(\frac{1}{r_1}-\frac{1}{r_2}\right)\sup_{\Omega\times\mathbb{R}}\beta.\]
This, together with \eqref{eq9.10}, yields
\[\begin{aligned}
  \overline{\mathcal{V}}(r_2;d_I)\le & \mu^\infty(d_I;\hat{c}_{r_2,\varepsilon_2})+\varepsilon_2 \\
  \le & \mu^\infty(d_I;\hat{c}_{r_1,\varepsilon_2})+\varepsilon_2+\left(\frac{1}{r_1}-\frac{1}{r_2}\right) \sup_{\Omega\times\mathbb{R}}\beta\\
  < & \mu^\infty(d_I;\hat{c}_{r_1,\varepsilon_2})+2\left(\frac{1}{r_1}-\frac{1}{r_2}\right) \sup_{\Omega\times\mathbb{R}}\beta\\
  \le & \overline{\mathcal{V}}(r_1;d_I)+2\left(\frac{1}{r_1}-\frac{1}{r_2}\right) \sup_{\Omega\times\mathbb{R}}\beta.
\end{aligned}\]
For any $r_0>0$, as $r\to r_0$ we have $\frac{1}{r}-\frac{1}{r_0}\to0$. As a result, the above inequality implies that $\limsup_{r\to r_0^+}\overline{\mathcal{V}}(r;d_I)\leq\overline{\mathcal{V}}(r_0;d_I)$, and the strict monotonicity gives $\liminf_{r\to r_0^+}\overline{\mathcal{V}}(r;d_I)\geq\overline{\mathcal{V}}(r_0;d_I)$. Hence $\overline{\mathcal{V}}(r;d_I)$ is right-continuous. The left-continuity can be proved analogously. This continuity holds for all $r\in(0,\infty)$, not just for $r>r_+$.

For $\underline{\mathcal{V}}(r;d_I)=\inf_{\hat{c}_r\in\mathscr{C}_r}\mu^\infty(d_I;\hat{c}_r)$, replacing the supremum with the infimum and reversing the inequalities in the above proof, we can similarly show that $\underline{\mathcal{V}}(r;d_I)$ is continuous and strictly increasing in $r\in(0,\infty)$.

Finally, as $r\to0^+$, $\gamma(x,t)-\frac{\beta(x,t)}{r}\to-\infty$ uniformly in $(x,t)\in\overline{\Omega}\times\mathbb{R}$, so all four functions tend to $-\infty$. As $r\to+\infty$, $\gamma(x,t)-\frac{\beta(x,t)}{r}\to\gamma(x,t)>0$ uniformly in $(x,t)\in\overline{\Omega}\times\mathbb{R}$, so all four functions tend to positive values. Combined with strict monotonicity, each of the four equations has exactly one positive solution.

This finishes the proof.
\end{proof}

\medskip
By Lemma \ref{lemma9.19}, we know that the numbers $r^{-}_{0}$, $r^{+}_{0}$, $r^{-}_{\infty}$ and $r^{+}_{\infty}$ in Theorem \ref{theorem9.4}(iii) and (iv) are well-defined.
\begin{lemma}\label{lemma9.20}
Let $d_I>0$ be given.
\begin{enumerate}[{\rm (i)}]
  \item If $\beta,\gamma\in C^{0,1}(\overline{\Omega}\times\mathbb{R})$ satisfies $\sup_{t\in\mathbb{R}}(\|\partial_t \beta(\cdot,t)\|_{\infty}+\|\partial_t \gamma(\cdot,t)\|_{\infty})<\infty$, then
      \[\lim\limits_{\omega\to0}\underline{\mathcal{R}}_0(\omega)=r^{-}_{0}\quad \text{and}\quad \lim\limits_{\omega\to0}\overline{\mathcal{R}}_0(\omega)=r^{+}_{0}.\]
  \item If $\beta,\gamma\in\mathscr{X}$, then
      \[ \lim\limits_{\omega\to\infty}\underline{\mathcal{R}}_0(\omega)=r^{-}_{\infty}\quad \text{and}\quad \lim\limits_{\omega\to\infty}\overline{\mathcal{R}}_0(\omega)=r^{+}_{\infty}.\]
\end{enumerate}
\end{lemma}
\begin{proof}
Using Theorem \ref{theorem1.6}, we have
 \[\lim_{\omega\to 0}\limsup\limits_{T\to\infty}\fint^{T}_{0}H(s;\omega,\gamma-\beta/r^{-}_{0})\,\mathrm{d}s =\overline{\mathcal{U}}(r^{-}_{0})=0,\]
 \[\lim_{\omega\to 0}\liminf\limits_{T\to\infty}\fint^{T}_{0}H(s;\omega,\gamma-\beta/r^{+}_{0})\,\mathrm{d}s =\underline{\mathcal{U}}(r^{+}_{0})=0,\]
 \[\lim_{\omega\to \infty}\limsup\limits_{T\to\infty}\fint^{T}_{0}H(s;\omega,\gamma-\beta/r^{-}_{\infty}) \,\mathrm{d}s=\sup_{\hat{c}\in\mathscr{C}_{r^{-}_{\infty}}}\mu^\infty(d_I;\hat{c})=0\]
and
 \[\lim_{\omega\to \infty}\liminf\limits_{T\to\infty}\fint^{T}_{0}H(s;\omega,\gamma-\beta/r^{+}_{\infty}) \,\mathrm{d}s=\inf_{\hat{c}\in\mathscr{C}_{r^{+}_{\infty}}}\mu^\infty(d_I;\hat{c})=0.\]
Similarly to Lemma \ref{lemma9.15}, the result can be established with only slight changes to the proof. This completes the proof.
\end{proof}

\medskip
Theorem \ref{theorem9.4} follows from Lemmas \ref{lemma9.15}-\ref{lemma9.18} and \ref{lemma9.20}.


\begin{thebibliography}{99}\setlength{\itemsep}{-1mm}\linespread{1}\selectfont

\bibitem{Allen2007Asymptotic} L.~J.~S. Allen, B.~M. Bolker, Y.~Lou, and A.~L. Nevai, Asymptotic profiles of the steady states for an {$SIS$} epidemic patch model, {\it SIAM J. Appl. Math.}, {\bf 67} (2007), 1283--1309.

\bibitem{Allen2008Asymptotic} L.~J.~S. Allen, B.~M. Bolker, Y.~Lou, and A.~L. Nevai, Asymptotic profiles of the steady states for an {SIS} epidemic reaction-diffusion model, {\it Discrete Contin. Dyn. Syst.}, {\bf 21} (2008), 1--20.

\bibitem{Bai2020Asymptotic} X.~Bai and X.~He, Asymptotic behavior of the principal eigenvalue for cooperative periodic-parabolic systems and applications, J. Differential Equations, 269 (2020), 9868--9903.

\bibitem{Bai2023Dynamics} X.~Bai, X.~He, and W.-M. Ni, Dynamics of a periodic-parabolic {L}otka-{V}olterra competition-diffusion system in heterogeneous environments, {\it J. Eur. Math. Soc.}, {\bf 25} (2023), 4583--4637.

\bibitem{Berestycki2025Generalized} H.~Berestycki, G. Nadin, and L. Rossi, Generalized principal eigenvalues for parabolic operators in bounded domains, arxiv:2502.16522, (2025).

\bibitem{Berestycki1994The} H.~Berestycki, L.~Nirenberg, and S.~R.~S. Varadhan, The principal eigenvalue and maximum principle for second-order elliptic operators in general domains, {\it Comm. Pure Appl. Math.}, {\bf 47} (1994), 47--92.

\bibitem{Cantrell2021On} R.~S. Cantrell and K.-Y. Lam, On the evolution of slow dispersal in multispecies communities, {\it SIAM J. Math. Anal.}, {\bf 53} (2021), 4933--4964.

\bibitem{Cantrell2021Ideal} R.~S. Cantrell, C. Cosner, and K.-Y. Lam, Ideal free dispersal under general spatial heterogeneity and time periodicity, {\it SIAM J. Appl. Math.}, {\bf 81} (2021), 789--813.

\bibitem{Castellano2024Multiplicity} K.~Castellano and R.~B. Salako, Multiplicity of endemic equilibria for a diffusive {SIS} epidemic model with mass-action, {\it SIAM J. Appl. Math.}, {\bf 84} (2024), 732--755.

\bibitem{Chen2012Effects} X.~Chen and Y.~Lou, Effects of diffusion and advection on the  smallest eigenvalue of an elliptic operator and their applications, {\it Indiana Univ. Math. J.}, {\bf 61} (2012), 45--80.

\bibitem{Chow1993Floquet} S.-N. Chow, K. Lu, and J. Mallet-Paret,~Floquet theory for parabolic differential equations, {\it J. Differential Equations}, {\bf 109} (1993), 147-200.

\bibitem{Chow1995Floquet} S.-N. Chow, K.~Lu, and J.~Mallet-Paret, Floquet bundles for scalar parabolic equations, {\it Arch. Rational Mech. Anal.}, {\bf 129} (1995), 245--304.

\bibitem{Cui2017Dynamics} R.~Cui, K.-Y. Lam, and Y.~Lou, Dynamics and asymptotic profiles of steady states of an epidemic model in advective environments, {\it J. Differential Equations}, {\bf 263} (2017), 2343--2373.

\bibitem{Cui2021Concentration} R.~Cui, H.~Li, R.~Peng, and M.~Zhou, Concentration behavior of endemic equilibrium for a reaction-diffusion-advection {SIS} epidemic model with mass action infection mechanism, {\it Calc. Var. Partial Differential Equations}, {\bf 60} (2021), 184 (pp. 1--38).

\bibitem{Cui2016A} R.~Cui and Y.~Lou, A spatial {SIS} model in advective heterogeneous  environments, {\it J. Differential Equations}, {\bf 261} (2016), 3305--3343.

\bibitem{Devinatz1974The} A.~Devinatz, R.~Ellis, and A.~Friedman, The asymptotic behavior of the first real eigenvalue of second order elliptic operators with a small  parameter in the highest derivatives. {II}, {\it Indiana Univ. Math. J.}, {\bf 23} (1973/74), 991--1011.

\bibitem{Evans1989A} L.C. Evans and P.E. Souganidis, A PDE approach to geometric optics for certain semilinear parabolic equations, {\it Indiana Univ. Math. J.}, {\bf 38} (1989), 141--172.

\bibitem{Fabes1986A} E.~B.~Fabes, N.~Garofalo, and S.~Salsa, A backward Harnack inequality and Fatou theorem for nonnegative solutions of parabolic equations, {\it Illinois J. Math.}, {\bf 30} (1986), no.~4, 536--565.

\bibitem{Fabes1997Behavior} E.~B.~Fabes and M.~V.~Safonov, Behavior near the boundary of positive solutions of second order parabolic equations, {\it J. Fourier Anal. Appl.}, {\bf 3} (1997), no.~Special Issue, 871--882.

\bibitem{Fabes1999Behavior} E.~Fabes, M.~Safonov, and Y.~Yuan, Behavior near the boundary of positive solutions of second order parabolic equations. II, {\it Trans. Amer. Math. Soc.}, {\bf 351} (1999), no.~12, 4947--4961.

\bibitem{Fink2006Almost} A.~M. Fink, Almost Periodic Differential Equations. Lecture Notes in Mathematics 377. Springer, Berlin, 1974.

\bibitem{Fleming1986PDE} W. H. Fleming and P. E. Souganidis, PDE-viscosity solution approach to some problems of large deviations, {\it Ann. Scuola Norm. Sup. Pisa Cl. Sci.}, (4) {\bf 13} (1986), 171-192.

\bibitem{Gao2024A} D.~Gao, C.~Lei, R.~Peng, and B.~Zhang, A diffusive {SIS} epidemic model with saturated incidence function in a heterogeneous environment, {\it Nonlinearity}, {\bf 37} (2024), 025002 (pp. 1--40).

\bibitem{Gilbarg2015Elliptic} D.~Gilbarg and N. S. Trudinger, Elliptic Partial Differential Equations of Second Order, Springer, 2015.

\bibitem{Hamel2011Rearrangement} F.~Hamel, N.~Nadirashvili, and E.~Russ, Rearrangement inequalities and applications to isoperimetric problems for eigenvalues, {\it Ann. of Math.}  , {\bf 174}(2) (2011), 647--755.

\bibitem{Hess1991Periodic} P.~Hess, Periodic-parabolic boundary value problems and positivity, vol.~247 of Pitman Research Notes in Mathematics Series, Longman Scientific \& Technical, Harlow; copublished in the United States with John Wiley \& Sons, Inc., New York, 1991.


\bibitem{Huska2006Exponential} J.~H\'{u}ska, Exponential separation and principal {F}loquet bundles for linear parabolic equations on general bounded domains: the divergence case, {\it Indiana Univ. Math. J.}, {\bf 55} (2006), 1015--1043.

\bibitem{Huska2006Harnack} J.~H\'{u}ska, Harnack inequality and exponential separation for oblique derivative problems on {L}ipschitz domains, {\it J. Differential Equations}, {\bf 226} (2006), 541--557.

\bibitem{Huska2008Exponential} J.~H\'{u}ska, Exponential separation and principal {F}loquet bundles for linear parabolic equations on general bounded domains: nondivergence case, {\it Trans. Amer. Math. Soc.}, {\bf 360} (2008), 4639--4679.

\bibitem{Huska2004The} J.~H\'{u}ska and P.~Pol\'{a}\v{c}ik, The principal {F}loquet bundle and exponential separation for linear parabolic equations, {\it J. Dynam. Differential Equations}, {\bf 16} (2004), 347--375.

\bibitem{Huska2007Harnack} J.~H\'{u}ska, P.~Pol\'{a}\v{c}ik, and M.~V. Safonov, Harnack inequalities, exponential separation, and perturbations of principal {F}loquet bundles for linear parabolic equations, {\it Ann. Inst. H. Poincar\'{e} C Anal. Non Lin\'{e}aire}, {\bf 24} (2007), 711--739.

\bibitem{Hutson2001The} V. Hutson, K. Mischaikow, and P. Pol\'{a}\v{c}ik, The evolution of dispersal rates in a heterogeneous time-periodic environment, {\it J. Math. Biol.}, {\bf 43} (2001), no. 6, 501--533.

\bibitem{Jiang2018A} D.~Jiang, Z.-C. Wang, and L.~Zhang, A reaction-diffusion-advection {SIS} epidemic model in a spatially-temporally heterogeneous environment, {\it Discrete Contin. Dyn. Syst. Ser. B}, {\bf 23} (2018), 4557--4578.

\bibitem{Koike1987An} S. Koike, An asymptotic formula for solutions of Hamilton-Jacobi-Bellman equations, {\it Nonlinear Anal.}, {\bf 11} (1987), 429-436.

\bibitem{Kuto2017Concentration} K.~Kuto, H.~Matsuzawa, and R.~Peng, Concentration profile of endemic equilibrium of a reaction-diffusion-advection {SIS} epidemic model, {\it Calc. Var. Partial Differential Equations}, {\bf 56} (2017), 112 (pp. 1--28).

\bibitem{Lam2022Introduction} K.-Y. Lam and Y.~Lou, Introduction to reaction-diffusion equations:  Theory and applications to spatial ecology and evolutionary biology. Lecture Notes on Mathematical Modelling in the Life Sciences. Springer Nature, Cham, 2022.

\bibitem{Lam2024The} K.-Y. Lam{} and Y.~Lou, The principal {F}loquet bundle and the dynamics of fast diffusing communities, {\it Trans. Amer. Math. Soc.}, {\bf 377} (2024), 1--29.

\bibitem{Lam2023A} K.-Y. La{}m, Y.~Lou, and B.~Perthame, A Hamilton-Jacobi approach to evolution of dispersal, {\it Comm. Partial Differential Equations}, {\bf 48} (2023), 86--118.

\bibitem{Li2019Dynamics} H.~Li and R.~Peng, Dynamics and asymptotic profiles of endemic equilibrium for {SIS} epidemic patch models, {\it J. Math. Biol.}, {\bf 79} (2019), 1279--1317.

\bibitem{Li2025Asymptotic} Y.~Li, S.~Ruan, and Z.-C. Wang, Asymptotic behavior of endemic  equilibria for a {SIS} epidemic model in convective environments, {\it J. Differential Equations}, {\bf 426} (2025), 606--659.

\bibitem{Li2026Effects} Y.~Li, S.~Ruan and Z.-C.~Wang, Effects of Advection and Diffusion on the Principal Eigenvalue of a Periodic-Parabolic Operator, {\it Indiana Univ. Math. J.},  {\bf 75} (2026), no. 2, 429--486.

\bibitem{Li2026Optimization} Y.~Li, S.~Ruan and Z.-C.~Wang, Optimization of the principal eigenvalue of a cooperative elliptic system with drifts. {\it Sci. China Math.}, {\bf 69} (2026), online.

\bibitem{Lieberman1984The} G.~M.~Lieberman, The nonlinear oblique derivative problem for quasilinear elliptic equations, {\it Nonlinear Anal.}, {\bf 8} (1984), no.~5, 499--523.

\bibitem{Lieberman1987Local} G.~M.~Lieberman, Local estimates for subsolutions and supersolutions of oblique derivative problems for general second order elliptic equations, {\it Trans. Amer. Math. Soc.}, {\bf 302} (1987), no.~1, 41--67.

\bibitem{Liu2025Asymptotic} S.~Liu, Asymptotic and global analysis of principal eigenvalues for linear time-periodic parabolic systems, {\it J. Math. Pures Appl.}, {\bf 203(9)} (2025), 103781, (pp. 1--44).

\bibitem{Liu2022Classifying} S.~Liu and Y.~Lou, Classifying the level set of principal eigenvalue for time-periodic parabolic operators and applications, {\it J. Funct. Anal.}, {\bf 282} (2022), 109338, (pp. 1--43).

\bibitem{Liu2019Monotonicity} S. Liu, Y. Lou, R. Peng, and M. Zhou, Monotonicity of the principal eigenvalue for a linear time-periodic parabolic operator, {\it Proc. Amer. Math. Soc.}, {\bf 147} (2019), no. 12, 5291--5302.

\bibitem{Liu2021AsymptoticsI} S.~Liu, Y.~Lou, R.~Peng, and M.~Zhou, Asymptotics of the principal eigenvalue for a linear time-periodic parabolic operator {I}: large advection, {\it SIAM J. Math. Anal.}, {\bf 53} (2021), 5243--5277.

\bibitem{Liu2021AsymptoticsII} S.~Liu, Y.~Lou, R.~Peng, and M.~Zhou, Asymptotics of the principal eigenvalue for a linear time-periodic parabolic operator {II}: {S}mall diffusion, {\it Trans. Amer. Math. Soc.}, {\bf 374} (2021), 4895--4930.

\bibitem{Lou2006Evolution} Y.~Lou and T.~Nagylaki, Evolution of a semilinear parabolic system for migration and selection without dominance, {\it J. Differential Equations}, {\bf 225} (2006), 624--665.

\bibitem{Mierczynski1995Globally} J.~Mierczy\'{n}ski, Globally positive solutions of linear parabolic {PDE}s of second order with {R}obin boundary conditions, {\it J. Math. Anal. Appl.}, {\bf 209} (1997), 47--59.

\bibitem{Mierczynski2008Spectral} J.~Mierczy\'{n}ski and W.~Shen, Spectral theory for random and nonautonomous parabolic equations and applications, vol.~139 of Chapman \& Hall/CRC Monographs and Surveys in Pure and Applied Mathematics, CRC Press, Boca Raton, FL, 2008.

\bibitem{Mierczynski2008Time} J.~Mierczy\'{n}ski and W.~Shen, Time averaging for nonautonomous/random linear parabolic equations, {\it Discrete Contin. Dyn. Syst. Ser. B}, {\bf 9} (2008), 661--699.

\bibitem{Mierczynski2011Persistence} J.~Mierczy\'{n}ski and W.~Shen{}, Persistence in forward nonautonomous competitive systems of parabolic equations, {\it J. Dynam. Differential Equations}, {\bf 23} (2011), 551--571.

\bibitem{Mierczynski2013Spectral} J.~Mierczy\'{n}ski and W.~Shen, Spectral theory for forward nonautonomous parabolic equations and applications, in Infinite dimensional dynamical systems, vol.~64 of Fields Inst. Commun., Springer, New York, 2013, 57--99.

\bibitem{Moser1964A} J.~Moser, A Harnack inequality for parabolic differential equations, {\it Comm. Pure Appl. Math.}, {\bf 17} (1964), 101--134.

\bibitem{Moser1971On} J. Moser, On a pointwise estimate for parabolic differential equations, {\it Comm. Pure Appl. Math.}, {\bf 24} (1971), 727-740.

\bibitem{Peng2009Asymptotic} R.~Peng, Asymptotic profiles of the positive steady state for an  {SIS} epidemic reaction-diffusion model. {I}, {\it J. Differential Equations}, {\bf 247} (2009), 1096--1119.

\bibitem{Peng2025Spatial} R.~Peng, R.~Salako, and Y.~Wu, Spatial profiles of a reaction-diffusion epidemic model with nonlinear incidence mechanism and constant total population, {\it SIAM J. Math. Anal.}, {\bf 57} (2025), 1586--1620.

\bibitem{Peng2025Spatial2} R.~Peng{}, R.~B. Salako, and Y.~Wu, Spatial profiles of a reaction-diffusion epidemic model with nonlinear incidence mechanism and varying total population, {\it Nonlinearity}, {\bf 38} (2025), 045006 (pp. 1--35).

\bibitem{Peng2012A} R.~Peng and X.-Q. Zhao, A reaction-diffusion {SIS} epidemic model in a time-periodic environment, {\it Nonlinearity}, {\bf 25} (2012), 1451--1471.

\bibitem{Peng2015Effects} R.~Peng and X.-Q. Zhao, Effects of diffusion and advection on the principal eigenvalue of a periodic-parabolic problem with applications, {\it Calc. Var. Partial Differential Equations}, {\bf 54} (2015), 1611--1642.

\bibitem{Polacik2005On} P. Pol\'{a}\v{c}ik, On uniqueness of positive entire solutions and other properties of linear parabolic equations, {\it Discrete Contin. Dyn. Syst.}, {\bf 12}(1) (2005), 13-26.

\bibitem{Polacik1993Exponential} P.~Pol\'{a}\v{c}ik and I.~Tere\v{s}\v{c}\'{a}k, Exponential separation and invariant bundles for maps in ordered {B}anach spaces with applications to parabolic equations, {\it J. Dynam. Differential Equations}, {\bf 5} (1993), 279--303.

\bibitem{Shen2007Almost} W. Shen, Y. Yi, Almost automorphic and almost periodic dynamics in skew-product semiflows, {\it Mem. Am. Math. Soc.}, {\bf 136}(647) (1998), 1--93.

\bibitem{Shen2007Spectral} W.~Shen and G.~T. Vickers, Spectral theory for general nonautonomous/random dispersal evolution operators, {\it J. Differential Equations}, {\bf 235} (2007), 262--297.

\bibitem{Strauss2007Partial} W. A. Strauss, Partial differential equations: An introduction. John Wiley \& Sons, 2007.

\bibitem{Wang2015A} B.-G.~Wang, W.-T. Li, and Z.-C. Wang, A reaction-diffusion {SIS} epidemic model in an almost periodic environment, {\it Z. Angew. Math. Phys.}, {\bf 66} (2015), 3085--3108.

\bibitem{Wang2012Basic} W.~Wang and X.-Q. Zhao, Basic reproduction numbers for reaction-diffusion epidemic models, {\it SIAM J. Appl. Dyn. Syst.}, {\bf 11} (2012), 1652--1673.

\bibitem{Wu2016Asymptotic} Y.~Wu and X.~Zou, Asymptotic profiles of steady states for a diffusive {SIS} epidemic model with mass action infection mechanism, {\it J. Differential Equations}, {\bf 261} (2016), 4424--4447.

\bibitem{Zhang2021Asymptotic} L.~Zhang and X.-Q. Zhao, Asymptotic behavior of the basic reproduction ratio for periodic reaction-diffusion systems, {\it SIAM J. Math. Anal.}, {\bf 53} (2021), 6873--6909.

\end{thebibliography}
\end{document}